\documentclass[12pt,dvipsnames]{amsart}
\usepackage{color}
\usepackage[colorlinks=true, pdfstartview=FitV,,, linkcolor=blue, citecolor=blue, urlcolor=blue]{hyperref}
\usepackage{geometry}
\usepackage{amsmath, amsthm, amssymb, shuffle, tikz}
\usepackage{verbatim}
\usepackage{enumitem}
\usepackage[all]{xy}
\usetikzlibrary{arrows}
\usepackage{xcolor}
\usepackage[colorinlistoftodos]{todonotes}

\usepackage{thm-restate}

\definecolor{applegreen}{rgb}{0.55,0.71,0.0}

\definecolor{darkblue}{rgb}{0.0, 0.0, 0.7}

\definecolor{cinnabar}{rgb}{0.89, 0.26, 0.2}
\definecolor{blue-violet}{rgb}{0.54, 0.17, 0.89}

\emergencystretch=\maxdimen
\newtheorem{Theorem}{Theorem}[section]
\newtheorem{Proposition}[Theorem]{Proposition}
\newtheorem{Corollary}[Theorem]{Corollary}
\newtheorem{Lemma}[Theorem]{Lemma}

\theoremstyle{definition}
\newtheorem{Example}[Theorem]{Example}
\newtheorem{Remark}[Theorem]{Remark}
\newtheorem{Definition}[Theorem]{Definition}
\newtheorem{Problem}[Theorem]{Problem}

\def\SSym{\operatorname{\mathsf{\mathfrak{S}Sym}}}
\def\YSym{\operatorname{\mathsf{\mathcal{Y}Sym}}}
\def\PSym{\operatorname{\mathsf{PSym}}}
    \def\TSym{\operatorname{\mathsf{TSym}}}
\def\STSym{\operatorname{\mathsf{STSym}}}

\def\PQSym{\operatorname{\mathsf{PQSym}}}
\def\NCQSym{\operatorname{\mathsf{NCQSym}}}

\def\WQSym{\operatorname{\mathsf{WQSym}}}

\newcommand{\cH}{\mathcal{H}}

\newcommand{\std}{\operatorname{\mathtt{std}}}
\newcommand{\pack}{\operatorname{\mathtt{pack}}}
\newcommand{\perm}{\text{perm}}
\newcommand{\tree}{\text{tree}}
\newcommand{\ptree}{\text{ptree}}

\newcommand{\GD}{\operatorname{\mathtt{GDes}}}
\newcommand{\Des}{\operatorname{\mathtt{Des}}}
\newcommand{\iInv}{\operatorname{\mathtt{iInv}}}
\newcommand{\Inv}{\operatorname{\mathtt{Inv}}}
\newcommand{\PBT}{\text{PBT}}
\newcommand{\Perm}{\text{Perm}}
\newcommand{\PF}{\text{PF}}
\newcommand{\PT}{\text{PT}}
\newcommand{\PW}{\text{PW}}
\newcommand{\LPT}{\text{LPT}}
\newcommand{\GSP}{\text{GSP}}
\newcommand{\gsp}{\text{gsp}}
\newcommand{\ideg}{\text{ideg}}
\newcommand{\cmpt}{\operatorname{\mathtt{cmpt}}}
\newcommand{\cmpts}{\operatorname{\mathtt{cmpts}}}
\newcommand{\setpar}{\operatorname{\mathtt{setpar}}}
\newcommand{\delrpt}{\operatorname{\mathtt{init}}}
\newcommand{\pw}{\operatorname{\mathtt{pw}}}
\newcommand{\allow}{\operatorname{\mathtt{allow}}}
\newcommand{\Span}{\operatorname{\mathtt{Span}}}

\begin{document}

\title[Hopf algebras of parking functions and decorated planar trees]{Hopf algebras of parking functions and decorated planar trees}

\author[Bergeron]{Nantel Bergeron}
\address[N. Bergeron]
{Department of Mathematics and Statistics \\ York University\\
	Toronto, Ontario, Canada}
\email{bergeron@mathstat.yorku.ca}
\urladdr{\url{http://www.math.yorku.ca/bergeron/}}
\thanks{Bergeron is partially supported by NSERC and the York Research Chair in applied algebra. }

\author[Gonz\'alez D'Le\'on]{Rafael S. Gonz\'alez D'Le\'on }
\address[R.\ S.\ Gonz\'alez D'Le\'on]{Departamento de Matem\'aticas\\Pontificia Universidad Javeriana\\Bogot\'a\\Colombia} 
\email{gonzalezdrafael@javeriana.edu.co}
\urladdr{\url{http://dleon.combinatoria.co}}

\author[Li]{Shu Xiao Li}
\address[S.\ X.\ Li]
{School of Mathematical Sciences \\Dalian University of Technology\\
	Dalian, Liaoning, China}
\email{lishuxiao@dlut.edu.cn}

\author[Pang]{C. Y. Amy Pang}
\address[C. Y. A.\ Pang]{Department of Mathematics\\Hong Kong Baptist University\\Hong Kong} 
\email{amypang@hkbu.edu.hk}
\urladdr{\url{http://amypang.github.io}}

\thanks{Pang and Li are partially supported by the Hong Kong Research Grants Council grant ECS 22300017.}

\author[Vargas]{Yannic Vargas}
\address[Y.\ Vargas]{Departamento de Matem\'aticas\\Instituto Venezolano de Investigaci\'on Cient\'ifica\\Caracas\\Venezuela} 

\address[Y.\ Vargas]{Institut fur Mathematik\\University of Potsdam\\ Campus Golm, Haus 9 Karl-LiebknechtStr. \\24-25 D-14476 Potsdam\\ Germany}
\email{yannicmath@gmail.com}

\thanks{This work started at the Algebraic Combinatorics workshop at the Fields institute.}
\maketitle

\begin{abstract}
	We construct three new combinatorial Hopf algebras based on the Loday-Ronco operations on planar binary trees. The first and second algebras are defined on planar trees and labeled planar trees extending the Loday-Ronco and Malvenuto-Reutenauer Hopf algebras respectively. We show that the latter is bidendriform which implies that is also free, cofree, and self-dual. The third algebra involves a new visualization of parking functions as decorated binary trees; it is also bidendriform, free, cofree, and self-dual, and therefore abstractly isomorphic to the algebra $\PQSym$ of Novelli and Thibon. 
	
	We define partial orders on the objects indexing each of these three Hopf algebras, one of which, when restricting to $(m+1)$-ary trees, coarsens the $m$-Tamari order of Bergeron and Pr\'eville-Ratelle. We show that multiplication of dual fundamental basis elements are given by intervals in each of these orders.
	
	Finally, we use an axiomatized version of the techniques of Aguiar and Sottile on the Malvenuto-Reutenauer Hopf algebra to define a monomial basis on each of our Hopf algebras, and to show that comultiplication is cofree on the monomial elements. This in particular, implies the cofreeness of the Hopf algebra on planar trees. We also find explicit positive formulas for the multiplication on monomial basis and a cancellation-free and grouping-free formula for the antipode of monomial elements.
\end{abstract}

\tableofcontents

\setcounter{tocdepth}{3}

\section{Introduction}

The algebraic structure of a combinatorial Hopf algebra $(\cH,\cdot,\Delta)$  can often be understood in terms of a poset on the underlying family of combinatorial objects indexing a basis. For instance, in the Malvenuto-Reutenauer Hopf algebra of permutations $\SSym$ \cite{MR95}, the product of fundamental basis elements, and of dual-fundamental basis elements, are sums over an interval in the right and left weak orders, respectively. 
Another example is the multiplication in the Loday-Ronco Hopf algebra of binary trees $\YSym^*$ \cite{LR98} which is the sum of an interval in the Tamari order. 
Furthermore, the relationship between these two particular algebras is encoded through poset morphisms going both ways, from permutations to binary trees and vice versa. In \cite{AS05} these morphisms are shown to form a Galois connection between the two posets. These two morphisms also reveal a classical connection between these two families of combinatorial objects: the set of permutations of  $[n]:=\{1,2,\dots, n\}$ is in bijection with the set of increasing binary trees with $n$ internal nodes, which are planar rooted binary trees whose set of internal nodes are decorated with a bijection to $[n]$ in a way that any path away from the root has an increasing sequence of labels. A consequence of this interpretation of permutations as labeled planar binary trees is that forgetting the labels induces a map between permutations and planar binary trees. In Figure \ref{fig:permutations_of_3} we illustrate the $6$ permutations of $[3]$ in one-line notation besides their corresponding increasing binary tree.

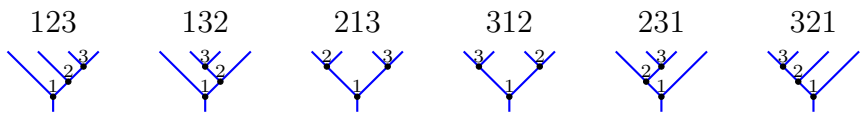
\begin{figure}[h]
    \begin{center}
		\begin{tikzpicture}[scale=0.2,baseline=0pt]
		\draw[blue, thick] (0,-1) -- (0,0);
		\draw[blue, thick] (0,0) -- (3,3);
		\draw[blue, thick] (0,0) -- (-3,3);
		\draw[blue, thick] (1,1) -- (-1,3);
		\draw[blue, thick] (2,2) -- (1,3);
		\filldraw[black] (0,0) circle (5pt)  {};
		\filldraw[black] (1,1) circle (5pt)  {};
		\filldraw[black] (2,2) circle (5pt)  {};
		\node at (0,0.7)[font=\fontsize{7pt}{0}]{$1$};
		\node at (1,1.7)[font=\fontsize{7pt}{0}]{$2$};
		\node at (2,2.7)[font=\fontsize{7pt}{0}]{$3$};
		
		\draw[blue, thick] (10,-1) -- (10,0);
		\draw[blue, thick] (10,0) -- (13,3);
		\draw[blue, thick] (10,0) -- (7,3);
		\draw[blue, thick] (11,1) -- (9,3);
		\draw[blue, thick] (10,2) -- (11,3);
		\filldraw[black] (10,0) circle (5pt)  {};
		\filldraw[black] (11,1) circle (5pt)  {};
		\filldraw[black] (10,2) circle (5pt)  {};
		\node at (10,0.7)[font=\fontsize{7pt}{0}]{$1$};
		\node at (11,1.7)[font=\fontsize{7pt}{0}]{$2$};
		\node at (10,2.7)[font=\fontsize{7pt}{0}]{$3$};
		
		\draw[blue, thick] (20,-1) -- (20,0);
		\draw[blue, thick] (20,0) -- (23,3);
		\draw[blue, thick] (20,0) -- (17,3);
		\draw[blue, thick] (22,2) -- (21,3);
		\draw[blue, thick] (18,2) -- (19,3);
		\filldraw[black] (20,0) circle (5pt)  {};
		\filldraw[black] (22,2) circle (5pt)  {};
		\filldraw[black] (18,2) circle (5pt)  {};
		\node at (20,0.7)[font=\fontsize{7pt}{0}]{$1$};
		\node at (18,2.7)[font=\fontsize{7pt}{0}]{$2$};
		\node at (22,2.7)[font=\fontsize{7pt}{0}]{$3$};
		
		\draw[blue, thick] (30,-1) -- (30,0);
		\draw[blue, thick] (30,0) -- (33,3);
		\draw[blue, thick] (30,0) -- (27,3);
		\draw[blue, thick] (32,2) -- (31,3);
		\draw[blue, thick] (28,2) -- (29,3);
		\filldraw[black] (30,0) circle (5pt)  {};
		\filldraw[black] (32,2) circle (5pt)  {};
		\filldraw[black] (28,2) circle (5pt)  {};
		\node at (30,0.7)[font=\fontsize{7pt}{0}]{$1$};
		\node at (28,2.7)[font=\fontsize{7pt}{0}]{$3$};
		\node at (32,2.7)[font=\fontsize{7pt}{0}]{$2$};
		
		\draw[blue, thick] (40,-1) -- (40,0);
		\draw[blue, thick] (40,0) -- (43,3);
		\draw[blue, thick] (40,0) -- (37,3);
		\draw[blue, thick] (39,1) -- (41,3);
		\draw[blue, thick] (40,2) -- (39,3);
		\filldraw[black] (40,0) circle (5pt)  {};
		\filldraw[black] (39,1) circle (5pt)  {};
		\filldraw[black] (40,2) circle (5pt)  {};
		\node at (40,0.7)[font=\fontsize{7pt}{0}]{$1$};
		\node at (39,1.7)[font=\fontsize{7pt}{0}]{$2$};
		\node at (40,2.7)[font=\fontsize{7pt}{0}]{$3$};
		
		\draw[blue, thick] (50,-1) -- (50,0);
		\draw[blue, thick] (50,0) -- (47,3);
		\draw[blue, thick] (50,0) -- (53,3);
		\draw[blue, thick] (49,1) -- (51,3);
		\draw[blue, thick] (48,2) -- (49,3);
		\filldraw[black] (50,0) circle (5pt)  {};
		\filldraw[black] (49,1) circle (5pt)  {};
		\filldraw[black] (48,2) circle (5pt)  {};
		\node at (50,0.7)[font=\fontsize{7pt}{0}]{$1$};
		\node at (49,1.7)[font=\fontsize{7pt}{0}]{$2$};
		\node at (48,2.7)[font=\fontsize{7pt}{0}]{$3$};
		
			\node at (0,5) {$123$};
	\node at (10,5) {$132$};
	\node at (20,5) {$213$};
	\node at (30,5) {$312$};
	\node at (40,5) {$231$};
	\node at (50,5) {$321$};
		\end{tikzpicture}
	\end{center}
    \caption{Permutations of $[3]$.}
    \label{fig:permutations_of_3}
\end{figure}
\subsection{New combinatorial Hopf algebras on trees}
Inspired on the relationship between $\YSym$ and $\SSym$, we study three new combinatorial Hopf algebras supported on the set $\PT$ of planar rooted trees and $\LPT$ of labeled planar rooted trees. The definition of the multiplication and comultiplication on the fundamental basis of these algebras are generalizations of the classical operations of ``lightening splitting'' and ``grafting" defined on binary trees (see Figure \ref{fig:splitting_by_lightening}).

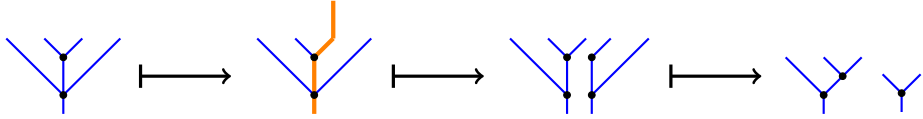
\begin{figure}[h]
    \centering
    \begin{tikzpicture}
	\begin{scope}[scale=0.25,baseline=0pt, shift={(-13.2,0)}]
		\draw[blue, thick] (0,-1) -- (0,0);
		\draw[blue, thick] (0,0) -- (3,3);
		\draw[blue, thick] (0,0) -- (-3,3);
		\draw[blue, thick] (0,0) -- (0,2);
		\draw[blue, thick] (0,2) -- (-1,3);
		\draw[blue, thick] (0,2) -- (1,3);
		\filldraw[black] (0,0) circle (5pt)  {};
		\filldraw[black] (0,2) circle (5pt)  {};
		\end{scope}

		\begin{scope}[scale=0.25,baseline=0pt]
		\draw[orange, ultra thick] (0,-1) -- (0,2);
		\draw[blue, thick] (0,0) -- (3,3);
		\draw[blue, thick] (0,0) -- (-3,3);
		\draw[orange, ultra thick] (1,3) -- (1,5);
		\draw[blue, thick] (0,2) -- (-1,3);
		\draw[orange, ultra thick] (0,2) -- (1,3);
		\filldraw[black] (0,0) circle (5pt)  {};
		\filldraw[black] (0,2) circle (5pt)  {};
		\end{scope}
		\begin{scope}[scale=0.25,baseline=0pt, shift={(13.3,0)}]
		\draw[blue, thick] (0,-1) -- (0,0);
		\draw[blue, thick] (0,0) -- (-3,3);
		\draw[blue, thick] (0,0) -- (0,2);
		\draw[blue, thick] (0,2) -- (-1,3);
		\draw[blue, thick] (0,2) -- (1,3);
		\filldraw[black] (0,0) circle (5pt)  {};
		\filldraw[black] (0,2) circle (5pt)  {};
		\end{scope}
		\begin{scope}[scale=0.25,baseline=0pt, shift={(14.6,0)}]
		\draw[blue, thick] (0,-1) -- (0,0);
		\draw[blue, thick] (0,0) -- (3,3);
		\draw[blue, thick] (0,0) -- (0,2);
		\draw[blue, thick] (0,2) -- (1,3);
		\filldraw[black] (0,0) circle (5pt)  {};
		\filldraw[black] (0,2) circle (5pt)  {};
		\end{scope}
		
		\begin{scope}[scale=0.25,baseline=0pt, shift={(26.8,0)}]
		\draw[blue, thick] (0,-1) -- (0,0);
		\draw[blue, thick] (0,0) -- (-2,2);
		\draw[blue, thick] (0,0) -- (2,2);
		\draw[blue, thick] (1,1) -- (0,2);
		\filldraw[black] (0,0) circle (5pt)  {};
		\filldraw[black] (1,1) circle (5pt)  {};
		\end{scope}
		\begin{scope}[scale=0.25,baseline=0pt, shift={(30.9,0.1)}]
		\draw[blue, thick] (0,-1) -- (0,0);
		\draw[blue, thick] (0,0) -- (1,1);
		\draw[blue, thick] (0,0) -- (-1,1);
		\filldraw[black] (0,0) circle (5pt)  {};
		\end{scope}
		\begin{scope}[shift={(-3.45,0)}]
		\node (v1) at (1,0.25) {};
		\node (v2) at (2.5,0.25) {};
		\draw[|->,very thick]  (v1) edge (v2);	
		\end{scope}
		\begin{scope}[shift={(-0.125,0)}]
		\node (v1) at (1,0.25) {};
		\node (v2) at (2.5,0.25) {};
		\draw[|->,very thick]  (v1) edge (v2);	
		\end{scope}
		\begin{scope}[shift={(3.525,0)}]
		\node (v1) at (1,0.25) {};
		\node (v2) at (2.5,0.25) {};
		\draw[|->,very thick]  (v1) edge (v2);	
		\end{scope}

\end{tikzpicture}
    \caption{Splitting by lightening at the leaf $i+1$ with $i=2$}
    \label{fig:splitting_by_lightening}
\end{figure}

The first two combinatorial Hopf algebras extend directly  $\YSym$ and $\SSym$  to include all planar rooted trees. The fundamental basis in the former Hopf algebra is indexed on $\PT$ extending that of $\YSym$. There are however two different natural extensions of the Hopf structure of $\YSym$ to $\PT$ depending on whether we permit any splitting by lightening at any leaf $i+1$  (which we will denote $t\mapsto ({}^it,t^i)$) or only some splittings which we call ``allowable''. Both constructions preserve the number of leaves but the allowable splitting is defined so that the number of internal nodes is also preserved.

\begin{Theorem}[Proposition \ref{prop:tsym}]\label{thm:iso_first}
	The two Hopf algebra structures on $\PT$ induced by the operation of splitting by lightening are isomorphic.
\end{Theorem}

\begin{Remark}
    Theorem \ref{thm:iso_first} illustrates a common theme in this area: multiple (a priori) different constructions on the same set of combinatorial objects could lead to isomorphic Hopf algebras. From an algebraic combinatorics point of view, however, the chosen basis and construction is often relevant in the context where these Hopf algebras are being used.
    \end{Remark}

We call $\TSym$ the Hopf algebra on $\PT$ together with the multiplication and comultiplication defined by allowable splitting by lightning.

Our extension of $\SSym$ to the context of planar trees comes from the concept of the Stirling permutations introduced by Gessel and Stanley in \cite{GS78}. A (generalized) Stirling permutation can be viewed as a labeled planar tree whose labels are increasing in any path away from the root (see Figure \ref{fig:stirling_tree_example}).  In particular, for planar trees whose internal nodes all have $3$ children, the generalized Stirling permutations are exactly the Stirling permutations defined in \cite{GS78}. They can also be regarded as $212$ avoiding packed words (as defined in \cite{J99}).

\begin{figure}[h]
    \centering
    \begin{tikzpicture}[scale=0.2,baseline=0pt]

	\draw[blue, thick] (0,-1) -- (0,0);
	\draw[blue, thick] (0,0) -- (-5.25,4);
	\draw[blue, thick] (0,0) -- (7,6);
	\draw[blue, thick] (0,0) -- (0,4);
	\draw[blue, thick] (-5.25,4) -- (-7.25,6);
	\draw[blue, thick] (-5.25,4) -- (-3.25,6);
	\draw[blue, thick] (-4.25,5) -- (-5.25,6);

	\draw[blue, thick] (0,4) -- (2,6);
	\draw[blue, thick] (0,4) -- (0,6);
	\draw[blue, thick] (0,4) -- (-2,6);
	\filldraw[black] (0,0) circle (5pt)  {};
	\filldraw[black] (-5.25,4) circle (5pt)  {};
	\filldraw[black] (0,4) circle (5pt)  {};
	\filldraw[black] (-4.25,5) circle (5pt)  {};
	\node at (-0.4,1)[font=\fontsize{7pt}{0}]{$1$};
	\node at (0.4,1)[font=\fontsize{7pt}{0}]{$1$};
	\node [font=\fontsize{7pt}{0}] at (-5.25,4.7) {$2$};
	\node [font=\fontsize{7pt}{0}] at (-4.25,5.7) {$4$};
	\node [font=\fontsize{7pt}{0}] at (-0.4,5) {$3$};
	\node [font=\fontsize{7pt}{0}] at (0.4,5) {$3$};
	
\node at (0,8) {241331};
\end{tikzpicture}
    \caption{Generalized Stirling permutation and its underlying planar tree}
    \label{fig:stirling_tree_example}
\end{figure}
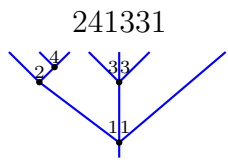

We define $\STSym$ as the Hopf algebra whose fundamental basis is indexed on the set $\GSP$ of generalized Stirling permutation and where its multiplication and comultiplication are defined by allowable splitting of the underlying planar tree.

A Hopf algebra $(\cH,\cdot,\Delta)$ has a  \emph{bidendriform} structure if it can be provided with operations $\ll$, $ \gg$, $\Delta_\ll$, and $\Delta_\gg$ satisfying $\cdot = \ll + \gg$, $\Delta = \Delta_\ll + \Delta_\gg$ and a set of relations described in  Section \ref{sec:bidendriform}. One of the important consequences for a Hopf algebra to be bidendriform is given by Foissy in \cite{F07,F12} where he proves that such an algebra is free, cofree and self-dual. Using this result we prove the following theorem.

\begin{Theorem}[Proposition \ref{prop:STSYMbidendriform} and Corollary \ref{cor:STSym_selfdual}]
    The Hopf algebra $\STSym$ has a bidendriform structure and hence is free, cofree and self-dual.
\end{Theorem}

Bergeron and Zabrocki show in \cite{BZ09} that the Hopf algebra $\NCQSym=\WQSym^*$ of noncommutative quasisymmetric functions, supported on packed words, is free and cofree. With the interpretation of Stirling permtuations as packed words we give an embedding $\STSym^*\to\NCQSym=\WQSym^*$.

We introduce yet a third new combinatorial Hopf algebra on labeled planar trees, the Hopf algebra of parking functions $\PSym$, in which a parking function is viewed as a pair  $(\sigma,t)\in\Perm_n\times\PBT_n$  of a binary tree and a permutation of the same degree such that $\Des(t)\subseteq\Des(\sigma)$ (see Section \ref{sec:background} for the definitions of these statistics). A consequence of our construction is the existence of natural Hopf algebra epimorphisms to $\SSym$ and $\YSym$ (see Figure \ref{fig:all_hopf_algebra_maps}).

\begin{Theorem}[Proposition \ref{prop:PSYM_bidendriform} and Corollary \ref{cor:PSYM_free_selfdual}]
The Hopf algebra $\PSym$ is a bidendriform bialgebra, and hence is free, cofree, self-dual, and isomorphic to the Hopf algebra $\PQSym$ defined in \cite{NT07}.
\end{Theorem}

It is worthy to mention that, though we do not prove it here, our construction of $\PQSym$ generalizes to the context of $m$-parking functions by using $m+1$-ary trees. This construction is likely related to the constructions in \cite{NT20}.
 
The following diagram summarizes the relationships between our three new algebras and other Hopf algebras. Note that the square commutes but the triangle on the left does not.
\begin{figure}[h]
    \begin{center}
	\begin{tikzpicture}
		\node at (0,0) {$\YSym$};
		\node at (3,0) {$\TSym$};
		\node at (3,2) {$\STSym$};
		\node at (0,2) {$\SSym$};
		\draw[black,thick, right hook->] (0.8,0) -- (2.2,0);
		\draw[black,thick, right hook->] (0.8,2) -- (2.2,2);
		\draw[black,thick, ->>] (0,1.5) -- (0,0.5);
		\draw[black,thick, ->>] (3,1.5) -- (3,0.5);
		\node at (0.3,1) {$\Pi$};
		\node at (3.6,1) {$\Pi_\ptree$};
		\node at (-3,1) {$\PSym$};
		\draw[black,thick, ->>] (-2.2,0.8) -- (-0.8,0.2);
		\draw[black,thick, ->>] (-2.2,1.2) -- (-0.8,1.8);
		\node at (-1.8,1.8) {$\Pi_\perm$};
		\node at (-1.7,0.3) {$\Pi_\tree$};
		\node at (-1,1) {$\not \equiv$};
		\node at (1.5,1) {$\equiv$};
	    \node at (6,2) {$\WQSym$};
	    \draw[black,thick, ->>]  (5.2,2) -- (3.8,2);
	\end{tikzpicture}
\end{center}
    \caption{Relations between the Hopf algebras discussed in this work}
    \label{fig:all_hopf_algebra_maps}
\end{figure}
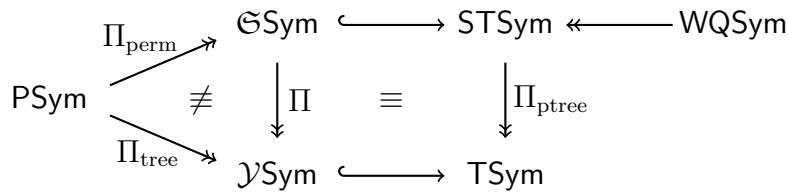

\subsection{Underlying posets and poset maps}
 We define partial orders naturally associated with the three families of combinatorial objects that index the Hopf algebra constructions discussed above.
 
 We define a Tamari order on the set $\PT$ of all planar trees that when restricting to any family of planar $(m+1)$-ary trees for a fixed $m$ coarsens the $m$-Tamari order defined in \cite{BP12} and when $m=1$ it coincides with the classical definition of the Tamari lattice. In the case of labeled planar trees, the set $\GSP$ of Stirling permutations inherits a planar weak order from the left weak order on packed words \cite{V15}. Much in the same flavor as in \cite{AS05} we show that the maps $\pi_{ptree}$ leaving the underlying planar unlabeled tree asociated to a Stirling permutation and a particular section $\iota_{gsp}$ associating to a planar tree its unique $213$-avoiding Stirling permutation are compatible as stated in the following proposition.

 \begin{Proposition}[Corollary \ref{cor:Galois_connection}]\label{prop:galois_connection}
 The maps $\pi_{ptree}$ and $\iota_{gsp}$ form a   Galois connection between the weak order on Stirling permutations and the Tamari order on planar trees. 
 \end{Proposition}
 
Via the identification of a parking function as a pair $(\sigma,t)\in\Perm_n\times\PBT_n$, parking functions get a natural partial order inherited from the product lattice on the weak and Tamari orders, which we call the parking order.

The definition of partial orders associated to the Hopf algebras $\TSym^*$, $\STSym^*$, and $\PSym^*$ allow us to see multiplication in the dual bases as sums over intervals in these posets. 

\begin{Theorem}[Propositions \ref{prop:interval_multiplication_TSYM}, \ref{prop:interval_multiplication_STSYM}, and \ref{prop:interval_multiplication_PSYM}]\label{prop:dualproduct-all}
	Let $F_f^*$, $F_g^*$ be elements in the dual basis in any of the graded duals $\TSym^*$, $\STSym^*$, or $\PSym^*$. Then,
		$$m(F_f^*\otimes F_g^*)=\sum_{f\backslash g\leq h\leq f/g}F_h^*,$$
where $\leq$ are the planar Tamari, weak and monomial partial orders respectively; and $f/g$ and $f\backslash g$ are obtained by grafting $f$ and $g$ according to Definitions \ref{def:graft-trees} and \ref{def:perm}.
\end{Theorem}

\subsection{Monomial basis and axioms} One way to show that a combinatorial Hopf algebra is free is by constructing an explicit basis with a free multiplication \cite{LR02,BZ09,NT06,NT07,PP18}, often of the form 
\begin{equation}
H_f=\sum_{g\leq f} G_g \label{eq:freebasis}
\end{equation}
where $G$ is the starting basis, and $\leq$ is a partial order on the underlying objects. However, besides \cite{BZ09}, most of the other works do not explore the comultiplication properties of their free $H$ bases -- are the comultiplication coefficients even non-negative? -- so it is not easy to use the new $H$ basis to define Hopf morphisms to other Hopf algebras. Furthermore, despite the common formulas of the type in (\ref{eq:freebasis}), in all these algebras there is no universal result from which these cases may be recovered as easy corollaries.

Aguiar-Sottile \cite{AS05,AS06} took with  $\YSym$ (dual to the Loday-Ronco algebra) and $\SSym$ an alternative direction resulting in a more complete approach. They make the dual construction to (\ref{eq:freebasis}): if $F$ is the starting ``fundamental'' basis (dual to $G$ above), define the ``monomial'' $M$ basis by 
\begin{equation}
M_f=\sum_{g\geq f} \mu(f,g) F_g, \quad \text{ equivalently }\quad F_f =\sum_{g\geq f}M_g. \label{eq:cofreebasis}
\end{equation}
They then show that the monomial basis is cofree, and that the primitive monomial basis elements are indexed by permutations or binary trees with no ``global descent". Playing with M\"obius inversion on the poset structure, one then obtains a positive multiplication formula and a cancellation-free antipode formula for the monomial basis. An extension of these techniques to packed words is key to Vargas' self-duality isomorphism for $\WQSym$ \cite{V20} - it was possible to compare monomial bases on $\WQSym$ and $\NCQSym\cong\WQSym^*$ thanks to the explicit multiplication formula.
These techniques were used also in  \cite{CS08,S20}

In this article we abstract the technique of Aguiar and Sottile by giving a set of axioms (see Section \ref{sec:axioms}) based on the interaction between an underlying poset on the set of combinatorial objects indexing the fundamental basis and the Hopf operations of a combinatorial Hopf algebra defined on the same family of objects. These axioms guarantee that the monomial basis defined via (\ref{eq:cofreebasis}) is cofree. We make abstract definitions of ``shifted concatenation" and ``global descents'' to give a unifying positive multiplication formula and a cancellation-free antipode formula where all terms have the same sign. Our axioms are easily transfered via poset morphisms, meaning that monomial bases may be defined for many combinatorial Hopf algebras on related objects, all in one sweep.

We show that  $\TSym$, $\STSym$ and $\PSym$ satisfy the axioms in Section \ref{sec:axioms}. We show this directly on   $\STSym$ and $\PSym$, and we use a transfer axiom to show that from the calculation in $\STSym$ we can conclude that the axioms are also satisfied in $\TSym$. As a consequence of these axioms, we obtain explicit formulas for the multiplication and comultiplication of monomial basis elements, and a cancellation-free grouping-free formula for their antipode. We have the following theorem.

\begin{Theorem}[Propositions \ref{prop:coproduct-monomial-tsym}--\ref{prop:antipode-monomial-tsym}, \ref{prop:coproduct-monomial-stsym}--\ref{prop:antipode-monomial-stsym},  \ref{prop:coproduct-monomial-parking}--\ref{prop:antipode-monomial-parking}] \label{theorem:general_monomial}
For the monomial bases $\{M_t\}$ of the combinatorial Hopf algebras $\TSym$, $\STSym$ and $\PSym$ defined as in equation \eqref{eq:cofreebasis} we have that 
\begin{enumerate}
    \item The comultiplication of monomial
basis elements is given by
\[
\Delta_{+}(M_{t})=\sum_{i\in\GD(f)}M_{{}^{i}t}\otimes M_{t^{i}}.
\]
\item The product of monomial basis elements has non-negative integer coefficients. 

\item In the antipode image $\mathcal{S}(M_t)$, all
terms have the same sign, equal to the parity of $|\GD(t)|+1$.
\end{enumerate}

\end{Theorem}
As a consequence of Theorem \ref{theorem:general_monomial} we obtain that $\TSym$ is cofree. An interesting open question is to determine if $\TSym$ is also free and self-dual.

\begin{Remark}
    Restricting $\TSym$ to the generating subset of planar $(m+1)$-ary trees for a fixed positive integer $m$ we obtain a subalgebra of $\TSym$ indexed by Fuss-Catalan objects that is likely isomorphic to the Hopf algebra defined in \cite{NT20}.
\end{Remark}

\begin{Remark}
Note that parts (1) and (2) of Theorem \ref{theorem:general_monomial} require independent axioms, so for a general Hopf algebra one can have (1) without having (2), as it is the case of the tableaux algebra studied in \cite{MR21}.
\end{Remark}
The organization of this article is as follows. In section \ref{sec:background}, we review the notation and background on the classical combinatorial Hopf algebras $\YSym$ and $\SSym$. In section \ref{sec:axioms}, we present a list of our monomial basis axioms and its consequences without proofs. Sections  \ref{sec:TSYM}, \ref{sec:STSYM} and \ref{sec:parking_functions} introduce respectively the Hopf algebras of planar trees, generalized Stirling permutations, and parking functions, detailing their algebraic properties and connections with the classical Hopf algebras $\YSym$ and $\SSym$. Lastly in section \ref{sec:proofs}, we provide the proofs for the theorems in section \ref{sec:axioms}.

\section{Background}\label{sec:background}

\subsection{Notations}\label{sec:notations}

In this section, we define a number of operations on planar trees and labeled planar trees. These operations also apply to binary trees and labeled binary trees as special cases.

\begin{Definition}
	A (rooted planar) \textbf{tree} is a rooted tree such that for each node its set of children is totally ordered and each node either has no child, i.e. a \textbf{leaf}, or have at least 2 children i.e. an \textbf{internal node}. By convention, there is an \textbf{empty tree} with no internal node and $1$ leaf, we draw it as ``$\,\begin{tikzpicture}[scale=0.2,baseline=0pt]\draw[blue, thick] (0,-1) -- (0,.5);\end{tikzpicture}\,$". 
	The $i$-th leaf of $t$ reading from left to right is called the $i$-th leaf of $t$  for short. The set of planar trees is denoted by $\PT$.	
	The \textbf{degree} of a tree $t$ is one less than number of leaves in $t$, denoted by $\deg(t)$. The \textbf{internal degree} of $t$ is number of internal nodes of $t$, denoted by $\ideg(t)$. (See for example Figure~\ref{fig:PT4})
\end{Definition}

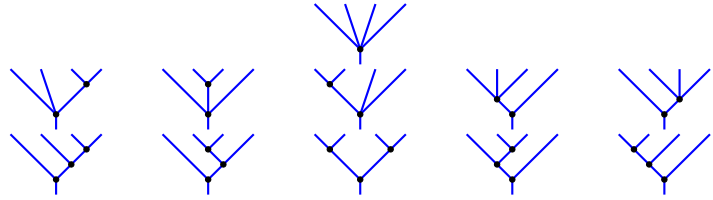
\begin{figure}
\begin{center}
	\begin{tikzpicture}[scale=0.2,baseline=0pt]
	\draw[blue, thick] (0,-1) -- (0,0);
	\draw[blue, thick] (0,0) -- (3,3);
	\draw[blue, thick] (0,0) -- (-3,3);
	\draw[blue, thick] (0,0) -- (-1,3);
	\draw[blue, thick] (0,0) -- (1,3);
	\filldraw[black] (0,0) circle (5pt)  {};
	\end{tikzpicture}
	
	\begin{tikzpicture}[scale=0.2,baseline=0pt]
	\draw[blue, thick] (0,-1) -- (0,0);
	\draw[blue, thick] (0,0) -- (3,3);
	\draw[blue, thick] (0,0) -- (-3,3);
	\draw[blue, thick] (0,0) -- (-1,3);
	\draw[blue, thick] (2,2) -- (1,3);
	\filldraw[black] (0,0) circle (5pt)  {};
	\filldraw[black] (2,2) circle (5pt)  {};
	
	\draw[blue, thick] (10,-1) -- (10,0);
	\draw[blue, thick] (10,0) -- (13,3);
	\draw[blue, thick] (10,0) -- (7,3);
	\draw[blue, thick] (10,0) -- (10,2);
	\draw[blue, thick] (10,2) -- (9,3);
	\draw[blue, thick] (10,2) -- (11,3);
	\filldraw[black] (10,0) circle (5pt)  {};
	\filldraw[black] (10,2) circle (5pt)  {};
	
	\draw[blue, thick] (20,-1) -- (20,0);
	\draw[blue, thick] (20,0) -- (23,3);
	\draw[blue, thick] (20,0) -- (17,3);
	\draw[blue, thick] (20,0) -- (21,3);
	\draw[blue, thick] (18,2) -- (19,3);
	\filldraw[black] (20,0) circle (5pt)  {};
	\filldraw[black] (18,2) circle (5pt)  {};
	
	\draw[blue, thick] (30,-1) -- (30,0);
	\draw[blue, thick] (30,0) -- (33,3);
	\draw[blue, thick] (30,0) -- (27,3);
	\draw[blue, thick] (29,1) -- (31,3);
	\draw[blue, thick] (29,1) -- (29,3);
	\filldraw[black] (30,0) circle (5pt)  {};
	\filldraw[black] (29,1) circle (5pt)  {};
	
	\draw[blue, thick] (40,-1) -- (40,0);
	\draw[blue, thick] (40,0) -- (37,3);
	\draw[blue, thick] (40,0) -- (43,3);
	\draw[blue, thick] (41,1) -- (41,3);
	\draw[blue, thick] (41,1) -- (39,3);
	\filldraw[black] (40,0) circle (5pt)  {};
	\filldraw[black] (41,1) circle (5pt)  {};
	\end{tikzpicture}
	
	\begin{tikzpicture}[scale=0.2,baseline=0pt]
	\draw[blue, thick] (0,-1) -- (0,0);
	\draw[blue, thick] (0,0) -- (3,3);
	\draw[blue, thick] (0,0) -- (-3,3);
	\draw[blue, thick] (1,1) -- (-1,3);
	\draw[blue, thick] (2,2) -- (1,3);
	\filldraw[black] (0,0) circle (5pt)  {};
	\filldraw[black] (1,1) circle (5pt)  {};
	\filldraw[black] (2,2) circle (5pt)  {};
	
	\draw[blue, thick] (10,-1) -- (10,0);
	\draw[blue, thick] (10,0) -- (13,3);
	\draw[blue, thick] (10,0) -- (7,3);
	\draw[blue, thick] (11,1) -- (9,3);
	\draw[blue, thick] (10,2) -- (11,3);
	\filldraw[black] (10,0) circle (5pt)  {};
	\filldraw[black] (11,1) circle (5pt)  {};
	\filldraw[black] (10,2) circle (5pt)  {};
	
	\draw[blue, thick] (20,-1) -- (20,0);
	\draw[blue, thick] (20,0) -- (23,3);
	\draw[blue, thick] (20,0) -- (17,3);
	\draw[blue, thick] (22,2) -- (21,3);
	\draw[blue, thick] (18,2) -- (19,3);
	\filldraw[black] (20,0) circle (5pt)  {};
	\filldraw[black] (22,2) circle (5pt)  {};
	\filldraw[black] (18,2) circle (5pt)  {};
	
	\draw[blue, thick] (30,-1) -- (30,0);
	\draw[blue, thick] (30,0) -- (33,3);
	\draw[blue, thick] (30,0) -- (27,3);
	\draw[blue, thick] (29,1) -- (31,3);
	\draw[blue, thick] (30,2) -- (29,3);
	\filldraw[black] (30,0) circle (5pt)  {};
	\filldraw[black] (29,1) circle (5pt)  {};
	\filldraw[black] (30,2) circle (5pt)  {};
	
	\draw[blue, thick] (40,-1) -- (40,0);
	\draw[blue, thick] (40,0) -- (37,3);
	\draw[blue, thick] (40,0) -- (43,3);
	\draw[blue, thick] (39,1) -- (41,3);
	\draw[blue, thick] (38,2) -- (39,3);
	\filldraw[black] (40,0) circle (5pt)  {};
	\filldraw[black] (39,1) circle (5pt)  {};
	\filldraw[black] (38,2) circle (5pt)  {};
	\end{tikzpicture}
\end{center}	
\caption{\sl Planar trees with 4 leaves, i.e. with degree 3.
The trees in row $i$ have $i$ internal nodes.}\label{fig:PT4}
\end{figure}

\begin{Definition}
	A \textbf{labeled (planar) tree} is a tree $f$ together with a map $\kappa_f:\{\text{internal nodes of }f\}\to\mathbb{N}$ that assigns a natural number to each internal node of $f$. The set of labeled trees is denoted by $\LPT$. By abuse of notation, we use the symbol $f$ for both labeled tree and its underlying tree. If $\kappa_f$ is an injection, then we may define the \textbf{standardization} of $f$, denoted by $\std(f)$, to be the labeled tree whose underlying tree is the same as $f$ and the labels are renumbered to $\{1,2,\dots,\ideg(f)\}$ while keeping the relative order.
\end{Definition}

In this paper, we view permutations, parking functions and increasing trees as labeled  trees. We use the symbols $s,t,r$ for trees, $\sigma,\tau,\rho$ for permutations, $f,g,h$ for parking functions and $u,v,w$ for generalized Stirling permutations.

We describe below some operations on trees, that will be used to define the comultiplication and multiplication of the associated Hopf algebras.

An orange lightening hits the tree $t$ at a leaf, and goes along the path from that leaf to the root. This lightening splits the tree into two pieces, and by contracting unnecessary nodes, we obtain an ordered pair of two trees $(t_1,t_2)$. We denote this by $t\mapsto(t_1,t_2)$ if $(t_1,t_2)$, and call it a \textbf{deconcatenation} of $t$. When it is useful to specify that the lightening is at the $(i+1)$-th leaf, so $\deg t_1 =i$, we may write $({}^it,t^i)$  in place of $(t_1,t_2)$.

\begin{Example}\label{ex:split}
	Let $t=\begin{tikzpicture}[scale=0.25,baseline=0pt]
		\draw[blue, thick] (0,-1) -- (0,0);
		\draw[blue, thick] (0,0) -- (3,3);
		\draw[blue, thick] (0,0) -- (-3,3);
		\draw[blue, thick] (0,0) -- (0,2);
		\draw[blue, thick] (0,2) -- (-1,3);
		\draw[blue, thick] (0,2) -- (1,3);
		\filldraw[black] (0,0) circle (5pt)  {};
		\filldraw[black] (0,2) circle (5pt)  {};
		\end{tikzpicture}$. To calculate ${}^2 t$ and $t^2$, consider a lightening hitting at the third leaf:
	
	
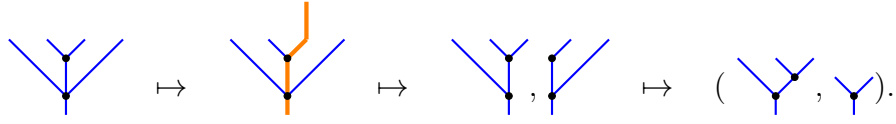
\begin{center}
		\begin{tikzpicture}[scale=0.25,baseline=0pt]
		\draw[blue, thick] (0,-1) -- (0,0);
		\draw[blue, thick] (0,0) -- (3,3);
		\draw[blue, thick] (0,0) -- (-3,3);
		\draw[blue, thick] (0,0) -- (0,2);
		\draw[blue, thick] (0,2) -- (-1,3);
		\draw[blue, thick] (0,2) -- (1,3);
		\filldraw[black] (0,0) circle (5pt)  {};
		\filldraw[black] (0,2) circle (5pt)  {};
		\end{tikzpicture}$\quad\mapsto\quad$
		\begin{tikzpicture}[scale=0.25,baseline=0pt]
		\draw[orange, ultra thick] (0,-1) -- (0,2);
		\draw[blue, thick] (0,0) -- (3,3);
		\draw[blue, thick] (0,0) -- (-3,3);
		\draw[orange, ultra thick] (1,3) -- (1,5);
		\draw[blue, thick] (0,2) -- (-1,3);
		\draw[orange, ultra thick] (0,2) -- (1,3);
		\filldraw[black] (0,0) circle (5pt)  {};
		\filldraw[black] (0,2) circle (5pt)  {};
		\end{tikzpicture}$\quad\mapsto\quad$
		\begin{tikzpicture}[scale=0.25,baseline=0pt]
		\draw[blue, thick] (0,-1) -- (0,0);
		\draw[blue, thick] (0,0) -- (-3,3);
		\draw[blue, thick] (0,0) -- (0,2);
		\draw[blue, thick] (0,2) -- (-1,3);
		\draw[blue, thick] (0,2) -- (1,3);
		\filldraw[black] (0,0) circle (5pt)  {};
		\filldraw[black] (0,2) circle (5pt)  {};
		\end{tikzpicture},
		\begin{tikzpicture}[scale=0.25,baseline=0pt]
		\draw[blue, thick] (0,-1) -- (0,0);
		\draw[blue, thick] (0,0) -- (3,3);
		\draw[blue, thick] (0,0) -- (0,2);
		\draw[blue, thick] (0,2) -- (1,3);
		\filldraw[black] (0,0) circle (5pt)  {};
		\filldraw[black] (0,2) circle (5pt)  {};
		\end{tikzpicture}$\quad\mapsto\quad($
		\begin{tikzpicture}[scale=0.25,baseline=0pt]
		\draw[blue, thick] (0,-1) -- (0,0);
		\draw[blue, thick] (0,0) -- (-2,2);
		\draw[blue, thick] (0,0) -- (2,2);
		\draw[blue, thick] (1,1) -- (0,2);
		\filldraw[black] (0,0) circle (5pt)  {};
		\filldraw[black] (1,1) circle (5pt)  {};
		\end{tikzpicture},
		\begin{tikzpicture}[scale=0.25,baseline=0pt]
		\draw[blue, thick] (0,-1) -- (0,0);
		\draw[blue, thick] (0,0) -- (1,1);
		\draw[blue, thick] (0,0) -- (-1,1);
		\filldraw[black] (0,0) circle (5pt)  {};
		\end{tikzpicture}$)$.
	\end{center}
\end{Example}

Similarly, $k-1$ lightenings hitting a tree $t$ split it into $k$ trees, such process is denoted by $t\mapsto(t_1,\dots,t_k)$. Multiple lightenings may hit the same leaf, thus $(t_1,\dots,t_k)$ is determined by a \emph{weak composition} of $\deg(t)+k$ with $k$ parts i.e. a sequence of nonnegative integers $(a_1,a_2,\dots,a_k)$ such that $\sum_{i=1}^k a_i=\deg(t)+k$.

\begin{Example}\label{ex:allowable}
	$$\begin{tikzpicture}[scale=0.25,baseline=0pt]
	\draw[blue, thick] (0,-1) -- (0,0);
	\draw[blue, thick] (0,0) -- (7,7);
	\draw[blue, thick] (0,0) -- (-7,7);
	\draw[blue, thick] (-5,5) -- (-3,7);
	\draw[blue, thick] (-4,6) -- (-5,7);
	\draw[blue, thick] (0,0) -- (1,5);
	\draw[blue, thick] (1,5) -- (-1,7);
	\draw[blue, thick] (1,5) -- (1,7);
	\draw[blue, thick] (1,5) -- (3,7);
	\draw[blue, thick] (6,6) -- (5,7);
	\draw[orange, ultra thick] (0,0) -- (-5,5) -- (-4,6) -- (-5,7) -- (-5,9);
	\draw[red, ultra thick] (0,0) -- (6,6) -- (5,7) -- (5,9);
	\filldraw[black] (0,0) circle (5pt)  {};
	\filldraw[black] (1,5) circle (5pt)  {};
	\filldraw[black] (6,6) circle (5pt)  {};
	\filldraw[black] (-5,5) circle (5pt)  {};
	\filldraw[black] (-4,6) circle (5pt)  {};
	\end{tikzpicture}\quad\mapsto\quad (\begin{tikzpicture}[scale=0.25,baseline=-0.5pt]
	\draw[blue, thick] (0,-1) -- (0,0);
	\draw[blue, thick] (0,0) -- (1,1);
	\draw[blue, thick] (0,0) -- (-1,1);
	\filldraw[black] (0,0) circle (5pt)  {};
	\end{tikzpicture}, \begin{tikzpicture}[scale=0.25,baseline=-0.5pt]
	\draw[blue, thick] (0,-1) -- (0,0);
	\draw[blue, thick] (0,0) -- (5,5);
	\draw[blue, thick] (0,0) -- (-5,5);
	\draw[blue, thick] (-4,4) -- (-3,5);
	\draw[blue, thick] (0,0) -- (1,3);
	\draw[blue, thick] (1,3) -- (-1,5);
	\draw[blue, thick] (1,3) -- (1,5);
	\draw[blue, thick] (1,3) -- (3,5);
	\filldraw[black] (0,0) circle (5pt)  {};
	\filldraw[black] (1,3) circle (5pt)  {};
	\filldraw[black] (-4,4) circle (5pt)  {};
	\end{tikzpicture},\begin{tikzpicture}[scale=0.25,baseline=-0.5pt]
	\draw[blue, thick] (0,-1) -- (0,0);
	\draw[blue, thick] (0,0) -- (1,1);
	\draw[blue, thick] (0,0) -- (-1,1);
	\filldraw[black] (0,0) circle (5pt)  {};
	\end{tikzpicture}).$$
\end{Example}

We say $t\mapsto(t_1,\dots,t_k)$ is \textbf{allowable} if $\ideg(t)=\ideg(t_1)+\cdots+\ideg(t_k)$. The splitting in example \ref{ex:split} is not allowable while the one in example \ref{ex:allowable} is. Note that if $t$ is a binary tree, then all $t\mapsto(t_1,\dots,t_k)$ are allowable.

\begin{Definition} \label{def:graft-trees}
	Let $s$ and $t$ be two trees, we define two operations $/$ and $\backslash$ as follows
	\begin{itemize}
		\item $s/ t$ is the tree obtained by identifying the root of $s$ with the left-most leaf of $t$.
		\item $s\backslash t$ is the tree obtained by identifying the root of $t$ with the right-most leaf of $s$.
	\end{itemize}
\end{Definition}

\begin{Example}
	Let $s=\begin{tikzpicture}[scale=0.2,baseline=0pt]
	\draw[blue, thick] (0,-1) -- (0,0);
	\draw[blue, thick] (0,0) -- (2,2);
	\draw[blue, thick] (0,0) -- (-2,2);
	\draw[blue, thick] (-1,1) -- (0,2);
	\filldraw[black] (0,0) circle (5pt)  {};
	\filldraw[black] (-1,1) circle (5pt)  {};
	\end{tikzpicture}$ and $t=\begin{tikzpicture}[scale=0.2,baseline=0pt]
	\draw[red, thick] (0,-1) -- (0,0);
	\draw[red, thick] (0,0) -- (2,2);
	\draw[red, thick] (0,0) -- (-2,2);
	\draw[red, thick] (1,1) -- (0,2);
	\filldraw[black] (0,0) circle (5pt)  {};
	\filldraw[black] (1,1) circle (5pt)  {};
	\end{tikzpicture}$. Then
	
	$s/t=$\begin{tikzpicture}[scale=0.2,baseline=0pt]
	\draw[red, thick] (0,-1) -- (0,0);
	\draw[red, thick] (0,0) -- (4,4);
	\draw[red, thick] (3,3) -- (2,4);
	\draw[red, thick] (0,0) -- (-2,2);
	\draw[blue, thick] (-2,2) -- (-4,4);
	\draw[blue, thick] (-2,2) -- (0,4);
	\draw[blue, thick] (-3,3) -- (-2,4);
	\filldraw[black] (0,0) circle (5pt)  {};
	\filldraw[black] (3,3) circle (5pt)  {};
	\filldraw[black] (-2,2) circle (5pt)  {};
	\filldraw[black] (-3,3) circle (5pt)  {};
	\end{tikzpicture} and
	$s\backslash t=$\begin{tikzpicture}[scale=0.2,baseline=0pt]
	\draw[blue, thick] (0,-1) -- (0,0);
	\draw[blue, thick] (0,0) -- (-4,4);
	\draw[blue, thick] (-3,3) -- (-2,4);
	\draw[blue, thick] (0,0) -- (2,2);
	\draw[red, thick] (2,2) -- (4,4);
	\draw[red, thick] (2,2) -- (0,4);
	\draw[red, thick] (3,3) -- (2,4);
	\filldraw[black] (0,0) circle (5pt)  {};
	\filldraw[black] (-3,3) circle (5pt)  {};
	\filldraw[black] (3,3) circle (5pt)  {};
	\filldraw[black] (2,2) circle (5pt)  {};
	\end{tikzpicture}.
\end{Example}

It is easy to see that both the above operations are associative. Note also that, if $\deg s=i$, then ${}^i (s/ t)={}^i (s\backslash t)=s$ and $(s/ t)^i=(s\backslash t)^i=t$.

Let $s$ be a tree with $k$ leaves, and $t_1,\dots,t_k$ be $k$ trees, we define $s\shuffle(t_{1},\dots,t_{k})$ to be the tree obtained by identifying the root of $t_i$ with the $i$-th leaf of $s$ for all $i$.

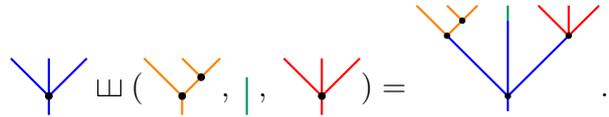
\begin{Example}
	$$\begin{tikzpicture}[scale=0.25,baseline=0pt]
	\draw[blue, thick] (0,-1) -- (0,0);
	\draw[blue, thick] (0,0) -- (-2,2);
	\draw[blue, thick] (0,0) -- (0,2);
	\draw[blue, thick] (0,0) -- (2,2);
	\filldraw[black] (0,0) circle (5pt)  {};
	\end{tikzpicture}\shuffle (\begin{tikzpicture}[scale=0.25,baseline=0pt]
	\draw[orange, thick] (0,-1) -- (0,0);
	\draw[orange, thick] (0,0) -- (-2,2);
	\draw[orange, thick] (1,1) -- (0,2);
	\draw[orange, thick] (0,0) -- (2,2);
	\filldraw[black] (0,0) circle (5pt)  {};
	\filldraw[black] (1,1) circle (5pt)  {};
	\end{tikzpicture},~\begin{tikzpicture}[scale=0.25,baseline=0pt]
	\draw[ForestGreen, thick] (0,-1) -- (0,1);
	\end{tikzpicture}~,~\begin{tikzpicture}[scale=0.25,baseline=0pt]
	\draw[red, thick] (0,-1) -- (0,0);
	\draw[red, thick] (0,0) -- (-2,2);
	\draw[red, thick] (0,0) -- (0,2);
	\draw[red, thick] (0,0) -- (2,2);
	\filldraw[black] (0,0) circle (5pt)  {};
	\end{tikzpicture})=\begin{tikzpicture}[scale=0.2,baseline=0pt]
	\draw[blue, thick] (0,-1) -- (0,0);
	\draw[blue, thick] (0,0) -- (-4,4);
	\draw[blue, thick] (0,0) -- (0,5);
	\draw[blue, thick] (0,0) -- (4,4);
	\draw[orange, thick] (-4,4) -- (-6,6);
	\draw[orange, thick] (-4,4) -- (-2,6);
	\draw[orange, thick] (-3,5) -- (-4,6);
	\draw[ForestGreen, thick] (0,5) -- (0,6);
	\draw[red, thick] (4,4) -- (6,6);
	\draw[red, thick] (4,4) -- (4,6);
	\draw[red, thick] (4,4) -- (2,6);
	\filldraw[black] (0,0) circle (5pt)  {};
	\filldraw[black] (-4,4) circle (5pt)  {};
	\filldraw[black] (4,4) circle (5pt)  {};
	\filldraw[black] (-3,5) circle (5pt)  {};
	\end{tikzpicture}.$$
\end{Example}

We now define similar operations on labeled trees.

Let $f$ be a labeled tree and $f\mapsto(f_1,\dots,f_k)$ be allowable. Note that when the splitting is allowable, each internal node of $f$ exists in exactly one $f_i$ i.e. it does not exist or gets contracted in all other parts. Therefore, $\kappa_f$ induces a labeling on each $f_i$, namely $\kappa_{f_i}(v)=\kappa_f(v)$. By abuse of notation, such operation is denoted by $f\mapsto(f_1,\dots,f_k)$ where $f_i$ are labeled trees. 

Let $f$ be a labeled tree such that $\kappa_f$ is a bijection between its internal nodes and $\{1,\dots,\ideg(f)\}$. Assume $f\mapsto(f_1,f_2)$ is allowable with $\deg(f_1)=i$. We will use a notation where the degree of the split is specified and denote the standardization of $f_1$ by ${}^if$ and the standardization of $f_2$ by $f^i$, i.e. ${}^if=\std(f_1)$ and $f^i=\std(f_2)$.

\begin{Example}
	$$\begin{tikzpicture}[scale=0.25,baseline=0pt]
	\draw[blue, thick] (0,-1) -- (0,0);
	\draw[blue, thick] (0,0) -- (7,7);
	\draw[blue, thick] (0,0) -- (-7,7);
	\draw[blue, thick] (-5,5) -- (-3,7);
	\draw[blue, thick] (-4,6) -- (-5,7);
	\draw[blue, thick] (0,0) -- (1,5);
	\draw[blue, thick] (1,5) -- (-1,7);
	\draw[blue, thick] (1,5) -- (1,7);
	\draw[blue, thick] (1,5) -- (3,7);
	\draw[blue, thick] (6,6) -- (5,7);
	\draw[orange, ultra thick] (0,0) -- (-5,5) -- (-4,6) -- (-5,7) -- (-5,9);
	\draw[red, ultra thick] (0,0) -- (6,6) -- (5,7) -- (5,9);
	\filldraw[black] (0,0) circle (5pt)  {};
	\filldraw[black] (1,5) circle (5pt)  {};
	\filldraw[black] (6,6) circle (5pt)  {};
	\filldraw[black] (-5,5) circle (5pt)  {};
	\filldraw[black] (-4,6) circle (5pt)  {};
	\node at (-0.4,1)[font=\fontsize{7pt}{0}]{$1$};
	\node at (0.5,1)[font=\fontsize{7pt}{0}]{$1$};
	\node at (0.6,6)[font=\fontsize{7pt}{0}]{$3$};
	\node at (1.4,6)[font=\fontsize{7pt}{0}]{$3$};
	\node at (6,6.7)[font=\fontsize{7pt}{0}]{$2$};
	\node at (-4,6.7)[font=\fontsize{7pt}{0}]{$5$};
	\node at (-5,5.7)[font=\fontsize{7pt}{0}]{$4$};
	\end{tikzpicture}\mapsto\  (\begin{tikzpicture}[scale=0.25,baseline=-1pt]
	\draw[blue, thick] (0,-1) -- (0,0);
	\draw[blue, thick] (0,0) -- (1,1);
	\draw[blue, thick] (0,0) -- (-1,1);
	\filldraw[black] (0,0) circle (5pt)  {};
	\node at (0,0.7)[font=\fontsize{7pt}{0}]{$4$};
	\end{tikzpicture}, \begin{tikzpicture}[scale=0.25,baseline=-1pt]
	\draw[blue, thick] (0,-1) -- (0,0);
	\draw[blue, thick] (0,0) -- (5,5);
	\draw[blue, thick] (0,0) -- (-5,5);
	\draw[blue, thick] (-4,4) -- (-3,5);
	\draw[blue, thick] (0,0) -- (1,3);
	\draw[blue, thick] (1,3) -- (-1,5);
	\draw[blue, thick] (1,3) -- (1,5);
	\draw[blue, thick] (1,3) -- (3,5);
	\filldraw[black] (0,0) circle (5pt)  {};
	\filldraw[black] (1,3) circle (5pt)  {};
	\filldraw[black] (-4,4) circle (5pt)  {};
	\node at (-0.3,1)[font=\fontsize{7pt}{0}]{$1$};
	\node at (0.7,1)[font=\fontsize{7pt}{0}]{$1$};
	\node at (0.6,4)[font=\fontsize{7pt}{0}]{$3$};
	\node at (1.4,4)[font=\fontsize{7pt}{0}]{$3$};
	\node at (-4,4.7)[font=\fontsize{7pt}{0}]{$5$};
	\end{tikzpicture},\begin{tikzpicture}[scale=0.25,baseline=-1pt]
	\draw[blue, thick] (0,-1) -- (0,0);
	\draw[blue, thick] (0,0) -- (1,1);
	\draw[blue, thick] (0,0) -- (-1,1);
	\filldraw[black] (0,0) circle (5pt)  {};
	\node at (0,0.7)[font=\fontsize{7pt}{0}]{$2$};
	\end{tikzpicture}).$$
\end{Example}

\begin{Definition}\label{def:perm}
	Let $f$ and $g$ be labeled trees, we define two operations $/$ and $\backslash$ as follows
	\begin{itemize}
		\item $f/g$ is the labeled tree obtained by first shifting labels of $f$ up by $\ideg(g)$, and then identifying the root of $f$ with the left-most leaf of $g$.
		\item $f\backslash g$ is the labeled tree obtained by first shifting labels of $g$ up by $\ideg(f)$, and then identifying the root of $g$ with the right-most leaf of $f$.
	\end{itemize}
\end{Definition}

\begin{Example}
	Let $f=\begin{tikzpicture}[scale=0.2,baseline=0pt]
	\draw[blue, thick] (0,-1) -- (0,0);
	\draw[blue, thick] (0,0) -- (2,2);
	\draw[blue, thick] (0,0) -- (-2,2);
	\draw[blue, thick] (-1,1) -- (0,2);
	\filldraw[black] (0,0) circle (5pt)  {};
	\filldraw[black] (-1,1) circle (5pt)  {};
	\node at (0,0.7)[font=\fontsize{7pt}{0}]{$1$};
	\node at (-1,1.7)[font=\fontsize{7pt}{0}]{$2$};
	\end{tikzpicture}$ and $g=\begin{tikzpicture}[scale=0.2,baseline=0pt]
	\draw[red, thick] (0,-1) -- (0,0);
	\draw[red, thick] (0,0) -- (2,2);
	\draw[red, thick] (0,0) -- (-2,2);
	\draw[red, thick] (1,1) -- (0,2);
	\filldraw[black] (0,0) circle (5pt)  {};
	\filldraw[black] (1,1) circle (5pt)  {};
	\node at (0,0.7)[font=\fontsize{7pt}{0}]{$1$};
	\node at (1,1.7)[font=\fontsize{7pt}{0}]{$2$};
	\end{tikzpicture}$. Then
	
	$f/g=$\begin{tikzpicture}[scale=0.2,baseline=0pt]
	\draw[red, thick] (0,-1) -- (0,0);
	\draw[red, thick] (0,0) -- (4,4);
	\draw[red, thick] (3,3) -- (2,4);
	\draw[red, thick] (0,0) -- (-2,2);
	\draw[blue, thick] (-2,2) -- (-4,4);
	\draw[blue, thick] (-2,2) -- (0,4);
	\draw[blue, thick] (-3,3) -- (-2,4);
	\filldraw[black] (0,0) circle (5pt)  {};
	\filldraw[black] (3,3) circle (5pt)  {};
	\filldraw[black] (-2,2) circle (5pt)  {};
	\filldraw[black] (-3,3) circle (5pt)  {};
	\node at (0,0.7)[font=\fontsize{7pt}{0}]{$1$};
	\node at (3,3.7)[font=\fontsize{7pt}{0}]{$2$};
	\node at (-2,2.7)[font=\fontsize{7pt}{0}]{$3$};
	\node at (-3,3.7)[font=\fontsize{7pt}{0}]{$4$};
	\end{tikzpicture} and
	$f\backslash g=$\begin{tikzpicture}[scale=0.2,baseline=0pt]
	\draw[blue, thick] (0,-1) -- (0,0);
	\draw[blue, thick] (0,0) -- (-4,4);
	\draw[blue, thick] (-3,3) -- (-2,4);
	\draw[blue, thick] (0,0) -- (2,2);
	\draw[red, thick] (2,2) -- (4,4);
	\draw[red, thick] (2,2) -- (0,4);
	\draw[red, thick] (3,3) -- (2,4);
	\filldraw[black] (0,0) circle (5pt)  {};
	\filldraw[black] (-3,3) circle (5pt)  {};
	\filldraw[black] (3,3) circle (5pt)  {};
	\filldraw[black] (2,2) circle (5pt)  {};
	\node at (0,0.7)[font=\fontsize{7pt}{0}]{$1$};
	\node at (-3,3.7)[font=\fontsize{7pt}{0}]{$2$};
	\node at (2,2.7)[font=\fontsize{7pt}{0}]{$3$};
	\node at (3,3.7)[font=\fontsize{7pt}{0}]{$4$};
	\end{tikzpicture}.
\end{Example}

It is easy to see that both the above operations are associative. As for unlabeled trees, we again have that, if $\deg f=i$, then ${}^i (f/ g)={}^i (f\backslash g)=f$ and $(f/ g)^i=(f\backslash g)^i=g$.

Let $f$ be a labeled tree with $k$ leaves, and $g_1,\dots,g_k$ be $k$ labeled trees, we define $f\overleftarrow{\shuffle}(g_{1},\dots,g_{k})$ to be the labeled tree obtained by first shifting the labels of each $g_i$ up by $\ideg(f)$, and then identifying the root of $g_i$ with the $i$-th leaf of $f$ for all $i$.

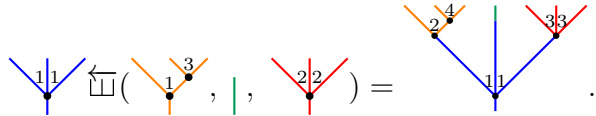
\begin{Example}
	$$\begin{tikzpicture}[scale=0.25,baseline=0pt]
	\draw[blue, thick] (0,-1) -- (0,0);
	\draw[blue, thick] (0,0) -- (-2,2);
	\draw[blue, thick] (0,0) -- (0,2);
	\draw[blue, thick] (0,0) -- (2,2);
	\filldraw[black] (0,0) circle (5pt)  {};
	\node at (-0.4,1)[font=\fontsize{7pt}{0}]{$1$};
	\node at (0.4,1)[font=\fontsize{7pt}{0}]{$1$};
	\end{tikzpicture}\overleftarrow{\shuffle} (\begin{tikzpicture}[scale=0.25,baseline=0pt]
	\draw[orange, thick] (0,-1) -- (0,0);
	\draw[orange, thick] (0,0) -- (-2,2);
	\draw[orange, thick] (1,1) -- (0,2);
	\draw[orange, thick] (0,0) -- (2,2);
	\filldraw[black] (0,0) circle (5pt)  {};
	\filldraw[black] (1,1) circle (5pt)  {};
	\node at (0,0.7)[font=\fontsize{7pt}{0}]{$1$};
	\node at (1,1.7)[font=\fontsize{7pt}{0}]{$3$};
	\end{tikzpicture},~\begin{tikzpicture}[scale=0.25,baseline=0pt]
	\draw[ForestGreen, thick] (0,-1) -- (0,1);
	\end{tikzpicture}~,~\begin{tikzpicture}[scale=0.25,baseline=0pt]
	\draw[red, thick] (0,-1) -- (0,0);
	\draw[red, thick] (0,0) -- (-2,2);
	\draw[red, thick] (0,0) -- (0,2);
	\draw[red, thick] (0,0) -- (2,2);
	\filldraw[black] (0,0) circle (5pt)  {};
	\node at (-0.4,1)[font=\fontsize{7pt}{0}]{$2$};
	\node at (0.4,1)[font=\fontsize{7pt}{0}]{$2$};
	\end{tikzpicture})=\begin{tikzpicture}[scale=0.2,baseline=0pt]
	\draw[blue, thick] (0,-1) -- (0,0);
	\draw[blue, thick] (0,0) -- (-4,4);
	\draw[blue, thick] (0,0) -- (0,5);
	\draw[blue, thick] (0,0) -- (4,4);
	\draw[orange, thick] (-4,4) -- (-6,6);
	\draw[orange, thick] (-4,4) -- (-2,6);
	\draw[orange, thick] (-3,5) -- (-4,6);
	\draw[ForestGreen, thick] (0,5) -- (0,6);
	\draw[red, thick] (4,4) -- (6,6);
	\draw[red, thick] (4,4) -- (4,6);
	\draw[red, thick] (4,4) -- (2,6);
	\filldraw[black] (0,0) circle (5pt)  {};
	\filldraw[black] (-4,4) circle (5pt)  {};
	\filldraw[black] (4,4) circle (5pt)  {};
	\filldraw[black] (-3,5) circle (5pt)  {};
	\node at (-0.4,1)[font=\fontsize{7pt}{0}]{$1$};
	\node at (0.4,1)[font=\fontsize{7pt}{0}]{$1$};
	\node at (-4,4.7)[font=\fontsize{7pt}{0}]{$2$};
	\node at (-3,5.7)[font=\fontsize{7pt}{0}]{$4$};
	\node at (3.6,5)[font=\fontsize{7pt}{0}]{$3$};
	\node at (4.4,5)[font=\fontsize{7pt}{0}]{$3$};
	\end{tikzpicture}.$$
\end{Example}

Consider the map $\pi:\LPT\to \PT$ by forgetting the node labels. It is easy to see that $\pi$ preserves the operations we defined above.

\begin{Proposition}\label{prop:pi-property}The map $\pi$ satisfies the following properties:
\begin{enumerate}
\item ${}^{i}(\pi f)=\pi\left({}^{i}f\right)$, $(\pi f)^{i}=\pi\left(f^{i}\right)$;
\item $\pi(g/h)=(\pi g)/(\pi h)$;
\item $\pi(g\backslash h)=(\pi g)\backslash (\pi h)$;
\item $\pi(f\overleftarrow{\shuffle}(g_{1},\dots,g_{k}))=(\pi f)\shuffle(\pi g_{1},\dots,\pi g_{k})$.
\end{enumerate}
\end{Proposition}

\subsection{Binary trees and the Tamari order}\label{sec:tamari}

\begin{Definition}
	A (rooted planar) \textbf{binary tree} is a tree where each node has either $0$ or $2$ children, the left child and the right child. The set of planar binary trees is denoted by $\PBT$. Let $t$ be a binary tree of degree $n$, we define its \textbf{descent} set to be $\Des(t)=\{1\leq i\leq n-1: \text{the }(i+1)\text{-th leaf of }t\text{ is a right child}\}$.
\end{Definition}

\begin{Example}
	{Binary trees with $3$ internal nodes.
		
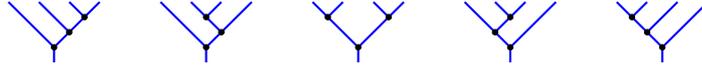
\begin{center}
		\begin{tikzpicture}[scale=0.2,baseline=0pt]
			\draw[blue, thick] (0,-1) -- (0,0);
			\draw[blue, thick] (0,0) -- (3,3);
			\draw[blue, thick] (0,0) -- (-3,3);
			\draw[blue, thick] (1,1) -- (-1,3);
			\draw[blue, thick] (2,2) -- (1,3);
			\filldraw[black] (0,0) circle (5pt)  {};
			\filldraw[black] (1,1) circle (5pt)  {};
			\filldraw[black] (2,2) circle (5pt)  {};
	
			\draw[blue, thick] (10,-1) -- (10,0);
			\draw[blue, thick] (10,0) -- (13,3);
			\draw[blue, thick] (10,0) -- (7,3);
			\draw[blue, thick] (11,1) -- (9,3);
			\draw[blue, thick] (10,2) -- (11,3);
			\filldraw[black] (10,0) circle (5pt)  {};
			\filldraw[black] (11,1) circle (5pt)  {};
			\filldraw[black] (10,2) circle (5pt)  {};
		
			\draw[blue, thick] (20,-1) -- (20,0);
			\draw[blue, thick] (20,0) -- (23,3);
			\draw[blue, thick] (20,0) -- (17,3);
			\draw[blue, thick] (22,2) -- (21,3);
			\draw[blue, thick] (18,2) -- (19,3);
			\filldraw[black] (20,0) circle (5pt)  {};
			\filldraw[black] (22,2) circle (5pt)  {};
			\filldraw[black] (18,2) circle (5pt)  {};
		
			\draw[blue, thick] (30,-1) -- (30,0);
			\draw[blue, thick] (30,0) -- (33,3);
			\draw[blue, thick] (30,0) -- (27,3);
			\draw[blue, thick] (29,1) -- (31,3);
			\draw[blue, thick] (30,2) -- (29,3);
			\filldraw[black] (30,0) circle (5pt)  {};
			\filldraw[black] (29,1) circle (5pt)  {};
			\filldraw[black] (30,2) circle (5pt)  {};
		
			\draw[blue, thick] (40,-1) -- (40,0);
			\draw[blue, thick] (40,0) -- (37,3);
			\draw[blue, thick] (40,0) -- (43,3);
			\draw[blue, thick] (39,1) -- (41,3);
			\draw[blue, thick] (38,2) -- (39,3);
			\filldraw[black] (40,0) circle (5pt)  {};
			\filldraw[black] (39,1) circle (5pt)  {};
			\filldraw[black] (38,2) circle (5pt)  {};
		\end{tikzpicture}
	\end{center}
	Their descent sets are $\emptyset, \{2\}, \{1\}, \{2\}, \{1,2\}$ respectively.}
\end{Example}

A left rotation at node $x$ is the following operation, where $y$ is the right child of $x$, and $A,B,C$ are branches.

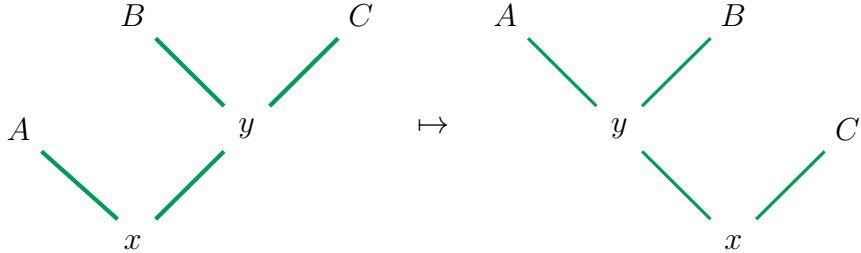
\begin{center}
\begin{tikzpicture}[scale=0.8]
	\node at (0,0) {\begin{tikzpicture}
		\node at (0,0) {$x$};
		\node at (1.5,1.5) {$y$};
		\node at (-1.5,1.5) {$A$};
		\node at (0,3) {$B$};
		\node at (3,3) {$C$};
		
		\draw[ForestGreen,ultra thick] (0.3,0.3) -- (1.2,1.2);
		\draw[ForestGreen,ultra thick] (-0.2,0.3) -- (-1.2,1.2);
		\draw[ForestGreen,ultra thick] (1.8,1.8) -- (2.7,2.7);
		\draw[ForestGreen,ultra thick] (1.2,1.8) -- (0.3,2.7);
		\end{tikzpicture}};
	\node at (4,0) {$\mapsto$};
	\node at (8,0) {\begin{tikzpicture}
		\node at (0,0) {$x$};
		\node at (-1.5,1.5) {$y$};
		\node at (-3,3) {$A$};
		\node at (0,3) {$B$};
		\node at (1.5,1.5) {$C$};
		
		\draw[ForestGreen,very thick] (0.3,0.3) -- (1.2,1.2);
		\draw[ForestGreen,very thick] (-0.3,0.3) -- (-1.2,1.2);
		\draw[ForestGreen,very thick] (-1.8,1.8) -- (-2.7,2.7);
		\draw[ForestGreen,very thick] (-1.2,1.8) -- (-0.3,2.7);
		\end{tikzpicture}};
\end{tikzpicture}
\end{center}

\begin{Definition}
The \textit{Tamari order}, $\leq_{T}$, is a partial order on PBT whose covering relations are left rotations. (see for example Figure~\ref{fig:tamari34})
\end{Definition}

\begin{Remark}
	In the literature, trees are sometimes drawn downwards. Some people use the dual poset as the Tamari poset, depending on reflecting or rotating the tree upwards. Our present convention is the dual of the Tamari order used in \cite{AS06}. 
\end{Remark}

\begin{figure}
	\begin{center}
		\begin{tikzpicture}
		\node at (0,0) {\begin{tikzpicture}[scale=0.25,baseline=0pt]
			\node at (0,0) {\begin{tikzpicture}[scale=0.4]
				\draw[ForestGreen, ultra thick] (-3,6.5) -- (-3,8.5);
				\draw[ForestGreen, ultra thick] (1,1.5) -- (2.7,6);
				\draw[ForestGreen, ultra thick] (1,14) -- (2.7,9.5);
				\draw[ForestGreen, ultra thick] (-1,14) -- (-2.7,12);
				\draw[ForestGreen, ultra thick] (-1,1.5) -- (-2.7,3.5);
				\node () at (0,0){\begin{tikzpicture}[scale=0.2,baseline=0pt]
					\draw[blue, thick] (0,-1) -- (0,0);
					\draw[blue, thick] (0,0) -- (3,3);
					\draw[blue, thick] (2,2) -- (1,3);
					\draw[blue, thick] (1,1) -- (-1,3);
					\draw[blue, thick] (0,0) -- (-3,3);
					\filldraw[black] (0,0) circle (5pt)  {};
					\filldraw[black] (1,1) circle (5pt)  {};
					\filldraw[black] (2,2) circle (5pt)  {};\end{tikzpicture}};
				\node () at (-3,5){\begin{tikzpicture}[scale=0.25,baseline=0pt]
					\draw[blue, thick] (0,-1) -- (0,0);
					\draw[blue, thick] (0,0) -- (3,3);
					\draw[blue, thick] (0,2) -- (1,3);
					\draw[blue, thick] (1,1) -- (-1,3);
					\draw[blue, thick] (0,0) -- (-3,3);
					\filldraw[black] (0,0) circle (5pt)  {};
					\filldraw[black] (1,1) circle (5pt)  {};
					\filldraw[black] (0,2) circle (5pt)  {};\end{tikzpicture}};
				\node () at (-3,10){\begin{tikzpicture}[scale=0.25,baseline=0pt]
					\draw[blue, thick] (0,-1) -- (0,0);
					\draw[blue, thick] (0,0) -- (3,3);
					\draw[blue, thick] (-1,1) -- (1,3);
					\draw[blue, thick] (0,2) -- (-1,3);
					\draw[blue, thick] (0,0) -- (-3,3);
					\filldraw[black] (0,0) circle (5pt)  {};
					\filldraw[black] (-1,1) circle (5pt)  {};
					\filldraw[black] (0,2) circle (5pt)  {};\end{tikzpicture}};
				\node () at (3,7.5){\begin{tikzpicture}[scale=0.25,baseline=0pt]
					\draw[blue, thick] (0,-1) -- (0,0);
					\draw[blue, thick] (0,0) -- (3,3);
					\draw[blue, thick] (2,2) -- (1,3);
					\draw[blue, thick] (-2,2) -- (-1,3);
					\draw[blue, thick] (0,0) -- (-3,3);
					\filldraw[black] (0,0) circle (5pt)  {};
					\filldraw[black] (2,2) circle (5pt)  {};
					\filldraw[black] (-2,2) circle (5pt)  {};\end{tikzpicture}};
				\node () at (0,15){\begin{tikzpicture}[scale=0.25,baseline=0pt]
					\draw[blue, thick] (0,-1) -- (0,0);
					\draw[blue, thick] (0,0) -- (3,3);
					\draw[blue, thick] (-1,1) -- (1,3);
					\draw[blue, thick] (-2,2) -- (-1,3);
					\draw[blue, thick] (0,0) -- (-3,3);
					\filldraw[black] (0,0) circle (5pt)  {};
					\filldraw[black] (-1,1) circle (5pt)  {};
					\filldraw[black] (-2,2) circle (5pt)  {};\end{tikzpicture}};
		\end{tikzpicture}};
			\end{tikzpicture}};
		\node at (8,0) {
		\begin{tikzpicture}[scale=0.25,baseline=0pt]
		\draw[ForestGreen, ultra thick] (-2.2,1.5) -- (-8.7,5.4);
		\draw[ForestGreen, ultra thick] (-10,8) -- (-10,10);
		\draw[ForestGreen, ultra thick] (-10,14) -- (-10,16);
		\draw[ForestGreen, ultra thick] (2.2,1.5) -- (13.7,8.4);
		\draw[ForestGreen, ultra thick] (15,11) -- (15,13);
		\draw[ForestGreen, ultra thick] (15,17) -- (15,19);
		\draw[ForestGreen, ultra thick] (-7.8,19.5) -- (3.7,26.4);
		\draw[ForestGreen, ultra thick] (-2.2,7.5) -- (-8.7,11.4);
		\draw[ForestGreen, ultra thick] (-7.2,16.5) -- (-8.7,17.4);
		\draw[ForestGreen, ultra thick] (2.8,10.5) -- (-3.7,14.4);
		\draw[ForestGreen, ultra thick] (7.8,13.5) -- (1.3,17.4);
		\draw[ForestGreen, ultra thick] (12.8,16.5) -- (6.3,20.4);
		\draw[ForestGreen, ultra thick] (12.8,22.5) -- (6.3,26.4);
		\draw[ForestGreen, ultra thick] (12.8,10.5) -- (11.3,11.4);
		\draw[ForestGreen, ultra thick] (0,2) -- (0,4);
		\draw[ForestGreen, ultra thick] (5,23) -- (5,25);
		\draw[ForestGreen, ultra thick] (2.2,7.5) -- (3.7,8.4);
		\draw[ForestGreen, ultra thick] (7.2,10.5) -- (8.7,11.4);
		\draw[ForestGreen, ultra thick] (-2.8,16.5) -- (-1.3,17.4);
		\draw[ForestGreen, ultra thick] (2.2,19.5) -- (3.7,20.4);
		\draw[ForestGreen, ultra thick] (-7.8,7.5) -- (-5.8,8.7);
		\draw[ForestGreen, ultra thick] (10.7,18.6) -- (13.7,20.4);
		\draw[ForestGreen, ultra thick] (-4.3,9.6) -- (-0.8,11.6);
		\draw[ForestGreen, ultra thick] (0.7,12.6) -- (4.2,14.6);
		\draw[ForestGreen, ultra thick] (5.7,15.6) -- (9.2,17.6);
		\node () at (0,0){\begin{tikzpicture}[scale=0.1]
			\draw[blue, thick] (0,-1) -- (0,0);
			\draw[blue, thick] (0,0) -- (4,4);
			\draw[blue, thick] (0,0) -- (-4,4);
			\draw[blue, thick] (1,1) -- (-2,4);
			\draw[blue, thick] (2,2) -- (0,4);
			\draw[blue, thick] (3,3) -- (2,4);
			\filldraw[black] (0,0) circle (5pt)  {};
			\filldraw[black] (1,1) circle (5pt)  {};
			\filldraw[black] (2,2) circle (5pt)  {};
			\filldraw[black] (3,3) circle (5pt)  {};\end{tikzpicture}};
		\node () at (-10,6){\begin{tikzpicture}[scale=0.1]
			\draw[blue, thick] (0,-1) -- (0,0);
			\draw[blue, thick] (0,0) -- (4,4);
			\draw[blue, thick] (0,0) -- (-4,4);
			\draw[blue, thick] (-3,3) -- (-2,4);
			\draw[blue, thick] (2,2) -- (0,4);
			\draw[blue, thick] (3,3) -- (2,4);
			\filldraw[black] (0,0) circle (5pt)  {};
			\filldraw[black] (-3,3) circle (5pt)  {};
			\filldraw[black] (2,2) circle (5pt)  {};
			\filldraw[black] (3,3) circle (5pt)  {};\end{tikzpicture}};
		\node () at (0,6){\begin{tikzpicture}[scale=0.1]
			\draw[blue, thick] (0,-1) -- (0,0);
			\draw[blue, thick] (0,0) -- (4,4);
			\draw[blue, thick] (0,0) -- (-4,4);
			\draw[blue, thick] (1,1) -- (-2,4);
			\draw[blue, thick] (2,2) -- (0,4);
			\draw[blue, thick] (1,3) -- (2,4);
			\filldraw[black] (0,0) circle (5pt)  {};
			\filldraw[black] (1,1) circle (5pt)  {};
			\filldraw[black] (2,2) circle (5pt)  {};
			\filldraw[black] (1,3) circle (5pt)  {};\end{tikzpicture}};
		\node () at (-10,12){\begin{tikzpicture}[scale=0.1]
			\draw[blue, thick] (0,-1) -- (0,0);
			\draw[blue, thick] (0,0) -- (4,4);
			\draw[blue, thick] (0,0) -- (-4,4);
			\draw[blue, thick] (-3,3) -- (-2,4);
			\draw[blue, thick] (2,2) -- (0,4);
			\draw[blue, thick] (1,3) -- (2,4);
			\filldraw[black] (0,0) circle (5pt)  {};
			\filldraw[black] (-3,3) circle (5pt)  {};
			\filldraw[black] (2,2) circle (5pt)  {};
			\filldraw[black] (1,3) circle (5pt)  {};\end{tikzpicture}};
		\node () at (-10,18){\begin{tikzpicture}[scale=0.1]
			\draw[blue, thick] (0,-1) -- (0,0);
			\draw[blue, thick] (0,0) -- (4,4);
			\draw[blue, thick] (0,0) -- (-4,4);
			\draw[blue, thick] (-3,3) -- (-2,4);
			\draw[blue, thick] (-1,1) -- (2,4);
			\draw[blue, thick] (1,3) -- (0,4);
			\filldraw[black] (0,0) circle (5pt)  {};
			\filldraw[black] (-3,3) circle (5pt)  {};
			\filldraw[black] (-1,1) circle (5pt)  {};
			\filldraw[black] (1,3) circle (5pt)  {};\end{tikzpicture}};
		\node () at (5,9){\begin{tikzpicture}[scale=0.1]
			\draw[blue, thick] (0,-1) -- (0,0);
			\draw[blue, thick] (0,0) -- (4,4);
			\draw[blue, thick] (0,0) -- (-4,4);
			\draw[blue, thick] (1,1) -- (-2,4);
			\draw[blue, thick] (0,2) -- (2,4);
			\draw[blue, thick] (1,3) -- (0,4);
			\filldraw[black] (0,0) circle (5pt)  {};
			\filldraw[black] (1,1) circle (5pt)  {};
			\filldraw[black] (0,2) circle (5pt)  {};
			\filldraw[black] (1,3) circle (5pt)  {};\end{tikzpicture}};
		\node () at (-5,15){\begin{tikzpicture}[scale=0.1]
			\draw[blue, thick] (0,-1) -- (0,0);
			\draw[blue, thick] (0,0) -- (4,4);
			\draw[blue, thick] (0,0) -- (-4,4);
			\draw[blue, thick] (-1,1) -- (2,4);
			\draw[blue, thick] (0,2) -- (-2,4);
			\draw[blue, thick] (1,3) -- (0,4);
			\filldraw[black] (0,0) circle (5pt)  {};
			\filldraw[black] (-1,1) circle (5pt)  {};
			\filldraw[black] (0,2) circle (5pt)  {};
			\filldraw[black] (1,3) circle (5pt)  {};\end{tikzpicture}};
		\node () at (0,18){\begin{tikzpicture}[scale=0.1]
			\draw[blue, thick] (0,-1) -- (0,0);
			\draw[blue, thick] (0,0) -- (4,4);
			\draw[blue, thick] (0,0) -- (-4,4);
			\draw[blue, thick] (-1,1) -- (2,4);
			\draw[blue, thick] (0,2) -- (-2,4);
			\draw[blue, thick] (-1,3) -- (0,4);
			\filldraw[black] (0,0) circle (5pt)  {};
			\filldraw[black] (-1,1) circle (5pt)  {};
			\filldraw[black] (0,2) circle (5pt)  {};
			\filldraw[black] (-1,3) circle (5pt)  {};\end{tikzpicture}};
		\node () at (10,12){\begin{tikzpicture}[scale=0.1]
			\draw[blue, thick] (0,-1) -- (0,0);
			\draw[blue, thick] (0,0) -- (4,4);
			\draw[blue, thick] (0,0) -- (-4,4);
			\draw[blue, thick] (1,1) -- (-2,4);
			\draw[blue, thick] (0,2) -- (2,4);
			\draw[blue, thick] (-1,3) -- (0,4);
			\filldraw[black] (0,0) circle (5pt)  {};
			\filldraw[black] (1,1) circle (5pt)  {};
			\filldraw[black] (0,2) circle (5pt)  {};
			\filldraw[black] (-1,3) circle (5pt)  {};\end{tikzpicture}};
		\node () at (15,9){\begin{tikzpicture}[scale=0.1]
			\draw[blue, thick] (0,-1) -- (0,0);
			\draw[blue, thick] (0,0) -- (4,4);
			\draw[blue, thick] (0,0) -- (-4,4);
			\draw[blue, thick] (1,1) -- (-2,4);
			\draw[blue, thick] (-1,3) -- (0,4);
			\draw[blue, thick] (3,3) -- (2,4);
			\filldraw[black] (0,0) circle (5pt)  {};
			\filldraw[black] (1,1) circle (5pt)  {};
			\filldraw[black] (-1,3) circle (5pt)  {};
			\filldraw[black] (3,3) circle (5pt)  {};\end{tikzpicture}};
		\node () at (15,15){\begin{tikzpicture}[scale=0.1]
			\draw[blue, thick] (0,-1) -- (0,0);
			\draw[blue, thick] (0,0) -- (4,4);
			\draw[blue, thick] (0,0) -- (-4,4);
			\draw[blue, thick] (-2,2) -- (0,4);
			\draw[blue, thick] (-1,3) -- (-2,4);
			\draw[blue, thick] (3,3) -- (2,4);
			\filldraw[black] (0,0) circle (5pt)  {};
			\filldraw[black] (-2,2) circle (5pt)  {};
			\filldraw[black] (-1,3) circle (5pt)  {};
			\filldraw[black] (3,3) circle (5pt)  {};\end{tikzpicture}};
		\node () at (15,21){\begin{tikzpicture}[scale=0.1]
			\draw[blue, thick] (0,-1) -- (0,0);
			\draw[blue, thick] (0,0) -- (4,4);
			\draw[blue, thick] (0,0) -- (-4,4);
			\draw[blue, thick] (-2,2) -- (0,4);
			\draw[blue, thick] (-3,3) -- (-2,4);
			\draw[blue, thick] (3,3) -- (2,4);
			\filldraw[black] (0,0) circle (5pt)  {};
			\filldraw[black] (-2,2) circle (5pt)  {};
			\filldraw[black] (-3,3) circle (5pt)  {};
			\filldraw[black] (3,3) circle (5pt)  {};\end{tikzpicture}};
		\node () at (5,21){\begin{tikzpicture}[scale=0.1]
			\draw[blue, thick] (0,-1) -- (0,0);
			\draw[blue, thick] (0,0) -- (4,4);
			\draw[blue, thick] (0,0) -- (-4,4);
			\draw[blue, thick] (-2,2) -- (0,4);
			\draw[blue, thick] (-1,3) -- (-2,4);
			\draw[blue, thick] (-1,1) -- (2,4);
			\filldraw[black] (0,0) circle (5pt)  {};
			\filldraw[black] (-2,2) circle (5pt)  {};
			\filldraw[black] (-1,3) circle (5pt)  {};
			\filldraw[black] (-1,1) circle (5pt)  {};\end{tikzpicture}};
		\node () at (5,27){\begin{tikzpicture}[scale=0.1]
			\draw[blue, thick] (0,-1) -- (0,0);
			\draw[blue, thick] (0,0) -- (4,4);
			\draw[blue, thick] (0,0) -- (-4,4);
			\draw[blue, thick] (-2,2) -- (0,4);
			\draw[blue, thick] (-3,3) -- (-2,4);
			\draw[blue, thick] (-1,1) -- (2,4);
			\filldraw[black] (0,0) circle (5pt)  {};
			\filldraw[black] (-2,2) circle (5pt)  {};
			\filldraw[black] (-3,3) circle (5pt)  {};
			\filldraw[black] (-1,1) circle (5pt)  {};\end{tikzpicture}};
	\end{tikzpicture}};
	\end{tikzpicture}
\end{center}
   \caption{\sl The Tamari order on binary trees of degree $3$ and $4$.}\label{fig:tamari34}
\end{figure}
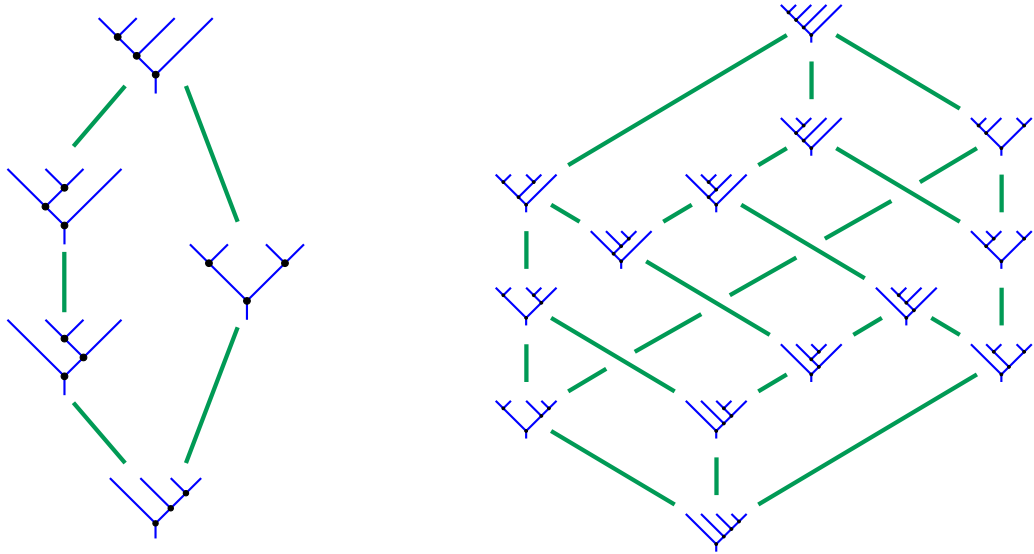

\subsection{Hopf algebra of binary trees}

In this section, we recall the Loday-Ronco Hopf algebra of binary trees \cite{LR98,AS06}.

\begin{Definition}
	As a graded vector space, (the dual of) the Loday-Ronco Hopf algebra, $\YSym$ is  $\displaystyle\bigoplus_{n\geq 0}\Bbbk$-$\PBT_n$ where $\PBT_n$ is the set of binary trees with $n+1$ leaves. The fundamental basis of $\YSym$ is $\{F_t:t\in\PBT\}$.
\end{Definition}

The comultiplication is by deconcatenation
$$\Delta(F_t)=\sum_{t\mapsto(t_1,t_2)}F_{t_1}\otimes F_{t_2}=\sum_{i=0}^{\deg f}F_{{}^it}\otimes F_{t^i}.$$
In what follows, let $\Delta_+$ be defined by $\Delta= Id\otimes 1 + 1\otimes Id \ +\  \Delta_+$.

\begin{Example}
	Consider the binary tree $t=$
	\begin{tikzpicture}[scale=0.2,baseline=5pt]
		\draw[blue, thick] (0,-1) -- (0,0);
		\draw[blue, thick] (0,0) -- (5,5);
		\draw[blue, thick] (0,0) -- (-5,5);
		\draw[blue, thick] (2,2) -- (-1,5);
		\draw[blue, thick] (1,3) -- (3,5);
		\draw[blue, thick] (2,4) -- (1,5);
		\draw[blue, thick] (-4,4) -- (-3,5);
		\filldraw[black] (0,0) circle (5pt)  {};
		\filldraw[black] (2,2) circle (5pt)  {};
		\filldraw[black] (1,3) circle (5pt)  {};
		\filldraw[black] (2,4) circle (5pt)  {};
		\filldraw[black] (-4,4) circle (5pt)  {};
	\end{tikzpicture},
	$$\Delta_+(F_t)=
	F_{\begin{tikzpicture}[scale=0.2,baseline=0pt]
		\draw[blue, thick] (0,-1) -- (0,0);
		\draw[blue, thick] (0,0) -- (1,1);
		\draw[blue, thick] (0,0) -- (-1,1);
		\filldraw[black] (0,0) circle (5pt)  {};
	\end{tikzpicture}}\otimes
	F_{\begin{tikzpicture}[scale=0.2,baseline=0pt]
		\draw[blue, thick] (0,-1) -- (0,0);
		\draw[blue, thick] (0,0) -- (4,4);
		\draw[blue, thick] (0,0) -- (-4,4);
		\draw[blue, thick] (1,1) -- (-2,4);
		\draw[blue, thick] (0,2) -- (2,4);
		\draw[blue, thick] (1,3) -- (0,4);
		\filldraw[black] (0,0) circle (5pt)  {};
		\filldraw[black] (1,1) circle (5pt)  {};
		\filldraw[black] (0,2) circle (5pt)  {};
		\filldraw[black] (1,3) circle (5pt)  {};
	\end{tikzpicture}}+
	F_{\begin{tikzpicture}[scale=0.2,baseline=0pt]
		\draw[blue, thick] (0,-1) -- (0,0);
		\draw[blue, thick] (0,0) -- (2,2);
		\draw[blue, thick] (0,0) -- (-2,2);
		\draw[blue, thick] (-1,1) -- (0,2);
		\filldraw[black] (0,0) circle (5pt)  {};
		\filldraw[black] (-1,1) circle (5pt)  {};
	\end{tikzpicture}}\otimes
	F_{\begin{tikzpicture}[scale=0.2,baseline=0pt]
		\draw[blue, thick] (0,-1) -- (0,0);
		\draw[blue, thick] (0,0) -- (3,3);
		\draw[blue, thick] (0,0) -- (-3,3);
		\draw[blue, thick] (-1,1) -- (1,3);
		\draw[blue, thick] (0,2) -- (-1,3);
		\filldraw[black] (0,0) circle (5pt)  {};
		\filldraw[black] (-1,1) circle (5pt)  {};
		\filldraw[black] (0,2) circle (5pt)  {};
	\end{tikzpicture}}+
	F_{\begin{tikzpicture}[scale=0.2,baseline=0pt]
		\draw[blue, thick] (0,-1) -- (0,0);
		\draw[blue, thick] (0,0) -- (3,3);
		\draw[blue, thick] (0,0) -- (-3,3);
		\draw[blue, thick] (2,2) -- (1,3);
		\draw[blue, thick] (-2,2) -- (-1,3);
		\filldraw[black] (0,0) circle (5pt)  {};
		\filldraw[black] (2,2) circle (5pt)  {};
		\filldraw[black] (-2,2) circle (5pt)  {};
	\end{tikzpicture}}\otimes
	F_{\begin{tikzpicture}[scale=0.2,baseline=0pt]
		\draw[blue, thick] (0,-1) -- (0,0);
		\draw[blue, thick] (0,0) -- (2,2);
		\draw[blue, thick] (0,0) -- (-2,2);
		\draw[blue, thick] (-1,1) -- (0,2);
		\filldraw[black] (0,0) circle (5pt)  {};
		\filldraw[black] (-1,1) circle (5pt)  {};
	\end{tikzpicture}}+
	F_{\begin{tikzpicture}[scale=0.2,baseline=0pt]
		\draw[blue, thick] (0,-1) -- (0,0);
		\draw[blue, thick] (0,0) -- (4,4);
		\draw[blue, thick] (0,0) -- (-4,4);
		\draw[blue, thick] (2,2) -- (0,4);
		\draw[blue, thick] (3,3) -- (2,4);
		\draw[blue, thick] (-3,3) -- (-2,4);
		\filldraw[black] (0,0) circle (5pt)  {};
		\filldraw[black] (2,2) circle (5pt)  {};
		\filldraw[black] (3,3) circle (5pt)  {};
		\filldraw[black] (-3,3) circle (5pt)  {};
	\end{tikzpicture}}\otimes
	F_{\begin{tikzpicture}[scale=0.2,baseline=0pt]
		\draw[blue, thick] (0,-1) -- (0,0);
		\draw[blue, thick] (0,0) -- (1,1);
		\draw[blue, thick] (0,0) -- (-1,1);
		\filldraw[black] (0,0) circle (5pt)  {};
	\end{tikzpicture}}.$$
\end{Example}

It is not hard to check that this comultiplication is coassociative.

To compute the product $F_s\cdot F_t$ in $\YSym$, first deconcatenate $t$ into $\deg(s)+1$ pieces, then attach them to the leaves of $s$:

$$F_s\cdot F_t=\sum_{t\mapsto(t_1,\dots,t_{\deg(s)+1})}F_{s\shuffle (t_1,\dots,t_{\deg(s)+1})}.$$

\begin{Example}
	Let $s=\begin{tikzpicture}[scale=0.2,baseline=0pt]
	\draw[blue, thick] (0,-1) -- (0,0);
	\draw[blue, thick] (0,0) -- (2,2);
	\draw[blue, thick] (0,0) -- (-2,2);
	\draw[blue, thick] (-1,1) -- (0,2);
	\filldraw[black] (0,0) circle (5pt)  {};
	\filldraw[black] (-1,1) circle (5pt)  {};
	\end{tikzpicture}$ and $t=\begin{tikzpicture}[scale=0.2,baseline=0pt]
	\draw[red, thick] (0,-1) -- (0,0);
	\draw[red, thick] (0,0) -- (2,2);
	\draw[red, thick] (0,0) -- (-2,2);
	\draw[red, thick] (1,1) -- (0,2);
	\filldraw[black] (0,0) circle (5pt)  {};
	\filldraw[black] (1,1) circle (5pt)  {};
	\end{tikzpicture}$. Then
	$$F_s\cdot F_t=
	F_{\begin{tikzpicture}[scale=0.2,baseline=0pt]
		\draw[blue, thick] (0,-1) -- (0,0);
		\draw[blue, thick] (0,0) -- (2,2);
		\draw[blue, thick] (0,0) -- (-4,4);
		\draw[red, thick] (2,2) -- (4,4);
		\draw[red, thick] (2,2) -- (0,4);
		\draw[red, thick] (3,3) -- (2,4);
		\draw[blue, thick] (-3,3) -- (-2,4);
		\filldraw[black] (0,0) circle (5pt)  {};
		\filldraw[black] (2,2) circle (5pt)  {};
		\filldraw[black] (3,3) circle (5pt)  {};
		\filldraw[black] (-3,3) circle (5pt)  {};
	\end{tikzpicture}}+
	F_{\begin{tikzpicture}[scale=0.2,baseline=0pt]
		\draw[blue, thick] (0,-1) -- (0,0);
		\draw[blue, thick] (0,0) -- (4,4);
		\draw[blue, thick] (0,0) -- (-4,4);
		\draw[blue, thick] (-1,1) -- (0,2);
		\draw[red, thick] (0,2) -- (2,4);
		\draw[red, thick] (0,2) -- (-2,4);
		\draw[red, thick] (1,3) -- (0,4);
		\filldraw[black] (0,0) circle (5pt)  {};
		\filldraw[black] (-1,1) circle (5pt)  {};
		\filldraw[black] (0,2) circle (5pt)  {};
		\filldraw[black] (1,3) circle (5pt)  {};
	\end{tikzpicture}}+
	F_{\begin{tikzpicture}[scale=0.2,baseline=0pt]
		\draw[blue, thick] (0,-1) -- (0,0);
		\draw[blue, thick] (0,0) -- (4,4);
		\draw[blue, thick] (0,0) -- (-2,2);
		\draw[red, thick] (-2,2) -- (-4,4);
		\draw[blue, thick] (-1,1) -- (2,4);
		\draw[red, thick] (-2,2) -- (0,4);
		\draw[red, thick] (-1,3) -- (-2,4);
		\filldraw[black] (0,0) circle (5pt)  {};
		\filldraw[black] (-1,1) circle (5pt)  {};
		\filldraw[black] (-2,2) circle (5pt)  {};
		\filldraw[black] (-1,3) circle (5pt)  {};
	\end{tikzpicture}}+
	F_{\begin{tikzpicture}[scale=0.2,baseline=0pt]
		\draw[blue, thick] (0,-1) -- (0,0);
		\draw[blue, thick] (0,0) -- (3,3);
		\draw[red, thick] (3,3) -- (4,4);
		\draw[blue, thick] (0,0) -- (-4,4);
		\draw[red, thick] (3,3) -- (2,4);
		\draw[blue, thick] (-2,2) -- (-1,3);
		\draw[red, thick] (-1,3) -- (0,4);
		\draw[red, thick] (-1,3) -- (-2,4);
		\filldraw[black] (0,0) circle (5pt)  {};
		\filldraw[black] (3,3) circle (5pt)  {};
		\filldraw[black] (-2,2) circle (5pt)  {};
		\filldraw[black] (-1,3) circle (5pt)  {};
	\end{tikzpicture}}+
	F_{\begin{tikzpicture}[scale=0.2,baseline=0pt]
		\draw[blue, thick] (0,-1) -- (0,0);
		\draw[blue, thick] (0,0) -- (3,3);
		\draw[red, thick] (3,3) -- (4,4);
		\draw[blue, thick] (0,0) -- (-3,3);
		\draw[red, thick] (-3,3) -- (-4,4);
		\draw[red, thick] (3,3) -- (2,4);
		\draw[blue, thick] (-2,2) -- (0,4);
		\draw[red, thick] (-3,3) -- (-2,4);
		\filldraw[black] (0,0) circle (5pt)  {};
		\filldraw[black] (3,3) circle (5pt)  {};
		\filldraw[black] (-2,2) circle (5pt)  {};
		\filldraw[black] (-3,3) circle (5pt)  {};
	\end{tikzpicture}}+
	F_{\begin{tikzpicture}[scale=0.2,baseline=0pt]
		\draw[blue, thick] (0,-1) -- (0,0);
		\draw[blue, thick] (0,0) -- (4,4);
		\draw[blue, thick] (0,0) -- (-3,3);
		\draw[red, thick] (-3,3) -- (-4,4);
		\draw[blue, thick] (-1,1) -- (1,3);
		\draw[red, thick] (1,3) -- (2,4);
		\draw[red, thick] (-3,3) -- (-2,4);
		\draw[red, thick] (1,3) -- (0,4);
		\filldraw[black] (0,0) circle (5pt)  {};
		\filldraw[black] (-1,1) circle (5pt)  {};
		\filldraw[black] (-3,3) circle (5pt)  {};
		\filldraw[black] (1,3) circle (5pt)  {};
	\end{tikzpicture}}$$
\end{Example}

\begin{Remark}
Note that our present multiplication is in the opposite order of that from \cite{AS06}.
\end{Remark}

It is known that the dual multiplication can be described using intervals on the Tamari order. (This looks different from the formula in 
\cite{LR02} because we took the dual of their convention for the Tamari order.)

\begin{Proposition}\label{prop:tree}\cite{LR02}
	Let $s$ and $t$ be two binary trees, and let $F_s^*$, $F_t^*$ be their dual basis element in the graded dual $\YSym^*$. Then,
	$$m(F_s^*\otimes F_t^*)=\sum_{s\backslash t\leq_{T} u\leq_{T} s/t}F_u^*.$$
\end{Proposition}

In \cite{AS06} the authors use the Tamari order to define a cofree basis of $\YSym$.

\begin{Definition}
	The monomial basis $\{M_t\}$ of $\YSym$, is defined by $\displaystyle F_s=\sum_{t\geq_{T} s}M_t$. The monomial basis is well-defined and is indeed a basis because of the triangularity obtained via M\"{o}bius inversion.
\end{Definition}

To define the comultiplication on the monomial basis, we introduce the idea of global descent.

\begin{Definition}
	Let $t$ be a binary tree of degree $n$. Its \textbf{global descent} set is $\GD(t)=\{i\in[n-1]:t=s/r \text{ for some }s, r\text{ and }\deg(s)=i\}$.
\end{Definition}

For example, when $t=\begin{tikzpicture}[scale=0.2,baseline=0pt]
		\draw[blue, thick] (0,-1) -- (0,0);
		\draw[blue, thick] (0,0) -- (7,7);
		\draw[blue, thick] (5,5) -- (3,7);
		\draw[blue, thick] (4,6) -- (5,7);
		\draw[blue, thick] (0,0) -- (-7,7);
		\draw[red, thick] (-6,6) -- (-5,7);
		\draw[blue, thick] (-1,5) -- (-3,7);
		\draw[blue, thick] (-2,6) -- (-1,7);
		\draw[red, thick] (-3,3) -- (1,7);
		\filldraw[black] (0,0) circle (5pt)  {};
		\filldraw[black] (5,5) circle (5pt)  {};
		\filldraw[black] (4,6) circle (5pt)  {};
		\filldraw[black] (-6,6) circle (5pt)  {};
		\filldraw[black] (-1,5) circle (5pt)  {};
		\filldraw[black] (-2,6) circle (5pt)  {};
		\filldraw[black] (-3,3) circle (5pt)  {};
	\end{tikzpicture}$, we have that $\GD(t)=\{1,4\}$, as it is indicated by the red branches.

 \begin{Proposition}[Theorems 5.1, 4.2, and 6.1 \cite{AS06}]
\label{prop:ysym-monomial}  
\begin{enumerate}
    \item The coproduct of monomial
basis elements is given by
\[
\Delta_{+}(M_{t})=\sum_{i\in\GD(f)}M_{{}^{i}t}\otimes M_{t^{i}}.
\]
\item The product of monomial basis elements has non-negative integer coefficients. 

\item In the antipode image $\mathcal{S}(M_t)$, all
terms have the same sign, equal to the parity of $|\GD(t)|+1$.
\end{enumerate}

\end{Proposition}

The references above give exact formulas for the multiplication and antipode; these may be obtained using Proposition \ref{prop:axiomysym} and the theorems in Section \ref{sec:axioms-monomials}.

\subsection{Permutations and weak order}\label{sec:permweak}
In this article we view permutations via their correspondence to labeled binary trees (see for example \cite[Prop. 2.3]{LR98}).

\begin{Definition}
	A \textbf{permutation} of degree $n$ is a labeled binary tree with $n$ internal nodes, where the labeling is a bijection to the set $\{1,2,\dots,n\}$ and such that the label of each internal node is smaller than any of the labels of its children. The set of permutations is denoted by $\Perm$. The \textbf{descent} set of a permutation $\sigma$ is the descent set of its underlying binary tree, denoted by $\Des(\sigma)$. In Figure \ref{fig:permutations_of_3} we illustrate the set of permutations of $[3]$.
\end{Definition}

This definition is equivalent to the classical definition of a permutations as a bijection $\sigma:[n]\rightarrow [n]$. For a permutation $\sigma$, we define $\sigma(i)$ to be the $i$-th label reading from left to right of the open area above the internal nodes. This is the reason we put the label in the area above the node. The descent set of $\sigma$ can be equivalently defined in the classical definition as $\Des(\sigma)=\{i:\sigma(i)>\sigma(i+1)\}$.

\begin{Definition}
	For a permutation $\sigma$, its \textbf{inversion} set is $\Inv(\sigma)=\{(i,j):i<j,\sigma(i)>\sigma(j)\}$. For permutations of $n$, we can construct a partial order by inclusion of inversion sets. This order is called the \textbf{(left) weak order}, denoted by $\leq_{w}$. (see for example Figure~\ref{fig:weak4})
	
	The \textbf{inversed-inversion} set of $\sigma$ is $\iInv(\sigma)=\{(a,b):a<b,\sigma^{-1}(a)>\sigma^{-1}(b)\}$. The weak order can be equivalently defined as: $\sigma$ is covered by $\tau$ if $\tau=T_a\circ\sigma$ and $|\iInv(\tau)|=|\iInv(\sigma)|+1$ where $T_a=(a,a+1)$ is the simple transposition.
\end{Definition}

\begin{figure}
\begin{center}
	\begin{tikzpicture}
	\node at (10,0) {
	\begin{tikzpicture}[scale=0.22,baseline=0pt]
	\draw[ForestGreen, ultra thick] (4,2) -- (12,7);
	\draw[ForestGreen, ultra thick] (-4,2) -- (-12,7);
	\draw[ForestGreen, ultra thick] (0,3) -- (0,5);
	\draw[ForestGreen, ultra thick] (17,10.5) -- (19,16);
	\draw[ForestGreen, ultra thick] (-17,10.5) -- (-19,16);
	\draw[ForestGreen, ultra thick] (3,11) -- (8,16);
	\draw[ForestGreen, ultra thick] (12,11) -- (6.6,13.7);
	\draw[ForestGreen, ultra thick] (5.4,14.3) -- (2,16);
	\draw[ForestGreen, ultra thick] (-3,11) -- (-8,16);
	\draw[ForestGreen, ultra thick] (-12,11) -- (-6.6,13.7);
	\draw[ForestGreen, ultra thick] (-5.4,14.3) -- (-2,16);
	\draw[ForestGreen, ultra thick] (21.5,20.5) -- (24,26);
	\draw[ForestGreen, ultra thick] (18.5,20.5) -- (17.4,23.1);
	\draw[ForestGreen, ultra thick] (16.8,24.2) -- (16,26);
	\draw[ForestGreen, ultra thick] (-21.5,20.5) -- (-24,26);
	\draw[ForestGreen, ultra thick] (-18.5,20.5) -- (-16,26);
	\draw[ForestGreen, ultra thick] (11.5,20.5) -- (23,27);
	\draw[ForestGreen, ultra thick] (13.5,30.5) -- (11,36);
	\draw[ForestGreen, ultra thick] (-1.5,20.5) -- (-2.6,23.1);
	\draw[ForestGreen, ultra thick] (-3.2,24.2) -- (-4,26);
	\draw[ForestGreen, ultra thick] (-8.5,20.5) -- (3,27);
	\draw[ForestGreen, ultra thick] (3.5,30.5) -- (2.4,33.1);
	\draw[ForestGreen, ultra thick] (1.8,34.2) -- (1,36);
	\draw[ForestGreen, ultra thick] (-3.5,30.5) -- (8,37);
	\draw[ForestGreen, ultra thick] (-11.5,20.5) -- (-14,26);
	\draw[ForestGreen, ultra thick] (8.5,20.5) -- (6,26);
	\draw[ForestGreen, ultra thick] (23.5,30.5) -- (21,36);
	\draw[ForestGreen, ultra thick] (-6.5,30.5) -- (-9,36);
	\draw[ForestGreen, ultra thick] (-16.5,30.5) -- (-17.6,33.1);
	\draw[ForestGreen, ultra thick] (-18.2,34.2) -- (-19,36);
	\draw[ForestGreen, ultra thick] (-23.5,30.5) -- (-12,37);
	\draw[ForestGreen, ultra thick] (16.5,30.5) -- (19,36);
	\draw[ForestGreen, ultra thick] (-25,30.5) -- (-22,37);
	\draw[ForestGreen, ultra thick] (-16.5,46) -- (-19,41);
	\draw[ForestGreen, ultra thick] (-12.5,47) -- (-1.5,41);
	\draw[ForestGreen, ultra thick] (-2.5,47) -- (-5.5,44);
	\draw[ForestGreen, ultra thick] (-6.5,43) -- (-8.5,41);
	\draw[ForestGreen, ultra thick] (-12.5,50.5) -- (-3,56);
	\draw[ForestGreen, ultra thick] (16.5,46) -- (19,41);
	\draw[ForestGreen, ultra thick] (12.5,47) -- (1.5,41);
	\draw[ForestGreen, ultra thick] (2.5,47) -- (5.5,44);
	\draw[ForestGreen, ultra thick] (6.5,43) -- (8.5,41);
	\draw[ForestGreen, ultra thick] (12.5,50.5) -- (3,56);
	\draw[ForestGreen, ultra thick] (0,50.5) -- (0,53);
	
	\node () at (0,0){\begin{tikzpicture}[scale=0.2,baseline=0pt]
		\draw[blue, thick] (0,-1) -- (0,0);
		\draw[blue, thick] (0,0) -- (4,4);
		\draw[blue, thick] (0,0) -- (-4,4);
		\draw[blue, thick] (1,1) -- (-2,4);
		\draw[blue, thick] (2,2) -- (0,4);
		\draw[blue, thick] (3,3) -- (2,4);
		\filldraw[black] (0,0) circle (5pt)  {};
		\filldraw[black] (1,1) circle (5pt)  {};
		\filldraw[black] (2,2) circle (5pt)  {};
		\filldraw[black] (3,3) circle (5pt)  {};
		\node at (0,0.7)[font=\fontsize{7pt}{0}]{$1$};
		\node at (1,1.7)[font=\fontsize{7pt}{0}]{$2$};
		\node at (2,2.7)[font=\fontsize{7pt}{0}]{$3$};
		\node at (3,3.7)[font=\fontsize{7pt}{0}]{$4$};
		\end{tikzpicture}};
	\node () at (-15,8){\begin{tikzpicture}[scale=0.2,baseline=0pt]
		\draw[blue, thick] (0,-1) -- (0,0);
		\draw[blue, thick] (0,0) -- (4,4);
		\draw[blue, thick] (0,0) -- (-4,4);
		\draw[blue, thick] (-3,3) -- (-2,4);
		\draw[blue, thick] (2,2) -- (0,4);
		\draw[blue, thick] (3,3) -- (2,4);
		\filldraw[black] (0,0) circle (5pt)  {};
		\filldraw[black] (-3,3) circle (5pt)  {};
		\filldraw[black] (2,2) circle (5pt)  {};
		\filldraw[black] (3,3) circle (5pt)  {};
		\node at (0,0.7)[font=\fontsize{7pt}{0}]{$1$};
		\node at (-3,3.7)[font=\fontsize{7pt}{0}]{$2$};
		\node at (2,2.7)[font=\fontsize{7pt}{0}]{$3$};
		\node at (3,3.7)[font=\fontsize{7pt}{0}]{$4$};
		\end{tikzpicture}};
	\node () at (-20,18){\begin{tikzpicture}[scale=0.2,baseline=0pt]
		\draw[blue, thick] (0,-1) -- (0,0);
		\draw[blue, thick] (0,0) -- (4,4);
		\draw[blue, thick] (0,0) -- (-4,4);
		\draw[blue, thick] (-3,3) -- (-2,4);
		\draw[blue, thick] (2,2) -- (0,4);
		\draw[blue, thick] (3,3) -- (2,4);
		\filldraw[black] (0,0) circle (5pt)  {};
		\filldraw[black] (-3,3) circle (5pt)  {};
		\filldraw[black] (2,2) circle (5pt)  {};
		\filldraw[black] (3,3) circle (5pt)  {};
		\node at (0,0.7)[font=\fontsize{7pt}{0}]{$1$};
		\node at (-3,3.7)[font=\fontsize{7pt}{0}]{$3$};
		\node at (2,2.7)[font=\fontsize{7pt}{0}]{$2$};
		\node at (3,3.7)[font=\fontsize{7pt}{0}]{$4$};
		\end{tikzpicture}};
	\node () at (-25,28){\begin{tikzpicture}[scale=0.2,baseline=0pt]
		\draw[blue, thick] (0,-1) -- (0,0);
		\draw[blue, thick] (0,0) -- (4,4);
		\draw[blue, thick] (0,0) -- (-4,4);
		\draw[blue, thick] (-3,3) -- (-2,4);
		\draw[blue, thick] (2,2) -- (0,4);
		\draw[blue, thick] (3,3) -- (2,4);
		\filldraw[black] (0,0) circle (5pt)  {};
		\filldraw[black] (-3,3) circle (5pt)  {};
		\filldraw[black] (2,2) circle (5pt)  {};
		\filldraw[black] (3,3) circle (5pt)  {};
		\node at (0,0.7)[font=\fontsize{7pt}{0}]{$1$};
		\node at (-3,3.7)[font=\fontsize{7pt}{0}]{$4$};
		\node at (2,2.7)[font=\fontsize{7pt}{0}]{$2$};
		\node at (3,3.7)[font=\fontsize{7pt}{0}]{$3$};
		\end{tikzpicture}};
	\node () at (0,8){\begin{tikzpicture}[scale=0.2,baseline=0pt]
		\draw[blue, thick] (0,-1) -- (0,0);
		\draw[blue, thick] (0,0) -- (4,4);
		\draw[blue, thick] (0,0) -- (-4,4);
		\draw[blue, thick] (1,1) -- (-2,4);
		\draw[blue, thick] (-1,3) -- (0,4);
		\draw[blue, thick] (3,3) -- (2,4);
		\filldraw[black] (0,0) circle (5pt)  {};
		\filldraw[black] (1,1) circle (5pt)  {};
		\filldraw[black] (-1,3) circle (5pt)  {};
		\filldraw[black] (3,3) circle (5pt)  {};
		\node at (0,0.7)[font=\fontsize{7pt}{0}]{$1$};
		\node at (1,1.7)[font=\fontsize{7pt}{0}]{$2$};
		\node at (-1,3.7)[font=\fontsize{7pt}{0}]{$3$};
		\node at (3,3.7)[font=\fontsize{7pt}{0}]{$4$};
		\end{tikzpicture}};
	\node () at (10,18){\begin{tikzpicture}[scale=0.2,baseline=0pt]
		\draw[blue, thick] (0,-1) -- (0,0);
		\draw[blue, thick] (0,0) -- (4,4);
		\draw[blue, thick] (0,0) -- (-4,4);
		\draw[blue, thick] (1,1) -- (-2,4);
		\draw[blue, thick] (-1,3) -- (0,4);
		\draw[blue, thick] (3,3) -- (2,4);
		\filldraw[black] (0,0) circle (5pt)  {};
		\filldraw[black] (1,1) circle (5pt)  {};
		\filldraw[black] (-1,3) circle (5pt)  {};
		\filldraw[black] (3,3) circle (5pt)  {};
		\node at (0,0.7)[font=\fontsize{7pt}{0}]{$1$};
		\node at (1,1.7)[font=\fontsize{7pt}{0}]{$2$};
		\node at (3,3.7)[font=\fontsize{7pt}{0}]{$3$};
		\node at (-1,3.7)[font=\fontsize{7pt}{0}]{$4$};
		\end{tikzpicture}};
	\node () at (15,8){\begin{tikzpicture}[scale=0.2,baseline=0pt]
		\draw[blue, thick] (0,-1) -- (0,0);
		\draw[blue, thick] (0,0) -- (4,4);
		\draw[blue, thick] (0,0) -- (-4,4);
		\draw[blue, thick] (1,1) -- (-2,4);
		\draw[blue, thick] (2,2) -- (0,4);
		\draw[blue, thick] (1,3) -- (2,4);
		\filldraw[black] (0,0) circle (5pt)  {};
		\filldraw[black] (1,1) circle (5pt)  {};
		\filldraw[black] (2,2) circle (5pt)  {};
		\filldraw[black] (1,3) circle (5pt)  {};
		\node at (0,0.7)[font=\fontsize{7pt}{0}]{$1$};
		\node at (1,1.7)[font=\fontsize{7pt}{0}]{$2$};
		\node at (2,2.7)[font=\fontsize{7pt}{0}]{$3$};
		\node at (1,3.7)[font=\fontsize{7pt}{0}]{$4$};
		\end{tikzpicture}};
	\node () at (20,18){\begin{tikzpicture}[scale=0.2,baseline=0pt]
		\draw[blue, thick] (0,-1) -- (0,0);
		\draw[blue, thick] (0,0) -- (4,4);
		\draw[blue, thick] (0,0) -- (-4,4);
		\draw[blue, thick] (1,1) -- (-2,4);
		\draw[blue, thick] (0,2) -- (2,4);
		\draw[blue, thick] (1,3) -- (0,4);
		\filldraw[black] (0,0) circle (5pt)  {};
		\filldraw[black] (1,1) circle (5pt)  {};
		\filldraw[black] (0,2) circle (5pt)  {};
		\filldraw[black] (1,3) circle (5pt)  {};
		\node at (0,0.7)[font=\fontsize{7pt}{0}]{$1$};
		\node at (1,1.7)[font=\fontsize{7pt}{0}]{$2$};
		\node at (0,2.7)[font=\fontsize{7pt}{0}]{$3$};
		\node at (1,3.7)[font=\fontsize{7pt}{0}]{$4$};
		\end{tikzpicture}};
	\node () at (25,28){\begin{tikzpicture}[scale=0.2,baseline=0pt]
		\draw[blue, thick] (0,-1) -- (0,0);
		\draw[blue, thick] (0,0) -- (4,4);
		\draw[blue, thick] (0,0) -- (-4,4);
		\draw[blue, thick] (1,1) -- (-2,4);
		\draw[blue, thick] (0,2) -- (2,4);
		\draw[blue, thick] (-1,3) -- (0,4);
		\filldraw[black] (0,0) circle (5pt)  {};
		\filldraw[black] (1,1) circle (5pt)  {};
		\filldraw[black] (0,2) circle (5pt)  {};
		\filldraw[black] (-1,3) circle (5pt)  {};
		\node at (0,0.7)[font=\fontsize{7pt}{0}]{$1$};
		\node at (1,1.7)[font=\fontsize{7pt}{0}]{$2$};
		\node at (0,2.7)[font=\fontsize{7pt}{0}]{$3$};
		\node at (-1,3.7)[font=\fontsize{7pt}{0}]{$4$};
		\end{tikzpicture}};
	\node () at (15,28){\begin{tikzpicture}[scale=0.2,baseline=0pt]
		\draw[blue, thick] (0,-1) -- (0,0);
		\draw[blue, thick] (0,0) -- (4,4);
		\draw[blue, thick] (0,0) -- (-4,4);
		\draw[blue, thick] (-1,1) -- (2,4);
		\draw[blue, thick] (0,2) -- (-2,4);
		\draw[blue, thick] (1,3) -- (0,4);
		\filldraw[black] (0,0) circle (5pt)  {};
		\filldraw[black] (-1,1) circle (5pt)  {};
		\filldraw[black] (0,2) circle (5pt)  {};
		\filldraw[black] (1,3) circle (5pt)  {};
		\node at (0,0.7)[font=\fontsize{7pt}{0}]{$1$};
		\node at (-1,1.7)[font=\fontsize{7pt}{0}]{$2$};
		\node at (0,2.7)[font=\fontsize{7pt}{0}]{$3$};
		\node at (1,3.7)[font=\fontsize{7pt}{0}]{$4$};
		\end{tikzpicture}};
	\node () at (20,38){\begin{tikzpicture}[scale=0.2,baseline=0pt]
		\draw[blue, thick] (0,-1) -- (0,0);
		\draw[blue, thick] (0,0) -- (4,4);
		\draw[blue, thick] (0,0) -- (-4,4);
		\draw[blue, thick] (-1,1) -- (2,4);
		\draw[blue, thick] (0,2) -- (-2,4);
		\draw[blue, thick] (-1,3) -- (0,4);
		\filldraw[black] (0,0) circle (5pt)  {};
		\filldraw[black] (-1,1) circle (5pt)  {};
		\filldraw[black] (0,2) circle (5pt)  {};
		\filldraw[black] (-1,3) circle (5pt)  {};
		\node at (0,0.7)[font=\fontsize{7pt}{0}]{$1$};
		\node at (-1,1.7)[font=\fontsize{7pt}{0}]{$2$};
		\node at (0,2.7)[font=\fontsize{7pt}{0}]{$3$};
		\node at (-1,3.7)[font=\fontsize{7pt}{0}]{$4$};
		\end{tikzpicture}};
	\node () at (15,48){\begin{tikzpicture}[scale=0.2,baseline=0pt]
		\draw[blue, thick] (0,-1) -- (0,0);
		\draw[blue, thick] (0,0) -- (4,4);
		\draw[blue, thick] (0,0) -- (-4,4);
		\draw[blue, thick] (-2,2) -- (0,4);
		\draw[blue, thick] (-1,3) -- (-2,4);
		\draw[blue, thick] (-1,1) -- (2,4);
		\filldraw[black] (0,0) circle (5pt)  {};
		\filldraw[black] (-2,2) circle (5pt)  {};
		\filldraw[black] (-1,3) circle (5pt)  {};
		\filldraw[black] (-1,1) circle (5pt)  {};
		\node at (0,0.7)[font=\fontsize{7pt}{0}]{$1$};
		\node at (-1,1.7)[font=\fontsize{7pt}{0}]{$2$};
		\node at (-2,2.7)[font=\fontsize{7pt}{0}]{$3$};
		\node at (-1,3.7)[font=\fontsize{7pt}{0}]{$4$};
		\end{tikzpicture}};
	\node () at (0,56){\begin{tikzpicture}[scale=0.2,baseline=0pt]
		\draw[blue, thick] (0,-1) -- (0,0);
		\draw[blue, thick] (0,0) -- (4,4);
		\draw[blue, thick] (0,0) -- (-4,4);
		\draw[blue, thick] (-2,2) -- (0,4);
		\draw[blue, thick] (-3,3) -- (-2,4);
		\draw[blue, thick] (-1,1) -- (2,4);
		\filldraw[black] (0,0) circle (5pt)  {};
		\filldraw[black] (-2,2) circle (5pt)  {};
		\filldraw[black] (-3,3) circle (5pt)  {};
		\filldraw[black] (-1,1) circle (5pt)  {};
		\node at (0,0.7)[font=\fontsize{7pt}{0}]{$1$};
		\node at (-1,1.7)[font=\fontsize{7pt}{0}]{$2$};
		\node at (-2,2.7)[font=\fontsize{7pt}{0}]{$3$};
		\node at (-3,3.7)[font=\fontsize{7pt}{0}]{$4$};
		\end{tikzpicture}};
	\node () at (0,18){\begin{tikzpicture}[scale=0.2,baseline=0pt]
		\draw[blue, thick] (0,-1) -- (0,0);
		\draw[blue, thick] (0,0) -- (4,4);
		\draw[blue, thick] (0,0) -- (-4,4);
		\draw[blue, thick] (-3,3) -- (-2,4);
		\draw[blue, thick] (2,2) -- (0,4);
		\draw[blue, thick] (1,3) -- (2,4);
		\filldraw[black] (0,0) circle (5pt)  {};
		\filldraw[black] (-3,3) circle (5pt)  {};
		\filldraw[black] (2,2) circle (5pt)  {};
		\filldraw[black] (1,3) circle (5pt)  {};
		\node at (0,0.7)[font=\fontsize{7pt}{0}]{$1$};
		\node at (-3,3.7)[font=\fontsize{7pt}{0}]{$2$};
		\node at (2,2.7)[font=\fontsize{7pt}{0}]{$3$};
		\node at (1,3.7)[font=\fontsize{7pt}{0}]{$4$};
		\end{tikzpicture}};
	\node () at (-5,28){\begin{tikzpicture}[scale=0.2,baseline=0pt]
		\draw[blue, thick] (0,-1) -- (0,0);
		\draw[blue, thick] (0,0) -- (4,4);
		\draw[blue, thick] (0,0) -- (-4,4);
		\draw[blue, thick] (-3,3) -- (-2,4);
		\draw[blue, thick] (2,2) -- (0,4);
		\draw[blue, thick] (1,3) -- (2,4);
		\filldraw[black] (0,0) circle (5pt)  {};
		\filldraw[black] (-3,3) circle (5pt)  {};
		\filldraw[black] (2,2) circle (5pt)  {};
		\filldraw[black] (1,3) circle (5pt)  {};
		\node at (0,0.7)[font=\fontsize{7pt}{0}]{$1$};
		\node at (-3,3.7)[font=\fontsize{7pt}{0}]{$3$};
		\node at (2,2.7)[font=\fontsize{7pt}{0}]{$2$};
		\node at (1,3.7)[font=\fontsize{7pt}{0}]{$4$};
		\end{tikzpicture}};
	\node () at (-10,38){\begin{tikzpicture}[scale=0.2,baseline=0pt]
		\draw[blue, thick] (0,-1) -- (0,0);
		\draw[blue, thick] (0,0) -- (4,4);
		\draw[blue, thick] (0,0) -- (-4,4);
		\draw[blue, thick] (-3,3) -- (-2,4);
		\draw[blue, thick] (2,2) -- (0,4);
		\draw[blue, thick] (1,3) -- (2,4);
		\filldraw[black] (0,0) circle (5pt)  {};
		\filldraw[black] (-3,3) circle (5pt)  {};
		\filldraw[black] (2,2) circle (5pt)  {};
		\filldraw[black] (1,3) circle (5pt)  {};
		\node at (0,0.7)[font=\fontsize{7pt}{0}]{$1$};
		\node at (-3,3.7)[font=\fontsize{7pt}{0}]{$4$};
		\node at (2,2.7)[font=\fontsize{7pt}{0}]{$2$};
		\node at (1,3.7)[font=\fontsize{7pt}{0}]{$3$};
		\end{tikzpicture}};
	\node () at (-10,18){\begin{tikzpicture}[scale=0.2,baseline=0pt]
		\draw[blue, thick] (0,-1) -- (0,0);
		\draw[blue, thick] (0,0) -- (4,4);
		\draw[blue, thick] (0,0) -- (-4,4);
		\draw[blue, thick] (-2,2) -- (0,4);
		\draw[blue, thick] (-1,3) -- (-2,4);
		\draw[blue, thick] (3,3) -- (2,4);
		\filldraw[black] (0,0) circle (5pt)  {};
		\filldraw[black] (-2,2) circle (5pt)  {};
		\filldraw[black] (-1,3) circle (5pt)  {};
		\filldraw[black] (3,3) circle (5pt)  {};
		\node at (0,0.7)[font=\fontsize{7pt}{0}]{$1$};
		\node at (-2,2.7)[font=\fontsize{7pt}{0}]{$2$};
		\node at (-1,3.7)[font=\fontsize{7pt}{0}]{$3$};
		\node at (3,3.7)[font=\fontsize{7pt}{0}]{$4$};
		\end{tikzpicture}};
	\node () at (5,28){\begin{tikzpicture}[scale=0.2,baseline=0pt]
		\draw[blue, thick] (0,-1) -- (0,0);
		\draw[blue, thick] (0,0) -- (4,4);
		\draw[blue, thick] (0,0) -- (-4,4);
		\draw[blue, thick] (-2,2) -- (0,4);
		\draw[blue, thick] (-1,3) -- (-2,4);
		\draw[blue, thick] (3,3) -- (2,4);
		\filldraw[black] (0,0) circle (5pt)  {};
		\filldraw[black] (-2,2) circle (5pt)  {};
		\filldraw[black] (-1,3) circle (5pt)  {};
		\filldraw[black] (3,3) circle (5pt)  {};
		\node at (0,0.7)[font=\fontsize{7pt}{0}]{$1$};
		\node at (-2,2.7)[font=\fontsize{7pt}{0}]{$2$};
		\node at (-1,3.7)[font=\fontsize{7pt}{0}]{$4$};
		\node at (3,3.7)[font=\fontsize{7pt}{0}]{$3$};
		\end{tikzpicture}};
	\node () at (0,38){\begin{tikzpicture}[scale=0.2,baseline=0pt]
		\draw[blue, thick] (0,-1) -- (0,0);
		\draw[blue, thick] (0,0) -- (4,4);
		\draw[blue, thick] (0,0) -- (-4,4);
		\draw[blue, thick] (-2,2) -- (0,4);
		\draw[blue, thick] (-1,3) -- (-2,4);
		\draw[blue, thick] (3,3) -- (2,4);
		\filldraw[black] (0,0) circle (5pt)  {};
		\filldraw[black] (-2,2) circle (5pt)  {};
		\filldraw[black] (-1,3) circle (5pt)  {};
		\filldraw[black] (3,3) circle (5pt)  {};
		\node at (0,0.7)[font=\fontsize{7pt}{0}]{$1$};
		\node at (-2,2.7)[font=\fontsize{7pt}{0}]{$3$};
		\node at (-1,3.7)[font=\fontsize{7pt}{0}]{$4$};
		\node at (3,3.7)[font=\fontsize{7pt}{0}]{$2$};
		\end{tikzpicture}};
	\node () at (10,38){\begin{tikzpicture}[scale=0.2,baseline=0pt]
		\draw[blue, thick] (0,-1) -- (0,0);
		\draw[blue, thick] (0,0) -- (4,4);
		\draw[blue, thick] (0,0) -- (-4,4);
		\draw[blue, thick] (-3,3) -- (-2,4);
		\draw[blue, thick] (-1,1) -- (2,4);
		\draw[blue, thick] (1,3) -- (0,4);
		\filldraw[black] (0,0) circle (5pt)  {};
		\filldraw[black] (-3,3) circle (5pt)  {};
		\filldraw[black] (-1,1) circle (5pt)  {};
		\filldraw[black] (1,3) circle (5pt)  {};
		\node at (0,0.7)[font=\fontsize{7pt}{0}]{$1$};
		\node at (-1,1.7)[font=\fontsize{7pt}{0}]{$2$};
		\node at (-3,3.7)[font=\fontsize{7pt}{0}]{$3$};
		\node at (1,3.7)[font=\fontsize{7pt}{0}]{$4$};
		\end{tikzpicture}};
	\node () at (0,48){\begin{tikzpicture}[scale=0.2,baseline=0pt]
		\draw[blue, thick] (0,-1) -- (0,0);
		\draw[blue, thick] (0,0) -- (4,4);
		\draw[blue, thick] (0,0) -- (-4,4);
		\draw[blue, thick] (-3,3) -- (-2,4);
		\draw[blue, thick] (-1,1) -- (2,4);
		\draw[blue, thick] (1,3) -- (0,4);
		\filldraw[black] (0,0) circle (5pt)  {};
		\filldraw[black] (-3,3) circle (5pt)  {};
		\filldraw[black] (-1,1) circle (5pt)  {};
		\filldraw[black] (1,3) circle (5pt)  {};
		\node at (0,0.7)[font=\fontsize{7pt}{0}]{$1$};
		\node at (-1,1.7)[font=\fontsize{7pt}{0}]{$2$};
		\node at (-3,3.7)[font=\fontsize{7pt}{0}]{$4$};
		\node at (1,3.7)[font=\fontsize{7pt}{0}]{$3$};
		\end{tikzpicture}};
	\node () at (-15,28){\begin{tikzpicture}[scale=0.2,baseline=0pt]
		\draw[blue, thick] (0,-1) -- (0,0);
		\draw[blue, thick] (0,0) -- (4,4);
		\draw[blue, thick] (0,0) -- (-4,4);
		\draw[blue, thick] (-2,2) -- (0,4);
		\draw[blue, thick] (-3,3) -- (-2,4);
		\draw[blue, thick] (3,3) -- (2,4);
		\filldraw[black] (0,0) circle (5pt)  {};
		\filldraw[black] (-2,2) circle (5pt)  {};
		\filldraw[black] (-3,3) circle (5pt)  {};
		\filldraw[black] (3,3) circle (5pt)  {};
		\node at (0,0.7)[font=\fontsize{7pt}{0}]{$1$};
		\node at (-2,2.7)[font=\fontsize{7pt}{0}]{$2$};
		\node at (-3,3.7)[font=\fontsize{7pt}{0}]{$3$};
		\node at (3,3.7)[font=\fontsize{7pt}{0}]{$4$};
		\end{tikzpicture}};
	\node () at (-20,38){\begin{tikzpicture}[scale=0.2,baseline=0pt]
		\draw[blue, thick] (0,-1) -- (0,0);
		\draw[blue, thick] (0,0) -- (4,4);
		\draw[blue, thick] (0,0) -- (-4,4);
		\draw[blue, thick] (-2,2) -- (0,4);
		\draw[blue, thick] (-3,3) -- (-2,4);
		\draw[blue, thick] (3,3) -- (2,4);
		\filldraw[black] (0,0) circle (5pt)  {};
		\filldraw[black] (-2,2) circle (5pt)  {};
		\filldraw[black] (-3,3) circle (5pt)  {};
		\filldraw[black] (3,3) circle (5pt)  {};
		\node at (0,0.7)[font=\fontsize{7pt}{0}]{$1$};
		\node at (-2,2.7)[font=\fontsize{7pt}{0}]{$2$};
		\node at (-3,3.7)[font=\fontsize{7pt}{0}]{$4$};
		\node at (3,3.7)[font=\fontsize{7pt}{0}]{$3$};
		\end{tikzpicture}};
	\node () at (-15,48){\begin{tikzpicture}[scale=0.2,baseline=0pt]
		\draw[blue, thick] (0,-1) -- (0,0);
		\draw[blue, thick] (0,0) -- (4,4);
		\draw[blue, thick] (0,0) -- (-4,4);
		\draw[blue, thick] (-2,2) -- (0,4);
		\draw[blue, thick] (-3,3) -- (-2,4);
		\draw[blue, thick] (3,3) -- (2,4);
		\filldraw[black] (0,0) circle (5pt)  {};
		\filldraw[black] (-2,2) circle (5pt)  {};
		\filldraw[black] (-3,3) circle (5pt)  {};
		\filldraw[black] (3,3) circle (5pt)  {};
		\node at (0,0.7)[font=\fontsize{7pt}{0}]{$1$};
		\node at (-2,2.7)[font=\fontsize{7pt}{0}]{$3$};
		\node at (-3,3.7)[font=\fontsize{7pt}{0}]{$4$};
		\node at (3,3.7)[font=\fontsize{7pt}{0}]{$2$};
		\end{tikzpicture}};
	\end{tikzpicture}};
	\end{tikzpicture}
\end{center}
\caption{\sl The weak order on permutations of $[4]$.}\label{fig:weak4}
\end{figure}
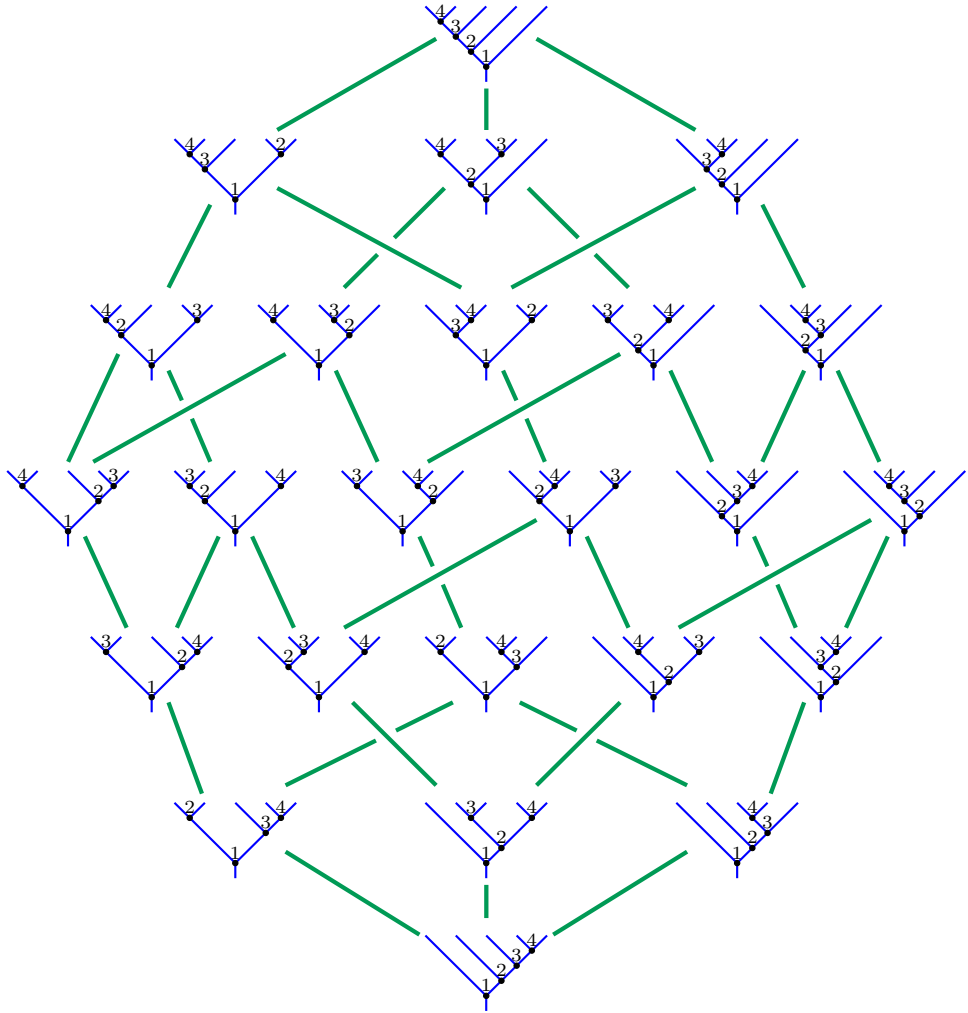

\subsection{Hopf algebra of permutations} \label{sec:ssym}
In this section, we recall the Malvenuto-Reutenauer Hopf algebra of permutations \cite{MR95}.

\begin{Definition}\label{def:ssym}
	As a graded vector space, Malvenuto-Reutenauer Hopf algebra, $\SSym$ is  $\displaystyle\bigoplus_{n\geq 0}\Bbbk\Perm_n$ where $\Perm_n$ is the set of permutations of degree $n$. By convention $\Perm_0$ is the set containing the empty permutation. The fundamental basis of $\SSym$ is $\{F_\sigma:\sigma\in\Perm\}$.
\end{Definition}

The multiplication and comultiplication rules for $\SSym$ are
$$F_\sigma\cdot F_\tau=\sum_{\tau\mapsto(\tau_1,\dots,\tau_{\deg(\sigma)+1})}F_{\sigma\overleftarrow{\shuffle}(\tau_1,\dots,\tau_{\deg(\sigma)+1})},$$
$$\Delta(F_\sigma)=\sum_{i=0}^{\deg(\sigma)}F_{{}^i\sigma}\otimes F_{\sigma^i}=\sum_{\sigma\mapsto(\sigma_1,\sigma_2)}F_{\std(\sigma_1)}\otimes F_{\std(\sigma_2)}.$$

\begin{Example}
		Let $\sigma=\begin{tikzpicture}[scale=0.2,baseline=0pt]
		\draw[blue, thick] (0,-1) -- (0,0);
		\draw[blue, thick] (0,0) -- (2,2);
		\draw[blue, thick] (0,0) -- (-2,2);
		\draw[blue, thick] (-1,1) -- (0,2);
		\filldraw[black] (0,0) circle (5pt)  {};
		\filldraw[black] (-1,1) circle (5pt)  {};
		\node at (0,0.7)[font=\fontsize{7pt}{0}]{$1$};
		\node at (-1,1.7)[font=\fontsize{7pt}{0}]{$2$};
	\end{tikzpicture}$ and $\tau=\begin{tikzpicture}[scale=0.2,baseline=0pt]
		\draw[red, thick] (0,-1) -- (0,0);
		\draw[red, thick] (0,0) -- (2,2);
		\draw[red, thick] (0,0) -- (-2,2);
		\draw[red, thick] (1,1) -- (0,2);
		\filldraw[black] (0,0) circle (5pt)  {};
		\filldraw[black] (1,1) circle (5pt)  {};
		\node at (0,0.7)[font=\fontsize{7pt}{0}]{$1$};
		\node at (1,1.7)[font=\fontsize{7pt}{0}]{$2$};
	\end{tikzpicture}$. Then
		$$F_\sigma\cdot F_\tau=
	F_{\begin{tikzpicture}[scale=0.2,baseline=0pt]
		\draw[blue, thick] (0,-1) -- (0,0);
		\draw[blue, thick] (0,0) -- (2,2);
		\draw[blue, thick] (0,0) -- (-4,4);
		\draw[red, thick] (2,2) -- (4,4);
		\draw[red, thick] (2,2) -- (0,4);
		\draw[red, thick] (3,3) -- (2,4);
		\draw[blue, thick] (-3,3) -- (-2,4);
		\filldraw[black] (0,0) circle (5pt)  {};
		\filldraw[black] (2,2) circle (5pt)  {};
		\filldraw[black] (3,3) circle (5pt)  {};
		\filldraw[black] (-3,3) circle (5pt)  {};
		\node at (0,0.7)[font=\fontsize{7pt}{0}]{$1$};
		\node at (-3,3.7)[font=\fontsize{7pt}{0}]{$2$};
		\node at (2,2.7)[font=\fontsize{7pt}{0}]{$3$};
		\node at (3,3.7)[font=\fontsize{7pt}{0}]{$4$};
	\end{tikzpicture}}+
	F_{\begin{tikzpicture}[scale=0.2,baseline=0pt]
		\draw[blue, thick] (0,-1) -- (0,0);
		\draw[blue, thick] (0,0) -- (4,4);
		\draw[blue, thick] (0,0) -- (-4,4);
		\draw[blue, thick] (-1,1) -- (0,2);
		\draw[red, thick] (0,2) -- (2,4);
		\draw[red, thick] (0,2) -- (-2,4);
		\draw[red, thick] (1,3) -- (0,4);
		\filldraw[black] (0,0) circle (5pt)  {};
		\filldraw[black] (-1,1) circle (5pt)  {};
		\filldraw[black] (0,2) circle (5pt)  {};
		\filldraw[black] (1,3) circle (5pt)  {};
		\node at (0,0.7)[font=\fontsize{7pt}{0}]{$1$};
		\node at (-1,1.7)[font=\fontsize{7pt}{0}]{$2$};
		\node at (0,2.7)[font=\fontsize{7pt}{0}]{$3$};
		\node at (1,3.7)[font=\fontsize{7pt}{0}]{$4$};
	\end{tikzpicture}}+
	F_{\begin{tikzpicture}[scale=0.2,baseline=0pt]
		\draw[blue, thick] (0,-1) -- (0,0);
		\draw[blue, thick] (0,0) -- (4,4);
		\draw[blue, thick] (0,0) -- (-2,2);
		\draw[red, thick] (-2,2) -- (-4,4);
		\draw[blue, thick] (-1,1) -- (2,4);
		\draw[red, thick] (-2,2) -- (0,4);
		\draw[red, thick] (-1,3) -- (-2,4);
		\filldraw[black] (0,0) circle (5pt)  {};
		\filldraw[black] (-1,1) circle (5pt)  {};
		\filldraw[black] (-2,2) circle (5pt)  {};
		\filldraw[black] (-1,3) circle (5pt)  {};
		\node at (0,0.7)[font=\fontsize{7pt}{0}]{$1$};
		\node at (-1,1.7)[font=\fontsize{7pt}{0}]{$2$};
		\node at (-2,2.7)[font=\fontsize{7pt}{0}]{$3$};
		\node at (-1,3.7)[font=\fontsize{7pt}{0}]{$4$};
	\end{tikzpicture}}+
	F_{\begin{tikzpicture}[scale=0.2,baseline=0pt]
		\draw[blue, thick] (0,-1) -- (0,0);
		\draw[blue, thick] (0,0) -- (3,3);
		\draw[red, thick] (3,3) -- (4,4);
		\draw[blue, thick] (0,0) -- (-4,4);
		\draw[red, thick] (3,3) -- (2,4);
		\draw[blue, thick] (-2,2) -- (-1,3);
		\draw[red, thick] (-1,3) -- (0,4);
		\draw[red, thick] (-1,3) -- (-2,4);
		\filldraw[black] (0,0) circle (5pt)  {};
		\filldraw[black] (3,3) circle (5pt)  {};
		\filldraw[black] (-2,2) circle (5pt)  {};
		\filldraw[black] (-1,3) circle (5pt)  {};
		\node at (0,0.7)[font=\fontsize{7pt}{0}]{$1$};
		\node at (-2,2.7)[font=\fontsize{7pt}{0}]{$2$};
		\node at (-1,3.7)[font=\fontsize{7pt}{0}]{$3$};
		\node at (3,3.7)[font=\fontsize{7pt}{0}]{$4$};
	\end{tikzpicture}}+
	F_{\begin{tikzpicture}[scale=0.2,baseline=0pt]
		\draw[blue, thick] (0,-1) -- (0,0);
		\draw[blue, thick] (0,0) -- (3,3);
		\draw[red, thick] (3,3) -- (4,4);
		\draw[blue, thick] (0,0) -- (-3,3);
		\draw[red, thick] (-3,3) -- (-4,4);
		\draw[red, thick] (3,3) -- (2,4);
		\draw[blue, thick] (-2,2) -- (0,4);
		\draw[red, thick] (-3,3) -- (-2,4);
		\filldraw[black] (0,0) circle (5pt)  {};
		\filldraw[black] (3,3) circle (5pt)  {};
		\filldraw[black] (-2,2) circle (5pt)  {};
		\filldraw[black] (-3,3) circle (5pt)  {};
		\node at (0,0.7)[font=\fontsize{7pt}{0}]{$1$};
		\node at (-2,2.7)[font=\fontsize{7pt}{0}]{$2$};
		\node at (-3,3.7)[font=\fontsize{7pt}{0}]{$3$};
		\node at (3,3.7)[font=\fontsize{7pt}{0}]{$4$};
	\end{tikzpicture}}+
	F_{\begin{tikzpicture}[scale=0.2,baseline=0pt]
		\draw[blue, thick] (0,-1) -- (0,0);
		\draw[blue, thick] (0,0) -- (4,4);
		\draw[blue, thick] (0,0) -- (-3,3);
		\draw[red, thick] (-3,3) -- (-4,4);
		\draw[blue, thick] (-1,1) -- (1,3);
		\draw[red, thick] (1,3) -- (2,4);
		\draw[red, thick] (-3,3) -- (-2,4);
		\draw[red, thick] (1,3) -- (0,4);
		\filldraw[black] (0,0) circle (5pt)  {};
		\filldraw[black] (-1,1) circle (5pt)  {};
		\filldraw[black] (-3,3) circle (5pt)  {};
		\filldraw[black] (1,3) circle (5pt)  {};
		\node at (0,0.7)[font=\fontsize{7pt}{0}]{$1$};
		\node at (-1,1.7)[font=\fontsize{7pt}{0}]{$2$};
		\node at (-3,3.7)[font=\fontsize{7pt}{0}]{$3$};
		\node at (1,3.7)[font=\fontsize{7pt}{0}]{$4$};
	\end{tikzpicture}}$$
	Let $\sigma=$
	\begin{tikzpicture}[scale=0.2,baseline=0pt]
		\draw[blue, thick] (0,-1) -- (0,0);
		\draw[blue, thick] (0,0) -- (5,5);
		\draw[blue, thick] (0,0) -- (-5,5);
		\draw[blue, thick] (2,2) -- (-1,5);
		\draw[blue, thick] (1,3) -- (3,5);
		\draw[blue, thick] (2,4) -- (1,5);
		\draw[blue, thick] (-4,4) -- (-3,5);
		\filldraw[black] (0,0) circle (5pt)  {};
		\filldraw[black] (2,2) circle (5pt)  {};
		\filldraw[black] (1,3) circle (5pt)  {};
		\filldraw[black] (2,4) circle (5pt)  {};
		\filldraw[black] (-4,4) circle (5pt)  {};
		\node at (0,0.7)[font=\fontsize{7pt}{0}]{$1$};
		\node at (2,2.7)[font=\fontsize{7pt}{0}]{$2$};
		\node at (1,3.7)[font=\fontsize{7pt}{0}]{$3$};
		\node at (-4,4.7)[font=\fontsize{7pt}{0}]{$4$};
		\node at (2,4.7)[font=\fontsize{7pt}{0}]{$5$};
	\end{tikzpicture}. Then
	$$\Delta_+(F_\sigma)=
	F_{\begin{tikzpicture}[scale=0.2,baseline=0pt]
		\draw[blue, thick] (0,-1) -- (0,0);
		\draw[blue, thick] (0,0) -- (1,1);
		\draw[blue, thick] (0,0) -- (-1,1);
		\filldraw[black] (0,0) circle (5pt)  {};
		\node at (0,0.7)[font=\fontsize{7pt}{0}]{$1$};
	\end{tikzpicture}}\otimes
	F_{\begin{tikzpicture}[scale=0.2,baseline=0pt]
		\draw[blue, thick] (0,-1) -- (0,0);
		\draw[blue, thick] (0,0) -- (4,4);
		\draw[blue, thick] (0,0) -- (-4,4);
		\draw[blue, thick] (1,1) -- (-2,4);
		\draw[blue, thick] (0,2) -- (2,4);
		\draw[blue, thick] (1,3) -- (0,4);
		\filldraw[black] (0,0) circle (5pt)  {};
		\filldraw[black] (1,1) circle (5pt)  {};
		\filldraw[black] (0,2) circle (5pt)  {};
		\filldraw[black] (1,3) circle (5pt)  {};
		\node at (0,0.7)[font=\fontsize{7pt}{0}]{$1$};
		\node at (1,1.7)[font=\fontsize{7pt}{0}]{$2$};
		\node at (0,2.7)[font=\fontsize{7pt}{0}]{$3$};
		\node at (1,3.7)[font=\fontsize{7pt}{0}]{$4$};
	\end{tikzpicture}}+
	F_{\begin{tikzpicture}[scale=0.2,baseline=0pt]
		\draw[blue, thick] (0,-1) -- (0,0);
		\draw[blue, thick] (0,0) -- (2,2);
		\draw[blue, thick] (0,0) -- (-2,2);
		\draw[blue, thick] (-1,1) -- (0,2);
		\filldraw[black] (0,0) circle (5pt)  {};
		\filldraw[black] (-1,1) circle (5pt)  {};
		\node at (0,0.7)[font=\fontsize{7pt}{0}]{$1$};
		\node at (-1,1.7)[font=\fontsize{7pt}{0}]{$2$};
	\end{tikzpicture}}\otimes
	F_{\begin{tikzpicture}[scale=0.2,baseline=0pt]
		\draw[blue, thick] (0,-1) -- (0,0);
		\draw[blue, thick] (0,0) -- (3,3);
		\draw[blue, thick] (0,0) -- (-3,3);
		\draw[blue, thick] (-1,1) -- (1,3);
		\draw[blue, thick] (0,2) -- (-1,3);
		\filldraw[black] (0,0) circle (5pt)  {};
		\filldraw[black] (-1,1) circle (5pt)  {};
		\filldraw[black] (0,2) circle (5pt)  {};
		\node at (0,0.7)[font=\fontsize{7pt}{0}]{$1$};
		\node at (-1,1.7)[font=\fontsize{7pt}{0}]{$2$};
		\node at (0,2.7)[font=\fontsize{7pt}{0}]{$3$};
	\end{tikzpicture}}+
	F_{\begin{tikzpicture}[scale=0.2,baseline=0pt]
		\draw[blue, thick] (0,-1) -- (0,0);
		\draw[blue, thick] (0,0) -- (3,3);
		\draw[blue, thick] (0,0) -- (-3,3);
		\draw[blue, thick] (2,2) -- (1,3);
		\draw[blue, thick] (-2,2) -- (-1,3);
		\filldraw[black] (0,0) circle (5pt)  {};
		\filldraw[black] (2,2) circle (5pt)  {};
		\filldraw[black] (-2,2) circle (5pt)  {};
		\node at (0,0.7)[font=\fontsize{7pt}{0}]{$1$};
		\node at (2,2.7)[font=\fontsize{7pt}{0}]{$2$};
		\node at (-2,2.7)[font=\fontsize{7pt}{0}]{$3$};
	\end{tikzpicture}}\otimes
	F_{\begin{tikzpicture}[scale=0.2,baseline=0pt]
		\draw[blue, thick] (0,-1) -- (0,0);
		\draw[blue, thick] (0,0) -- (2,2);
		\draw[blue, thick] (0,0) -- (-2,2);
		\draw[blue, thick] (-1,1) -- (0,2);
		\filldraw[black] (0,0) circle (5pt)  {};
		\filldraw[black] (-1,1) circle (5pt)  {};
		\node at (0,0.7)[font=\fontsize{7pt}{0}]{$1$};
		\node at (-1,1.7)[font=\fontsize{7pt}{0}]{$2$};
	\end{tikzpicture}}+
	F_{\begin{tikzpicture}[scale=0.2,baseline=0pt]
		\draw[blue, thick] (0,-1) -- (0,0);
		\draw[blue, thick] (0,0) -- (4,4);
		\draw[blue, thick] (0,0) -- (-4,4);
		\draw[blue, thick] (2,2) -- (0,4);
		\draw[blue, thick] (3,3) -- (2,4);
		\draw[blue, thick] (-3,3) -- (-2,4);
		\filldraw[black] (0,0) circle (5pt)  {};
		\filldraw[black] (2,2) circle (5pt)  {};
		\filldraw[black] (3,3) circle (5pt)  {};
		\filldraw[black] (-3,3) circle (5pt)  {};
		\node at (0,0.7)[font=\fontsize{7pt}{0}]{$1$};
		\node at (2,2.7)[font=\fontsize{7pt}{0}]{$2$};
		\node at (-3,3.7)[font=\fontsize{7pt}{0}]{$3$};
		\node at (3,3.7)[font=\fontsize{7pt}{0}]{$4$};
	\end{tikzpicture}}\otimes
	F_{\begin{tikzpicture}[scale=0.2,baseline=0pt]
		\draw[blue, thick] (0,-1) -- (0,0);
		\draw[blue, thick] (0,0) -- (1,1);
		\draw[blue, thick] (0,0) -- (-1,1);
		\filldraw[black] (0,0) circle (5pt)  {};
		\node at (0,0.7)[font=\fontsize{7pt}{0}]{$1$};
	\end{tikzpicture}}$$
\end{Example}

It is known that the dual multiplication can be described using intervals on the weak order. (Note that \cite{LR02} visualizes permutations as decreasing trees, so their definition of $\sigma\backslash \tau$ is what we call $\sigma/\tau$, and vice versa.)

\begin{Proposition}\label{prop:perm}\cite{LR02}
	Let $F_\sigma$ and $F_\tau$ be two permutations, and let $F_\sigma^*$, $F_\tau^*$ be their dual basis element in the graded dual $\SSym^*$. Then,
	$$m(F_\sigma^*\otimes F_\tau^*)=\sum_{\sigma\backslash \tau\leq_{w} \rho\leq_{w} \sigma/\tau}F_\rho^*.$$
\end{Proposition}

The weak order also leads to a cofree basis of $\SSym$.

\begin{Definition}\cite{AS05}
	The monomial basis of $\SSym$, $\{M_\sigma\}$, is defined using M\"obius inversion by the formula  $\displaystyle F_\tau=\sum_{\sigma\geq_{w} \tau}M_\sigma$. 
\end{Definition}

In \cite{AS05}, the authors provide formulas for the coproduct, product and antipode of these monomial basis elements. As in the case for $\YSym$,
the idea of global descent is important.

\begin{Definition} \label{def:permutation-globaldescent}
	Let $\sigma$ be a permutation of degree $n$. Its global descent set is $\GD(\sigma)=\{i\in[n-1]:\sigma=\tau/\rho \text{ for some }\tau, \rho\text{ and }\deg(\tau)=i\}$. Equivalently, $\GD(\sigma)=\{i\in[n-1]:\sigma(a)>\sigma(b)\text{ for all }a\leq i< b\}$.
\end{Definition}

For example, when $\sigma=\begin{tikzpicture}[scale=0.2,baseline=0pt]
		\draw[blue, thick] (0,-1) -- (0,0);
		\draw[blue, thick] (0,0) -- (7,7);
		\draw[blue, thick] (5,5) -- (3,7);
		\draw[blue, thick] (4,6) -- (5,7);
		\draw[blue, thick] (0,0) -- (-7,7);
		\draw[blue, thick] (-6,6) -- (-5,7);
		\draw[blue, thick] (-1,5) -- (-3,7);
		\draw[blue, thick] (-2,6) -- (-1,7);
		\draw[red, thick] (-3,3) -- (1,7);
		\filldraw[black] (0,0) circle (5pt)  {};
		\filldraw[black] (5,5) circle (5pt)  {};
		\filldraw[black] (4,6) circle (5pt)  {};
		\filldraw[black] (-6,6) circle (5pt)  {};
		\filldraw[black] (-1,5) circle (5pt)  {};
		\filldraw[black] (-2,6) circle (5pt)  {};
		\filldraw[black] (-3,3) circle (5pt)  {};
		\node at (0,0.7)[font=\fontsize{7pt}{0}]{$1$};
		\node at (5,5.7)[font=\fontsize{7pt}{0}]{$2$};
		\node at (4,6.7)[font=\fontsize{7pt}{0}]{$3$};
		\node at (-6,6.7)[font=\fontsize{7pt}{0}]{$5$};
		\node at (-1,5.7)[font=\fontsize{7pt}{0}]{$6$};
		\node at (-2,6.7)[font=\fontsize{7pt}{0}]{$7$};
		\node at (-3,3.7)[font=\fontsize{7pt}{0}]{$4$};
	\end{tikzpicture}$, we have that $\GD(\sigma)=\{4\}$, as indicated by the red branch.
	
\begin{Proposition} \label{prop:coproduct-monomial-ssym} [Theorems 3.1, 4.1, and 5.5 \cite{AS05}]
\begin{enumerate}
\item The coproduct of monomial
basis elements in $\SSym$ is given by
\[
\Delta_{+}(M_{\sigma})=\sum_{i\in\GD(\sigma)}M_{{}^{i}\sigma}\otimes M_{\sigma^{i}}.
\]
\item The product of monomial basis elements has non-negative
integer coefficients.
\item In the antipode image $\mathcal{S}(M_\sigma)$, all
terms have the same sign, being the parity of $|\GD(\sigma)|+1$.
\end{enumerate}
\end{Proposition}

In Examples \ref{ex:monomialcoef} and \ref{ex:below} we illustrate how to calculate the product and the antipode in the monomial basis.

Recall from the end of Section \ref{sec:notations} the forgetful projection $\pi:\LPT\to \PT$, which restricts to $\pi:\Perm\to\PBT$ sending a permutation to its underlying binary tree. It follows from Proposition \ref{prop:pi-property} that the projection $\Pi:\SSym\to\YSym$, $F_\sigma\mapsto F_{\pi(\sigma)}$ is a Hopf morphism. 

\begin{Remark}
In \cite{AS06} they construct $\Pi$ by visualising permutations as decreasing trees instead. The two maps are equivalent by transforming our increasing trees into decreasing trees via the map $i \leftrightarrow n-i$ on the labels. This is a Hopf morphism from $\SSym$ to a version of $\SSym$ with multiplication in the opposite order, which explains why our version of $\YSym$ has multiplication in the opposite order from the one in \cite{AS06}.
\end{Remark}

In the opposite direction, there is a section map $\iota:\PBT\rightarrow\Perm$, sending $t$ to the unique $213$-avoiding permutation whose underlying tree is $t$. The construction is given at the end of Section \ref{sec:gsp}, and Lemma \ref{lem:iota-max} shows that this is the maximal preimage of $t$ under $\pi$. 

We note additional well-known order-theoretic properties of the maps $\pi$ and $\iota$, which are special cases of Proposition \ref{prop:lsb} and Lemma \ref{lem:iota-gd} regarding planar trees which we will prove in Section \ref{sec:STSYM}.

\begin{Proposition}\label{prop:pi-property-bin}The maps $\pi:\Perm\rightarrow\PBT$ and  $\iota:\PBT\rightarrow\Perm$ satisfy the following properties:
\begin{enumerate}
	\item $\pi$ is order-preserving: $\sigma \leq_{w} \tau  \implies \pi(\sigma) \leq_{T} \pi (\tau)$;
	\item $\iota$ is order-preserving: $s \leq_{T} t \implies \iota(s) \leq_{w} \iota(t)$;
	\item $\pi$ preserves the least upper bound: $\pi(\sigma\vee \tau)=\pi(\sigma)\vee\pi(\tau)$;
	\item $\iota$ preserves global descents: $\GD(t)=\GD(\iota(t))$.
\end{enumerate}
\end{Proposition}
 
It follows from properties (1) and (2) above that $\pi$ and $\iota$ form a Galois connection (see \cite{AS06} for the definition), which gives the following formula (\cite[Th. 3.1]{AS06}) for the images of monomial basis elements under the projection $\Pi$, which will also follow from our Theorem \ref{thm:monomialquotient} below.

\begin{equation}
\Pi(M_{\sigma})=
\begin{cases}
M_{\pi(\sigma)} & \text{if } \sigma\text{ is $213$-avoiding},\\
0 & \text{otherwise}.
\end{cases}
\end{equation}

\section{Axioms for monomial basis properties}\label{sec:axioms}\label{sec:axioms-monomials}
This section gives abstract conditions on the interaction between the poset of basis elements and the Hopf operations, for the behaviour of ``monomial'' bases seen in $\YSym$ (Proposition \ref{prop:ysym-monomial}) and $\SSym$ (Propositions \ref{prop:coproduct-monomial-ssym}, \ref{prop:product-monomial-ssym}, \ref{prop:antipode-monomial-ssym}). We delay the proofs of the main theorems to Section \ref{sec:proofs} as they are technical. The ideas in this Section mimic the ones used in \cite{AS05,AS06}.

Throughout this section, let $P_{n}$ be a poset for each integer
$n\geq0$. Let $\mathcal{H}$ be a graded vector space with a pair of bases $\{F_{f}:f\in P_{n}\}$ and $\{M_{f}:f\in P_{n}\}$,
related by $F_{f}=\sum_{g\geq f}M_{g}$. We will say that respectively they are the \emph{fundamental} and \emph{monomial} bases.

\subsection{Axioms for the coproduct formula} \label{sec:axioms-coproduct}
Here are the axioms necessary for the comultiplication of the monomial basis
to be ``deconcatenation at global descents''.
\begin{enumerate}
\item[($\Delta$1).] $\mathcal{H}$ is a coalgebra and, for each $f\in P_{n}$, there exists a subset $\allow(f)\subseteq\{1,\dots,n-1\}$
such that
\begin{equation}
\Delta_{+}(F_{f})=\sum_{i\in \allow(f)}F_{{}^{i}f}\otimes F_{f^{i}}\label{eq:fcoproductformula}
\end{equation}
for some ${}^{i}f\in P_{i}$ and $f^{i}\in P_{n-i}$; and if $f'\geq f$,
then $\allow(f')=\allow(f)$. In particular, $\allow$ is constant over the connected components of $P_n$.
\item[($\Delta$2).] If  $f\leq f'\in P_n$, then ${}^{i}f\leq{}^{i}f'$ and $f^{i}\leq f'^{i}$.
\item[($\Delta$3).] Given $g\in P_i$, $h\in P_{n-i}$
there exists a unique $g/h=\max\{f\in P_n|{}^{i}f=g,f^{i}=h\}$,
and, if $g\leq g'\in P_i$, $h\leq h'\in P_{n-i}$, then $g/h\leq g'/h'$.
\end{enumerate}

\begin{Remark} \label{rem:coproductaxiom-poset}
In axiom ($\Delta$1), in place of demanding that $\mathcal{H}$ is already a coalgebra, we may instead specify conditions on the poset and the operations $\allow$, $^if$, $f^i$ so that Equation (\ref{eq:fcoproductformula}) defines a coalgebra structure. These conditions are:
\begin{enumerate}
    \item for $f\in P_n$ and each pair $i<j\in [1,n-1]$, we have $j\in\allow(f)$ and $i\in\allow({}^{j}f)$ if and only if $i\in\allow(f)$ and $j-i\in\allow(f^i)$;
    \item under the condition in (1), $^{i}\left({}^{j}f\right)={}^{i}f$, $\left({}^{j}f\right)^{i} ={}^{j-i}\left(f^{i}\right)$; $ f^{j}=\left(f^{i}\right)^{j-i}$.
\end{enumerate}
\end{Remark}

Under the axioms ($\Delta$1)-($\Delta$3), we may define the global descent
set of each $f\in P_{n}$ by
\[
\GD(f):=\{i\in\allow(f):f={}^{i}f/f^{i}\}.
\]

To describe iterated coproducts, it is necessary to define the deconcatenation
of $f$ into more than two parts, an abstract analogue of the operation
$t\mapsto(t_{1},\dots,t_{k})$ for trees. Given $S\subseteq\allow(f)$,
let $f|S$ be a $k$-tuple defined inductively
by $f|\{i\}:=({}^{i}f,f^{i})$ and
\begin{equation}
f|\{i_{1},\dots,i_{k},j\}:=({}^{j}f|\{i_{1},\dots,i_{k}\},f^{j}) \label{eq:deconcatenation-defn}
\end{equation}
(for $i_{1}<\cdots<i_{k}<j$; Lemma \ref{lem:kfactor-coproductaxiom}
will show that $S\backslash\{j\}\subseteq\allow({}^{j}f)$).). By coassociativity, $f|S$ may be defined
in many equivalent ways. 

\begin{Theorem}\label{thm:monomialcoproduct}  If $\mathcal{H}$
is a coalgebra satisfying axioms ($\Delta$1)-($\Delta$3), then
\begin{align}
\Delta_{+}(M_{f}) & =\sum_{i\in\GD(f)}M_{{}^{i}f}\otimes M_{f^{i}},\label{eq:mcoproductformula}\\
\Delta_{+}^{[k]}(M_{f}) & =\sum_{S\subseteq\GD(f),\ |S|=k-1}M_{(f|S)_{1}}\otimes\cdots\otimes M_{(f|S)_{k}}.\label{eq:mcoproductformula-multi}
\end{align}
 \end{Theorem}

\subsection{Axioms for the product formula}

Next, we give axioms for the multiplication in a monomial basis to be computable
with the argument of \cite[Th. 4.1]{AS05}. One of the key ingredients
(condition ($m$1) below) is that the product of fundamental basis elements
should be computable by ``shuffles'', which we abstract as follows:

\begin{Definition}\label{def:shuffle} Let $C_{i},C_{j}$ be connected components of $P_{n},P_{m}$
respectively. A \textbf{shuffle} $\zeta$ on $C_{i},C_{j}$ is a
function $\zeta:C_{i}\times C_{j}\rightarrow P_{n+m}$. 

\end{Definition}

Here is the complete list of relevant axioms:
\begin{enumerate}
\item[($m$0).]  Each connected component of each $P_{n}$ admits least upper bounds.
\end{enumerate}

\noindent
For any ordered pair of connected components $C_{i},C_{j}$ of $P_{n},P_{m}$
respectively, there is a set $Sh(C_{i},C_{j})$ of shuffles such
that:
\begin{enumerate}
\item[($m$1).]  $\mathcal{H}$ is an algebra and for $f\in C_{i}$ and $g\in C_{j}$, 
\begin{equation}
F_{f}F_{g}=\sum_{\zeta\in Sh(C_{i},C_{j})}F_{\zeta(f,g)}. \label{eq:fproductformula}
\end{equation}
\item[($m$2).]  For $f\leq f'\in C_{i}$ and $g\leq g'\in C_{j}$ and $\zeta\in Sh(C_{i},C_{j})$,
we have $\zeta(f,g)\leq\zeta(f',g')$.
\item[($m$3).]  For $f_{1},f_{2}\in C_{i}$ and $g_{1},g_{2}\in C_{j}$ and $\zeta\in Sh(C_{i},C_{j})$,
we have $\zeta(f_{1}\vee f_{2},g_{1}\vee g_{2})\leq\zeta(f_{1},g_{1})\vee\zeta(f_{2},g_{2})$.
\end{enumerate}

Fix $\zeta\in Sh(C_{i},C_{j})$.
It follows from ($m$2) that for all $f\in C_{i}$ and all $g\in C_{j}$,
the elements $\zeta(f,g)$ all lie in the same connected component.  Denote this output component
by $\cmpt(\zeta(C_{i},C_{j}))$. Then we may define inductively
\begin{align}
Sh(C_{j_{1}},\dots,C_{j_{k}}) & :=\{(\zeta,\zeta'):\zeta\in Sh(C_{j_{1}},C_{j_{2}}),\zeta'\in Sh(\cmpt(\zeta(C_{j_{1}},C_{j_{2}})),C_{j_{3}},\dots,C_{j_{k}})\},\label{eq:shuffle-multi}\\
(\zeta,\zeta')(f_{1},\dots,f_{k}) & :=\zeta'(\zeta(f_{1},f_{2}),f_{3},\dots,f_{k}).\nonumber 
\end{align}

\begin{Remark}
 \label{rem:productaxiom-poset}
Similar to Remark \ref{rem:coproductaxiom-poset}, we may replace the requirement in axiom ($m$1) that $\mathcal{H}$ is an algebra by the following condition on shuffles
for $f\in C_{i}, g\in C_{j}, h\in C_{k}$:
\begin{align*}
&\{\zeta'(\zeta(f,g),h) :\zeta\in Sh(C_{i},C_{j}),\zeta'\in Sh(\cmpt(\zeta(C_{i},C_{j})),C_{k})\} \\
= & \{\zeta'(f,\zeta(g,h)) :\zeta\in Sh(C_{j},C_{k}),\zeta'\in Sh(C_{i},\cmpt(\zeta(C_{j},C_{k})))\}.
\end{align*}
This ensures Equation (\ref{eq:fproductformula}) defines an associative multiplication.
\end{Remark}

\begin{Example} \label{ex:shuffle}
For $\SSym$ and $\YSym$, we consider the weak order (Section~\ref{sec:permweak}) and the Tamari order (Section~\ref{sec:tamari}), respectively.
In each case, $C_n=P_n$ is a unique connected component.
The set $Sh(C_n,C_m)$ can be identified with the set of sequences $1\le i_1 \le i_2\le \cdots\le i_n\le m+1$.
For $\zeta\in Sh(C_n,C_m)$ we define $\zeta:C_{n}\times C_{m}\rightarrow P_{n+m}$
as $\zeta(f,g)=f\overleftarrow{\shuffle}(g_{1},\dots,g_{k})$,
where $g\mapsto(g_{1},\dots,g_{k})$ is the lightening splitting of $g$ at the leaf positions
given by $\zeta$. \end{Example}

Under the axioms ($m$0)-($m$3), the product coefficients for a monomial basis
is calculated as follows: given $f\in C_{i}$ and $g\in C_{j}$, set
\begin{equation}
A_{f,g}^{h}:=\left\{ \zeta\in Sh(C_{i},C_{j})\middle|\:\begin{array}{c}
\zeta(f,g)\leq h,\\
\mbox{and if }f'\geq f,g'\geq g\mbox{ satisfy }\zeta(f',g')\leq h,\\
\mbox{then }f'=f,g'=g.
\end{array}\right\} ,\label{eq:a-def}
\end{equation}
and, given $f_{1}\in C_{1},f_{2}\in C_{2},\dots,f_{k}\in C_{k}$,
set
\begin{equation}
A_{f_{1},\dots,f_{k}}^{h}:=\left\{ \zeta\in Sh(C_{j_{1}},\dots,C_{j_{k}})\middle|\:\begin{array}{c}
\zeta(f_{1},\dots,f_{k})\leq h,\\
\mbox{and if }f_{i}'\geq f_{i}\mbox{ satisfy }\zeta(f'_{1},\dots,f'_{k})\leq h,\\
\mbox{then }f_{i}'=f_{i}\text{ for all }i.
\end{array}\right\} .\label{eq:a-def-multi}
\end{equation}

Example~\ref{ex:monomialcoef} below shows how to calculate $A_{f,g}^{h}$ for permutations.
To see what the coefficients $\alpha_{f,g}^h=|A_{f,g}^{h}|$ look like for the tree-based
algebras in this paper, see Proposition~\ref{prop:product-monomial-ssym} below.

\begin{Theorem}\label{thm:monomialproduct} If $\mathcal{H}$ is
an algebra satisfying axioms ($m$0)-($m$3), then, for any $f\in P_{i}$ and $g\in P_{n-i}$, 
\[
M_{f}\cdot M_{g}=\sum_{h\in P_{n}}\alpha_{f,g}^{h}M_{h}
\]
 where $\alpha_{f,g}^h=|A_{f,g}^{h}|$ given in (\ref{eq:a-def}). Moreover, for any 
integer $k>0$ and $f_{i}\in P_{n_{i}}$ with $n_{1}+\cdots+ n_{k}=n$, we have
\[
M_{f_{1}}\cdot M_{f_{2}}\cdots M_{f_{k}}=\sum_{h\in P_{n}}\alpha_{f_{1},\dots,f_{k}}^{h} M_{h}
\]
 where $\alpha_{f_{1},\dots,f_{k}}^{h}=|A_{f_{1},\dots,f_{k}}^{h}|$  given  in (\ref{eq:a-def-multi}).\end{Theorem}
 
 For our algebras of trees, the theorem above takes a form similar to the following proposition.
 
 \begin{Proposition} \label{prop:product-monomial-ssym} \cite[Th. 4.1]{AS05}
The product of monomial basis elements in $\SSym$ has non-negative
integer coefficients. More specifically, if $\sigma$ has $k$ leaves,
then $M_{\sigma}\cdot M_{\tau}=\sum_{\rho}\alpha_{\sigma,\tau}^{\rho}M_{\rho}$,
where $\alpha_{\sigma,\tau}^{\rho}$ is the number of splittings $\tau\mapsto(\tau_{1},\dots,\tau_{k})$ that satisfy the following
conditions:
\begin{enumerate}
\item $\sigma\overleftarrow{\shuffle}(\tau_{1},\dots,\tau_{k})\leq_{w}\rho$;
\item for $\sigma'\geq_{w}\sigma$ and $\tau'\geq_{w}\tau$, if
$\tau'\mapsto(\tau'_{1},\dots,\tau'_{k})$ satisfies $\deg\tau'_{i}=\deg\tau_{i}$
and $\sigma'\overleftarrow{\shuffle}(\tau'_{1},\dots,\tau'_{k})\leq_{w}\rho$,
then $\sigma'=\sigma$ and $\tau'=\tau$.
\end{enumerate}
\end{Proposition}

\begin{Example}\label{ex:monomialcoef}
We compute $\alpha_{\sigma,\tau}^{\rho}$, the
coefficient of $M_{\rho}$ in the product $M_{\sigma}\cdot M_{\tau}$,
for $\sigma=\begin{tikzpicture}[scale=0.2,baseline=0pt]
		\draw[blue, thick] (0,-1) -- (0,0);
		\draw[blue, thick] (0,0) -- (2,2);
		\draw[blue, thick] (0,0) -- (-2,2);
		\draw[blue, thick] (-1,1) -- (0,2);
		\filldraw[black] (0,0) circle (5pt)  {};
		\filldraw[black] (-1,1) circle (5pt)  {};
		\node at (0,0.7)[font=\fontsize{7pt}{0}]{$1$};
		\node at (-1,1.7)[font=\fontsize{7pt}{0}]{$2$};
	\end{tikzpicture}$,
$\tau=\begin{tikzpicture}[scale=0.2,baseline=0pt]
		\draw[blue, thick] (0,-1) -- (0,0);
		\draw[blue, thick] (0,0) -- (2,2);
		\draw[blue, thick] (0,0) -- (-2,2);
		\draw[blue, thick] (1,1) -- (0,2);
		\filldraw[black] (0,0) circle (5pt)  {};
		\filldraw[black] (1,1) circle (5pt)  {};
		\node at (0,0.7)[font=\fontsize{7pt}{0}]{$1$};
		\node at (1,1.7)[font=\fontsize{7pt}{0}]{$2$};
	\end{tikzpicture}$,
$\rho = \begin{tikzpicture}[scale=0.2,baseline=0pt]
		\draw[blue, thick] (0,-1) -- (0,0);
		\draw[blue, thick] (0,0) -- (4,4);
		\draw[blue, thick] (0,0) -- (-4,4);
		\draw[blue, thick] (-2,2) -- (0,4);
		\draw[blue, thick] (-1,3) -- (-2,4);
		\draw[blue, thick] (-1,1) -- (2,4);
		\filldraw[black] (0,0) circle (5pt)  {};
		\filldraw[black] (-2,2) circle (5pt)  {};
		\filldraw[black] (-1,3) circle (5pt)  {};
		\filldraw[black] (-1,1) circle (5pt)  {};
		\node at (0,0.7)[font=\fontsize{7pt}{0}]{$1$};
		\node at (-1,1.7)[font=\fontsize{7pt}{0}]{$2$};
		\node at (-2,2.7)[font=\fontsize{7pt}{0}]{$3$};
		\node at (-1,3.7)[font=\fontsize{7pt}{0}]{$4$};
		\end{tikzpicture}$.
Since $\sigma$ has $k=3$ leaves, let us describe the choices of $\tau\mapsto(\tau_{1},\tau_{2},\tau_{3})$ such that
$\sigma\overleftarrow{\shuffle}(\tau_{1},\tau_{2},\tau_{3})\leq_w \rho$, i.e., 
the ones that satisfy condition (1) in Proposition \ref{prop:product-monomial-ssym}:
\begin{alignat*}{2}
\sigma\overleftarrow{\shuffle}\Big(\,
	\begin{tikzpicture}[scale=0.2,baseline=0pt]
		\draw[blue, thick] (0,-1) -- (0,0);
		\draw[blue, thick] (0,0) -- (2,2);
		\draw[blue, thick] (0,0) -- (-2,2);
		\draw[blue, thick] (1,1) -- (0,2);
		\filldraw[black] (0,0) circle (5pt)  {};
		\filldraw[black] (1,1) circle (5pt)  {};
		\node at (0,0.7)[font=\fontsize{7pt}{0}]{$1$};
		\node at (1,1.7)[font=\fontsize{7pt}{0}]{$2$};
	\end{tikzpicture},
	\begin{tikzpicture}[scale=0.25,baseline=0pt]
		\draw[blue, thick] (0,-1) -- (0,1);
	\end{tikzpicture},
	\begin{tikzpicture}[scale=0.25,baseline=0pt]
		\draw[blue, thick] (0,-1) -- (0,1);
	\end{tikzpicture}\,\Big) & 
	=\begin{tikzpicture}[scale=0.2,baseline=0pt]
		\draw[blue, thick] (0,-1) -- (0,0);
		\draw[blue, thick] (0,0) -- (4,4);
		\draw[blue, thick] (0,0) -- (-4,4);
		\draw[blue, thick] (-2,2) -- (0,4);
		\draw[blue, thick] (-1,3) -- (-2,4);
		\draw[blue, thick] (-1,1) -- (2,4);
		\filldraw[black] (0,0) circle (5pt)  {};
		\filldraw[black] (-2,2) circle (5pt)  {};
		\filldraw[black] (-1,3) circle (5pt)  {};
		\filldraw[black] (-1,1) circle (5pt)  {};
		\node at (0,0.7)[font=\fontsize{7pt}{0}]{$1$};
		\node at (-1,1.7)[font=\fontsize{7pt}{0}]{$2$};
		\node at (-2,2.7)[font=\fontsize{7pt}{0}]{$3$};
		\node at (-1,3.7)[font=\fontsize{7pt}{0}]{$4$};
		\end{tikzpicture}; \qquad&
\sigma\overleftarrow{\shuffle}\Big(\,
	\begin{tikzpicture}[scale=0.25,baseline=0pt]
		\draw[blue, thick] (0,-1) -- (0,1);
	\end{tikzpicture},
	\begin{tikzpicture}[scale=0.2,baseline=0pt]
		\draw[blue, thick] (0,-1) -- (0,0);
		\draw[blue, thick] (0,0) -- (2,2);
		\draw[blue, thick] (0,0) -- (-2,2);
		\draw[blue, thick] (1,1) -- (0,2);
		\filldraw[black] (0,0) circle (5pt)  {};
		\filldraw[black] (1,1) circle (5pt)  {};
		\node at (0,0.7)[font=\fontsize{7pt}{0}]{$1$};
		\node at (1,1.7)[font=\fontsize{7pt}{0}]{$2$};
	\end{tikzpicture},
	\begin{tikzpicture}[scale=0.25,baseline=0pt]
		\draw[blue, thick] (0,-1) -- (0,1);
	\end{tikzpicture}\,\Big) & 
	=\begin{tikzpicture}[scale=0.2,baseline=0pt]
		\draw[blue, thick] (0,-1) -- (0,0);
		\draw[blue, thick] (0,0) -- (4,4);
		\draw[blue, thick] (0,0) -- (-4,4);
		\draw[blue, thick] (-1,1) -- (2,4);
		\draw[blue, thick] (0,2) -- (-2,4);
		\draw[blue, thick] (1,3) -- (0,4);
		\filldraw[black] (0,0) circle (5pt)  {};
		\filldraw[black] (-1,1) circle (5pt)  {};
		\filldraw[black] (0,2) circle (5pt)  {};
		\filldraw[black] (1,3) circle (5pt)  {};
		\node at (0,0.7)[font=\fontsize{7pt}{0}]{$1$};
		\node at (-1,1.7)[font=\fontsize{7pt}{0}]{$2$};
		\node at (0,2.7)[font=\fontsize{7pt}{0}]{$3$};
		\node at (1,3.7)[font=\fontsize{7pt}{0}]{$4$};
		\end{tikzpicture};\\
\sigma\overleftarrow{\shuffle}\Big(\,
	\begin{tikzpicture}[scale=0.25,baseline=0pt]
		\draw[blue, thick] (0,-1) -- (0,1);
	\end{tikzpicture},
	\begin{tikzpicture}[scale=0.25,baseline=0pt]
		\draw[blue, thick] (0,-1) -- (0,1);
	\end{tikzpicture},
	\begin{tikzpicture}[scale=0.2,baseline=0pt]
		\draw[blue, thick] (0,-1) -- (0,0);
		\draw[blue, thick] (0,0) -- (2,2);
		\draw[blue, thick] (0,0) -- (-2,2);
		\draw[blue, thick] (1,1) -- (0,2);
		\filldraw[black] (0,0) circle (5pt)  {};
		\filldraw[black] (1,1) circle (5pt)  {};
		\node at (0,0.7)[font=\fontsize{7pt}{0}]{$1$};
		\node at (1,1.7)[font=\fontsize{7pt}{0}]{$2$};
	\end{tikzpicture}\,\Big) &  
	=\begin{tikzpicture}[scale=0.2,baseline=0pt]
		\draw[blue, thick] (0,-1) -- (0,0);
		\draw[blue, thick] (0,0) -- (4,4);
		\draw[blue, thick] (0,0) -- (-4,4);
		\draw[blue, thick] (-3,3) -- (-2,4);
		\draw[blue, thick] (2,2) -- (0,4);
		\draw[blue, thick] (3,3) -- (2,4);
		\filldraw[black] (0,0) circle (5pt)  {};
		\filldraw[black] (-3,3) circle (5pt)  {};
		\filldraw[black] (2,2) circle (5pt)  {};
		\filldraw[black] (3,3) circle (5pt)  {};
		\node at (0,0.7)[font=\fontsize{7pt}{0}]{$1$};
		\node at (-3,3.7)[font=\fontsize{7pt}{0}]{$2$};
		\node at (2,2.7)[font=\fontsize{7pt}{0}]{$3$};
		\node at (3,3.7)[font=\fontsize{7pt}{0}]{$4$};
		\end{tikzpicture}; \qquad&
\sigma\overleftarrow{\shuffle}\Big(\,
	\begin{tikzpicture}[scale=0.25,baseline=0pt]
		\draw[blue, thick] (0,-1) -- (0,1);
	\end{tikzpicture},
	\begin{tikzpicture}[scale=0.2,baseline=0pt]
		\draw[blue, thick] (0,-1) -- (0,0);
		\draw[blue, thick] (0,0) -- (1,1);
		\draw[blue, thick] (0,0) -- (-1,1);
		\filldraw[black] (0,0) circle (5pt)  {};
		\node at (0,0.7)[font=\fontsize{7pt}{0}]{$1$};
	\end{tikzpicture},
	\begin{tikzpicture}[scale=0.2,baseline=0pt]
		\draw[blue, thick] (0,-1) -- (0,0);
		\draw[blue, thick] (0,0) -- (1,1);
		\draw[blue, thick] (0,0) -- (-1,1);
		\filldraw[black] (0,0) circle (5pt)  {};
		\node at (0,0.7)[font=\fontsize{7pt}{0}]{$2$};
	\end{tikzpicture}\,\Big) & 
	=\begin{tikzpicture}[scale=0.2,baseline=0pt]
		\draw[blue, thick] (0,-1) -- (0,0);
		\draw[blue, thick] (0,0) -- (4,4);
		\draw[blue, thick] (0,0) -- (-4,4);
		\draw[blue, thick] (-2,2) -- (0,4);
		\draw[blue, thick] (-1,3) -- (-2,4);
		\draw[blue, thick] (3,3) -- (2,4);
		\filldraw[black] (0,0) circle (5pt)  {};
		\filldraw[black] (-2,2) circle (5pt)  {};
		\filldraw[black] (-1,3) circle (5pt)  {};
		\filldraw[black] (3,3) circle (5pt)  {};
		\node at (0,0.7)[font=\fontsize{7pt}{0}]{$1$};
		\node at (-2,2.7)[font=\fontsize{7pt}{0}]{$2$};
		\node at (-1,3.7)[font=\fontsize{7pt}{0}]{$3$};
		\node at (3,3.7)[font=\fontsize{7pt}{0}]{$4$};
		\end{tikzpicture}. 
\end{alignat*}
To check condition (2), we apply the above four shuffles to any $\sigma'\geq_{w}\sigma$
and any $\tau'\geq_{w}\tau$. Because $\sigma$ is already maximal,
we only need to consider $\tau'=\begin{tikzpicture}[scale=0.2,baseline=0pt]
		\draw[blue, thick] (0,-1) -- (0,0);
		\draw[blue, thick] (0,0) -- (2,2);
		\draw[blue, thick] (0,0) -- (-2,2);
		\draw[blue, thick] (-1,1) -- (0,2);
		\filldraw[black] (0,0) circle (5pt)  {};
		\filldraw[black] (-1,1) circle (5pt)  {};
		\node at (0,0.7)[font=\fontsize{7pt}{0}]{$1$};
		\node at (-1,1.7)[font=\fontsize{7pt}{0}]{$2$};
	\end{tikzpicture}$:
\begin{alignat*}{2}
\sigma \overleftarrow{\shuffle}\Big(\,
	\begin{tikzpicture}[scale=0.2,baseline=0pt]
		\draw[blue, thick] (0,-1) -- (0,0);
		\draw[blue, thick] (0,0) -- (2,2);
		\draw[blue, thick] (0,0) -- (-2,2);
		\draw[blue, thick] (-1,1) -- (0,2);
		\filldraw[black] (0,0) circle (5pt)  {};
		\filldraw[black] (-1,1) circle (5pt)  {};
		\node at (0,0.7)[font=\fontsize{7pt}{0}]{$1$};
		\node at (-1,1.7)[font=\fontsize{7pt}{0}]{$2$};
	\end{tikzpicture},
	\begin{tikzpicture}[scale=0.25,baseline=0pt]
		\draw[blue, thick] (0,-1) -- (0,1);
	\end{tikzpicture},
	\begin{tikzpicture}[scale=0.25,baseline=0pt]
		\draw[blue, thick] (0,-1) -- (0,1);
	\end{tikzpicture}\,\Big) & 
	=\begin{tikzpicture}[scale=0.2,baseline=0pt]
		\draw[blue, thick] (0,-1) -- (0,0);
		\draw[blue, thick] (0,0) -- (4,4);
		\draw[blue, thick] (0,0) -- (-4,4);
		\draw[blue, thick] (-2,2) -- (0,4);
		\draw[blue, thick] (-3,3) -- (-2,4);
		\draw[blue, thick] (-1,1) -- (2,4);
		\filldraw[black] (0,0) circle (5pt)  {};
		\filldraw[black] (-2,2) circle (5pt)  {};
		\filldraw[black] (-3,3) circle (5pt)  {};
		\filldraw[black] (-1,1) circle (5pt)  {};
		\node at (0,0.7)[font=\fontsize{7pt}{0}]{$1$};
		\node at (-1,1.7)[font=\fontsize{7pt}{0}]{$2$};
		\node at (-2,2.7)[font=\fontsize{7pt}{0}]{$3$};
		\node at (-3,3.7)[font=\fontsize{7pt}{0}]{$4$};
		\end{tikzpicture} \not\leq_{w}\rho; \qquad&
		\sigma\overleftarrow{\shuffle}\Big(\,
	\begin{tikzpicture}[scale=0.25,baseline=0pt]
		\draw[blue, thick] (0,-1) -- (0,1);
	\end{tikzpicture},
	\begin{tikzpicture}[scale=0.2,baseline=0pt]
		\draw[blue, thick] (0,-1) -- (0,0);
		\draw[blue, thick] (0,0) -- (2,2);
		\draw[blue, thick] (0,0) -- (-2,2);
		\draw[blue, thick] (-1,1) -- (0,2);
		\filldraw[black] (0,0) circle (5pt)  {};
		\filldraw[black] (-1,1) circle (5pt)  {};
		\node at (0,0.7)[font=\fontsize{7pt}{0}]{$1$};
		\node at (-1,1.7)[font=\fontsize{7pt}{0}]{$2$};
	\end{tikzpicture},
	\begin{tikzpicture}[scale=0.25,baseline=0pt]
		\draw[blue, thick] (0,-1) -- (0,1);
	\end{tikzpicture}\,\Big) & 
	=\begin{tikzpicture}[scale=0.2,baseline=0pt]
		\draw[blue, thick] (0,-1) -- (0,0);
		\draw[blue, thick] (0,0) -- (4,4);
		\draw[blue, thick] (0,0) -- (-4,4);
		\draw[blue, thick] (-1,1) -- (2,4);
		\draw[blue, thick] (0,2) -- (-2,4);
		\draw[blue, thick] (-1,3) -- (0,4);
		\filldraw[black] (0,0) circle (5pt)  {};
		\filldraw[black] (-1,1) circle (5pt)  {};
		\filldraw[black] (0,2) circle (5pt)  {};
		\filldraw[black] (-1,3) circle (5pt)  {};
		\node at (0,0.7)[font=\fontsize{7pt}{0}]{$1$};
		\node at (-1,1.7)[font=\fontsize{7pt}{0}]{$2$};
		\node at (0,2.7)[font=\fontsize{7pt}{0}]{$3$};
		\node at (-1,3.7)[font=\fontsize{7pt}{0}]{$4$};
		\end{tikzpicture} \leq_{w}\rho ;\\
\sigma\overleftarrow{\shuffle}\Big(\,
	\begin{tikzpicture}[scale=0.25,baseline=0pt]
		\draw[blue, thick] (0,-1) -- (0,1);
	\end{tikzpicture},
	\begin{tikzpicture}[scale=0.25,baseline=0pt]
		\draw[blue, thick] (0,-1) -- (0,1);
	\end{tikzpicture},
	\begin{tikzpicture}[scale=0.2,baseline=0pt]
		\draw[blue, thick] (0,-1) -- (0,0);
		\draw[blue, thick] (0,0) -- (2,2);
		\draw[blue, thick] (0,0) -- (-2,2);
		\draw[blue, thick] (-1,1) -- (0,2);
		\filldraw[black] (0,0) circle (5pt)  {};
		\filldraw[black] (-1,1) circle (5pt)  {};
		\node at (0,0.7)[font=\fontsize{7pt}{0}]{$1$};
		\node at (-1,1.7)[font=\fontsize{7pt}{0}]{$2$};
	\end{tikzpicture}\,\Big) &  
	=\begin{tikzpicture}[scale=0.2,baseline=0pt]
		\draw[blue, thick] (0,-1) -- (0,0);
		\draw[blue, thick] (0,0) -- (4,4);
		\draw[blue, thick] (0,0) -- (-4,4);
		\draw[blue, thick] (-3,3) -- (-2,4);
		\draw[blue, thick] (2,2) -- (0,4);
		\draw[blue, thick] (1,3) -- (2,4);
		\filldraw[black] (0,0) circle (5pt)  {};
		\filldraw[black] (-3,3) circle (5pt)  {};
		\filldraw[black] (2,2) circle (5pt)  {};
		\filldraw[black] (1,3) circle (5pt)  {};
		\node at (0,0.7)[font=\fontsize{7pt}{0}]{$1$};
		\node at (-3,3.7)[font=\fontsize{7pt}{0}]{$2$};
		\node at (2,2.7)[font=\fontsize{7pt}{0}]{$3$};
		\node at (1,3.7)[font=\fontsize{7pt}{0}]{$4$};
		\end{tikzpicture} \leq_{w}\rho ; \qquad&
\sigma\overleftarrow{\shuffle}\Big(\,
	\begin{tikzpicture}[scale=0.25,baseline=0pt]
		\draw[blue, thick] (0,-1) -- (0,1);
	\end{tikzpicture},
	\begin{tikzpicture}[scale=0.2,baseline=0pt]
		\draw[blue, thick] (0,-1) -- (0,0);
		\draw[blue, thick] (0,0) -- (1,1);
		\draw[blue, thick] (0,0) -- (-1,1);
		\filldraw[black] (0,0) circle (5pt)  {};
		\node at (0,0.7)[font=\fontsize{7pt}{0}]{$2$};
	\end{tikzpicture},
	\begin{tikzpicture}[scale=0.2,baseline=0pt]
		\draw[blue, thick] (0,-1) -- (0,0);
		\draw[blue, thick] (0,0) -- (1,1);
		\draw[blue, thick] (0,0) -- (-1,1);
		\filldraw[black] (0,0) circle (5pt)  {};
		\node at (0,0.7)[font=\fontsize{7pt}{0}]{$1$};
	\end{tikzpicture}\,\Big) & 
	=\begin{tikzpicture}[scale=0.2,baseline=0pt]
		\draw[blue, thick] (0,-1) -- (0,0);
		\draw[blue, thick] (0,0) -- (4,4);
		\draw[blue, thick] (0,0) -- (-4,4);
		\draw[blue, thick] (-2,2) -- (0,4);
		\draw[blue, thick] (-1,3) -- (-2,4);
		\draw[blue, thick] (3,3) -- (2,4);
		\filldraw[black] (0,0) circle (5pt)  {};
		\filldraw[black] (-2,2) circle (5pt)  {};
		\filldraw[black] (-1,3) circle (5pt)  {};
		\filldraw[black] (3,3) circle (5pt)  {};
		\node at (0,0.7)[font=\fontsize{7pt}{0}]{$1$};
		\node at (-2,2.7)[font=\fontsize{7pt}{0}]{$2$};
		\node at (-1,3.7)[font=\fontsize{7pt}{0}]{$4$};
		\node at (3,3.7)[font=\fontsize{7pt}{0}]{$3$};
		\end{tikzpicture} \leq_{w}\rho\,.
\end{alignat*}

Hence only the first shuffle satisfies (2) and $\alpha_{\sigma,\tau}^{\rho}=1$.
\end{Example}

\subsection{Axioms for the antipode formula} \label{sec:axioms-antipode}

In the following, we require  that $\mathcal{H}$ is a Hopf algebra satisfying axioms ($\Delta$1)-($\Delta$3) and ($m$0)-($m$3). 

\begin{Remark}\label{rem:translationrequirements}
    Similar to Remarks \ref{rem:coproductaxiom-poset} and \ref{rem:productaxiom-poset}, we may translate the requirement that $\mathcal{H}$ is a Hopf algebra into the following conditions on the poset operations: for $x \in C, y \in C'$,
    \begin{align*}
        \{^i(\zeta(x,y)) | \zeta\in Sh(C,C')\}&=\{ \zeta'({}^jx,{}^{i-j}y) | \zeta'\in Sh(\cmpt({}^jx),\cmpt({}^{i-j}y)), j\in [1,i-1]\}; \\
        \{(\zeta(x,y))^i | \zeta\in Sh(C,C')\} &=\{ \zeta'(x^j,y^{i-j}) | \zeta'\in Sh(\cmpt(x^j),\cmpt(y^{i-j})), j\in [1,i-1] \}.
    \end{align*}
    This will ensure the compatibility between the multiplication and comultiplication, defined by Equations \ref{eq:fproductformula} and \ref{eq:fcoproductformula}.
\end{Remark}

Fix $f\in P_{n}$ and $S\subseteq\GD(f)$. Let $f|S=(f_{1},\dots,f_{k})$ denote the splitting of $f$ at $S$
and let $C_i$ denote the connected component containing $f_{i}$, then we write $\cmpts(f|S)$
for the sequence $(C_{1},\dots,C_{k})$. From (\ref{eq:shuffle-multi}), each fixed shuffle $\zeta\in Sh(C_{1},C_{2})$
defines an injection 
\begin{align*}
 Sh(\cmpt(\zeta(C_{1},C_{2})),C_{3},\dots,C_{k})&\hookrightarrow Sh(C_{1},C_{2},C_{3},\dots,C_{k})\\
  \zeta'&\mapsto (\zeta,\zeta')\,.
  \end{align*}
Similarly, by associativity of the multiplication, each fixed shuffle $\zeta\in Sh(C_{i},C_{i+1})$
defines an injection 
 $$Sh(C_{1},\dots,C_{i-1},\cmpt(\zeta(C_{i},C_{i+1})),C_{i+2},\dots,C_{k})\hookrightarrow Sh(C_{1},\dots,C_{k}).$$

Takeuchi's antipode formula involves shuffles of $f|S$ for different $S$; we need to compare shuffles of $f|S$ with shuffles of $f|R$ in order to resolve the cancelation of terms. To do so, Axiom ($\mathcal{S}0$) below will ask for a well chosen $\zeta_*\in Sh(C,C')$ such that the operation defined by
$f\backslash g=\zeta_*(f,g)$ is associative and $f\backslash g\leq f/g$ for $f/g$ defined by ($\Delta$3).
Lemma \ref{lem:nesting-for-antipode}
will then show that the composition of  injections as above gives well-defined
maps 
$$\iota_{R,S}:Sh(\cmpts(f|R))\hookrightarrow Sh(\cmpts(f|S)),$$
whenever $R\subseteq S\subseteq\GD(f)$. As an example: if $S=\{i_{1},i_{2},i_{3},i_{4},i_{5},i_{6}\}$
and $R=\{i_{1},i_{3},i_{4}\}$, then $\zeta\in Sh(\cmpts(f|R))$ takes
as input four entries. Then $\iota_{R,S}(\zeta)\in Sh(\cmpts(f|S))$
is defined by
\[
(\iota_{R,S}\zeta)(h_{1},\dots,h_{7})=\zeta(h_{1},h_{2}\backslash h_{3},h_{4},h_{5}\backslash h_{6}\backslash h_{7}).
\]
\begin{enumerate}
\item[($\mathcal{S}$0).]  There is a binary associative operation $(f,g)\mapsto f\backslash g$
such that, for any connected components $C_{i}$ and $C_{j}$ we have a choice of $\zeta_*\in Sh(C_{i},C_{j})$ that gives $f\backslash g=\zeta_*(f,g)$ for all $f\in C_{i}$ and $g\in C_{j}$. 
Furthermore, $f\backslash g\leq f/g$.
Then we require that for any $R_{1},R_{2}\subseteq S\subseteq\GD(f)$, if $\zeta\in Sh(\cmpts(f|S))$
satisfies $\zeta=\iota_{R_{1},S}(\zeta_{1})=\iota_{R_{2},S}(\zeta_{2})$
for some $\zeta_{1}\in Sh(\cmpts(f|R_{1}))$ and  $\zeta_{2}\in Sh(\cmpts(f|R_{2}))$,
then it is possible to find $\zeta'\in Sh(\cmpts(f|R_{1}\cap R_{2}))$ such that $\zeta=\iota_{R_{1}\cap R_{2},S}(\zeta')$. 
\item[($\mathcal{S}$1).]  If $f'\geq f\in P_{n}$, then $\GD(f')\supseteq\GD(f)$;
\item[($\mathcal{S}$2).]  If $f_{i}'\geq f_{i}\in P_{n_{i}}$ for each $i$, then $f'_{1}/\cdots/f'_{k}\leq f_{1}/\cdots/f_{k}\vee f'_{1}\backslash\cdots\backslash f'_{k}$.
\item[($\mathcal{S}$3).]  For any $i_{1}<j_{1}<i_{2}<j_{2}<\cdots<i_{l}<j_{l}$, 
\begin{align*}
f_{1}/\cdots/f_{k} & \leq\left[(f_{1}/\cdots/f_{i_{1}})\backslash f_{i_{1}+1}\backslash\cdots\backslash f_{j_{1}-1}\backslash(f_{j_{1}}/\cdots/f_{i_{2}})\backslash f_{i_{2}+1}\backslash\cdots\backslash f_{j_{2}-1}\backslash\cdots\right]\\
 & \qquad\vee\left[f_{1}\backslash\cdots\backslash f_{i_{1}-1}\backslash(f_{i_{1}}/\cdots/f_{j_{1}})\backslash f_{j_{1}+1}\backslash\cdots\backslash f_{i_{2}-1}\backslash(f_{i_{2}}/\cdots/f_{j_{2}})\backslash\cdots\right].
\end{align*}
 \end{enumerate}

Under these conditions, the antipode formula for monomial basis elements
involves the following quantity: given $f,h\in P_{n}$, let $\beta_{f}^{h}$
be the number of $\zeta\in Sh(\cmpts(f|\GD(f)))$ satisfying the following three
conditions:
\begin{equation}
\begin{array}{ll}
\text{(i)} & \zeta(f|\GD(f))\leq h;\\
\text{(ii)} & \mbox{if }f'\geq f\mbox{ satisfies }\zeta(f'|\GD(f))\leq h,\mbox{ then }f=f';\\
\text{(iii)} & \mbox{if }\zeta=\iota_{\GD(f)\backslash\{i\},\GD(f)}\zeta' \\ &\qquad\text{ for some }i\in\GD(f)\text{ and  }\zeta'\in Sh(\cmpts(f|\GD(f)\backslash\{i\})),\\
 &\mbox{then }\zeta'(f|(\GD(f)\backslash\{i\}))\not\leq h.
\end{array}\label{eq:c-def-GD-2}
\end{equation}

\begin{Theorem}\label{thm:antipode} If $\mathcal{H}$ is a connected
Hopf algebra satisfying conditions ($\Delta$1)-($\Delta$3), ($m$0)-($m$3) and ($\mathcal{S}$0)-($\mathcal{S}$3), then the antipode is given by 
\[
\mathcal{S}(M_{f})=(-1)^{|\GD(f)|+1}\sum_{h} \beta_{f}^{h} M_{h}
\]
 where $\beta_{f}^{h}$ is defined in (\ref{eq:c-def-GD-2}).\end{Theorem}

\begin{Remark}\label{rem:s0trees} In this paper we are particularly interested in Hopf algebras with bases indexed by special types of labeled trees.
When the labeling is closed under the operations defined in Definition~\ref{def:perm}, we have that such algebras will satisfy ($\mathcal{S}$0).
To see this, we have $\zeta=\iota_{R,S}\zeta'$ if and only if, in $\zeta(h_{1},\dots,h_{k})$,
all nodes of $h_{i}$ occur to the left of $h_{i+1}$, for each
$i$ where $\deg h_{1}+\cdots+\deg h_{i}\in S\backslash R$. (This
observation is used towards the end of Example~\ref{ex:below}.) So if $\zeta=\iota_{R_{1},S}(\zeta_{1})=\iota_{R_{2},S}(\zeta_{2})$,
then, in $\zeta(h_{1},\dots,h_{k})$, all nodes of $h_{i}$ occur
to the left of $h_{i+1}$, for each $i$ where $\deg h_{1}+\cdots+\deg h_{i}\in(S\backslash R_{1})\cup(S\backslash R_{2})=S\backslash(R_{1}\cap R_{2})$.
\end{Remark}

\begin{Proposition} \label{prop:antipode-monomial-ssym} \cite[Th. 5.5]{AS05}
 In the antipode of monomial basis elements in
$\SSym$, all terms have the same sign. More specifically, if $\sigma\mapsto(\sigma_{1},\dots,\sigma_{k})$
is the result of splitting $\sigma$ at all its global descents, then
$\mathcal{S}(M_{\sigma})=(-1)^{\GD(\sigma)+1}\sum_{\tau}\gamma_{\sigma}^{\tau}M_{\tau}$
where $\gamma_{\sigma}^{\tau}$ is the number of splitting sets $\sigma_{2}\mapsto(\sigma_{2}^{(1)},\dots,\sigma_{2}^{(k_{2})}),\sigma_{3}\mapsto(\sigma_{3}^{(1)},\dots,\sigma_{3}^{(k_{3})}),\dots,\sigma_{k}\mapsto(\sigma_{k}^{(1)},\dots,\sigma_{k}^{(k_{k})})$
satisfying the following conditions:
\begin{enumerate}
\item[{\rm (1)}]
 $
\Big(\cdots\big((\sigma_{1}\overleftarrow{\shuffle}(\sigma_{2}^{(1)},\dots,\sigma_{2}^{(k_{2})}))\overleftarrow{\shuffle}(\sigma_{3}^{(1)},\dots,\sigma_{3}^{(k_{3})})\big)\cdots\Big)\overleftarrow{\shuffle}(\sigma_{k}^{(1)},\dots,\sigma_{k}^{(k_{k})})\leq_{w}\tau;
 $ \hfill\break and for now on we always assume that such iterated shuffle is performed from left to right and will not put the parenthesis.
\item[{\rm (2)}] For each $\sigma'>\sigma$, consider ${\sigma'}\mapsto({\sigma'}_{1},\dots,{\sigma'}_{k})$
and ${\sigma'}_{i}\mapsto ({\sigma'}_{i}^{(1)},\dots,{\sigma'}_{i}^{(k_{i})})$ for $2\leq i \leq k$
with  $\deg {\sigma'}_{i}^{(j)}=\deg {\sigma}_{i}^{(j)}$. Then 
   $${\sigma'}_{1}\overleftarrow{\shuffle}({\sigma'}_{2}^{(1)},\dots,{\sigma'}_{2}^{(k_{2})})\overleftarrow{\shuffle}({\sigma'}_{3}^{(1)},\dots,{\sigma'}_{3}^{(k_{3})})\cdots\overleftarrow{\shuffle}({\sigma'}_{k}^{(1)},\dots,{\sigma'}_{k}^{(k_{k})})\not\leq_{w}\tau;$$
\item[{\rm (3)}] Take any $i$ where all nodes from $\sigma_{i}$ appear to the left of all nodes from $\sigma_{i-1}$, so that the left hand side of (1)
may be computed as 
\begin{align*}
\sigma_{1}\overleftarrow{\shuffle}(\sigma_{2}^{(1)},\dots,\sigma_{2}^{(k_{2})})
     \cdots
     &\overleftarrow{\shuffle}(\sigma_{i-2}^{(1)},\dots,\sigma_{i-2}^{(k_{i-2})}) \\
     &\overleftarrow{\shuffle}((\sigma_{i-1}\backslash\sigma_{i})^{(1)},\dots,(\sigma_{i-1}\backslash\sigma_{i})^{(k')}) \\
   &\overleftarrow{\shuffle}(\sigma_{i+1}^{(1)},\dots,\sigma_{i+1}^{(k_{i+1})})\cdots\overleftarrow{\shuffle}(\sigma_{k}^{(1)},\dots,\sigma_{k}^{(k_{k})}),
   \end{align*}
for some splitting $\sigma_{i-1}\backslash\sigma_{i}\mapsto((\sigma_{i-1}\backslash\sigma_{i})^{(1)},\dots,(\sigma_{i-1}\backslash\sigma_{i})^{(k')})$.
Replacing $\sigma_{i-1}\backslash\sigma_{i}$ with $\sigma_{i-1}/\sigma_{i}$  gives a result $\not\leq_{w}\tau$.
\end{enumerate}
\end{Proposition}

\begin{Example}\label{ex:below}
Let us compute $\gamma_{\sigma}^{\tau}$, the coefficient
of $M_{\tau}$ (up to sign) in the antipode $S(M_{\sigma})$, for
$\sigma=\begin{tikzpicture}[scale=0.2,baseline=0pt]
		\draw[blue, thick] (0,-1) -- (0,0);
		\draw[blue, thick] (0,0) -- (4,4);
		\draw[blue, thick] (0,0) -- (-4,4);
		\draw[blue, thick] (-3,3) -- (-2,4);
		\draw[blue, thick] (-1,1) -- (2,4);
		\draw[blue, thick] (1,3) -- (0,4);
		\filldraw[black] (0,0) circle (5pt)  {};
		\filldraw[black] (-3,3) circle (5pt)  {};
		\filldraw[black] (-1,1) circle (5pt)  {};
		\filldraw[black] (1,3) circle (5pt)  {};
		\node at (0,0.7)[font=\fontsize{7pt}{0}]{$1$};
		\node at (-1,1.7)[font=\fontsize{7pt}{0}]{$2$};
		\node at (-3,3.7)[font=\fontsize{7pt}{0}]{$4$};
		\node at (1,3.7)[font=\fontsize{7pt}{0}]{$3$};
	\end{tikzpicture}$ and 
$\tau=\begin{tikzpicture}[scale=0.2,baseline=0pt]
		\draw[blue, thick] (0,-1) -- (0,0);
		\draw[blue, thick] (0,0) -- (4,4);
		\draw[blue, thick] (0,0) -- (-4,4);
		\draw[blue, thick] (-3,3) -- (-2,4);
		\draw[blue, thick] (2,2) -- (0,4);
		\draw[blue, thick] (3,3) -- (2,4);
		\filldraw[black] (0,0) circle (5pt)  {};
		\filldraw[black] (-3,3) circle (5pt)  {};
		\filldraw[black] (2,2) circle (5pt)  {};
		\filldraw[black] (3,3) circle (5pt)  {};
		\node at (0,0.7)[font=\fontsize{7pt}{0}]{$1$};
		\node at (-3,3.7)[font=\fontsize{7pt}{0}]{$2$};
		\node at (2,2.7)[font=\fontsize{7pt}{0}]{$3$};
		\node at (3,3.7)[font=\fontsize{7pt}{0}]{$4$};
	\end{tikzpicture}$. Splitting $\sigma$ at its global descents gives 
$$\sigma \mapsto \Big(
\sigma_1 = \begin{tikzpicture}[scale=0.2,baseline=0pt]
		\draw[blue, thick] (0,-1) -- (0,0);
		\draw[blue, thick] (0,0) -- (1,1);
		\draw[blue, thick] (0,0) -- (-1,1);
		\filldraw[black] (0,0) circle (5pt)  {};
		\node at (0,0.7)[font=\fontsize{7pt}{0}]{$1$};
	\end{tikzpicture},
\sigma_2 = \begin{tikzpicture}[scale=0.2,baseline=0pt]
		\draw[blue, thick] (0,-1) -- (0,0);
		\draw[blue, thick] (0,0) -- (2,2);
		\draw[blue, thick] (0,0) -- (-2,2);
		\draw[blue, thick] (1,1) -- (0,2);
		\filldraw[black] (0,0) circle (5pt)  {};
		\filldraw[black] (1,1) circle (5pt)  {};
		\node at (0,0.7)[font=\fontsize{7pt}{0}]{$1$};
		\node at (1,1.7)[font=\fontsize{7pt}{0}]{$2$};
	\end{tikzpicture},
\sigma_3 = \begin{tikzpicture}[scale=0.2,baseline=0pt]
		\draw[blue, thick] (0,-1) -- (0,0);
		\draw[blue, thick] (0,0) -- (1,1);
		\draw[blue, thick] (0,0) -- (-1,1);
		\filldraw[black] (0,0) circle (5pt)  {};
		\node at (0,0.7)[font=\fontsize{7pt}{0}]{$1$};
	\end{tikzpicture} \Big).$$ 
As $\sigma_{1}$ has 2 leaves and $\sigma_{2}$ has 3 leaves, we consider
$\sigma_{2}\mapsto(\sigma_{2}^{(1)},\sigma_{2}^{(2)})$ and $\sigma_{3}\mapsto(\sigma_{3}^{(1)},\sigma_{3}^{(2)},\sigma_{3}^{(3)},\sigma_{3}^{(4)})$
for the left hand side of Proposition~\ref{prop:antipode-monomial-ssym}  (1) to be computable.
Out of the three possibilities for $\sigma_{2}\mapsto(\sigma_{2}^{(1)},\sigma_{2}^{(2)})$
and four possibilities for $\sigma_{3}\mapsto(\sigma_{3}^{(1)},\sigma_{3}^{(2)},\sigma_{3}^{(3)},\sigma_{3}^{(4)})$,
twelve in total, the only two that satisfy Proposition~\ref{prop:antipode-monomial-ssym} (1) are
\begin{equation} \label{antipodeexample}
\begin{tikzpicture}[scale=0.2,baseline=0pt]
		\draw[blue, thick] (0,-1) -- (0,0);
		\draw[blue, thick] (0,0) -- (1,1);
		\draw[blue, thick] (0,0) -- (-1,1);
		\filldraw[black] (0,0) circle (5pt)  {};
		\node at (0,0.7)[font=\fontsize{7pt}{0}]{$1$};
	\end{tikzpicture}
\overleftarrow{\shuffle}\Big(~
\begin{tikzpicture}[scale=0.2,baseline=0pt]
		\draw[red, thick] (0,-1) -- (0,0);
		\draw[red, thick] (0,0) -- (1,1);
		\draw[red, thick] (0,0) -- (-1,1);
		\filldraw[black] (0,0) circle (5pt)  {};
		\node at (0,0.7)[font=\fontsize{7pt}{0}]{$1$};
	\end{tikzpicture}~,~
\begin{tikzpicture}[scale=0.2,baseline=0pt]
		\draw[red, thick] (0,-1) -- (0,0);
		\draw[red, thick] (0,0) -- (1,1);
		\draw[red, thick] (0,0) -- (-1,1);
		\filldraw[black] (0,0) circle (5pt)  {};
		\node at (0,0.7)[font=\fontsize{7pt}{0}]{$2$};
	\end{tikzpicture}~\Big) 
\overleftarrow{\shuffle}\Big(~
	\begin{tikzpicture}[scale=0.25,baseline=0pt]
		\draw[ForestGreen, thick] (0,-1) -- (0,1);
	\end{tikzpicture}~,~
	\begin{tikzpicture}[scale=0.25,baseline=0pt]
		\draw[ForestGreen, thick] (0,-1) -- (0,1);
	\end{tikzpicture}~,~
	\begin{tikzpicture}[scale=0.25,baseline=0pt]
		\draw[ForestGreen, thick] (0,-1) -- (0,1);
	\end{tikzpicture}~,~
	\begin{tikzpicture}[scale=0.2,baseline=0pt]
		\draw[ForestGreen, thick] (0,-1) -- (0,0);
		\draw[ForestGreen, thick] (0,0) -- (1,1);
		\draw[ForestGreen, thick] (0,0) -- (-1,1);
		\filldraw[black] (0,0) circle (5pt)  {};
		\node at (0,0.7)[font=\fontsize{7pt}{0}]{$1$};
	\end{tikzpicture}~\Big) 
	=\begin{tikzpicture}[scale=0.2,baseline=0pt]
		\draw[blue, thick] (0,-1) -- (0,0);
		\draw[blue, thick] (0,0) -- (3,3);
		\draw[blue, thick] (0,0) -- (-3,3);
		\draw[red, thick] (-3,3) -- (-4,4);
		\draw[red, thick] (-3,3) -- (-2,4);
		\draw[red, thick] (3,3) -- (2,4);
		\draw[red, thick] (3,3) -- (4,4);
		\draw[ForestGreen, thick] (-4,4) -- (-5,5);
		\draw[ForestGreen, thick] (-2,4) -- (-1,5);
		\draw[ForestGreen, thick] (2,4) -- (1,5);
		\draw[ForestGreen, thick] (4,4) -- (3,5);
		\draw[ForestGreen, thick] (4,4) -- (5,5);
		\filldraw[black] (0,0) circle (5pt)  {};
		\filldraw[black] (-3,3) circle (5pt)  {};
		\filldraw[black] (3,3) circle (5pt)  {};
		\filldraw[black] (4,4) circle (5pt)  {};
		\node at (0,0.7)[font=\fontsize{7pt}{0}]{$1$};
		\node at (-3,3.7)[font=\fontsize{7pt}{0}]{$2$};
		\node at (3,3.7)[font=\fontsize{7pt}{0}]{$3$};
		\node at (4,4.7)[font=\fontsize{7pt}{0}]{$4$};
	\end{tikzpicture}; \qquad
	\begin{tikzpicture}[scale=0.2,baseline=0pt]
		\draw[blue, thick] (0,-1) -- (0,0);
		\draw[blue, thick] (0,0) -- (1,1);
		\draw[blue, thick] (0,0) -- (-1,1);
		\filldraw[black] (0,0) circle (5pt)  {};
		\node at (0,0.7)[font=\fontsize{7pt}{0}]{$1$};
	\end{tikzpicture}
\overleftarrow{\shuffle}\Big(~
	\begin{tikzpicture}[scale=0.25,baseline=0pt]
		\draw[red, thick] (0,-1) -- (0,1);
	\end{tikzpicture}~,~
	\begin{tikzpicture}[scale=0.2,baseline=0pt]
		\draw[red, thick] (0,-1) -- (0,0);
		\draw[red, thick] (0,0) -- (2,2);
		\draw[red, thick] (0,0) -- (-2,2);
		\draw[red, thick] (1,1) -- (0,2);
		\filldraw[black] (0,0) circle (5pt)  {};
		\filldraw[black] (1,1) circle (5pt)  {};
		\node at (0,0.7)[font=\fontsize{7pt}{0}]{$1$};
		\node at (1,1.7)[font=\fontsize{7pt}{0}]{$2$};
	\end{tikzpicture}~\Big) 
\overleftarrow{\shuffle}\Big(~
	\begin{tikzpicture}[scale=0.25,baseline=0pt]
		\draw[ForestGreen, thick] (0,-1) -- (0,1);
	\end{tikzpicture}~,~
	\begin{tikzpicture}[scale=0.25,baseline=0pt]
		\draw[ForestGreen, thick] (0,-1) -- (0,1);
	\end{tikzpicture}~,~
	\begin{tikzpicture}[scale=0.25,baseline=0pt]
		\draw[ForestGreen, thick] (0,-1) -- (0,1);
	\end{tikzpicture}~,~
	\begin{tikzpicture}[scale=0.2,baseline=0pt]
		\draw[ForestGreen, thick] (0,-1) -- (0,0);
		\draw[ForestGreen, thick] (0,0) -- (1,1);
		\draw[ForestGreen, thick] (0,0) -- (-1,1);
		\filldraw[black] (0,0) circle (5pt)  {};
		\node at (0,0.7)[font=\fontsize{7pt}{0}]{$1$};
	\end{tikzpicture}~\Big) 
	=\begin{tikzpicture}[scale=0.2,baseline=0pt]
		\draw[blue, thick] (0,-1) -- (0,0);
		\draw[blue, thick] (0,0) -- (1,1);
		\draw[red, thick] (1,1) -- (3,3);
		\draw[ForestGreen, thick] (3,3) -- (4,4);
		\draw[blue, thick] (0,0) -- (-1,1);
		\draw[red, thick] (1,1) -- (-1,3);
		\draw[red, thick] (2,2) -- (1,3);
		\draw[red, thick] (-1,1) -- (-3,3);
		\draw[ForestGreen, thick] (-3,3) -- (-4,4);
		\draw[ForestGreen, thick] (-1,3) -- (-2,4);
		\draw[ForestGreen, thick] (1,3) -- (0,4);
		\draw[ForestGreen, thick] (3,3) -- (2,4);
		\filldraw[black] (0,0) circle (5pt)  {};
		\filldraw[black] (1,1) circle (5pt)  {};
		\filldraw[black] (2,2) circle (5pt)  {};
		\filldraw[black] (3,3) circle (5pt)  {};
		\node at (0,0.7)[font=\fontsize{7pt}{0}]{$1$};
		\node at (1,1.7)[font=\fontsize{7pt}{0}]{$2$};
		\node at (2,2.7)[font=\fontsize{7pt}{0}]{$3$};
		\node at (3,3.7)[font=\fontsize{7pt}{0}]{$4$};
		\end{tikzpicture}.
\end{equation}To check condition (2): the only $\sigma'>\sigma$ is $\sigma'=\begin{tikzpicture}[scale=0.2,baseline=0pt]
		\draw[blue, thick] (0,-1) -- (0,0);
		\draw[blue, thick] (0,0) -- (4,4);
		\draw[blue, thick] (0,0) -- (-4,4);
		\draw[blue, thick] (-2,2) -- (0,4);
		\draw[blue, thick] (-3,3) -- (-2,4);
		\draw[blue, thick] (-1,1) -- (2,4);
		\filldraw[black] (0,0) circle (5pt)  {};
		\filldraw[black] (-2,2) circle (5pt)  {};
		\filldraw[black] (-3,3) circle (5pt)  {};
		\filldraw[black] (-1,1) circle (5pt)  {};
		\node at (0,0.7)[font=\fontsize{7pt}{0}]{$1$};
		\node at (-1,1.7)[font=\fontsize{7pt}{0}]{$2$};
		\node at (-2,2.7)[font=\fontsize{7pt}{0}]{$3$};
		\node at (-3,3.7)[font=\fontsize{7pt}{0}]{$4$};
\end{tikzpicture}$, and the analogous calculations are
\begin{equation*}
\begin{tikzpicture}[scale=0.2,baseline=0pt]
		\draw[blue, thick] (0,-1) -- (0,0);
		\draw[blue, thick] (0,0) -- (1,1);
		\draw[blue, thick] (0,0) -- (-1,1);
		\filldraw[black] (0,0) circle (5pt)  {};
		\node at (0,0.7)[font=\fontsize{7pt}{0}]{$1$};
	\end{tikzpicture}
\overleftarrow{\shuffle}\Big(
\begin{tikzpicture}[scale=0.2,baseline=0pt]
		\draw[blue, thick] (0,-1) -- (0,0);
		\draw[blue, thick] (0,0) -- (1,1);
		\draw[blue, thick] (0,0) -- (-1,1);
		\filldraw[black] (0,0) circle (5pt)  {};
		\node at (0,0.7)[font=\fontsize{7pt}{0}]{$2$};
	\end{tikzpicture},
\begin{tikzpicture}[scale=0.2,baseline=0pt]
		\draw[blue, thick] (0,-1) -- (0,0);
		\draw[blue, thick] (0,0) -- (1,1);
		\draw[blue, thick] (0,0) -- (-1,1);
		\filldraw[black] (0,0) circle (5pt)  {};
		\node at (0,0.7)[font=\fontsize{7pt}{0}]{$1$};
	\end{tikzpicture}\Big) 
\overleftarrow{\shuffle}\Big(
	\begin{tikzpicture}[scale=0.25,baseline=0pt]
		\draw[blue, thick] (0,-1) -- (0,1);
	\end{tikzpicture},
	\begin{tikzpicture}[scale=0.25,baseline=0pt]
		\draw[blue, thick] (0,-1) -- (0,1);
	\end{tikzpicture},
	\begin{tikzpicture}[scale=0.25,baseline=0pt]
		\draw[blue, thick] (0,-1) -- (0,1);
	\end{tikzpicture},
	\begin{tikzpicture}[scale=0.2,baseline=0pt]
		\draw[blue, thick] (0,-1) -- (0,0);
		\draw[blue, thick] (0,0) -- (1,1);
		\draw[blue, thick] (0,0) -- (-1,1);
		\filldraw[black] (0,0) circle (5pt)  {};
		\node at (0,0.7)[font=\fontsize{7pt}{0}]{$1$};
	\end{tikzpicture}\Big) 
	=\begin{tikzpicture}[scale=0.2,baseline=0pt]
		\draw[blue, thick] (0,-1) -- (0,0);
		\draw[blue, thick] (0,0) -- (4,4);
		\draw[blue, thick] (0,0) -- (-4,4);
		\draw[blue, thick] (-3,3) -- (-2,4);
		\draw[blue, thick] (2,2) -- (0,4);
		\draw[blue, thick] (3,3) -- (2,4);
		\filldraw[black] (0,0) circle (5pt)  {};
		\filldraw[black] (-3,3) circle (5pt)  {};
		\filldraw[black] (2,2) circle (5pt)  {};
		\filldraw[black] (3,3) circle (5pt)  {};
		\node at (0,0.7)[font=\fontsize{7pt}{0}]{$1$};
		\node at (-3,3.7)[font=\fontsize{7pt}{0}]{$3$};
		\node at (2,2.7)[font=\fontsize{7pt}{0}]{$2$};
		\node at (3,3.7)[font=\fontsize{7pt}{0}]{$4$};
	\end{tikzpicture}\not \leq_{w}\tau; \qquad
\begin{tikzpicture}[scale=0.2,baseline=0pt]
		\draw[blue, thick] (0,-1) -- (0,0);
		\draw[blue, thick] (0,0) -- (1,1);
		\draw[blue, thick] (0,0) -- (-1,1);
		\filldraw[black] (0,0) circle (5pt)  {};
		\node at (0,0.7)[font=\fontsize{7pt}{0}]{$1$};
	\end{tikzpicture}
\overleftarrow{\shuffle}\Big(
	\begin{tikzpicture}[scale=0.25,baseline=0pt]
		\draw[blue, thick] (0,-1) -- (0,1);
	\end{tikzpicture},
	\begin{tikzpicture}[scale=0.2,baseline=0pt]
		\draw[blue, thick] (0,-1) -- (0,0);
		\draw[blue, thick] (0,0) -- (2,2);
		\draw[blue, thick] (0,0) -- (-2,2);
		\draw[blue, thick] (-1,1) -- (0,2);
		\filldraw[black] (0,0) circle (5pt)  {};
		\filldraw[black] (-1,1) circle (5pt)  {};
		\node at (0,0.7)[font=\fontsize{7pt}{0}]{$1$};
		\node at (-1,1.7)[font=\fontsize{7pt}{0}]{$2$};
	\end{tikzpicture}\Big) 
\overleftarrow{\shuffle}\Big(
	\begin{tikzpicture}[scale=0.25,baseline=0pt]
		\draw[blue, thick] (0,-1) -- (0,1);
	\end{tikzpicture},
	\begin{tikzpicture}[scale=0.25,baseline=0pt]
		\draw[blue, thick] (0,-1) -- (0,1);
	\end{tikzpicture},
	\begin{tikzpicture}[scale=0.25,baseline=0pt]
		\draw[blue, thick] (0,-1) -- (0,1);
	\end{tikzpicture},
	\begin{tikzpicture}[scale=0.2,baseline=0pt]
		\draw[blue, thick] (0,-1) -- (0,0);
		\draw[blue, thick] (0,0) -- (1,1);
		\draw[blue, thick] (0,0) -- (-1,1);
		\filldraw[black] (0,0) circle (5pt)  {};
		\node at (0,0.7)[font=\fontsize{7pt}{0}]{$1$};
	\end{tikzpicture}\Big) 
	=\begin{tikzpicture}[scale=0.2,baseline=0pt]
		\draw[blue, thick] (0,-1) -- (0,0);
		\draw[blue, thick] (0,0) -- (4,4);
		\draw[blue, thick] (0,0) -- (-4,4);
		\draw[blue, thick] (1,1) -- (-2,4);
		\draw[blue, thick] (-1,3) -- (0,4);
		\draw[blue, thick] (3,3) -- (2,4);
		\filldraw[black] (0,0) circle (5pt)  {};
		\filldraw[black] (1,1) circle (5pt)  {};
		\filldraw[black] (-1,3) circle (5pt)  {};
		\filldraw[black] (3,3) circle (5pt)  {};
		\node at (0,0.7)[font=\fontsize{7pt}{0}]{$1$};
		\node at (1,1.7)[font=\fontsize{7pt}{0}]{$2$};
		\node at (-1,3.7)[font=\fontsize{7pt}{0}]{$3$};
		\node at (3,3.7)[font=\fontsize{7pt}{0}]{$4$};
		\end{tikzpicture} \not \leq_{w}\tau,
\end{equation*}
so both lightening-split-choices satisfy (2). 

To check condition (3): for the shuffle on the left in (\ref{antipodeexample}),
all nodes of \textcolor{ForestGreen}{$\sigma_3$} are on the right
of all nodes of \textcolor{red}{$\sigma_2$}, so it may be computed
using $\sigma_{2}\backslash\sigma_{3}$; indeed for the splitting $\sigma_{2}\backslash\sigma_{3} = \begin{tikzpicture}[scale=0.2,baseline=0pt]
		\draw[blue, thick] (0,-1) -- (0,0);
		\draw[blue, thick] (0,0) -- (3,3);
		\draw[blue, thick] (0,0) -- (-3,3);
		\draw[blue, thick] (1,1) -- (-1,3);
		\draw[blue, thick] (2,2) -- (1,3);
		\filldraw[black] (0,0) circle (5pt)  {};
		\filldraw[black] (1,1) circle (5pt)  {};
		\filldraw[black] (2,2) circle (5pt)  {};
		\node at (0,0.7)[font=\fontsize{7pt}{0}]{$1$};
		\node at (1,1.7)[font=\fontsize{7pt}{0}]{$2$};
		\node at (2,2.7)[font=\fontsize{7pt}{0}]{$3$};
\end{tikzpicture}
\mapsto \Big(
\begin{tikzpicture}[scale=0.2,baseline=0pt]
		\draw[blue, thick] (0,-1) -- (0,0);
		\draw[blue, thick] (0,0) -- (1,1);
		\draw[blue, thick] (0,0) -- (-1,1);
		\filldraw[black] (0,0) circle (5pt)  {};
		\node at (0,0.7)[font=\fontsize{7pt}{0}]{$1$};
	\end{tikzpicture},
	\begin{tikzpicture}[scale=0.2,baseline=0pt]
		\draw[blue, thick] (0,-1) -- (0,0);
		\draw[blue, thick] (0,0) -- (2,2);
		\draw[blue, thick] (0,0) -- (-2,2);
		\draw[blue, thick] (1,1) -- (0,2);
		\filldraw[black] (0,0) circle (5pt)  {};
		\filldraw[black] (1,1) circle (5pt)  {};
		\node at (0,0.7)[font=\fontsize{7pt}{0}]{$2$};
		\node at (1,1.7)[font=\fontsize{7pt}{0}]{$3$};
		\end{tikzpicture}
\Big)$ we have 	$\begin{tikzpicture}[scale=0.2,baseline=0pt]
		\draw[blue, thick] (0,-1) -- (0,0);
		\draw[blue, thick] (0,0) -- (1,1);
		\draw[blue, thick] (0,0) -- (-1,1);
		\filldraw[black] (0,0) circle (5pt)  {};
		\node at (0,0.7)[font=\fontsize{7pt}{0}]{$1$};
	\end{tikzpicture}
\overleftarrow{\shuffle} \Big(
\begin{tikzpicture}[scale=0.2,baseline=0pt]
		\draw[blue, thick] (0,-1) -- (0,0);
		\draw[blue, thick] (0,0) -- (1,1);
		\draw[blue, thick] (0,0) -- (-1,1);
		\filldraw[black] (0,0) circle (5pt)  {};
		\node at (0,0.7)[font=\fontsize{7pt}{0}]{$1$};
	\end{tikzpicture},
	\begin{tikzpicture}[scale=0.2,baseline=0pt]
		\draw[blue, thick] (0,-1) -- (0,0);
		\draw[blue, thick] (0,0) -- (2,2);
		\draw[blue, thick] (0,0) -- (-2,2);
		\draw[blue, thick] (1,1) -- (0,2);
		\filldraw[black] (0,0) circle (5pt)  {};
		\filldraw[black] (1,1) circle (5pt)  {};
		\node at (0,0.7)[font=\fontsize{7pt}{0}]{$2$};
		\node at (1,1.7)[font=\fontsize{7pt}{0}]{$3$};
		\end{tikzpicture}
\Big)
=\begin{tikzpicture}[scale=0.2,baseline=0pt]
		\draw[blue, thick] (0,-1) -- (0,0);
		\draw[blue, thick] (0,0) -- (2,2);
		\draw[blue, thick] (0,0) -- (-3,3);
		\draw[blue, thick] (-3,3) -- (-4,4);
		\draw[blue, thick] (-3,3) -- (-2,4);
		\draw[blue, thick] (2,2) -- (0,4);
		\draw[blue, thick] (2,2) -- (3,3);
		\draw[blue, thick] (3,3) -- (2,4);
		\draw[blue, thick] (3,3) -- (4,4);
		\filldraw[black] (0,0) circle (5pt)  {};
		\filldraw[black] (-3,3) circle (5pt)  {};
		\filldraw[black] (2,2) circle (5pt)  {};
		\filldraw[black] (3,3) circle (5pt)  {};
		\node at (0,0.7)[font=\fontsize{7pt}{0}]{$1$};
		\node at (-3,3.7)[font=\fontsize{7pt}{0}]{$2$};
		\node at (2,2.7)[font=\fontsize{7pt}{0}]{$3$};
		\node at (3,3.7)[font=\fontsize{7pt}{0}]{$4$};
	\end{tikzpicture}$. Compare with the same calculation using $\sigma_{2}/\sigma_{3}$: $\sigma_{2}/\sigma_{3} = \begin{tikzpicture}[scale=0.2,baseline=0pt]
		\draw[blue, thick] (0,-1) -- (0,0);
		\draw[blue, thick] (0,0) -- (3,3);
		\draw[blue, thick] (0,0) -- (-3,3);
		\draw[blue, thick] (-1,1) -- (1,3);
		\draw[blue, thick] (0,2) -- (-1,3);
		\filldraw[black] (0,0) circle (5pt)  {};
		\filldraw[black] (-1,1) circle (5pt)  {};
		\filldraw[black] (0,2) circle (5pt)  {};
		\node at (0,0.7)[font=\fontsize{7pt}{0}]{$1$};
		\node at (-1,1.7)[font=\fontsize{7pt}{0}]{$2$};
		\node at (0,2.7)[font=\fontsize{7pt}{0}]{$3$};
\end{tikzpicture}
\mapsto \Big(
\begin{tikzpicture}[scale=0.2,baseline=0pt]
		\draw[blue, thick] (0,-1) -- (0,0);
		\draw[blue, thick] (0,0) -- (1,1);
		\draw[blue, thick] (0,0) -- (-1,1);
		\filldraw[black] (0,0) circle (5pt)  {};
		\node at (0,0.7)[font=\fontsize{7pt}{0}]{$2$};
	\end{tikzpicture},
	\begin{tikzpicture}[scale=0.2,baseline=0pt]
		\draw[blue, thick] (0,-1) -- (0,0);
		\draw[blue, thick] (0,0) -- (2,2);
		\draw[blue, thick] (0,0) -- (-2,2);
		\draw[blue, thick] (-1,1) -- (0,2);
		\filldraw[black] (0,0) circle (5pt)  {};
		\filldraw[black] (-1,1) circle (5pt)  {};
		\node at (0,0.7)[font=\fontsize{7pt}{0}]{$3$};
		\node at (-1,1.7)[font=\fontsize{7pt}{0}]{$1$};
	\end{tikzpicture}\Big) $ and 	$\begin{tikzpicture}[scale=0.2,baseline=0pt]
		\draw[blue, thick] (0,-1) -- (0,0);
		\draw[blue, thick] (0,0) -- (1,1);
		\draw[blue, thick] (0,0) -- (-1,1);
		\filldraw[black] (0,0) circle (5pt)  {};
		\node at (0,0.7)[font=\fontsize{7pt}{0}]{$1$};
	\end{tikzpicture}
\overleftarrow{\shuffle} \Big(
\begin{tikzpicture}[scale=0.2,baseline=0pt]
		\draw[blue, thick] (0,-1) -- (0,0);
		\draw[blue, thick] (0,0) -- (1,1);
		\draw[blue, thick] (0,0) -- (-1,1);
		\filldraw[black] (0,0) circle (5pt)  {};
		\node at (0,0.7)[font=\fontsize{7pt}{0}]{$2$};
	\end{tikzpicture},
	\begin{tikzpicture}[scale=0.2,baseline=0pt]
		\draw[blue, thick] (0,-1) -- (0,0);
		\draw[blue, thick] (0,0) -- (2,2);
		\draw[blue, thick] (0,0) -- (-2,2);
		\draw[blue, thick] (-1,1) -- (0,2);
		\filldraw[black] (0,0) circle (5pt)  {};
		\filldraw[black] (-1,1) circle (5pt)  {};
		\node at (0,0.7)[font=\fontsize{7pt}{0}]{$3$};
		\node at (-1,1.7)[font=\fontsize{7pt}{0}]{$1$};
	\end{tikzpicture}\Big) 
=\begin{tikzpicture}[scale=0.2,baseline=0pt]
		\draw[blue, thick] (0,-1) -- (0,0);
		\draw[blue, thick] (0,0) -- (4,4);
		\draw[blue, thick] (0,0) -- (-4,4);
		\draw[blue, thick] (-3,3) -- (-2,4);
		\draw[blue, thick] (2,2) -- (0,4);
		\draw[blue, thick] (1,3) -- (2,4);
		\filldraw[black] (0,0) circle (5pt)  {};
		\filldraw[black] (-3,3) circle (5pt)  {};
		\filldraw[black] (2,2) circle (5pt)  {};
		\filldraw[black] (1,3) circle (5pt)  {};
		\node at (0,0.7)[font=\fontsize{7pt}{0}]{$1$};
		\node at (-3,3.7)[font=\fontsize{7pt}{0}]{$3$};
		\node at (2,2.7)[font=\fontsize{7pt}{0}]{$2$};
		\node at (1,3.7)[font=\fontsize{7pt}{0}]{$4$};
		\end{tikzpicture} \not \leq_{w}\tau$, so this lightening-split-choice satisfies (3).

Because a permutation is completely determined by its node labels,
without reference to the binary tree, the above calculation can be
entirely computed using the node labels only. (But the tree calculation
is necessary on similar algebras with unlabeled trees, or where the
node labels do not determine the tree.) Let $d_{i}=\deg\sigma_{1}+\cdots+\deg\sigma_{i}$,
i.e. $d_{i}$ is the $i$th element of $\GD(\sigma)$. If $\{d_{i-1}+1,d_{i-1}+2,\dots,d_{i}\}$
appear setwise to the left of $\{d_{i}+1,d_{i}+2,\dots,d_{i+1}\}$
(in the example, 2 and 3 are both to the left of 4), then ``calculating
with $\sigma_{i}/\sigma_{i+1}$ instead of $\sigma_{i}\backslash\sigma_{i+1}$''
amounts to replacing $d_{i}+1,d_{i}+2,\dots,d_{i+1}$ with $d_{i-1}+1,d_{i-1}+2,\dots,d_{i-1}+(d_{i+1}-d_{i})$
(e.g. 4 with 2) and $d_{i-1}+1,d_{i-1}+2,\dots,d_{i}$ with $d_{i-1}+(d_{i+1}-d_{i})+1,\dots d_{i+1}$
(e.g. 2, 3 with 3, 4). We check condition (3) for the shuffle on the
right of (\ref{antipodeexample}) using this expedited method. Here,
1 appears to the left of 2 and 3, and $3124\not\leq_{w}\tau$;
2 and 3 appear to the left of 4, and $1342\not\leq_{w}\tau$,
so this lightening-split-choice also satisfies (3). 

Hence $\gamma_{\sigma}^{\tau}=2$. \end{Example}

\subsection{The inheritance of the axioms through subalgebras and quotienting}

This section contains two theorems for transferring  the above properties of monomial bases to subalgebras and quotient algebras. 

In this section, we will
assume that  $\mathcal{H}$ is a connected
Hopf algebra satisfying conditions ($\Delta$1)-($\Delta$3), ($m$0)-($m$3) and ($\mathcal{S}$0)-($\mathcal{S}$3).

First, we consider subalgebras that are spanned by induced subposets:

\begin{Theorem}\label{thm:monomialsubposet} Assume that $Q_{n}\subseteq P_{n}$ for each $n$, with the induced order, and
write $Q=\bigsqcup Q_{n}$.
\begin{itemize}
\item If for all $f,g\in Q$, we have ${}^{i}f,f^{i}\in Q$ for all
$i\in\allow(f)$ and $f/g\in Q$, then $\Span\{F_{f}:f\in Q\}$ is
a subcoalgebra of $\mathcal{H}$ satisfying ($\Delta$1)-($\Delta$3).
\item If for all $f,g\in Q$,
we have $\zeta(f,g)\in Q$ for all $\zeta\in Sh(\cmpt(f),\cmpt(g))$
and, when defined,  $f\vee g\in Q$, then $\Span\{F_{f}:f\in Q\}$ is a subalgebra of
$\mathcal{H}$ satisfying ($m$0)-($m$3).
\item If all above hypotheses on $Q$ hold,
and for all $f,g\in Q$, we have $f\backslash g\in Q$. Then $\Span\{F_{f}:f\in Q\}$
satisfies ($\mathcal{S}$0)-($\mathcal{S}$3).
\end{itemize}
\end{Theorem}

The proof is trivial. Note that, because $Q$ is closed under ``deconcatenation''
and ``maximal concatenation'' $/$, the meaning of $\GD$ is the
same in $Q$ as in $P$. However, unless $Q$ is obtained as a collection of upper order-ideals in $P$,
the monomial basis elements of $\Span\{F_{f}:f\in Q\}$ created using
$Q$ are not monomial basis elements of $\mathcal{H}$.

The next situation of inheritance of axioms is more technical but more interesting. 
For example, it will be used to pass from labeled to unlabeled trees. For the rest
of this section, we will assume we have a second graded connected Hopf algebra $\bar{\mathcal{H}}$ defined  by a poset as follow.
 Let $\bar{P}_{n}$ be posets for each integer  $n\geq0$.
Assume that $\bar{\mathcal{H}}$,  as  a graded vector
space, has a fundamental basis $\{F_{t}:t\in\bar{P}_{n}\}$. In both ${\mathcal{H}}$ and  $\bar{\mathcal{H}}$  we  define 
 a monomial basis $\{M_{f}:f\in P_{n}\}$ (resp. $\{M_{t}:t\in\bar{P}_{n}\}$), related by $F_{f}=\sum_{g\geq f}M_{g}$ (resp. $F_{t}=\sum_{s\geq t}M_{s}$).
Let $\pi_{n}:P_{n}\rightarrow\bar{P}_{n}$ be a sequence of surjective order-preserving
maps, and let $\Pi:\mathcal{H}\rightarrow\bar{\mathcal{H}}$ given
by $\Pi(F_{f})=F_{\pi_{\deg f}(f)}$ be the induced linear map. One
key axiom that  we need is the following:
\begin{enumerate}
\item[($\pi$0).]  For each $n$, there exists maps $\iota_{n}:\bar{P}_{n}\rightarrow P_{n}$
given by $\iota_{n}(t)=\max\{f\in P_{n}:\pi_{n}(f)=t\}$, and $\iota_{n}$
is order-preserving.
\end{enumerate}

Axiom ($\pi$0) requires that, for each $t\in \bar{P}_{n}$, it is possible to compute $\max\{f\in P_{n}:\pi_{n}(f)=t\}$. In particular, $\{f\in P_{n}:\pi_{n}(f)=t\}$ must lie in the same connected component of $P_{n}$.
Axiom ($\pi$0) implies that $\pi_n$ and $\iota_n$ form a \emph{Galois connection}, i.e. $\pi_n(f)\leq s$ if and only if $f\leq \iota_n(s)$. It follows that the connected components $C_{i}$ of $P_{n}$ are in bijection with the connected components $\bar{C}_{i}$ of $\bar{P}_{n}$. Since $\pi_n$ is part of a Galois connection it preserves least upper bounds, which we record as the following lemma.

\begin{Lemma}\label{lem:pipreserve-leastupperbound} If axiom ($\pi$0)
is satisfied, then $\pi_{n}(f\vee g)=\pi_{n}f\vee\pi_{n}g$ for all
$f,g\in C_{i}\subseteq P_{n}$. \qed \end{Lemma}

An immediate consequence of this Lemma is that ($m$0) is satisfied for $\bar{P}_{n}$. 
The other relevant axioms that we require from $\pi$ to be able to transfer axioms
($\Delta$1)-($\Delta$3), ($m$0)-($m$3), ($\mathcal{S}$0)-($\mathcal{S}$3) from $\mathcal{H}$ to  $\bar{\mathcal{H}}$ are:
\begin{enumerate}
\item[($\pi\Delta$1).]  for $t\in\bar{P}_{n}$, we have a map $\allow(t)\subseteq\{1,\ldots,n-1\}$ such that
$$\displaystyle\Delta_{+}(F_{t})=\sum_{i\in\allow(t)}F_{{}^{i}t}\otimes F_{t^{i}}$$
for some ${}^{i}t\in\bar{P}_{i}$, $t^{i}\in\bar{P}_{n-i}$; and for all $f\in P_{n}$, we have $\allow(f)=\allow(\pi_{n}f)$.
\item[($\pi\Delta$2).]  $\pi$ commutes with deconcatenation: ${}^{i}(\pi_{n}f)=\pi_{i}\left({}^{i}f\right)$,
$(\pi_{n}f)^{i}=\pi_{n-i}\left(f^{i}\right)$ for $f\in P_{n}$.
\item[($\pi m$1).]  For all connected components $\bar{C}_{i},\bar{C}_{j}$ of $\bar{P}_{n},\bar{P}_{m}$, we have a set $Sh(\bar{C}_{i},\bar{C}_{j})$ of shuffles such that, 
for $s\in\bar{C}_{i}$ and $t\in\bar{C}_{j}$, we have $F_{s}F_{t}=\sum_{\bar{\zeta}\in Sh(\bar{C}_{i},\bar{C}_{j})}F_{\bar{\zeta}(s,t)}$.
Axiom ($\pi$0) gives us corresponding connected components ${C}_{i},{C}_{j}$ of ${P}_{n},{P}_{m}$ and we require a 
bijection $Sh(C_{i},C_{j})\rightarrow Sh(\bar{C}_{i},\bar{C}_{j})$.
\item[($\pi m$2).]  For all $\zeta\in Sh(C_{i},C_{j})$, in bijection with $\bar{\zeta}\in Sh(\bar{C}_{i},\bar{C}_{j})$, and for all $f\in C_{i}$, $g\in C_{j}$,
we have $\pi_{n}(\zeta(f,g))=\bar{\zeta}(\pi_{n}(f),\pi_{m}(g))$. 
\item[($\pi\mathcal{S}$1).]  $\GD(\iota_{n}(t))=\GD(t)$ for $t\in\bar{P}_{n}$.
\item[($\pi\mathcal{S}$2).]  $\pi_{n}(g/h)=\pi_{i}(g)/\pi_{n-i}(h)$ for all $g\in P_{i}$, $h\in P_{n-i}$.
\end{enumerate}

\begin{Theorem}\label{thm:monomialquotient} For  $\mathcal{H},\bar{\mathcal{H}}$, if the maps $\pi_n\colon P_n\to\bar{P}_n$
satisfy ($\pi$0), then
\begin{equation}
\Pi(M_{f})=\begin{cases}
M_{t} & \text{if }f=\iota(t)\text{ for some }t;\\
0 & \text{otherwise}.
\end{cases}\label{eq:quotientmonomial}
\end{equation}
Furthermore:
\begin{itemize}
\item If  ($\pi\Delta$1)-($\pi\Delta$2) are satisfied, then $\Pi$ is a coalgebra morphism,
and $\bar{\mathcal{H}}$ also satisfies axioms ($\Delta$1)-($\Delta$3), with
$s/t$ defined as $\pi_{n}(\iota_{i}(s)/\iota_{n-i}(t))$ for $s\in\bar{P}_{i}$
and $t\in\bar{P}_{n-i}$.
\item If ($\pi m$1)-($\pi m$2) are satisfied, then $\Pi$ is an algebra morphism,
and $\bar{\mathcal{H}}$ also satisfies ($m$0)-($m$3).
\item If  ($\pi\Delta$1)-($\pi\Delta$2),
($\pi m$1)-($\pi m$2) and ($\pi\mathcal{S}$1)-($\pi\mathcal{S}$2) are satisfied, then $\bar{\mathcal{H}}$ also
satisfies ($\mathcal{S}$0)-($\mathcal{S}$3).
\end{itemize}
\end{Theorem}

See Section 7.4 for the proof. 
\begin{Remark}
  Note that this theorem transfers the full axioms from $\mathcal{H}$ to $\bar{\mathcal{H}}$, in contrast to \cite{AS06} where they only transfer the monomial basis formulas (for coproduct, product and antipode). The formula transfer uses only that $\pi_n$ and $\iota_n$ form a Galois connection, but transferring the axioms seems to require the stronger ($\pi$0) axiom. 
\end{Remark}

\subsection{Axioms on  \texorpdfstring{$\SSym$}{} and  \texorpdfstring{$\YSym$}{}}

Before using all the developed axiomatic machinery on new Hopf algebras we illustrate first how the axioms are satisfied by $\SSym$ and $\YSym$. We will use that the axioms are satisfied by these two Hopf algebras in our construction of a Hopf algebra on parking functions in Section \ref{sec:parking_functions}.

\begin{Proposition}\label{prop:axiomssym} $\SSym$ satisfies axioms
($\Delta$1)-($\Delta$3), ($m$0)-($m$3), ($\mathcal{S}$0)-($\mathcal{S}$3).\end{Proposition}
\begin{proof}
Axiom ($\Delta$1) is clear. ($\Delta$2) follows from \cite[Prop. 2.5.i]{AS05}.
($\Delta$3) follows from \cite[Prop. 2.5.ii]{AS05}: the first case shows
that $/$ is order-preserving, and the third case gives the maximality.
That $/$ is a one-sided inverse to $\sigma\mapsto({}^{i}\sigma,\sigma^{i})$
is clear from the definitions. 

Axiom ($m$0) is clear. Axiom ($m$1) is from the definition of the algebra (see Definition~\ref{def:ssym}). Axioms ($m$2) and ($m$3) are \cite[Prop. 2.10]{AS05}.

Axiom ($\mathcal{S}$0) follows from Remark \ref{rem:s0trees} and ($\mathcal{S}$1)
follows from \cite[Lem. 2.14]{AS05}. 

The intuition behind ($\mathcal{S}$2) is that all inversions of $\sigma'_{1}/\cdots/\sigma'_{k}$
are either ``within'' an $\sigma'_{i}$, in which case they are
in $\sigma'_{1}\backslash\cdots\backslash\sigma'_{k}$, or ``between''
two $\sigma'_{i}$s, in which case they are in $\sigma_{1}/\cdots/\sigma_{k}$.
It may be rigorously deduced from \cite[Prop. 2.16.iii]{AS05} as
follows: define $\sigma'=\sigma'_{1}/\cdots/\sigma'_{k}$ and $\sigma=\sigma_{1}/\cdots/\sigma_{k}$
(here $\sigma_{i},\sigma_{i}'$ play the role of $f_{i},f'_{i}$).
By axiom $\Delta3$, $\sigma'\geq\sigma$, so $\sigma'\vee\sigma=\sigma'$.
Let $R=\emptyset$ and $S=\{\deg\sigma_{1}+\cdots+\deg\sigma_{i}:i\in\{1,\dots,k\}\}$;
note that $S\subseteq\GD(\sigma)$ and $\sigma|S=(\sigma_{1},\dots,\sigma_{k})$.
Then \cite[Prop. 2.16.iii]{AS05} says $\sigma'_{S}\vee\sigma_{R}=(\sigma\vee\sigma')_{R\cap S}$,
where their notation $\sigma_{S}$ denotes $\sigma_{1}\backslash\cdots\backslash\sigma_{k}$.
Thus $\sigma_{R}=\sigma=\sigma_{1}/\cdots/\sigma_{k}$, $\sigma'_{S}=\sigma'_{1}\backslash\cdots\backslash\sigma'_{k}$,
$(\sigma\vee\sigma')_{R\cap S}=\sigma\vee\sigma'=\sigma'=\sigma'_{1}/\cdots/\sigma'_{k}$.

Axiom ($\mathcal{S}$3) also follows from \cite[Prop. 2.16.iii]{AS05},
by taking $\sigma=\sigma'=\sigma_{1}/\cdots/\sigma_{k}$, $R=\{d_{1},\dots,d_{i_{1}-1}\}\cup\{d_{j_{1}},\dots,d_{i_{2}-1}\}\cup\cdots\cup\{d_{j_{l-1}},\dots,d_{i_{l}-1}\}$,
$S=\{d_{i_{1}},\dots,d_{j_{1}-1}\}\cup\{d_{i_{2}},\dots,d_{j_{2}-1}\}\cup\cdots\cup\{d_{i_{l}},\dots,d_{k}\}$,
where $d_{i}=\deg f_{1}+\cdots+\deg f_{i}$.
\end{proof}
\begin{Proposition}\label{prop:axiomysym} $\YSym$ satisfies axioms
($\Delta$1)-($\Delta$3), ($m$0)-($m$3), ($\mathcal{S}$0)-($\mathcal{S}$3).\end{Proposition}

\begin{proof}
Use Theorem \ref{thm:monomialquotient} to bootstrap from $\SSym$,
as Proposition \ref{prop:pi-property-bin} provides the necessary
maps $\pi$ and $\iota$ satisfying ($\pi$0). It is clear that $\YSym$
satisfies axioms ($\pi\Delta$1) and ($\pi m$1). Axioms ($\pi\Delta$2),
($\pi m$2), ($\pi\mathcal{S}$1) and ($\pi\mathcal{S}$2) follow from Propositions
\ref{prop:pi-property} and \ref{prop:pi-property-bin}.
\end{proof}

\section{Hopf algebra of planar trees}\label{sec:TSYM}

\subsection{Two isomorphic Hopf structures}

In this section, we define two Hopf structures on the free vector space spanned by planar trees.

\begin{Definition}
	As a graded vector space, the Hopf algebra of planar trees, $\TSym$ is $\displaystyle\bigoplus_{n\geq 0}\Bbbk\PT_{n}$ where $\PT_{n}$ is the set of planar trees with $n+1$ leaves. By convention $\PT_{0}$ is the span of the empty tree. We denote two bases for $\TSym$ by $\{G_t:t\in\PT\}$ and $\{F_t:t\in\PT\}$.
\end{Definition}

Recall that a splitting $t\mapsto(t_1,t_2)$ is \textbf{allowable} if $\ideg(t)=\ideg(t_1)+\ideg(t_2)$. We can define two comultiplication structures on $\TSym$.
$$\Delta^A(\tilde{F}_t)=\sum_{t\mapsto(t_1,t_2)}\tilde{F}_{t_1}\otimes \tilde{F}_{t_2},$$
$$\Delta^B(F_t)=\sum_{\substack{\text{allowable} \\ t\mapsto(t_1,t_2)}}F_{t_1}\otimes F_{t_2}.$$

Then, we can define two multiplication structures on $\TSym$. Let $s,t\in \PT$,
$$m^A(\tilde{F}_s\otimes \tilde{F}_t)=\sum_{t\mapsto(t_1,\dots,t_{\deg(s)+1})}\tilde{F}_{s\shuffle (t_1,\dots,t_{\deg(s)+1})},$$
$$m^B(F_s\otimes F_t)=\sum_{\substack{\text{allowable}\\ t\mapsto(t_1,\dots,t_{\deg(s)+1})}}F_{s\shuffle (t_1,\dots,t_{\deg(s)+1})}.$$

It is not hard to see that both $(m^A,\Delta^A)$ and $(m^B,\Delta^B)$ preserve the degree, and $(m^B,\Delta^B)$ preserves the internal degree.

\begin{Example}
	Let $t=\begin{tikzpicture}[scale=0.2,baseline=0pt]
	\draw[blue, thick] (0,-1) -- (0,0);
	\draw[blue, thick] (0,0) -- (4,4);
	\draw[blue, thick] (0,0) -- (-4,4);
	\draw[blue, thick] (0,0) -- (1,3);
	\draw[blue, thick] (1,3) -- (2,4);
	\draw[blue, thick] (1,3) -- (0,4);
	\draw[blue, thick] (-3,3) -- (-2,4);
	\filldraw[black] (0,0) circle (5pt)  {};
	\filldraw[black] (1,3) circle (5pt)  {};
	\filldraw[black] (-3,3) circle (5pt)  {};
	\end{tikzpicture}$.
	
	$\Delta^A_+(\tilde{F}_t)=
	\tilde{F}_{\begin{tikzpicture}[scale=0.25,baseline=0pt]
	\draw[blue, thick] (0,-1) -- (0,0);
	\draw[blue, thick] (0,0) -- (1,1);
	\draw[blue, thick] (0,0) -- (-1,1);
	\filldraw[black] (0,0) circle (5pt)  {};
	\end{tikzpicture}}\otimes
	\tilde{F}_{\begin{tikzpicture}[scale=0.25,baseline=0pt]
	\draw[blue, thick] (0,-1) -- (0,0);
	\draw[blue, thick] (0,0) -- (-3,3);
	\draw[blue, thick] (0,0) -- (3,3);
	\draw[blue, thick] (0,0) -- (0,2);
	\draw[blue, thick] (0,2) -- (1,3);
	\draw[blue, thick] (0,2) -- (-1,3);
	\filldraw[black] (0,0) circle (5pt)  {};
	\filldraw[black] (0,2) circle (5pt)  {};
	\end{tikzpicture}}+
	\tilde{F}_{\begin{tikzpicture}[scale=0.25,baseline=0pt]
	\draw[blue, thick] (0,-1) -- (0,0);
	\draw[blue, thick] (0,0) -- (-2,2);
	\draw[blue, thick] (0,0) -- (2,2);
	\draw[blue, thick] (-1,1) -- (0,2);
	\filldraw[black] (0,0) circle (5pt)  {};
	\filldraw[black] (-1,1) circle (5pt)  {};
	\end{tikzpicture}}\otimes
	\tilde{F}_{\begin{tikzpicture}[scale=0.25,baseline=0pt]
	\draw[blue, thick] (0,-1) -- (0,0);
	\draw[blue, thick] (0,0) -- (-2,2);
	\draw[blue, thick] (0,0) -- (2,2);
	\draw[blue, thick] (-1,1) -- (0,2);
	\filldraw[black] (0,0) circle (5pt)  {};
	\filldraw[black] (-1,1) circle (5pt)  {};
	\end{tikzpicture}}+
	\tilde{F}_{\begin{tikzpicture}[scale=0.25,baseline=0pt]
	\draw[blue, thick] (0,-1) -- (0,0);
	\draw[blue, thick] (0,0) -- (-3,3);
	\draw[blue, thick] (0,0) -- (3,3);
	\draw[blue, thick] (2,2) -- (1,3);
	\draw[blue, thick] (-2,2) -- (-1,3);
	\filldraw[black] (0,0) circle (5pt)  {};
	\filldraw[black] (2,2) circle (5pt)  {};
	\filldraw[black] (-2,2) circle (5pt)  {};
	\end{tikzpicture}}\otimes
	\tilde{F}_{\begin{tikzpicture}[scale=0.25,baseline=0pt]
	\draw[blue, thick] (0,-1) -- (0,0);
	\draw[blue, thick] (0,0) -- (1,1);
	\draw[blue, thick] (0,0) -- (-1,1);
	\filldraw[black] (0,0) circle (5pt)  {};
	\end{tikzpicture}}$,
	
	$\Delta^B_+(F_t)=
	F_{\begin{tikzpicture}[scale=0.25,baseline=0pt]
	\draw[blue, thick] (0,-1) -- (0,0);
	\draw[blue, thick] (0,0) -- (1,1);
	\draw[blue, thick] (0,0) -- (-1,1);
	\filldraw[black] (0,0) circle (5pt)  {};
	\end{tikzpicture}}\otimes
	F_{\begin{tikzpicture}[scale=0.25,baseline=0pt]
	\draw[blue, thick] (0,-1) -- (0,0);
	\draw[blue, thick] (0,0) -- (-3,3);
	\draw[blue, thick] (0,0) -- (3,3);
	\draw[blue, thick] (0,0) -- (0,2);
	\draw[blue, thick] (0,2) -- (1,3);
	\draw[blue, thick] (0,2) -- (-1,3);
	\filldraw[black] (0,0) circle (5pt)  {};
	\filldraw[black] (0,2) circle (5pt)  {};
	\end{tikzpicture}}$.

\end{Example}

\begin{Example}
	Let $s=$\begin{tikzpicture}[scale=0.25,baseline=0pt]
	\draw[blue, thick] (0,-1) -- (0,0);
	\draw[blue, thick] (0,0) -- (-2,2);
	\draw[blue, thick] (0,0) -- (2,2);
	\draw[blue, thick] (1,1) -- (0,2);
	\filldraw[black] (0,0) circle (5pt)  {};
	\filldraw[black] (1,1) circle (5pt)  {};
	\end{tikzpicture}, $t=$\begin{tikzpicture}[scale=0.25,baseline=0pt]
	\draw[red, thick] (0,-1) -- (0,0);
	\draw[red, thick] (0,0) -- (-2,2);
	\draw[red, thick] (0,0) -- (2,2);
	\draw[red, thick] (0,0) -- (0,2);
	\filldraw[black] (0,0) circle (5pt)  {};
	\end{tikzpicture}.
	
	$m^A(\tilde{F}_s\otimes \tilde{F}_t)=
	\tilde{F}_{\begin{tikzpicture}[scale=0.2,baseline=0pt]
		\draw[blue, thick] (0,-1) -- (0,0);
		\draw[blue, thick] (0,0) -- (-4,4);
		\draw[blue, thick] (0,0) -- (2,2);
		\draw[blue, thick] (1,1) -- (-2,4);
		\draw[red, thick] (2,2) -- (0,4);
		\draw[red, thick] (2,2) -- (2,4);
		\draw[red, thick] (2,2) -- (4,4);
		\filldraw[black] (0,0) circle (5pt)  {};
		\filldraw[black] (1,1) circle (5pt)  {};
		\filldraw[black] (2,2) circle (5pt)  {};
	\end{tikzpicture}}+
	\tilde{F}_{\begin{tikzpicture}[scale=0.2,baseline=0pt]
		\draw[blue, thick] (0,-1) -- (0,0);
		\draw[blue, thick] (0,0) -- (-4,4);
		\draw[blue, thick] (0,0) -- (4,4);
		\draw[blue, thick] (1,1) -- (0,2);
		\draw[red, thick] (0,2) -- (-2,4);
		\draw[red, thick] (0,2) -- (0,4);
		\draw[red, thick] (0,2) -- (2,4);
		\filldraw[black] (0,0) circle (5pt)  {};
		\filldraw[black] (1,1) circle (5pt)  {};
		\filldraw[black] (0,2) circle (5pt)  {};
	\end{tikzpicture}}+
	\tilde{F}_{\begin{tikzpicture}[scale=0.2,baseline=0pt]
		\draw[blue, thick] (0,-1) -- (0,0);
		\draw[blue, thick] (0,0) -- (-4,4);
		\draw[blue, thick] (0,0) -- (4,4);
		\draw[blue, thick] (3,3) -- (2,4);
		\draw[red, thick] (-2,2) -- (0,4);
		\draw[red, thick] (-2,2) -- (-2,4);
		\draw[red, thick] (-2,2) -- (-4,4);
		\filldraw[black] (0,0) circle (5pt)  {};
		\filldraw[black] (-2,2) circle (5pt)  {};
		\filldraw[black] (3,3) circle (5pt)  {};
	\end{tikzpicture}}+
	\tilde{F}_{\begin{tikzpicture}[scale=0.2,baseline=0pt]
		\draw[blue, thick] (0,-1) -- (0,0);
		\draw[blue, thick] (0,0) -- (-4,4);
		\draw[blue, thick] (0,0) -- (3,3);
		\draw[blue, thick] (1,1) -- (-1,3);
		\draw[red, thick] (-1,3) -- (-2,4);
		\draw[red, thick] (-1,3) -- (0,4);
		\draw[red, thick] (3,3) -- (2,4);
		\draw[red, thick] (3,3) -- (4,4);
		\filldraw[black] (0,0) circle (5pt)  {};
		\filldraw[black] (1,1) circle (5pt)  {};
		\filldraw[black] (-1,3) circle (5pt)  {};
		\filldraw[black] (3,3) circle (5pt)  {};
	\end{tikzpicture}}+
	\tilde{F}_{\begin{tikzpicture}[scale=0.2,baseline=0pt]
		\draw[blue, thick] (0,-1) -- (0,0);
		\draw[blue, thick] (0,0) -- (-3,3);
		\draw[blue, thick] (0,0) -- (3,3);
		\draw[blue, thick] (2,2) -- (0,4);
		\draw[red, thick] (-3,3) -- (-4,4);
		\draw[red, thick] (-3,3) -- (-2,4);
		\draw[red, thick] (3,3) -- (2,4);
		\draw[red, thick] (3,3) -- (4,4);
		\filldraw[black] (0,0) circle (5pt)  {};
		\filldraw[black] (2,2) circle (5pt)  {};
		\filldraw[black] (3,3) circle (5pt)  {};
		\filldraw[black] (-3,3) circle (5pt)  {};
	\end{tikzpicture}}+
	\tilde{F}_{\begin{tikzpicture}[scale=0.2,baseline=0pt]
		\draw[blue, thick] (0,-1) -- (0,0);
		\draw[blue, thick] (0,0) -- (-3,3);
		\draw[blue, thick] (0,0) -- (4,4);
		\draw[blue, thick] (2,2) -- (1,3);
		\draw[red, thick] (-3,3) -- (-4,4);
		\draw[red, thick] (-3,3) -- (-2,4);
		\draw[red, thick] (1,3) -- (0,4);
		\draw[red, thick] (1,3) -- (2,4);
		\filldraw[black] (0,0) circle (5pt)  {};
		\filldraw[black] (2,2) circle (5pt)  {};
		\filldraw[black] (1,3) circle (5pt)  {};
		\filldraw[black] (-3,3) circle (5pt)  {};
	\end{tikzpicture}}$,
	
	$m^B(F_s\otimes F_t)=
	F_{\begin{tikzpicture}[scale=0.2,baseline=0pt]
		\draw[blue, thick] (0,-1) -- (0,0);
		\draw[blue, thick] (0,0) -- (-4,4);
		\draw[blue, thick] (0,0) -- (2,2);
		\draw[blue, thick] (1,1) -- (-2,4);
		\draw[red, thick] (2,2) -- (0,4);
		\draw[red, thick] (2,2) -- (2,4);
		\draw[red, thick] (2,2) -- (4,4);
		\filldraw[black] (0,0) circle (5pt)  {};
		\filldraw[black] (1,1) circle (5pt)  {};
		\filldraw[black] (2,2) circle (5pt)  {};
	\end{tikzpicture}}+
	F_{\begin{tikzpicture}[scale=0.2,baseline=0pt]
		\draw[blue, thick] (0,-1) -- (0,0);
		\draw[blue, thick] (0,0) -- (-4,4);
		\draw[blue, thick] (0,0) -- (4,4);
		\draw[blue, thick] (1,1) -- (0,2);
		\draw[red, thick] (0,2) -- (-2,4);
		\draw[red, thick] (0,2) -- (0,4);
		\draw[red, thick] (0,2) -- (2,4);
		\filldraw[black] (0,0) circle (5pt)  {};
		\filldraw[black] (1,1) circle (5pt)  {};
		\filldraw[black] (0,2) circle (5pt)  {};
	\end{tikzpicture}}+
	F_{\begin{tikzpicture}[scale=0.2,baseline=0pt]
		\draw[blue, thick] (0,-1) -- (0,0);
		\draw[blue, thick] (0,0) -- (-4,4);
		\draw[blue, thick] (0,0) -- (4,4);
		\draw[blue, thick] (3,3) -- (2,4);
		\draw[red, thick] (-2,2) -- (0,4);
		\draw[red, thick] (-2,2) -- (-2,4);
		\draw[red, thick] (-2,2) -- (-4,4);
		\filldraw[black] (0,0) circle (5pt)  {};
		\filldraw[black] (-2,2) circle (5pt)  {};
		\filldraw[black] (3,3) circle (5pt)  {};
	\end{tikzpicture}}$.
\end{Example}

\begin{Proposition}\label{prop:tsym}
	$(\TSym,m^A,\Delta^A)$ and $(\TSym,m^B,\Delta^B)$ are both graded Hopf algebras with respect to the degree. $(\TSym,m^B,\Delta^B)$ is a graded Hopf algebra with respect to the internal degree as well  (so it is bigraded).
\end{Proposition}

\begin{proof}
	It is clear that $\Delta^A$ and $\Delta^B$ are both coassociative. To see associativity, let $r,s,t$ be trees of degree $a-1,b-1,c-1$ respectively, then for $(\TSym,m^A,\Delta^A)$,
	\begin{align*}
	\tilde{F}_r\cdot (\tilde{F}_s\cdot \tilde{F}_t)&=\sum_{t\mapsto(t_1,\dots,t_b)}\sum_{\substack{u=s\shuffle(t_1,\dots,t_b) \\ u\mapsto(u_1,\dots,u_a)}}\tilde{F}_{r\shuffle(u_1,\dots,u_a)}\\
	&=\sum_{\substack{s \mapsto (s_1,\dots,s_a) \\ t \mapsto(t_1,\dots,t_{a+b-1})}}\tilde{F}_{\left(r\shuffle(s_1,\dots,s_a)\right)\shuffle(t_1,\dots,t_{a+b-1})}\\
	&=(\tilde{F}_r\cdot \tilde{F}_s)\cdot \tilde{F}_t.
	\end{align*}
	The second equality can be visualized using the following example.
		
		Let $r=\begin{tikzpicture}[scale=0.25,baseline=0pt]
			\draw[blue, thick] (0,-1) -- (0,0);
			\draw[blue, thick] (0,0) -- (-2,2);
			\draw[blue, thick] (0,0) -- (2,2);
			\draw[blue, thick] (0,0) -- (0,2);
			\filldraw[black] (0,0) circle (5pt)  {};
		\end{tikzpicture}$, $s=\begin{tikzpicture}[scale=0.25,baseline=0pt]
			\draw[red, thick] (0,-1) -- (0,0);
			\draw[red, thick] (0,0) -- (-4,4);
			\draw[red, thick] (0,0) -- (4,4);
			\draw[red, thick] (-3,3) -- (-2,4);
			\draw[red, thick] (2,2) -- (0,4);
			\draw[red, thick] (1,3) -- (2,4);
			\filldraw[black] (-3,3) circle (5pt)  {};
			\filldraw[black] (1,3) circle (5pt)  {};
			\filldraw[black] (2,2) circle (5pt)  {};
		\end{tikzpicture}$ and $t=\begin{tikzpicture}[scale=0.25,baseline=0pt]
			\draw[ForestGreen, thick] (0,-1) -- (0,0);
			\draw[ForestGreen, thick] (0,0) -- (-3,3);
			\draw[ForestGreen, thick] (0,0) -- (3,3);
			\draw[ForestGreen, thick] (0,0) -- (-1,3);
			\draw[ForestGreen, thick] (2,2) -- (1,3);
			\filldraw[black] (0,0) circle (5pt)  {};
			\filldraw[black] (2,2) circle (5pt)  {};
		\end{tikzpicture}$.
		
		The choice $t\mapsto\left(~
		\begin{tikzpicture}[scale=0.25,baseline=0pt]
			\draw[ForestGreen, thick] (0,-1) -- (0,1);
		\end{tikzpicture}~,~
		\begin{tikzpicture}[scale=0.25,baseline=0pt]
			\draw[ForestGreen, thick] (0,-1) -- (0,0);
			\draw[ForestGreen, thick] (0,0) -- (-2,2);
			\draw[ForestGreen, thick] (0,0) -- (2,2);
			\draw[ForestGreen, thick] (0,0) -- (0,2);
			\filldraw[black] (0,0) circle (5pt)  {};
		\end{tikzpicture}~,~
		\begin{tikzpicture}[scale=0.25,baseline=0pt]
			\draw[ForestGreen, thick] (0,-1) -- (0,1);
		\end{tikzpicture}~,~
		\begin{tikzpicture}[scale=0.25,baseline=0pt]
			\draw[ForestGreen, thick] (0,-1) -- (0,0);
			\draw[ForestGreen, thick] (0,0) -- (-1,1);
			\draw[ForestGreen, thick] (0,0) -- (1,1);
			\filldraw[black] (0,0) circle (5pt)  {};
		\end{tikzpicture}~,~
		\begin{tikzpicture}[scale=0.25,baseline=0pt]
			\draw[ForestGreen, thick] (0,-1) -- (0,1);
		\end{tikzpicture}~\right)$ makes $u=
		\begin{tikzpicture}[scale=0.25,baseline=0pt]
			\draw[red, thick] (0,-1) -- (0,0);
			\draw[red, thick] (0,0) -- (-5,5);
			\draw[red, thick] (0,0) -- (6,6);
			\draw[red, thick] (-4,4) -- (-3,5);
			\draw[ForestGreen, thick] (-5,5) -- (-7,7);
			\draw[ForestGreen, thick] (-3,5) -- (-5,7);
			\draw[ForestGreen, thick] (-3,5) -- (-3,7);
			\draw[ForestGreen, thick] (-3,5) -- (-1,7);
			\draw[ForestGreen, thick] (6,6) -- (7,7);
			\draw[ForestGreen, thick] (2,6) -- (1,7);
			\draw[red, thick] (4,4) -- (2,6);
			\draw[red, thick] (3,5) -- (4,6);
			\draw[ForestGreen, thick] (4,6) -- (3,7);
			\draw[ForestGreen, thick] (4,6) -- (5,7);
			\filldraw[black] (0,0) circle (5pt)  {};
			\filldraw[black] (-4,4) circle (5pt)  {};
			\filldraw[black] (-3,5) circle (5pt)  {};
			\filldraw[black] (4,4) circle (5pt)  {};
			\filldraw[black] (3,5) circle (5pt)  {};
			\filldraw[black] (4,6) circle (5pt)  {};
		\end{tikzpicture}$
		
		and the choice $u\mapsto\left(~
		\begin{tikzpicture}[scale=0.25,baseline=0pt]
			\draw[red, thick] (0,-1) -- (0,0);
			\draw[red, thick] (0,0) -- (-1,1);
			\draw[red, thick] (0,0) -- (1,1);
			\draw[ForestGreen, thick] (-1,1) -- (-2,2);
			\draw[ForestGreen, thick] (1,1) -- (2,2);
			\draw[ForestGreen, thick] (1,1) -- (0,2);
			\filldraw[black] (0,0) circle (5pt)  {};
			\filldraw[black] (1,1) circle (5pt)  {};
		\end{tikzpicture}~,~
		\begin{tikzpicture}[scale=0.25,baseline=0pt]
			\draw[red, thick] (0,-1) -- (0,0);
			\draw[red, thick] (0,0) -- (3,3);
			\draw[red, thick] (2,2) -- (1,3);
			\draw[red, thick] (0,0) -- (-3,3);
			\draw[ForestGreen, thick] (-3,3) -- (-4,4);
			\draw[ForestGreen, thick] (-3,3) -- (-2,4);
			\draw[ForestGreen, thick] (1,3) -- (0,4);
			\draw[ForestGreen, thick] (3,3) -- (4,4);
			\filldraw[black] (0,0) circle (5pt)  {};
			\filldraw[black] (2,2) circle (5pt)  {};
			\filldraw[black] (-3,3) circle (5pt)  {};
		\end{tikzpicture}~,~
		\begin{tikzpicture}[scale=0.25,baseline=0pt]
			\draw[red, thick] (0,-1) -- (0,0);
			\draw[red, thick] (0,0) -- (1,1);
			\draw[red, thick] (0,0) -- (-1,1);
			\draw[ForestGreen, thick] (1,1) -- (2,2);
			\draw[ForestGreen, thick] (-1,1) -- (-2,2);
			\draw[ForestGreen, thick] (-1,1) -- (0,2);
			\filldraw[black] (0,0) circle (5pt)  {};
			\filldraw[black] (-1,1) circle (5pt)  {};
		\end{tikzpicture}~\right)$ yields
		$\begin{tikzpicture}[scale=0.25,baseline=0pt]
			\draw[blue, thick] (0,-1) -- (0,0);
			\draw[blue, thick] (0,0) -- (-8,8);
			\draw[red, thick] (-8,8) -- (-9,9);
			\draw[red, thick] (-8,8) -- (-7,9);
			\draw[ForestGreen, thick] (-7,9) -- (-8,10);
			\draw[ForestGreen, thick] (-7,9) -- (-6,10);
			\draw[ForestGreen, thick] (-9,9) -- (-10,10);
			\draw[blue, thick] (0,0) -- (0,6);
			\draw[red, thick] (0,6) -- (-3,9);
			\draw[red, thick] (0,6) -- (3,9);
			\draw[red, thick] (2,8) -- (1,9);
			\draw[ForestGreen, thick] (-3,9) -- (-2,10);
			\draw[ForestGreen, thick] (-3,9) -- (-4,10);
			\draw[ForestGreen, thick] (1,9) -- (0,10);
			\draw[ForestGreen, thick] (3,9) -- (4,10);
			\draw[blue, thick] (0,0) -- (8,8);
			\draw[red, thick] (8,8) -- (9,9);
			\draw[red, thick] (8,8) -- (7,9);
			\draw[ForestGreen, thick] (7,9) -- (8,10);
			\draw[ForestGreen, thick] (7,9) -- (6,10);
			\draw[ForestGreen, thick] (9,9) -- (10,10);
			\filldraw[black] (0,0) circle (5pt)  {};
			\filldraw[black] (-8,8) circle (5pt)  {};
			\filldraw[black] (-7,9) circle (5pt)  {};
			\filldraw[black] (0,6) circle (5pt)  {};
			\filldraw[black] (-3,9) circle (5pt)  {};
			\filldraw[black] (2,8) circle (5pt)  {};
			\filldraw[black] (8,8) circle (5pt)  {};
			\filldraw[black] (7,9) circle (5pt)  {};
		\end{tikzpicture}$.
		
		On the other hand, the choice $s\mapsto\left(~
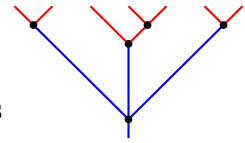
		\begin{tikzpicture}[scale=0.25,baseline=0pt]
			\draw[red, thick] (0,-1) -- (0,0);
			\draw[red, thick] (0,0) -- (-1,1);
			\draw[red, thick] (0,0) -- (1,1);
			\filldraw[black] (0,0) circle (5pt)  {};
		\end{tikzpicture}~,~
		\begin{tikzpicture}[scale=0.25,baseline=0pt]
			\draw[red, thick] (0,-1) -- (0,0);
			\draw[red, thick] (0,0) -- (-2,2);
			\draw[red, thick] (0,0) -- (2,2);
			\draw[red, thick] (1,1) -- (0,2);
			\filldraw[black] (0,0) circle (5pt)  {};
			\filldraw[black] (1,1) circle (5pt)  {};
		\end{tikzpicture}~,~
		\begin{tikzpicture}[scale=0.25,baseline=0pt]
			\draw[red, thick] (0,-1) -- (0,0);
			\draw[red, thick] (0,0) -- (-1,1);
			\draw[red, thick] (0,0) -- (1,1);
			\filldraw[black] (0,0) circle (5pt)  {};
		\end{tikzpicture}~\right)$ makes $
		\begin{tikzpicture}[scale=0.25,baseline=0pt]
			\draw[blue, thick] (0,-1) -- (0,0);
			\draw[blue, thick] (0,0) -- (-5,5);
			\draw[blue, thick] (0,0) -- (5,5);
			\draw[red, thick] (-5,5) -- (-6,6);
			\draw[red, thick] (-5,5) -- (-4,6);
			\draw[red, thick] (5,5) -- (4,6);
			\draw[red, thick] (5,5) -- (6,6);
			\draw[blue, thick] (0,0) -- (0,4);
			\draw[red, thick] (0,4) -- (-2,6);
			\draw[red, thick] (0,4) -- (2,6);
			\draw[red, thick] (1,5) -- (0,6);
			\filldraw[black] (0,0) circle (5pt)  {};
			\filldraw[black] (-5,5) circle (5pt)  {};
			\filldraw[black] (5,5) circle (5pt)  {};
			\filldraw[black] (0,4) circle (5pt)  {};
			\filldraw[black] (1,5) circle (5pt)  {};
		\end{tikzpicture}$
		
		and the choice $t\mapsto\left(~
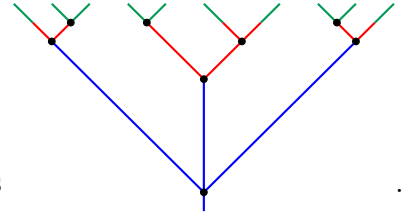
		\begin{tikzpicture}[scale=0.25,baseline=0pt]
			\draw[ForestGreen, thick] (0,-1) -- (0,1);
		\end{tikzpicture}~,~
		\begin{tikzpicture}[scale=0.25,baseline=0pt]
			\draw[ForestGreen, thick] (0,-1) -- (0,0);
			\draw[ForestGreen, thick] (0,0) -- (-1,1);
			\draw[ForestGreen, thick] (0,0) -- (1,1);
			\filldraw[black] (0,0) circle (5pt)  {};
		\end{tikzpicture}~,~
		\begin{tikzpicture}[scale=0.25,baseline=0pt]
			\draw[ForestGreen, thick] (0,-1) -- (0,0);
			\draw[ForestGreen, thick] (0,0) -- (-1,1);
			\draw[ForestGreen, thick] (0,0) -- (1,1);
			\filldraw[black] (0,0) circle (5pt)  {};
		\end{tikzpicture}~,~
		\begin{tikzpicture}[scale=0.25,baseline=0pt]
			\draw[ForestGreen, thick] (0,-1) -- (0,1);
		\end{tikzpicture}~,~
		\begin{tikzpicture}[scale=0.25,baseline=0pt]
			\draw[ForestGreen, thick] (0,-1) -- (0,1);
		\end{tikzpicture}~,~
		\begin{tikzpicture}[scale=0.25,baseline=0pt]
			\draw[ForestGreen, thick] (0,-1) -- (0,0);
			\draw[ForestGreen, thick] (0,0) -- (-1,1);
			\draw[ForestGreen, thick] (0,0) -- (1,1);
			\filldraw[black] (0,0) circle (5pt)  {};
		\end{tikzpicture}~,~
		\begin{tikzpicture}[scale=0.25,baseline=0pt]
			\draw[ForestGreen, thick] (0,-1) -- (0,1);
		\end{tikzpicture}~\right)$ yields
		$\begin{tikzpicture}[scale=0.25,baseline=0pt]
			\draw[blue, thick] (0,-1) -- (0,0);
			\draw[blue, thick] (0,0) -- (-8,8);
			\draw[red, thick] (-8,8) -- (-9,9);
			\draw[red, thick] (-8,8) -- (-7,9);
			\draw[ForestGreen, thick] (-7,9) -- (-8,10);
			\draw[ForestGreen, thick] (-7,9) -- (-6,10);
			\draw[ForestGreen, thick] (-9,9) -- (-10,10);
			\draw[blue, thick] (0,0) -- (0,6);
			\draw[red, thick] (0,6) -- (-3,9);
			\draw[red, thick] (0,6) -- (3,9);
			\draw[red, thick] (2,8) -- (1,9);
			\draw[ForestGreen, thick] (-3,9) -- (-2,10);
			\draw[ForestGreen, thick] (-3,9) -- (-4,10);
			\draw[ForestGreen, thick] (1,9) -- (0,10);
			\draw[ForestGreen, thick] (3,9) -- (4,10);
			\draw[blue, thick] (0,0) -- (8,8);
			\draw[red, thick] (8,8) -- (9,9);
			\draw[red, thick] (8,8) -- (7,9);
			\draw[ForestGreen, thick] (7,9) -- (8,10);
			\draw[ForestGreen, thick] (7,9) -- (6,10);
			\draw[ForestGreen, thick] (9,9) -- (10,10);
			\filldraw[black] (0,0) circle (5pt)  {};
			\filldraw[black] (-8,8) circle (5pt)  {};
			\filldraw[black] (-7,9) circle (5pt)  {};
			\filldraw[black] (0,6) circle (5pt)  {};
			\filldraw[black] (-3,9) circle (5pt)  {};
			\filldraw[black] (2,8) circle (5pt)  {};
			\filldraw[black] (8,8) circle (5pt)  {};
			\filldraw[black] (7,9) circle (5pt)  {};
		\end{tikzpicture}$.

	It is left to see the compatibility. Let $s,t$ be two trees, $\deg(s)=k$. We use the notation $\Delta_i(s)=\Delta_{i,k-i}(s)={}^{(i)}s\otimes s^{(k-i)}$, then for $(\TSym,m^A,\Delta^A)$,
	\begin{align*}
		\Delta_i(\tilde{F}_s\cdot \tilde{F}_t)&=\sum_{t\mapsto(t_1,\dots,t_{k+1})}\Delta_i\tilde{F}_{s\shuffle (t_1,\dots,t_{k+1})}\\
		&=\sum_{t\mapsto(t_1,\dots,t_{k+2})}\sum_{\deg(t_1)+\cdots+\deg(t_{n+1})=i-n}\tilde{F}_{^{(n)}s\shuffle (t_1,\dots,t_{n+1})}\otimes \tilde{F}_{s^{(k-n)}\shuffle(t_{n+2},\dots,t_{k+2})}\\
		&=\sum_{n=0}^i\Delta_n(\tilde{F}_s)\cdot\Delta_{i-n}(\tilde{F}_t).
	\end{align*}
	
	The second equality can be visualized using the following example.

		Let $s=\begin{tikzpicture}[scale=0.25,baseline=0pt]
			\draw[red, thick] (0,-1) -- (0,0);
			\draw[red, thick] (0,0) -- (-4,4);
			\draw[red, thick] (0,0) -- (4,4);
			\draw[red, thick] (-3,3) -- (-2,4);
			\draw[red, thick] (2,2) -- (0,4);
			\draw[red, thick] (1,3) -- (2,4);
			\filldraw[black] (-3,3) circle (5pt)  {};
			\filldraw[black] (1,3) circle (5pt)  {};
			\filldraw[black] (2,2) circle (5pt)  {};
		\end{tikzpicture}$, $t=\begin{tikzpicture}[scale=0.25,baseline=0pt]
			\draw[ForestGreen, thick] (0,-1) -- (0,0);
			\draw[ForestGreen, thick] (0,0) -- (-3,3);
			\draw[ForestGreen, thick] (0,0) -- (3,3);
			\draw[ForestGreen, thick] (0,0) -- (-1,3);
			\draw[ForestGreen, thick] (2,2) -- (1,3);
			\filldraw[black] (0,0) circle (5pt)  {};
			\filldraw[black] (2,2) circle (5pt)  {};
		\end{tikzpicture}$ and $i=2$.
		
		The choice $t\mapsto\left(~
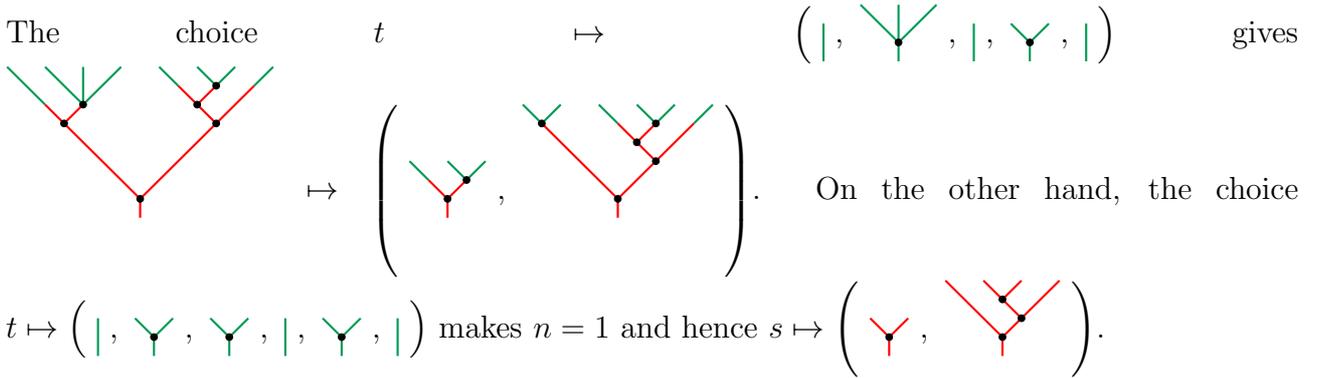
		\begin{tikzpicture}[scale=0.25,baseline=0pt]
			\draw[ForestGreen, thick] (0,-1) -- (0,1);
		\end{tikzpicture}~,~
		\begin{tikzpicture}[scale=0.25,baseline=0pt]
			\draw[ForestGreen, thick] (0,-1) -- (0,0);
			\draw[ForestGreen, thick] (0,0) -- (-2,2);
			\draw[ForestGreen, thick] (0,0) -- (2,2);
			\draw[ForestGreen, thick] (0,0) -- (0,2);
			\filldraw[black] (0,0) circle (5pt)  {};
		\end{tikzpicture}~,~
		\begin{tikzpicture}[scale=0.25,baseline=0pt]
			\draw[ForestGreen, thick] (0,-1) -- (0,1);
		\end{tikzpicture}~,~
		\begin{tikzpicture}[scale=0.25,baseline=0pt]
			\draw[ForestGreen, thick] (0,-1) -- (0,0);
			\draw[ForestGreen, thick] (0,0) -- (-1,1);
			\draw[ForestGreen, thick] (0,0) -- (1,1);
			\filldraw[black] (0,0) circle (5pt)  {};
		\end{tikzpicture}~,~
		\begin{tikzpicture}[scale=0.25,baseline=0pt]
			\draw[ForestGreen, thick] (0,-1) -- (0,1);
		\end{tikzpicture}~\right)$ gives $
		\begin{tikzpicture}[scale=0.25,baseline=0pt]
			\draw[red, thick] (0,-1) -- (0,0);
			\draw[red, thick] (0,0) -- (-5,5);
			\draw[red, thick] (0,0) -- (6,6);
			\draw[red, thick] (-4,4) -- (-3,5);
			\draw[ForestGreen, thick] (-5,5) -- (-7,7);
			\draw[ForestGreen, thick] (-3,5) -- (-5,7);
			\draw[ForestGreen, thick] (-3,5) -- (-3,7);
			\draw[ForestGreen, thick] (-3,5) -- (-1,7);
			\draw[ForestGreen, thick] (6,6) -- (7,7);
			\draw[red, thick] (4,4) -- (2,6);
			\draw[ForestGreen, thick] (2,6) -- (1,7);
			\draw[red, thick] (3,5) -- (4,6);
			\draw[ForestGreen, thick] (4,6) -- (3,7);
			\draw[ForestGreen, thick] (4,6) -- (5,7);
			\filldraw[black] (0,0) circle (5pt)  {};
			\filldraw[black] (-4,4) circle (5pt)  {};
			\filldraw[black] (-3,5) circle (5pt)  {};
			\filldraw[black] (4,4) circle (5pt)  {};
			\filldraw[black] (3,5) circle (5pt)  {};
			\filldraw[black] (4,6) circle (5pt)  {};
		\end{tikzpicture}\mapsto\left(~
		\begin{tikzpicture}[scale=0.25,baseline=0pt]
			\draw[red, thick] (0,-1) -- (0,0);
			\draw[red, thick] (0,0) -- (-1,1);
			\draw[red, thick] (0,0) -- (1,1);
			\draw[ForestGreen, thick] (-1,1) -- (-2,2);
			\draw[ForestGreen, thick] (1,1) -- (2,2);
			\draw[ForestGreen, thick] (1,1) -- (0,2);
			\filldraw[black] (0,0) circle (5pt)  {};
			\filldraw[black] (1,1) circle (5pt)  {};
		\end{tikzpicture}~,~
		\begin{tikzpicture}[scale=0.25,baseline=0pt]
			\draw[red, thick] (0,-1) -- (0,0);
			\draw[red, thick] (0,0) -- (-4,4);
			\draw[ForestGreen, thick] (-4,4) -- (-5,5);
			\draw[ForestGreen, thick] (-4,4) -- (-3,5);
			\draw[red, thick] (0,0) -- (4,4);
			\draw[red, thick] (2,2) -- (0,4);
			\draw[ForestGreen, thick] (0,4) -- (-1,5);
			\draw[red, thick] (1,3) -- (2,4);
			\draw[ForestGreen, thick] (4,4) -- (5,5);
			\draw[ForestGreen, thick] (2,4) -- (1,5);
			\draw[ForestGreen, thick] (2,4) -- (3,5);
			\filldraw[black] (0,0) circle (5pt)  {};
			\filldraw[black] (-4,4) circle (5pt)  {};
			\filldraw[black] (2,2) circle (5pt)  {};
			\filldraw[black] (1,3) circle (5pt)  {};
			\filldraw[black] (2,4) circle (5pt)  {};
		\end{tikzpicture}~\right)$. On the other hand, the choice $t\mapsto\left(~
		\begin{tikzpicture}[scale=0.25,baseline=0pt]
			\draw[ForestGreen, thick] (0,-1) -- (0,1);
		\end{tikzpicture}~,~
		\begin{tikzpicture}[scale=0.25,baseline=0pt]
			\draw[ForestGreen, thick] (0,-1) -- (0,0);
			\draw[ForestGreen, thick] (0,0) -- (-1,1);
			\draw[ForestGreen, thick] (0,0) -- (1,1);
			\filldraw[black] (0,0) circle (5pt)  {};
		\end{tikzpicture}~,~
		\begin{tikzpicture}[scale=0.25,baseline=0pt]
			\draw[ForestGreen, thick] (0,-1) -- (0,0);
			\draw[ForestGreen, thick] (0,0) -- (-1,1);
			\draw[ForestGreen, thick] (0,0) -- (1,1);
			\filldraw[black] (0,0) circle (5pt)  {};
		\end{tikzpicture}~,~
		\begin{tikzpicture}[scale=0.25,baseline=0pt]
			\draw[ForestGreen, thick] (0,-1) -- (0,1);
		\end{tikzpicture}~,~
		\begin{tikzpicture}[scale=0.25,baseline=0pt]
			\draw[ForestGreen, thick] (0,-1) -- (0,0);
			\draw[ForestGreen, thick] (0,0) -- (-1,1);
			\draw[ForestGreen, thick] (0,0) -- (1,1);
			\filldraw[black] (0,0) circle (5pt)  {};
		\end{tikzpicture}~,~
		\begin{tikzpicture}[scale=0.25,baseline=0pt]
			\draw[ForestGreen, thick] (0,-1) -- (0,1);
		\end{tikzpicture}~\right)$ makes $n=1$ and hence $s\mapsto\left(~
		\begin{tikzpicture}[scale=0.25,baseline=0pt]
			\draw[red, thick] (0,-1) -- (0,0);
			\draw[red, thick] (0,0) -- (1,1);
			\draw[red, thick] (0,0) -- (-1,1);
			\filldraw[black] (0,0) circle (5pt)  {};
		\end{tikzpicture}~,~
		\begin{tikzpicture}[scale=0.25,baseline=0pt]
			\draw[red, thick] (0,-1) -- (0,0);
			\draw[red, thick] (0,0) -- (-3,3);
			\draw[red, thick] (1,1) -- (-1,3);
			\draw[red, thick] (0,2) -- (1,3);
			\draw[red, thick] (0,0) -- (3,3);
			\filldraw[black] (0,0) circle (5pt)  {};
			\filldraw[black] (1,1) circle (5pt)  {};
			\filldraw[black] (0,2) circle (5pt)  {};
		\end{tikzpicture}~\right)$.
	
	The argument above works for $(\TSym,m^B,\Delta^B)$ as well, by taking the sum of all allowable splittings.
	
	Alternatively, one may prove the associativity of the multiplication and the compatibility using operadic methods, such as \cite[Th. 3.4]{FLS13}.
\end{proof}

\begin{Proposition}\label{prop:iso}
	The two Hopf structures $(\TSym,m^A,\Delta^A)$ and $(\TSym,m^B,\Delta^B)$ are isomorphic.
\end{Proposition}

\begin{proof}
	We define a partial order $\leq$ on planar trees. The covering relation is as follows. Fix a tree $t$ and an internal node $x$ of degree at least 3. Fix a branch of $x$ in the middle, and move it along the left child of $x$ that creates a new internal node $y$. This process gives a new tree $s$ and we define the covering relation $t\lessdot s$.
See Figure~\ref{fig:PTcover_and_order} to visualize the covering relation and an example of connected component of the partial order.
	
\begin{figure}
	\begin{center}
		\begin{tikzpicture}
		\node at (0,0){\begin{tikzpicture}
			\node at (0,0) {$x$};
			\node at (1.5,1.3) {$B$};
			\node at (-1.5,1.3) {$A$};
			
			\filldraw[Mahogany,ultra thick] (0,0.3) -- (0.5,1.2) -- (1.2,1.2);
			\filldraw[ForestGreen,ultra thick] (0,0.3) -- (-1.2,1.2) -- (0.5,1.2);
			
			\node at (0,2) {$\uparrow$};
			
			\node at (0,3) {$x$};
			\node at (-3,5.9) {$A$};
			\node at (1.5,4.5) {$B$};
			\filldraw[black] (-1.5,4.5) circle (3pt)  {};
			
			\draw[ForestGreen,very thick] (-0.3,3.3) -- (-1.2,4.2);
			\filldraw[Mahogany,very thick] (0.2,3.3) -- (1.2,4.2) -- (-0.8,4.2);
			\filldraw[ForestGreen,ultra thick] (-1.5,4.8) -- (-2.7,5.7) -- (-0.3,5.7);
			\end{tikzpicture}};
		\node at (8,0){
		\begin{tikzpicture}[scale=0.25,baseline=0pt]
		\draw[Mahogany, ultra thick] (-2.2,2) -- (-8.7,6);
		\draw[Mahogany, ultra thick] (-2.2,10) -- (-8.7,14);
		\draw[Mahogany, ultra thick] (7.8,10) -- (1.3,14);
		\draw[Mahogany, ultra thick] (7.8,18) -- (1.3,22);
		
		\draw[ForestGreen, ultra thick] (2.2,2) -- (8.7,6);
		\draw[ForestGreen, ultra thick] (2.2,10) -- (4.475,11.4);
		\draw[ForestGreen, ultra thick] (5.45,12) -- (8.7,14);
		\draw[ForestGreen, ultra thick] (-7.8,10) -- (-5.525,11.4);
		\draw[ForestGreen, ultra thick] (-4.55,12) -- (-1.3,14);
		\draw[ForestGreen, ultra thick] (-7.8,18) -- (-1.3,22);
		
		\draw[YellowOrange, ultra thick] (0,2.3) -- (0,5.7);
		\draw[YellowOrange, ultra thick] (-10,10.3) -- (-10,13.7);
		\draw[YellowOrange, ultra thick] (10,10.3) -- (10,13.7);
		\draw[YellowOrange, ultra thick] (0,18.3) -- (0,21.7);
		
		\node () at (0,0){\begin{tikzpicture}[scale=0.1]
			\draw[blue, thick] (0,-1) -- (0,0);
			\draw[blue, thick] (0,0) -- (6,6);
			\draw[blue, thick] (0,0) -- (-6,6);
			\draw[blue, thick] (0,0) -- (-2,4);
			\draw[blue, thick] (-2,4) -- (-4,6);
			\draw[blue, thick] (-2,4) -- (-2,6);
			\draw[blue, thick] (-2,4) -- (0,6);
			\draw[blue, thick] (0,0) -- (2,6);
			\draw[blue, thick] (5,5) -- (4,6);
			\filldraw[black] (0,0) circle (5pt)  {};
			\filldraw[black] (5,5) circle (5pt)  {};
			\filldraw[black] (-2,4) circle (5pt)  {};\end{tikzpicture}};
		\node () at (-10,8){\begin{tikzpicture}[scale=0.1]
			\draw[blue, thick] (0,-1) -- (0,0);
			\draw[blue, thick] (0,0) -- (6,6);
			\draw[blue, thick] (0,0) -- (-6,6);
			\draw[blue, thick] (-2,4) -- (-4,6);
			\draw[blue, thick] (-2,4) -- (-2,6);
			\draw[blue, thick] (-3,3) -- (0,6);
			\draw[blue, thick] (0,0) -- (2,6);
			\draw[blue, thick] (5,5) -- (4,6);
			\filldraw[black] (0,0) circle (5pt)  {};
			\filldraw[black] (5,5) circle (5pt)  {};
			\filldraw[black] (-3,3) circle (5pt)  {};
			\filldraw[black] (-2,4) circle (5pt)  {};\end{tikzpicture}};
		\node () at (0,8){\begin{tikzpicture}[scale=0.1]
			\draw[blue, thick] (0,-1) -- (0,0);
			\draw[blue, thick] (0,0) -- (6,6);
			\draw[blue, thick] (0,0) -- (-6,6);
			\draw[blue, thick] (-2,2) -- (-2,4);
			\draw[blue, thick] (-2,4) -- (-4,6);
			\draw[blue, thick] (-2,4) -- (-2,6);
			\draw[blue, thick] (-2,4) -- (0,6);
			\draw[blue, thick] (-2,2) -- (2,6);
			\draw[blue, thick] (5,5) -- (4,6);
			\filldraw[black] (0,0) circle (5pt)  {};
			\filldraw[black] (-2,2) circle (5pt)  {};
			\filldraw[black] (5,5) circle (5pt)  {};
			\filldraw[black] (-2,4) circle (5pt)  {};\end{tikzpicture}};
		\node () at (10,8){\begin{tikzpicture}[scale=0.1]
			\draw[blue, thick] (0,-1) -- (0,0);
			\draw[blue, thick] (0,0) -- (6,6);
			\draw[blue, thick] (0,0) -- (-6,6);
			\draw[blue, thick] (0,0) -- (-2,4);
			\draw[blue, thick] (-2,4) -- (-4,6);
			\draw[blue, thick] (-3,5) -- (-2,6);
			\draw[blue, thick] (-2,4) -- (0,6);
			\draw[blue, thick] (0,0) -- (2,6);
			\draw[blue, thick] (5,5) -- (4,6);
			\filldraw[black] (0,0) circle (5pt)  {};
			\filldraw[black] (5,5) circle (5pt)  {};
			\filldraw[black] (-3,5) circle (5pt)  {};
			\filldraw[black] (-2,4) circle (5pt)  {};\end{tikzpicture}};
		\node () at (-10,16){\begin{tikzpicture}[scale=0.1]
			\draw[blue, thick] (0,-1) -- (0,0);
			\draw[blue, thick] (0,0) -- (6,6);
			\draw[blue, thick] (0,0) -- (-6,6);
			\draw[blue, thick] (-2,4) -- (-4,6);
			\draw[blue, thick] (-2,4) -- (-2,6);
			\draw[blue, thick] (-3,3) -- (0,6);
			\draw[blue, thick] (-2,2) -- (2,6);
			\draw[blue, thick] (5,5) -- (4,6);
			\filldraw[black] (0,0) circle (5pt)  {};
			\filldraw[black] (-2,2) circle (5pt)  {};
			\filldraw[black] (5,5) circle (5pt)  {};
			\filldraw[black] (-3,3) circle (5pt)  {};
			\filldraw[black] (-2,4) circle (5pt)  {};\end{tikzpicture}};
		\node () at (0,16){\begin{tikzpicture}[scale=0.1]
			\draw[blue, thick] (0,-1) -- (0,0);
			\draw[blue, thick] (0,0) -- (6,6);
			\draw[blue, thick] (0,0) -- (-6,6);
			\draw[blue, thick] (-2,4) -- (-4,6);
			\draw[blue, thick] (-3,5) -- (-2,6);
			\draw[blue, thick] (-3,3) -- (0,6);
			\draw[blue, thick] (0,0) -- (2,6);
			\draw[blue, thick] (5,5) -- (4,6);
			\filldraw[black] (0,0) circle (5pt)  {};
			\filldraw[black] (-3,3) circle (5pt)  {};
			\filldraw[black] (5,5) circle (5pt)  {};
			\filldraw[black] (-3,5) circle (5pt)  {};
			\filldraw[black] (-2,4) circle (5pt)  {};\end{tikzpicture}};
		\node () at (10,16){\begin{tikzpicture}[scale=0.1]
			\draw[blue, thick] (0,-1) -- (0,0);
			\draw[blue, thick] (0,0) -- (6,6);
			\draw[blue, thick] (0,0) -- (-6,6);
			\draw[blue, thick] (-2,2) -- (-2,4);
			\draw[blue, thick] (-2,4) -- (-4,6);
			\draw[blue, thick] (-3,5) -- (-2,6);
			\draw[blue, thick] (-2,4) -- (0,6);
			\draw[blue, thick] (-2,2) -- (2,6);
			\draw[blue, thick] (5,5) -- (4,6);
			\filldraw[black] (0,0) circle (5pt)  {};
			\filldraw[black] (-2,2) circle (5pt)  {};
			\filldraw[black] (5,5) circle (5pt)  {};
			\filldraw[black] (-3,5) circle (5pt)  {};
			\filldraw[black] (-2,4) circle (5pt)  {};\end{tikzpicture}};
		\node () at (0,24){\begin{tikzpicture}[scale=0.1]
			\draw[blue, thick] (0,-1) -- (0,0);
			\draw[blue, thick] (0,0) -- (6,6);
			\draw[blue, thick] (0,0) -- (-6,6);
			\draw[blue, thick] (-2,4) -- (-4,6);
			\draw[blue, thick] (-3,5) -- (-2,6);
			\draw[blue, thick] (-3,3) -- (0,6);
			\draw[blue, thick] (-2,2) -- (2,6);
			\draw[blue, thick] (5,5) -- (4,6);
			\filldraw[black] (0,0) circle (5pt)  {};
			\filldraw[black] (-3,3) circle (5pt)  {};
			\filldraw[black] (-2,2) circle (5pt)  {};
			\filldraw[black] (5,5) circle (5pt)  {};
			\filldraw[black] (-3,5) circle (5pt)  {};
			\filldraw[black] (-2,4) circle (5pt)  {};\end{tikzpicture}};
		\end{tikzpicture}};
	\end{tikzpicture}
	\end{center}
	\caption{\sl Covering relation and a connected component of the partial order on planar trees needed for Proposition~\ref{prop:iso}}\label{fig:PTcover_and_order}
	\end{figure}
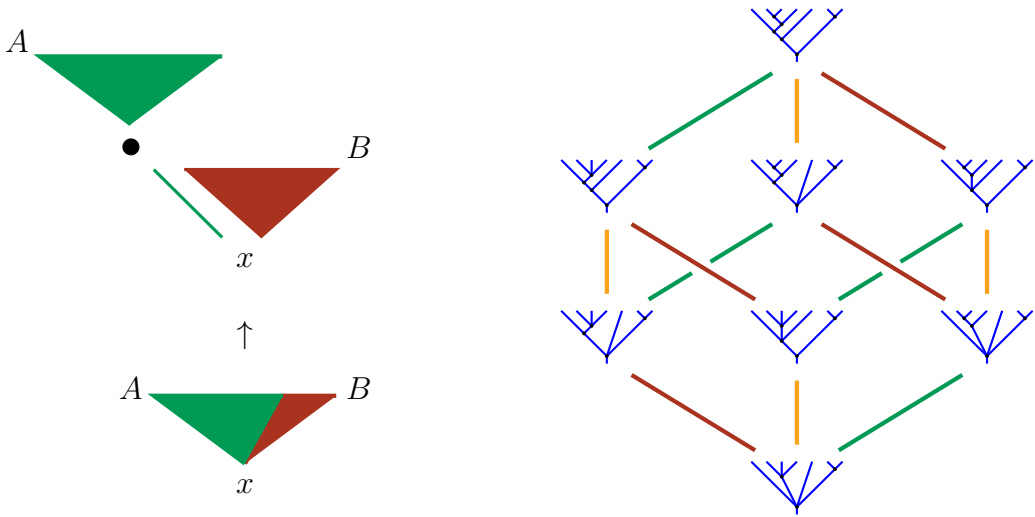

	Applying the basic move at each internal node of a tree $t$ generates a boolean poset. Since the basic moves at two different internal nodes are independent of each other, we see that each connected component in this order is a boolean poset.
	
	Fix a tree $t$, and an integer $i\in [1,\deg t-1]$. Construct the tree $t_{(i)}\geq t$ by applying the above cover relation to every node along the path from the $i+1$th leaf to the root, for the branch that is along this path. Then $t_{(i)}$ is a unique minimal tree  such that $t_{(i)}\geq t$ and the splitting at the $(i+1)$-th leaf of $t_{(i)}$ is allowable. Moreover, splitting $t_{(i)}$ and $t$ at the $(i+1)$-th leaf give the same result i.e. $({}^it,t^i)=({}^it_{(i)},t^i_{(i)})$. As an example, if we consider $t$ to be the bottom tree 
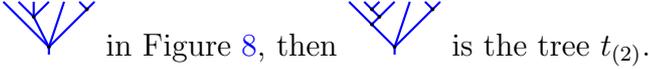
\begin{tikzpicture}[scale=0.1]
			\draw[blue, thick] (0,-1) -- (0,0);
			\draw[blue, thick] (0,0) -- (6,6);
			\draw[blue, thick] (0,0) -- (-6,6);
			\draw[blue, thick] (0,0) -- (-2,4);
			\draw[blue, thick] (-2,4) -- (-4,6);
			\draw[blue, thick] (-2,4) -- (-2,6);
			\draw[blue, thick] (-2,4) -- (0,6);
			\draw[blue, thick] (0,0) -- (2,6);
			\draw[blue, thick] (5,5) -- (4,6);
			\filldraw[black] (0,0) circle (5pt)  {};
			\filldraw[black] (5,5) circle (5pt)  {};
			\filldraw[black] (-2,4) circle (5pt)  {};\end{tikzpicture} in Figure \ref{fig:PTcover_and_order}, then \begin{tikzpicture}[scale=0.1]
			\draw[blue, thick] (0,-1) -- (0,0);
			\draw[blue, thick] (0,0) -- (6,6);
			\draw[blue, thick] (0,0) -- (-6,6);
			\draw[blue, thick] (-2,4) -- (-4,6);
			\draw[blue, thick] (-3,5) -- (-2,6);
			\draw[blue, thick] (-3,3) -- (0,6);
			\draw[blue, thick] (0,0) -- (2,6);
			\draw[blue, thick] (5,5) -- (4,6);
			\filldraw[black] (0,0) circle (5pt)  {};
			\filldraw[black] (-3,3) circle (5pt)  {};
			\filldraw[black] (5,5) circle (5pt)  {};
			\filldraw[black] (-3,5) circle (5pt)  {};
			\filldraw[black] (-2,4) circle (5pt)  {};\end{tikzpicture} is the tree $t_{(2)}$.

	We define a new basis $\{H_t\}$ such that $H_t=\displaystyle\sum_{s\geq t} F_s$. Then,
	
	$$\Delta_i^B(H_t)=\sum_{s\geq t}\Delta_i^B(F_s)=\sum_{s\geq t_{(i)}}\Delta_i^B(F_s)=\sum_{s_1\geq {}^it,s_2\geq t^i}F_{s_1}\otimes F_{s_2}=H_{{}^it}\otimes H_{t^i}.$$
	
	By construction, we know that $\Delta_i^A(\tilde{F}_t)=\tilde{F}_{{}^it}\otimes \tilde{F}_{t^i}$. And since the two multiplications $m^A$ and $m^B$ are based on the composed comultiplication $\Delta^A\circ\cdots\circ\Delta^A$ and $\Delta^B\circ\cdots\circ\Delta^B$, we have an isomorphism $(\TSym,m^B,\Delta^B)\to(\TSym,m^A,\Delta^A)$ such that $H_t\mapsto \tilde{F}_t$.
\end{proof}

In the rest of this paper, we only consider the operations $m^B,\Delta^B$, which we will denote by $m,\Delta$.

\begin{Remark}
	The Hopf algebra $\YSym$ does not have a bidendriform bialgebra structure. Hence, we do not expect one for $\TSym$.
\end{Remark}

\subsection{Planar Tamari order and monomial basis}

A left rotation is the operation on planar trees described in Figure~\ref{fig:Planar_Tamari}

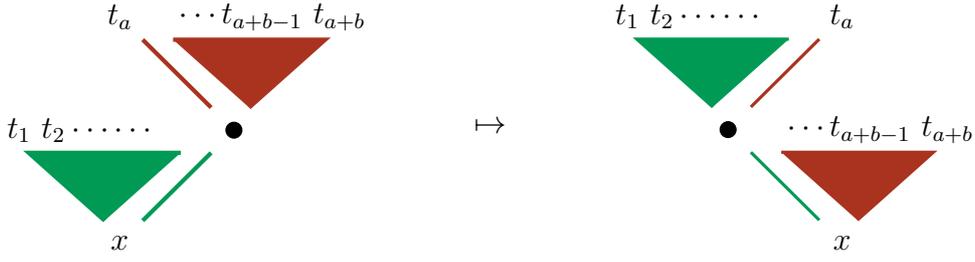
\begin{figure}
\begin{center}
	\begin{tikzpicture}
	\node at (0,0) {\begin{tikzpicture}
		\node at (0,0) {$x$};
		\filldraw[black] (1.5,1.5) circle (3pt)  {};
		\node at (-0.5,1.5) {$t_1~t_2\cdots\cdots$};
		\node at (0,3) {$t_a$};
		\node at (2,3) {$\cdots t_{a+b-1}~t_{a+b}$};
		
		\draw[ForestGreen,ultra thick] (0.3,0.3) -- (1.2,1.2);
		\filldraw[ForestGreen,ultra thick] (-0.2,0.3) -- (-1.2,1.2) -- (0.8,1.2);
		\filldraw[Mahogany,ultra thick] (1.7,1.8) -- (2.7,2.7) -- (0.7,2.7);
		\draw[Mahogany,ultra thick] (1.2,1.8) -- (0.3,2.7);
		\end{tikzpicture}};
	\node at (4,0) {$\mapsto$};
	\node at (8,0) {\begin{tikzpicture}
		\node at (0,0) {$x$};
		\filldraw[black] (-1.5,1.5) circle (3pt)  {};
		\node at (-2,3) {$t_1~t_2\cdots\cdots$};
		\node at (0,3) {$t_a$};
		\node at (0.5,1.5) {$\cdots t_{a+b-1}~t_{a+b}$};
		
		\filldraw[Mahogany,very thick] (0.2,0.3) -- (1.2,1.2) -- (-0.8,1.2);
		\draw[ForestGreen,very thick] (-0.3,0.3) -- (-1.2,1.2);
		\filldraw[ForestGreen,very thick] (-1.7,1.8) -- (-2.7,2.7) -- (-0.7,2.7);
		\draw[Mahogany,very thick] (-1.2,1.8) -- (-0.3,2.7);
		\end{tikzpicture}};
\end{tikzpicture}
\end{center}
\caption{\sl A left rotation: the cover relation of planar Tamari order}\label{fig:Planar_Tamari}
\end{figure}

\begin{Definition}\label{def:lpt}
We define the planar Tamari order on planar trees as follow. A cover relation $s\lessdot_{PT}t$ is defined if $t$ is obtained from $s$ by doing a left rotation at a node $x$, where $y$ is the right most child of $x$, and $B$ is the left most branch of $y$ as depicted Figure~\ref{fig:Planar_Tamari}. The planar left Tamari of size $n$ is the transitive closure of this relation. For example see Figure~\ref{fig:PTamari4} and ~\ref{fig:PTamari5}.
\end{Definition}

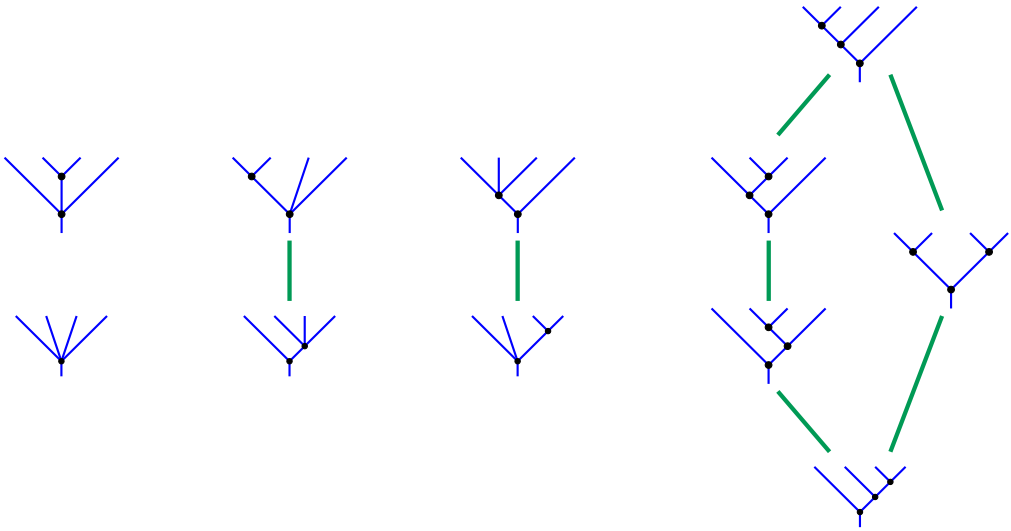
\begin{figure}
\begin{center}
\begin{tikzpicture}
\node at (0,0) {\begin{tikzpicture}[scale=0.4]
	\node () at (0,0){\begin{tikzpicture}[scale=0.2,baseline=0pt]
		\draw[blue, thick] (0,-1) -- (0,0);
		\draw[blue, thick] (0,0) -- (3,3);
		\draw[blue, thick] (0,0) -- (1,3);
		\draw[blue, thick] (0,0) -- (-1,3);
		\draw[blue, thick] (0,0) -- (-3,3);
		\filldraw[black] (0,0) circle (5pt)  {};\end{tikzpicture}};
	\node () at (0,5){\begin{tikzpicture}[scale=0.25,baseline=0pt]
		\draw[blue, thick] (0,-1) -- (0,0);
		\draw[blue, thick] (0,0) -- (3,3);
		\draw[blue, thick] (0,2) -- (1,3);
		\draw[blue, thick] (0,2) -- (-1,3);
		\draw[blue, thick] (0,0) -- (-3,3);
		\draw[blue, thick] (0,0) -- (0,2);
		\filldraw[black] (0,0) circle (5pt)  {};
		\filldraw[black] (0,2) circle (5pt)  {};\end{tikzpicture}};
\end{tikzpicture}};
\node at (3,0) {\begin{tikzpicture}[scale=0.4]
\draw[ForestGreen, ultra thick] (0,1.5) -- (0,3.5);
\node () at (0,0){\begin{tikzpicture}[scale=0.2,baseline=0pt]
	\draw[blue, thick] (0,-1) -- (0,0);
	\draw[blue, thick] (0,0) -- (3,3);
	\draw[blue, thick] (1,1) -- (1,3);
	\draw[blue, thick] (1,1) -- (-1,3);
	\draw[blue, thick] (0,0) -- (-3,3);
	\filldraw[black] (0,0) circle (5pt)  {};
	\filldraw[black] (1,1) circle (5pt)  {};\end{tikzpicture}};
\node () at (0,5){\begin{tikzpicture}[scale=0.25,baseline=0pt]
	\draw[blue, thick] (0,-1) -- (0,0);
	\draw[blue, thick] (0,0) -- (3,3);
	\draw[blue, thick] (0,0) -- (1,3);
	\draw[blue, thick] (-2,2) -- (-1,3);
	\draw[blue, thick] (0,0) -- (-3,3);
	\filldraw[black] (0,0) circle (5pt)  {};
	\filldraw[black] (-2,2) circle (5pt)  {};\end{tikzpicture}};
\end{tikzpicture}};
\node at (6,0) {\begin{tikzpicture}[scale=0.4]
\draw[ForestGreen, ultra thick] (0,1.5) -- (0,3.5);
\node () at (0,0){\begin{tikzpicture}[scale=0.2,baseline=0pt]
\draw[blue, thick] (0,-1) -- (0,0);
\draw[blue, thick] (0,0) -- (3,3);
\draw[blue, thick] (2,2) -- (1,3);
\draw[blue, thick] (0,0) -- (-1,3);
\draw[blue, thick] (0,0) -- (-3,3);
\filldraw[black] (0,0) circle (5pt)  {};
\filldraw[black] (2,2) circle (5pt)  {};\end{tikzpicture}};
\node () at (0,5){\begin{tikzpicture}[scale=0.25,baseline=0pt]
\draw[blue, thick] (0,-1) -- (0,0);
\draw[blue, thick] (0,0) -- (3,3);
\draw[blue, thick] (-1,1) -- (1,3);
\draw[blue, thick] (-1,1) -- (-1,3);
\draw[blue, thick] (0,0) -- (-3,3);
\filldraw[black] (0,0) circle (5pt)  {};
\filldraw[black] (-1,1) circle (5pt)  {};\end{tikzpicture}};
\end{tikzpicture}};
\node at (10.5,0) {\begin{tikzpicture}[scale=0.4]
\draw[ForestGreen, ultra thick] (-3,6.5) -- (-3,8.5);
\draw[ForestGreen, ultra thick] (1,1.5) -- (2.7,6);
\draw[ForestGreen, ultra thick] (1,14) -- (2.7,9.5);
\draw[ForestGreen, ultra thick] (-1,14) -- (-2.7,12);
\draw[ForestGreen, ultra thick] (-1,1.5) -- (-2.7,3.5);
\node () at (0,0){\begin{tikzpicture}[scale=0.2,baseline=0pt]
\draw[blue, thick] (0,-1) -- (0,0);
\draw[blue, thick] (0,0) -- (3,3);
\draw[blue, thick] (2,2) -- (1,3);
\draw[blue, thick] (1,1) -- (-1,3);
\draw[blue, thick] (0,0) -- (-3,3);
\filldraw[black] (0,0) circle (5pt)  {};
\filldraw[black] (1,1) circle (5pt)  {};
\filldraw[black] (2,2) circle (5pt)  {};\end{tikzpicture}};
\node () at (-3,5){\begin{tikzpicture}[scale=0.25,baseline=0pt]
\draw[blue, thick] (0,-1) -- (0,0);
\draw[blue, thick] (0,0) -- (3,3);
\draw[blue, thick] (0,2) -- (1,3);
\draw[blue, thick] (1,1) -- (-1,3);
\draw[blue, thick] (0,0) -- (-3,3);
\filldraw[black] (0,0) circle (5pt)  {};
\filldraw[black] (1,1) circle (5pt)  {};
\filldraw[black] (0,2) circle (5pt)  {};\end{tikzpicture}};
\node () at (-3,10){\begin{tikzpicture}[scale=0.25,baseline=0pt]
\draw[blue, thick] (0,-1) -- (0,0);
\draw[blue, thick] (0,0) -- (3,3);
\draw[blue, thick] (-1,1) -- (1,3);
\draw[blue, thick] (0,2) -- (-1,3);
\draw[blue, thick] (0,0) -- (-3,3);
\filldraw[black] (0,0) circle (5pt)  {};
\filldraw[black] (-1,1) circle (5pt)  {};
\filldraw[black] (0,2) circle (5pt)  {};\end{tikzpicture}};
\node () at (3,7.5){\begin{tikzpicture}[scale=0.25,baseline=0pt]
\draw[blue, thick] (0,-1) -- (0,0);
\draw[blue, thick] (0,0) -- (3,3);
\draw[blue, thick] (2,2) -- (1,3);
\draw[blue, thick] (-2,2) -- (-1,3);
\draw[blue, thick] (0,0) -- (-3,3);
\filldraw[black] (0,0) circle (5pt)  {};
\filldraw[black] (2,2) circle (5pt)  {};
\filldraw[black] (-2,2) circle (5pt)  {};\end{tikzpicture}};
\node () at (0,15){\begin{tikzpicture}[scale=0.25,baseline=0pt]
\draw[blue, thick] (0,-1) -- (0,0);
\draw[blue, thick] (0,0) -- (3,3);
\draw[blue, thick] (-1,1) -- (1,3);
\draw[blue, thick] (-2,2) -- (-1,3);
\draw[blue, thick] (0,0) -- (-3,3);
\filldraw[black] (0,0) circle (5pt)  {};
\filldraw[black] (-1,1) circle (5pt)  {};
\filldraw[black] (-2,2) circle (5pt)  {};\end{tikzpicture}};
\end{tikzpicture}};
\end{tikzpicture}
\end{center}
\caption{Planar Tamari order on trees with $4$ leaves.}\label{fig:PTamari4}
\end{figure}

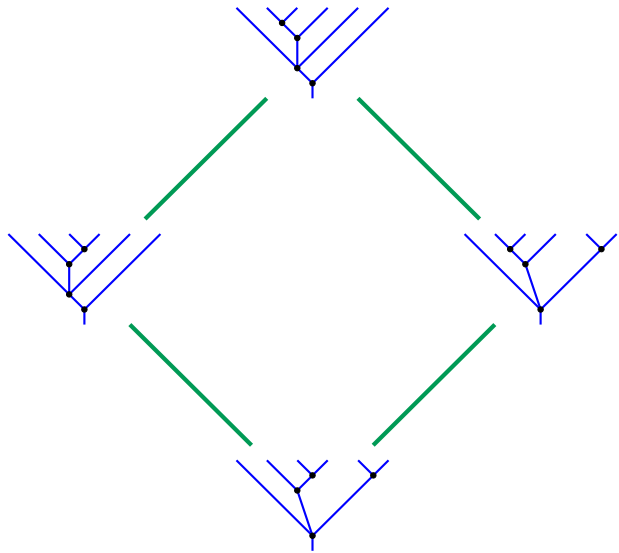
\begin{figure}
\begin{center}
\begin{tikzpicture}
\node at (0,0) {\begin{tikzpicture}
\node () at (0,0){\begin{tikzpicture}[scale=0.2,baseline=0pt]
\draw[blue, thick] (0,-1) -- (0,0);
\draw[blue, thick] (0,0) -- (5,5);
\draw[blue, thick] (4,4) -- (3,5);
\draw[blue, thick] (-1,3) -- (1,5);
\draw[blue, thick] (0,4) -- (-1,5);
\draw[blue, thick] (-1,3) -- (-3,5);
\draw[blue, thick] (0,0) -- (-5,5);
\draw[blue, thick] (0,0) -- (-1,3);
\filldraw[black] (0,0) circle (5pt)  {};
\filldraw[black] (4,4) circle (5pt)  {};
\filldraw[black] (-1,3) circle (5pt)  {};
\filldraw[black] (0,4) circle (5pt)  {};\end{tikzpicture}};
\node () at (-3,3){\begin{tikzpicture}[scale=0.2,baseline=0pt]
\draw[blue, thick] (0,-1) -- (0,0);
\draw[blue, thick] (0,0) -- (5,5);
\draw[blue, thick] (-1,1) -- (3,5);
\draw[blue, thick] (-1,3) -- (1,5);
\draw[blue, thick] (0,4) -- (-1,5);
\draw[blue, thick] (-1,3) -- (-3,5);
\draw[blue, thick] (0,0) -- (-5,5);
\draw[blue, thick] (-1,1) -- (-1,3);
\filldraw[black] (0,0) circle (5pt)  {};
\filldraw[black] (-1,1) circle (5pt)  {};
\filldraw[black] (-1,3) circle (5pt)  {};
\filldraw[black] (0,4) circle (5pt)  {};\end{tikzpicture}};
\node () at (3,3){\begin{tikzpicture}[scale=0.2,baseline=0pt]
\draw[blue, thick] (0,-1) -- (0,0);
\draw[blue, thick] (0,0) -- (5,5);
\draw[blue, thick] (4,4) -- (3,5);
\draw[blue, thick] (-1,3) -- (1,5);
\draw[blue, thick] (-2,4) -- (-1,5);
\draw[blue, thick] (-1,3) -- (-3,5);
\draw[blue, thick] (0,0) -- (-5,5);
\draw[blue, thick] (0,0) -- (-1,3);
\filldraw[black] (0,0) circle (5pt)  {};
\filldraw[black] (4,4) circle (5pt)  {};
\filldraw[black] (-1,3) circle (5pt)  {};
\filldraw[black] (-2,4) circle (5pt)  {};\end{tikzpicture}};
\node () at (0,6){\begin{tikzpicture}[scale=0.2,baseline=0pt]
\draw[blue, thick] (0,-1) -- (0,0);
\draw[blue, thick] (0,0) -- (5,5);
\draw[blue, thick] (-1,1) -- (3,5);
\draw[blue, thick] (-1,3) -- (1,5);
\draw[blue, thick] (-2,4) -- (-1,5);
\draw[blue, thick] (-1,3) -- (-3,5);
\draw[blue, thick] (0,0) -- (-5,5);
\draw[blue, thick] (-1,1) -- (-1,3);
\filldraw[black] (0,0) circle (5pt)  {};
\filldraw[black] (-1,1) circle (5pt)  {};
\filldraw[black] (-1,3) circle (5pt)  {};
\filldraw[black] (-2,4) circle (5pt)  {};\end{tikzpicture}};

\draw[ForestGreen,ultra thick] (0.8,0.8) -- (2.4,2.4);
\draw[ForestGreen,ultra thick] (-0.8,0.8) -- (-2.4,2.4);
\draw[ForestGreen,ultra thick] (-2.2,3.8) -- (-0.6,5.4);
\draw[ForestGreen,ultra thick] (2.2,3.8) -- (0.6,5.4);
\end{tikzpicture}};
\end{tikzpicture}
\end{center}
\caption{Another connected component of planar Tamari order.}\label{fig:PTamari5}
\end{figure}

\begin{Proposition}\label{prop:interval_multiplication_TSYM}
Let $F_s$ and $F_t$ be two basis elements in $\TSym$, and let $F_s^*$, $F_t^*$ be their dual basis elements in the graded dual $\TSym^*$. Then,
$$m(F_s^*\otimes F_t^*)=\sum_{s\backslash t\leq_{PT} r\leq_{PT} s/t}F_r^*.$$
\end{Proposition}

The proof will be given as Corollary \ref{cor:lst}, after introducing a partial order on generalized Stirling permutations.

\begin{Definition}
The monomial basis of $\TSym$ are defined as the elements satisfying $\displaystyle F_s=\sum_{t\geq_{PT} s}M_t$. The monomial basis is unique by triangularity and it forms a basis via M\"{o}bius inversion.
\end{Definition}

\begin{Definition}
Let $t$ be a planar tree of degree $n$. Its \textbf{global descent} set is $\GD(t)=\{i\in[n-1]:t=s/r\text{ for some }s,r\text{ and }\deg(s)=i\}$.
\end{Definition}

\begin{Proposition} \label{prop:coproduct-monomial-tsym} The coproduct of
monomial basis elements of $\TSym$ is given by 
\[
\Delta_{+}(M_{t})=\sum_{i\in\GD(f)}M_{{}^{i}t}\otimes M_{t^{i}}.
\]
\end{Proposition}
\begin{Proposition}\label{prop:product-monomial-tsym} The product
of monomial basis elements of $\TSym$ is given by $M_{s}\cdot M_{t}=\sum_{r}\alpha_{s,t}^{r}M_{r}$,
where $\alpha_{s,t}^{r}=|A_{s,t}^{r}|$ is defined as in (\ref{eq:a-def}).
\end{Proposition}
\begin{Proposition}\label{prop:antipode-monomial-tsym} The
antipode of monomial basis elements in $\TSym$ is given by $\mathcal{S}(M_{t})=(-1)^{\GD(t)+1}\sum_{s}\beta_{t}^{s}M_{s}$
where $\beta_{t}^{s}$ is defined as in (\ref{eq:c-def-GD-2}).
\end{Proposition}
\begin{proof}[Proof of Propositions \ref{prop:coproduct-monomial-tsym}-\ref{prop:antipode-monomial-tsym}]
We check the axioms of Theorem \ref{thm:monomialquotient}, for the
case $\bar{\mathcal{H}}=\TSym$ and $\mathcal{H}=\STSym$, the algebra
of generalized Stirling permutations of the next section. This would
provide the axioms for Theorems \ref{thm:monomialcoproduct}, \ref{thm:monomialproduct}
and \ref{thm:antipode}, which give respectively the results above.

The proofs of Propositions \ref{prop:coproduct-monomial-stsym}, \ref{prop:product-monomial-stsym}
and \ref{prop:antipode-monomial-stsym} below show that $\STSym$
satisfies all necessary axioms for Theorem \ref{thm:monomialquotient}.
Proposition \ref{prop:lsb} checks the conditions on $\pi$ and $\iota$,
and axiom ($\pi$0). It is clear that $\TSym$ satisfies axioms ($\pi\Delta$1)
and ($\pi m$1), and axioms ($\pi\Delta$2), ($\pi m$2), ($\pi\mathcal{S}$1)
and ($\pi\mathcal{S}$2) follow from Propositions \ref{prop:pi-property},
and Lemma \ref{lem:iota-gd}.
\end{proof}

\begin{Remark}
	When all internal nodes of a tree $t$ have $m+1$ children, we obtain $(m+1)$-ary trees. Restricting our planar Tamari order to $(m+1)$-ary trees gives a partial order that coarsens the $m$-Tamari order defined in \cite{BP12}, and the associated Hopf subalgebra of $\TSym$ is isomorphic to the one defined in \cite{NT20}. Moreover, for every given subset of positive integers $A \subseteq \mathbb{P}$ we obtain a Hopf subalgebra of $\TSym$ restricting to trees whose internal nodes are all in $A$. This is  related to a partition, studied by Ceballos and Gonz\'alez D'Le\'on in \cite{CG19}, of the set of planar trees  in terms of compositions, which they call \emph{signatures}. An open question is to find further Hopf constructions that are compatible with the notion of signature. 
\end{Remark}

\section{Hopf algebra of generalized Stirling permutations}\label{sec:STSYM}

\subsection{Generalized Stirling permutations and packed words}\label{sec:gsp}

In this section, we extend $\TSym$ to the context of generalized Stirling permutations, and we study their connections with packed words and ordered set compositions.

\begin{Definition}
	A \emph{generalized Stirling permutation} $u$ of degree $n$ is a labeled tree with $n+1$ leaves and $\kappa_u$ is a bijection from the set of internal nodes to $\{1,2,\dots,\ideg(u)\}$ such that the label of each internal node is less than the labels of its children. The set of generalized Stirling permutations is denoted by $\GSP$. See for example Figure~\ref{fig:GS3}.
\end{Definition}

\begin{figure}
\begin{center}
	\begin{tikzpicture}[scale=0.2,baseline=0pt]
	\draw[blue, thick] (0,-1) -- (0,0);
	\draw[blue, thick] (0,0) -- (3,3);
	\draw[blue, thick] (0,0) -- (-3,3);
	\draw[blue, thick] (0,0) -- (-1,3);
	\draw[blue, thick] (0,0) -- (1,3);
	\filldraw[black] (0,0) circle (5pt)  {};
	\node at (-1,1.5)[font=\fontsize{7pt}{0}]{$1$};
	\node at (0,1.5)[font=\fontsize{7pt}{0}]{$1$};
	\node at (1,1.5)[font=\fontsize{7pt}{0}]{$1$};
	\end{tikzpicture}
	
	\begin{tikzpicture}[scale=0.2,baseline=0pt]
	\draw[blue, thick] (0,-1) -- (0,0);
	\draw[blue, thick] (0,0) -- (3,3);
	\draw[blue, thick] (0,0) -- (-3,3);
	\draw[blue, thick] (0,0) -- (-1,3);
	\draw[blue, thick] (2,2) -- (1,3);
	\filldraw[black] (0,0) circle (5pt)  {};
	\filldraw[black] (2,2) circle (5pt)  {};
	\node at (-1,1.5)[font=\fontsize{7pt}{0}]{$1$};
	\node at (0.5,1.5)[font=\fontsize{7pt}{0}]{$1$};
	\node at (2,2.7)[font=\fontsize{7pt}{0}]{$2$};
	
	\draw[blue, thick] (10,-1) -- (10,0);
	\draw[blue, thick] (10,0) -- (13,3);
	\draw[blue, thick] (10,0) -- (7,3);
	\draw[blue, thick] (10,0) -- (10,2);
	\draw[blue, thick] (10,2) -- (9,3);
	\draw[blue, thick] (10,2) -- (11,3);
	\filldraw[black] (10,0) circle (5pt)  {};
	\filldraw[black] (10,2) circle (5pt)  {};
	\node at (9.5,1)[font=\fontsize{7pt}{0}]{$1$};
	\node at (10.5,1)[font=\fontsize{7pt}{0}]{$1$};
	\node at (10,2.7)[font=\fontsize{7pt}{0}]{$2$};
	
	\draw[blue, thick] (20,-1) -- (20,0);
	\draw[blue, thick] (20,0) -- (23,3);
	\draw[blue, thick] (20,0) -- (17,3);
	\draw[blue, thick] (20,0) -- (21,3);
	\draw[blue, thick] (18,2) -- (19,3);
	\filldraw[black] (20,0) circle (5pt)  {};
	\filldraw[black] (18,2) circle (5pt)  {};
	\node at (21,1.5)[font=\fontsize{7pt}{0}]{$1$};
	\node at (19.5,1.5)[font=\fontsize{7pt}{0}]{$1$};
	\node at (18,2.7)[font=\fontsize{7pt}{0}]{$2$};
	
	\draw[blue, thick] (30,-1) -- (30,0);
	\draw[blue, thick] (30,0) -- (33,3);
	\draw[blue, thick] (30,0) -- (27,3);
	\draw[blue, thick] (29,1) -- (31,3);
	\draw[blue, thick] (29,1) -- (29,3);
	\filldraw[black] (30,0) circle (5pt)  {};
	\filldraw[black] (29,1) circle (5pt)  {};
	\node at (30,0.7)[font=\fontsize{7pt}{0}]{$1$};
	\node at (28.5,2)[font=\fontsize{7pt}{0}]{$2$};
	\node at (29.5,2)[font=\fontsize{7pt}{0}]{$2$};
	
	\draw[blue, thick] (40,-1) -- (40,0);
	\draw[blue, thick] (40,0) -- (37,3);
	\draw[blue, thick] (40,0) -- (43,3);
	\draw[blue, thick] (41,1) -- (41,3);
	\draw[blue, thick] (41,1) -- (39,3);
	\filldraw[black] (40,0) circle (5pt)  {};
	\filldraw[black] (41,1) circle (5pt)  {};
	\node at (40,0.7)[font=\fontsize{7pt}{0}]{$1$};
	\node at (40.5,2)[font=\fontsize{7pt}{0}]{$2$};
	\node at (41.5,2)[font=\fontsize{7pt}{0}]{$2$};
	\end{tikzpicture}
	
	\begin{tikzpicture}[scale=0.2,baseline=0pt]
	\draw[blue, thick] (0,-1) -- (0,0);
	\draw[blue, thick] (0,0) -- (3,3);
	\draw[blue, thick] (0,0) -- (-3,3);
	\draw[blue, thick] (1,1) -- (-1,3);
	\draw[blue, thick] (2,2) -- (1,3);
	\filldraw[black] (0,0) circle (5pt)  {};
	\filldraw[black] (1,1) circle (5pt)  {};
	\filldraw[black] (2,2) circle (5pt)  {};
	\node at (0,0.7)[font=\fontsize{7pt}{0}]{$1$};
	\node at (1,1.7)[font=\fontsize{7pt}{0}]{$2$};
	\node at (2,2.7)[font=\fontsize{7pt}{0}]{$3$};
	
	\draw[blue, thick] (10,-1) -- (10,0);
	\draw[blue, thick] (10,0) -- (13,3);
	\draw[blue, thick] (10,0) -- (7,3);
	\draw[blue, thick] (11,1) -- (9,3);
	\draw[blue, thick] (10,2) -- (11,3);
	\filldraw[black] (10,0) circle (5pt)  {};
	\filldraw[black] (11,1) circle (5pt)  {};
	\filldraw[black] (10,2) circle (5pt)  {};
	\node at (10,0.7)[font=\fontsize{7pt}{0}]{$1$};
	\node at (11,1.7)[font=\fontsize{7pt}{0}]{$2$};
	\node at (10,2.7)[font=\fontsize{7pt}{0}]{$3$};
	
	\draw[blue, thick] (20,-1) -- (20,0);
	\draw[blue, thick] (20,0) -- (23,3);
	\draw[blue, thick] (20,0) -- (17,3);
	\draw[blue, thick] (22,2) -- (21,3);
	\draw[blue, thick] (18,2) -- (19,3);
	\filldraw[black] (20,0) circle (5pt)  {};
	\filldraw[black] (22,2) circle (5pt)  {};
	\filldraw[black] (18,2) circle (5pt)  {};
	\node at (20,0.7)[font=\fontsize{7pt}{0}]{$1$};
	\node at (18,2.7)[font=\fontsize{7pt}{0}]{$2$};
	\node at (22,2.7)[font=\fontsize{7pt}{0}]{$3$};
	
	\draw[blue, thick] (30,-1) -- (30,0);
	\draw[blue, thick] (30,0) -- (33,3);
	\draw[blue, thick] (30,0) -- (27,3);
	\draw[blue, thick] (32,2) -- (31,3);
	\draw[blue, thick] (28,2) -- (29,3);
	\filldraw[black] (30,0) circle (5pt)  {};
	\filldraw[black] (32,2) circle (5pt)  {};
	\filldraw[black] (28,2) circle (5pt)  {};
	\node at (30,0.7)[font=\fontsize{7pt}{0}]{$1$};
	\node at (32,2.7)[font=\fontsize{7pt}{0}]{$2$};
	\node at (28,2.7)[font=\fontsize{7pt}{0}]{$3$};
	
	\draw[blue, thick] (40,-1) -- (40,0);
	\draw[blue, thick] (40,0) -- (43,3);
	\draw[blue, thick] (40,0) -- (37,3);
	\draw[blue, thick] (39,1) -- (41,3);
	\draw[blue, thick] (40,2) -- (39,3);
	\filldraw[black] (40,0) circle (5pt)  {};
	\filldraw[black] (39,1) circle (5pt)  {};
	\filldraw[black] (40,2) circle (5pt)  {};
	\node at (40,0.7)[font=\fontsize{7pt}{0}]{$1$};
	\node at (39,1.7)[font=\fontsize{7pt}{0}]{$2$};
	\node at (40,2.7)[font=\fontsize{7pt}{0}]{$3$};
	
	\draw[blue, thick] (50,-1) -- (50,0);
	\draw[blue, thick] (50,0) -- (47,3);
	\draw[blue, thick] (50,0) -- (53,3);
	\draw[blue, thick] (49,1) -- (51,3);
	\draw[blue, thick] (48,2) -- (49,3);
	\filldraw[black] (50,0) circle (5pt)  {};
	\filldraw[black] (49,1) circle (5pt)  {};
	\filldraw[black] (48,2) circle (5pt)  {};
	\node at (50,0.7)[font=\fontsize{7pt}{0}]{$1$};
	\node at (49,1.7)[font=\fontsize{7pt}{0}]{$2$};
	\node at (48,2.7)[font=\fontsize{7pt}{0}]{$3$};
	\end{tikzpicture}
	\caption{Generalized Stirling permutations of degree 3.}\label{fig:GS3}
	\end{center}
\end{figure}
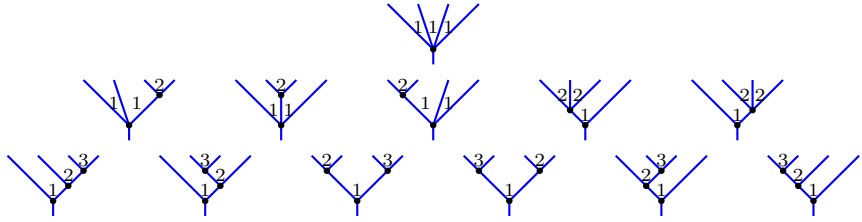

For each internal node $x$, we give the label $\kappa(x)$ to all regions between its children. Let $u$ be a generalized Stirling permutation of degree $n$, by reading the labels of the regions from left to right, we get a word $\mathbf{w}(u)=\mathbf{w}(u)_1\cdots \mathbf{w}(u)_n$. For example, let $u=$\begin{tikzpicture}[scale=0.2,baseline=0pt]
\draw[blue, thick] (0,-1) -- (0,0);
\draw[blue, thick] (0,0) -- (7,7);
\draw[blue, thick] (0,0) -- (-7,7);
\draw[blue, thick] (-5,5) -- (-3,7);
\draw[blue, thick] (-4,6) -- (-5,7);
\draw[blue, thick] (0,0) -- (1,5);
\draw[blue, thick] (1,5) -- (-1,7);
\draw[blue, thick] (1,5) -- (1,7);
\draw[blue, thick] (1,5) -- (3,7);
\draw[blue, thick] (6,6) -- (5,7);
\filldraw[black] (0,0) circle (5pt)  {};
\filldraw[black] (1,5) circle (5pt)  {};
\filldraw[black] (6,6) circle (5pt)  {};
\filldraw[black] (-5,5) circle (5pt)  {};
\filldraw[black] (-4,6) circle (5pt)  {};
\node at (-0.35,1.2)[font=\fontsize{7pt}{0}]{$1$};
\node at (0.55,1.2)[font=\fontsize{7pt}{0}]{$1$};
\node at (0.55,6.2)[font=\fontsize{7pt}{0}]{$3$};
\node at (1.45,6.2)[font=\fontsize{7pt}{0}]{$3$};
\node at (6,6.8)[font=\fontsize{7pt}{0}]{$2$};
\node at (-4,6.8)[font=\fontsize{7pt}{0}]{$5$};
\node at (-5,5.8)[font=\fontsize{7pt}{0}]{$4$};
\end{tikzpicture}. Then $\mathbf{w}(u)=4513312$. By abuse of notation, we use $u_i$ to mean $\mathbf{w}(u)_i$ for generalized Stirling permutation $u$.

\begin{Remark}
	Via $\mathbf{w}$, the set of generalized Stirling permutations in which all internal nodes have $3$ children is in bijection with the set of Stirling permutations defined in \cite{GS78}.
\end{Remark}

\begin{Definition}
	A word $w=w_1w_2\cdots w_n$ is a packed word if $w_i$ is a positive integer for $i=1,\dots, n$ and for each $1\leq a<w_i$, there exists a $j$ such that $w_j=a$. The set of packed words is denoted by $\PW$.
\end{Definition}

\begin{Remark}
	A packed word $w$ can be seen as a surjection $f_w\colon [n]\to[k]$  where $f_w(i)=w_i$ for some $k$. The sequence  $w^{-1}=(f_w^{-1}(1),f_w^{-1}(2), \ldots,f^{-1}(k))$ gives us an ordered set partition of $[n]$ with non-empty parts. We call such ordered set partition a {\sl set composition}. We thus have a bijection $w\leftrightarrow w^{-1}$  between pack words with maximal entry $k$ and set compositions with $k$ parts. See for example Figure~\ref{fig:PW_SC}. This bijection extend the notion of $w^{-1}$ for permutations. 
\end{Remark}

\begin{figure}
$$\begin{array}{llllllcllllll}
	&&&&&&w\leftrightarrow w^{-1}&&&&&&\\
		 111&&&&&&&&&&&& 123\\ \\
	 112 &121 &  211 &  221 &  122 &  212   & & 12|3 &  13|2 &  23|1 &  3|12  & 1|23 &  2|13\\ \\
         123&  132&  213&  312&  231&  321& & 1|2|3& 1|3|2&  2|1|3&  2|3|1&  3|1|2&  3|2|1\\
\end{array}$$
	\caption{\sl Packed words of length $3$ and their corresponding set compositions on the right. We use the shorthand notation $13|2$ instead of $(\{1,3\},\{2\})$}
	\label{fig:PW_SC}
\end{figure}

\begin{Definition}
	A word $w$ is said to be \emph{$212$-avoiding} if whenever $w_i=w_k$ and $i<j<k$, we must have $w_j\ge w_i$. A word $w$ is said to be \emph{$213$-avoiding} if whenever $i<j<k$, we must not have $w_k>w_i>w_j$.
\end{Definition}

\begin{Lemma}\label{lem:212}
	A packed word $w$ is $212$-avoiding if and only if there is a generalized Stirling permutation $u$ such that $w=\mathbf{w}(u)$.
\end{Lemma}

\begin{proof}
	Clearly, for any generalized Stirling permutation $u$, $\mathbf{w}(u)$ must be $212$-avoiding. Conversely, if a packed word $w$ is $212$-avoiding, we can write $w$ as $W_11W_21\cdots 1W_k$ where $W_i$ are subwords of $w$ that does not contain $1$. Since $W_i$ are $212$-avoiding, by induction on length of the packed words, we have corresponding generalized Stirling permutation $\mathbf{t}(W_i)$ (with labels shifted up). Then $w$ corresponds to the generalized Stirling permutation \begin{tikzpicture}[scale=0.2,baseline=0pt]
	\draw[blue, thick] (0,-1) -- (0,0);
	\draw[blue, thick] (0,0) -- (5,5);
	\draw[blue, thick] (0,0) -- (-1,5);
	\draw[blue, thick] (0,0) -- (0.7,2);
	\draw[blue, thick] (0,0) -- (-5,5);
	\filldraw[black] (0,0) circle (5pt)  {};
	\node at (-0.7,1.3)[font=\fontsize{7pt}{0}]{$1$};
	\node at (-5,5.7)[font=\fontsize{7pt}{0}]{$\mathbf{t}(W_1)$};
	\node at (5,5.7)[font=\fontsize{7pt}{0}]{$\mathbf{t}(W_k)$};
	\node at (-1,5.7)[font=\fontsize{7pt}{0}]{$\mathbf{t}(W_2)$};
	\node at (1,2.7)[font=\fontsize{7pt}{0}]{$\cdots$};
	\end{tikzpicture}.
\end{proof}

\begin{Remark}
 From the characterization in Lemma~\ref{lem:212} we deduce, by removing the largest entry in a generalized Stirling permutation, that the numbers $a_n$ of generalized Stirling permutation with degree $n$ satisfy the recurrence given by $a_0=1$ and
 $$ a_n = \sum_{k=1}^n (n-k+1) a_{n-k}.$$
From this we conclude that  $a_n = (n+1)a_{n-1}$ for $n\ge 1$ and hence $a_n=\frac{(n+1)!}{2}$.
\end{Remark}

Let $u$ be a $212$-avoiding packed word, we denote by $\mathbf{t}(u)$ the unique generalized Stirling permutation such that $\mathbf{w}(\mathbf{t}(u))=u$.

By forgetting the labels, we have a forgetful map $\pi_\ptree:\GSP\to\PT$ (this map agrees with $\pi$ of Proposition \ref{prop:pi-property}). In the other direction, the map $\iota_\gsp:\PT\to\GSP$ sends a tree $t$ to the unique $213$-avoiding generalized Stirling permutation $\iota_\gsp(t)$ such that $\pi_\ptree\circ\iota_\gsp(t)=t$. Indeed, the generalized Stirling permutation $\iota_\gsp(t)$ can be constructed by induction on the number of internal nodes as follows. The root of $t$ has label $1$. Suppose the root of $t$ has $k$ children, $x_1,\dots,x_k$. Let $s_i$ be the subtree of $t$ rooted at $x_i$. Then, we label $s_k$ by $\{2,\dots,\ideg(s_k)+1\}$ such that $s_k$ is $213$-avoiding. Such labeling exists by induction hypothesis. Similarly, we label $s_{k-1}$ by $\{\ideg(s_k)+2,\dots,\ideg(s_{k-1})+\ideg(s_{k})+1\}$ such that $s_{k-1}$ is $213$-avoiding. We keep doing this process to all $s_i$, and we obtain $\iota_\gsp(t)$.
If $t$ is a binary tree, then the $213$-avoiding permutation $\iota_\gsp(t)$ was denoted $\iota(t)$ in Section \ref{sec:ssym}.

Using a similar construction, there is a unique 312-avoiding generalized Stirling permutation $u$ such that $\pi_\ptree(u)=t$.

\subsection{Hopf algebra structure}

We define a Hopf structure on generalized Stirling permutations, and recall the Hopf structure on packed words \cite{NT06} dual to the Hopf algebra on set composition \cite{BZ09}. We then analyse their connections.

\begin{Definition}
	As a graded vector space, the Hopf algebra of generalized Stirling permutations, $\STSym$ is  $\displaystyle\bigoplus_{n\geq 0}\Bbbk\GSP_n$ where $\GSP_n$ is the set of generalized Stirling permutations of degree $n$. By convention $\GSP_0$ is the span of the empty generalized Stirling permutation. The fundamental basis of $\STSym$ is $\{F_u:u\in\GSP\}$.
\end{Definition}

The multiplication is given by shifted shuffle i.e. if $u$ and $v$ are generalized Stirling permutations and $u$ has $k$ leaves, then
$$\displaystyle m(F_u\otimes F_v)=\sum_{\substack{\text{allowable} \\ v\mapsto\left(v_1,\dots,v_k\right)}}F_{u\overleftarrow{\shuffle} \left(v_1,\dots,v_k\right)}.$$

The comultiplication is given by deconcatenation followed by standardization i.e.
$$\Delta(F_u)=\sum_{\substack{\text{allowable} \\ u\mapsto(u_1,u_2)}}F_{\std(u_1)}\otimes F_{\std(u_2)}.$$

\begin{Example}
	Let $u=$\begin{tikzpicture}[scale=0.25,baseline=0pt]
	\draw[blue, thick] (0,-1) -- (0,0);
	\draw[blue, thick] (0,0) -- (-2,2);
	\draw[blue, thick] (0,0) -- (2,2);
	\draw[blue, thick] (1,1) -- (0,2);
	\filldraw[black] (0,0) circle (5pt)  {};
	\filldraw[black] (1,1) circle (5pt)  {};
	\node at (0,0.7)[font=\fontsize{7pt}{0}]{$1$};
	\node at (1,1.7)[font=\fontsize{7pt}{0}]{$2$};
	\end{tikzpicture}, $v=$\begin{tikzpicture}[scale=0.25,baseline=0pt]
	\draw[red, thick] (0,-1) -- (0,0);
	\draw[red, thick] (0,0) -- (-2,2);
	\draw[red, thick] (0,0) -- (2,2);
	\draw[red, thick] (0,0) -- (0,2);
	\filldraw[black] (0,0) circle (5pt)  {};
	\node at (-0.4,1)[font=\fontsize{7pt}{0}]{$1$};
	\node at (0.4,1)[font=\fontsize{7pt}{0}]{$1$};
	\end{tikzpicture}.
	
	$m(F_u\otimes F_v)=
	F_{\begin{tikzpicture}[scale=0.2,baseline=0pt]
	\draw[blue, thick] (0,-1) -- (0,0);
	\draw[blue, thick] (0,0) -- (-4,4);
	\draw[blue, thick] (0,0) -- (2,2);
	\draw[blue, thick] (1,1) -- (-2,4);
	\draw[red, thick] (2,2) -- (0,4);
	\draw[red, thick] (2,2) -- (2,4);
	\draw[red, thick] (2,2) -- (4,4);
	\filldraw[black] (0,0) circle (5pt)  {};
	\filldraw[black] (1,1) circle (5pt)  {};
	\filldraw[black] (2,2) circle (5pt)  {};
	\node at (0,0.7)[font=\fontsize{7pt}{0}]{$1$};
	\node at (1,1.7)[font=\fontsize{7pt}{0}]{$2$};
	\node at (1.6,3)[font=\fontsize{7pt}{0}]{$3$};
	\node at (2.4,3)[font=\fontsize{7pt}{0}]{$3$};
	\end{tikzpicture}}+
	F_{\begin{tikzpicture}[scale=0.2,baseline=0pt]
	\draw[blue, thick] (0,-1) -- (0,0);
	\draw[blue, thick] (0,0) -- (-4,4);
	\draw[blue, thick] (0,0) -- (4,4);
	\draw[blue, thick] (1,1) -- (0,2);
	\draw[red, thick] (0,2) -- (-2,4);
	\draw[red, thick] (0,2) -- (0,4);
	\draw[red, thick] (0,2) -- (2,4);
	\filldraw[black] (0,0) circle (5pt)  {};
	\filldraw[black] (1,1) circle (5pt)  {};
	\filldraw[black] (0,2) circle (5pt)  {};
	\node at (0,0.7)[font=\fontsize{7pt}{0}]{$1$};
	\node at (1,1.7)[font=\fontsize{7pt}{0}]{$2$};
	\node at (-0.4,3)[font=\fontsize{7pt}{0}]{$3$};
	\node at (0.4,3)[font=\fontsize{7pt}{0}]{$3$};
	\end{tikzpicture}}+
	F_{\begin{tikzpicture}[scale=0.2,baseline=0pt]
	\draw[blue, thick] (0,-1) -- (0,0);
	\draw[blue, thick] (0,0) -- (-4,4);
	\draw[blue, thick] (0,0) -- (4,4);
	\draw[blue, thick] (3,3) -- (2,4);
	\draw[red, thick] (-2,2) -- (0,4);
	\draw[red, thick] (-2,2) -- (-2,4);
	\draw[red, thick] (-2,2) -- (-4,4);
	\filldraw[black] (0,0) circle (5pt)  {};
	\filldraw[black] (-2,2) circle (5pt)  {};
	\filldraw[black] (3,3) circle (5pt)  {};
	\node at (0,0.7)[font=\fontsize{7pt}{0}]{$1$};
	\node at (3,3.7)[font=\fontsize{7pt}{0}]{$2$};
	\node at (-1.6,3)[font=\fontsize{7pt}{0}]{$3$};
	\node at (-2.4,3)[font=\fontsize{7pt}{0}]{$3$};
	\end{tikzpicture}}$.
	
	Let $u=$\begin{tikzpicture}[scale=0.25,baseline=0pt]
	\draw[blue, thick] (0,-1) -- (0,0);
	\draw[blue, thick] (0,0) -- (7,7);
	\draw[blue, thick] (0,0) -- (-7,7);
	\draw[blue, thick] (-5,5) -- (-3,7);
	\draw[blue, thick] (-4,6) -- (-5,7);
	\draw[blue, thick] (0,0) -- (1,5);
	\draw[blue, thick] (1,5) -- (-1,7);
	\draw[blue, thick] (1,5) -- (1,7);
	\draw[blue, thick] (1,5) -- (3,7);
	\draw[blue, thick] (6,6) -- (5,7);
	\filldraw[black] (0,0) circle (5pt)  {};
	\filldraw[black] (1,5) circle (5pt)  {};
	\filldraw[black] (6,6) circle (5pt)  {};
	\filldraw[black] (-5,5) circle (5pt)  {};
	\filldraw[black] (-4,6) circle (5pt)  {};
	\node at (-0.4,1)[font=\fontsize{7pt}{0}]{$1$};
	\node at (0.5,1)[font=\fontsize{7pt}{0}]{$1$};
	\node at (0.6,6)[font=\fontsize{7pt}{0}]{$3$};
	\node at (1.4,6)[font=\fontsize{7pt}{0}]{$3$};
	\node at (6,6.7)[font=\fontsize{7pt}{0}]{$2$};
	\node at (-4,6.7)[font=\fontsize{7pt}{0}]{$5$};
	\node at (-5,5.7)[font=\fontsize{7pt}{0}]{$4$};
	\end{tikzpicture}.
	
	$\Delta_+(F_u)=
	F_{\begin{tikzpicture}[scale=0.25,baseline=0pt]
		\draw[blue, thick] (0,-1) -- (0,0);
		\draw[blue, thick] (0,0) -- (1,1);
		\draw[blue, thick] (0,0) -- (-1,1);
		\filldraw[black] (0,0) circle (5pt)  {};
		\node at (0,0.7)[font=\fontsize{7pt}{0}]{$1$};
	\end{tikzpicture}}\otimes
	F_{\begin{tikzpicture}[scale=0.25,baseline=0pt]
		\draw[blue, thick] (0,-1) -- (0,0);
		\draw[blue, thick] (0,0) -- (6,6);
		\draw[blue, thick] (0,0) -- (-6,6);
		\draw[blue, thick] (-5,5) -- (-4,6);
		\draw[blue, thick] (0,0) -- (0,4);
		\draw[blue, thick] (0,4) -- (0,6);
		\draw[blue, thick] (0,4) -- (-2,6);
		\draw[blue, thick] (0,4) -- (2,6);
		\draw[blue, thick] (5,5) -- (4,6);
		\filldraw[black] (0,0) circle (5pt)  {};
		\filldraw[black] (0,4) circle (5pt)  {};
		\filldraw[black] (5,5) circle (5pt)  {};
		\filldraw[black] (-5,5) circle (5pt)  {};
		\node at (-0.4,1)[font=\fontsize{7pt}{0}]{$1$};
		\node at (0.4,1)[font=\fontsize{7pt}{0}]{$1$};
		\node at (-0.4,5)[font=\fontsize{7pt}{0}]{$3$};
		\node at (0.4,5)[font=\fontsize{7pt}{0}]{$3$};
		\node at (5,5.7)[font=\fontsize{7pt}{0}]{$2$};
		\node at (-5,5.7)[font=\fontsize{7pt}{0}]{$4$};
	\end{tikzpicture}}+
	F_{\begin{tikzpicture}[scale=0.25,baseline=0pt]
		\draw[blue, thick] (0,-1) -- (0,0);
		\draw[blue, thick] (0,0) -- (-2,2);
		\draw[blue, thick] (0,0) -- (2,2);
		\draw[blue, thick] (1,1) -- (0,2);
		\filldraw[black] (0,0) circle (5pt)  {};
		\filldraw[black] (1,1) circle (5pt)  {};
		\node at (0,0.7)[font=\fontsize{7pt}{0}]{$1$};
		\node at (1,1.7)[font=\fontsize{7pt}{0}]{$2$};
	\end{tikzpicture}}\otimes
	F_{\begin{tikzpicture}[scale=0.25,baseline=0pt]
		\draw[blue, thick] (0,-1) -- (0,0);
		\draw[blue, thick] (0,0) -- (5,5);
		\draw[blue, thick] (0,0) -- (-5,5);
		\draw[blue, thick] (0,0) -- (-1,3);
		\draw[blue, thick] (-1,3) -- (-3,5);
		\draw[blue, thick] (-1,3) -- (-1,5);
		\draw[blue, thick] (-1,3) -- (1,5);
		\draw[blue, thick] (4,4) -- (3,5);
		\filldraw[black] (0,0) circle (5pt)  {};
		\filldraw[black] (-1,3) circle (5pt)  {};
		\filldraw[black] (4,4) circle (5pt)  {};
		\node at (-0.6,1)[font=\fontsize{7pt}{0}]{$1$};
		\node at (0.4,1)[font=\fontsize{7pt}{0}]{$1$};
		\node at (-1.4,4)[font=\fontsize{7pt}{0}]{$3$};
		\node at (-0.6,4)[font=\fontsize{7pt}{0}]{$3$};
		\node at (4,4.7)[font=\fontsize{7pt}{0}]{$2$};
	\end{tikzpicture}}+
	F_{\begin{tikzpicture}[scale=0.25,baseline=0pt]
		\draw[blue, thick] (0,-1) -- (0,0);
		\draw[blue, thick] (0,0) -- (6,6);
		\draw[blue, thick] (0,0) -- (-6,6);
		\draw[blue, thick] (-4,4) -- (-2,6);
		\draw[blue, thick] (-3,5) -- (-4,6);
		\draw[blue, thick] (0,0) -- (2,4);
		\draw[blue, thick] (2,4) -- (0,6);
		\draw[blue, thick] (2,4) -- (2,6);
		\draw[blue, thick] (2,4) -- (4,6);
		\filldraw[black] (0,0) circle (5pt)  {};
		\filldraw[black] (2,4) circle (5pt)  {};
		\filldraw[black] (-4,4) circle (5pt)  {};
		\filldraw[black] (-3,5) circle (5pt)  {};
		\node at (-0.2,1)[font=\fontsize{7pt}{0}]{$1$};
		\node at (1.5,2)[font=\fontsize{7pt}{0}]{$1$};
		\node at (1.6,5)[font=\fontsize{7pt}{0}]{$2$};
		\node at (2.4,5)[font=\fontsize{7pt}{0}]{$2$};
		\node at (-3,5.7)[font=\fontsize{7pt}{0}]{$4$};
		\node at (-4,4.7)[font=\fontsize{7pt}{0}]{$3$};
	\end{tikzpicture}}\otimes
	F_{\begin{tikzpicture}[scale=0.25,baseline=0pt]
		\draw[blue, thick] (0,-1) -- (0,0);
		\draw[blue, thick] (0,0) -- (1,1);
		\draw[blue, thick] (0,0) -- (-1,1);
		\filldraw[black] (0,0) circle (5pt)  {};
		\node at (0,0.7)[font=\fontsize{7pt}{0}]{$1$};
	\end{tikzpicture}}$.
\end{Example}

Under these operations, $\STSym$ forms a Hopf algebra. By forgetting the labels, we have a Hopf projection $\Pi_\ptree:\STSym\to\TSym$, $F_u\mapsto F_{\pi_\ptree(u)}$.

\begin{Proposition}
	$\STSym$ is a Hopf algebra.
\end{Proposition}
\begin{proof}
	The associativity, coassociativity and compatibility follow from the ones for $\TSym$. One only need to check that the labels are shifted and standardized correctly. The arguments in the proof of Proposition \ref{prop:tsym} can be transferred almost identically, by using allowable splitting, shifted shuffle and standardization in proper places.
\end{proof}

We then have the following commutative square of Hopf algebras.

\begin{center}
	\begin{tikzpicture}
		\node at (0,0) {$\YSym$};
		\node at (3,0) {$\TSym$};
		\node at (3,2) {$\STSym$};
		\node at (0,2) {$\SSym$};
		\draw[black,thick, right hook->] (0.8,0) -- (2.2,0);
		\draw[black,thick, right hook->] (0.8,2) -- (2.2,2);
		\draw[black,thick, ->>] (0,1.5) -- (0,0.5);
		\draw[black,thick, ->>] (3,1.5) -- (3,0.5);
	\end{tikzpicture}
\end{center}

\subsection{Connections to packed words}

We first recall that the dual Hopf algebra of packed words is $\NCQSym=\displaystyle\WQSym^*=\bigoplus_{n\geq 0}\Bbbk\PW_n^{-1}$ where $\PW_n$ is the set of packed words of length $n$. We use the $Q$ basis $\{Q_{u^{-1}}:u\in\PW\}$ as introduced in \cite{BZ09}.

For a word $w_1 \cdots w_n$, let $\pack(w_1 \cdots w_n)$ be the unique packed word $u_1\cdots u_n$ such that $u_i< u_j$ if and only if $w_i< w_j$. The product and coproduct formula for the $Q$ basis are as follows.

$$Q_{u^{-1}}\cdot Q_{v^{-1}}=\sum_{\substack{
	w=w_1\cdots w_{n+m},\\
	\pack(w_1\cdots w_n)=u,\\
	\pack(w_{n+1}\cdots w_{n+m})=v,\\
	\{w_1,\dots,w_n\}\cap\{w_{n+1},\dots,w_{n+m}\}=\emptyset}}Q_{w^{-1}},$$
where $n=\deg(u)$ and $m=\deg(v)$, and
$$\Delta(Q_{w^{-1}})=\sum_{i=0}^{n+m}\sum_{\substack{u_1<\cdots<u_n\\ v_1<\cdots<v_m\\ w_{u_j}\leq i,w_{v_j}>i}}Q_{\pack(w_{u_1}\cdots w_{u_n})^{-1}}\otimes Q_{\pack(w_{v_1}\cdots w_{v_m})^{-1}},$$
where $\deg(w)=n+m$.

\begin{Proposition}\label{prop:wqsym-embedding}
	The linear map $\STSym^*\to\NCQSym$ defined by $F_u^*\mapsto Q_{u^{-1}}$ is a Hopf embedding.
\end{Proposition}

\begin{proof}
	In terms of packed words, the product and coproduct of the $F$ basis are given by
	
	$$F_u\cdot F_v=\sum_{w\in (u\overline{\shuffle} v)\cap \GSP}F_w,$$
	where $u\overline{\shuffle} v$ denotes the shifted shuffle i.e. the set of shuffles of the two words $u_1\cdots u_{\deg(u)}$ and $(v_1+\deg(u))\cdots(v_{\deg(v)}+\deg(u))$.
	$$\Delta(F_w)=\sum_{\substack{0\leq i\leq n \\ \{w_1,\dots,w_i\}\cap\{w_{i+1},\dots,w_n\}=\emptyset}} F_{\pack(w_1\cdots w_i)}\otimes F_{\pack(w_{i+1}\cdots w_n)},$$
	where $n=\deg(w)$.
	
	The proof is completed by rewriting the formula above in the dual basis, and comparing to the formula of $Q$ basis.
\end{proof}

\begin{Remark}
	For a generalized Stirling permutation $u$, the underlying set partition of $u^{-1}$ is a noncrossing partition i.e. whenever $a<b<c<d$ and $a,c\in B_i$ for some $i$, we must not have $b,d\in B_j$ for some $j$.
	Hence the embedding  $\STSym^*\hookrightarrow \NCQSym $ can be further refined with the Hopf subalgebra of non-crossing set compositions inside the Hopf algebra of set compositions.
	Using the basis $Q$ of~\cite{BZ09} we have that the subspace $\NCQSym^{nc}$ spanned by $Q_B$ where the parts of $B$ are non-crossing is a Hopf subalgebra. Hence
	$$\STSym^*\hookrightarrow \NCQSym^{nc}\hookrightarrow \NCQSym .$$
		\end{Remark}

\begin{Problem} Taking the dual in the embedding of Proposition~\ref{prop:wqsym-embedding} gives us a projection $\WQSym\to\STSym$, given by 
\[Q_{u^{-1}}^* \mapsto 
\begin{cases}F_u & \text{if } u \in \GSP,\\ 
0 & \text{otherwise}.
\end{cases}\]
In light of Corollary~\ref{cor:STSym_selfdual}, it is reasonable to ask for finding an explicit Hopf embedding $\STSym\to\WQSym$.

\end{Problem}

\subsection{Bidendriform structure}\label{sec:bidendriform}
In this section we show that $\STSym$ is free, cofree and self-dual. We achieve these by giving $\STSym$ a bidendriform structure.

A Hopf algebra $(\cH,\cdot,\Delta)$ has a \emph{bidendriform} structure if it has the operations $\ll, \gg,\Delta_\ll,\Delta_\gg$ such that
\begin{enumerate}
	\item $a\cdot b=a\ll b+a\gg b$
	\item $(a\ll b)\ll c = a\ll(b\cdot c)$
	\item $a\gg (b\gg c)=(a\cdot b)\gg c$
	\item $a\gg(b\ll c) = (a\gg b)\ll c$
	\item $\Delta_+=\Delta_\ll+\Delta_\gg$
	\item $(\Delta_\ll \otimes id)\circ\Delta_\ll=(id\otimes\Delta_+)\circ\Delta_\ll$
	\item $(id \otimes \Delta_\gg)\circ\Delta_\gg=(\Delta_+\otimes id)\circ\Delta_\gg$
	\item $(\Delta_\gg \otimes id)\circ\Delta_\ll=(id\otimes\Delta_\ll)\circ\Delta_\gg$
	\item $\Delta_\gg(a\gg b)=a'b_\gg'\otimes a''\gg b_\gg''+a'\otimes a''\gg b+b_\gg'\otimes a\gg b_\gg''+ab_\gg'\otimes b_\gg''+a\otimes b$
	\item $\Delta_\gg(a\ll b)=a'b_\gg'\otimes a''\ll b_\gg''+a'\otimes a''\ll b+b_\gg'\otimes a\ll b_\gg''$
	\item $\Delta_\ll(a\gg b)=a'b_\ll'\otimes a''\gg b_\ll''+b_\ll'\otimes a\gg b_\ll''+ab_\ll'\otimes b_\ll''$
	\item $\Delta_\ll(a\ll b)=a'b_\ll'\otimes a''\ll b_\ll''+b_\ll'\otimes a\ll b_\ll''+a'b\otimes a''+b\otimes a$
\end{enumerate}
where the pairs $(x',x''),(x_\ll',x_\ll''),(x_\gg',x_\gg '')$ are all possible elements in $\Delta(x),\Delta_\ll(x),\Delta_\gg(x)$, as in Sweedler's notation.

We now present a bidendriform structure on $\STSym$, but leave the details for the interested reader.

\begin{enumerate}
	\item $\displaystyle F_u\ll F_v:=\sum_{w\in u\shuffle v,w_n\leq |u|}F_w$
	\item $\displaystyle F_u\gg F_v:=\sum_{w\in u\shuffle v,w_n> |u|}F_w$
	\item $\displaystyle\Delta_\ll(F_w):=\sum_{w\mapsto(u,v),\max(w)\in v}F_u\otimes F_v$
	\item $\displaystyle\Delta_\gg(F_w):=\sum_{w\mapsto(u,v),\max(w)\in u}F_u\otimes F_v.$
\end{enumerate}

\begin{Proposition}\label{prop:STSYMbidendriform}
The operations $\ll, \gg,\Delta_\ll,\Delta_\gg$ defined above puts a bidendriform structure on $\STSym$.
\end{Proposition}

The proof of this Proposition is technical and lengthy, yet straightforward.  It would take too much space to include here and we prefer to leave it for the interested reader.  The following corollary follows from the results of Foissy in \cite{F07,F12}.

\begin{Corollary}\label{cor:STSym_selfdual}
The Hopf algebra $\STSym$ is free, cofree and self-dual.
\end{Corollary}

\begin{Problem}
	Find an explicit Hopf isomorphism $\STSym\to\STSym^*$.
\end{Problem}

\subsection{Weak order on packed words}
Our partial order on generalized Stirling permutations follows from the (left) weak order of packed words, as defined in \cite{V15}.

For a packed word $w$ we denote by $w^{-1}(i)$ the $i$th component of $w^{-1}$. For example, if $w=1322122$, then $w^{-1}=15|3467|2$ and $w^{-1}(2)=\{3,4,6,7\}$.
We define the \textbf{inversed-inversion} set of a packed word $w$ to be 
 $$\iInv(w)=\{(a,b):a<b\text{ and }\min w^{-1}(a) > \max w^{-1}(b)\}. $$
In the example above $\iInv(1322122)=\{(2,3)\}$.
 
 For a packed word $w$ define $T_a(w)$ as the  packed word obtained from $w$ by interchanging all $a$'s and $a+1$'s, i.e. if $u=T_a(w)$  then
$$u_i=\begin{cases}
a & \text{if }w_i=a+1,\\
a+1 & \text{if }w_i=a,\\
w_i & \text{otherwise.}
\end{cases}$$

 The planar weak order $\leq_{Pw}$ on packed words is as follows: $u$ is covered by $w$ if and only if there exists $a$ such that $|\iInv(w)| = |\iInv(u)|+1$ and  $u=T_a(w)$.
 
By abuse of notations, we also use $\leq_{Pw}$ as the partial order on generalized Stirling permutations i.e. $u\leq_{Pw} v$ if and only if $\mathbf{w}(u)\leq_{Pw}\mathbf{w}(v)$.

The following proposition shows that, if $T_a(w)$ covers $w$, then all occurrences of $a$ in $w$ are to the left of all occurences of $a+1$.

\begin{Proposition}\label{prop:lpw} For a packed word $w$, $T_a(w)$ covers $w$ if and only if $\max w^{-1}(a)<\min w^{-1}(a+1)$.
In particular, if there exists $i<j<k$ with $w_i=w_k=a$, $w_j=a+1$, then $w$ and $T_a(w)$ are not comparable.
\end{Proposition}

\begin{proof}
	For any $b$ and $c$ satisfying $b<a$ and $a+1<c$  we have that,
	\begin{enumerate}
		\item $(b,a)\in\iInv(w)$ if and only if $(b,a+1)\in\iInv(T_a(w))$,
		\item $(b,a+1)\in\iInv(w)$ if and only if $(b,a)\in\iInv(T_a(w))$,
		\item $(a,c)\in\iInv(w)$ if and only if $(a+1,c)\in\iInv(T_a(w))$,
		\item $(a+1,c)\in\iInv(w)$ if and only if $(a,c)\in\iInv(T_a(w))$,
		\item $(b,c)\in\iInv(w)$ if and only if $(b,c)\in\iInv(T_a(w))$.
	\end{enumerate}
	Hence, $|\iInv(T_a(w))|=|\iInv(w)|+1$ if and only if $(a,a+1)\notin\iInv(w)$ and $(a,a+1)\in\iInv (T_a(w))$. By definition of $\iInv$, the second condition is equivalent to $\min  T_a(w)^{-1}(a)>\max T_a(w)^{-1}(a+1)$, i.e. $\min w^{-1}(a+1)>\max w^{-1}(a)$, which also ensures $(a,a+1)\notin\iInv(w)$.
	
	If there exists $i<j<k$ with $w_i=w_k=a$, $w_j=a+1$, then $(a,a+1)\notin\iInv(w)$ and $(a,a+1)\not\in\iInv T_a(w)$, hence $|\iInv(T_a(w))|=|\iInv(w)|$ so that $T_a(w)$ and $w$ are incomparable.
\end{proof}

The goal of the remainder of this section is Theorem \ref{thm:lpw}, describing the connected components under the weak order and  relating this order on packed words to the weak order of permutations.

\begin{Definition}
	For each packed word $w$, let $w^{-1} = (B_1,B_2,\ldots,B_k)$ with $B_i=w^{-1}(i)$. We define $\setpar(w)=\{B_1,B_2,\dots,B_k\}$ to be the underlying set partition of $w^{-1}$, and 
	$\delrpt(w)=\sigma$ to be the \emph{initial permutation}, obtained by taking only the first appearance of each letter in $w$. In other words, $\delrpt(w)=\sigma$ is the unique permutation $\sigma\colon[k]\to[k]$ such that $\min B_{\sigma(1)}<\min B_{\sigma(2)}<\cdots<\min B_{\sigma(k)}$. (Our notation follows \cite{G16}, note that \cite{V20} calls this the left-standardization.)
Given a set composition $B=(B_1,\dots,B_k)$ we denote by $w=B^{-1}$ the packed word such that $B=w^{-1}$.
\end{Definition}

\begin{Definition}
	Given a set partition  ${\mathcal B}=\{B_1,B_2,\dots,B_k\}$ such that $\min B_1<\cdots<\min B_k$ and a permutation $\sigma\colon[k]\to[k]$,
	we define $\pw({\mathcal B},\sigma) = (B_{\sigma^{-1}(1)},B_{\sigma^{-1}(2)},\dots,B_{\sigma^{-1}(k)})^{-1}$. For $w=\pw({\mathcal B},\sigma)$ we
	have $\setpar(w)={\mathcal B}$ and $\delrpt(w)=\sigma$.
\end{Definition}

\begin{figure}
	\begin{center}
		\begin{tikzpicture}[scale=0.8]
		\node at (0,0) {\begin{tikzpicture}[scale=0.48]
			\draw[ForestGreen, ultra thick] (0,1.5) -- (0,3.5);
			\draw[ForestGreen, ultra thick] (0,6.5) -- (0,8.5);
			
			\node () at (0,0) {$1213$};
			\node () at (0,5) {$1312$};
			\node () at (0,10) {$2321$};
			\end{tikzpicture}};
		\node at (5,0) {\begin{tikzpicture}[scale=0.4]
			\draw[ForestGreen, ultra thick] (5,1.7) -- (5,3);
			\draw[ForestGreen, ultra thick] (5,6.7) -- (5,8);
			
			\node () at (5,0) {\begin{tikzpicture}[scale=0.2,baseline=0pt]
				\draw[blue, thick] (0,-1) -- (0,0);
				\draw[blue, thick] (0,0) -- (4,4);
				\draw[blue, thick] (3,3) -- (2,4);
				\draw[blue, thick] (0,0) -- (-1,3);
				\draw[blue, thick] (-1,3) -- (0,4);
				\draw[blue, thick] (-1,3) -- (-2,4);
				\draw[blue, thick] (0,0) -- (-4,4);
				\filldraw[black] (0,0) circle (5pt)  {};
				\filldraw[black] (-1,3) circle (5pt)  {};
				\filldraw[black] (3,3) circle (5pt)  {};
				\node at (-0.6,1)[font=\fontsize{7pt}{0}]{$1$};
				\node at (0.4,1)[font=\fontsize{7pt}{0}]{$1$};
				\node at (3,3.7)[font=\fontsize{7pt}{0}]{$3$};
				\node at (-1,3.7)[font=\fontsize{7pt}{0}]{$2$};
				\end{tikzpicture}};
			\node () at (5,5) {\begin{tikzpicture}[scale=0.2,baseline=0pt]
				\draw[blue, thick] (0,-1) -- (0,0);
				\draw[blue, thick] (0,0) -- (4,4);
				\draw[blue, thick] (3,3) -- (2,4);
				\draw[blue, thick] (0,0) -- (-1,3);
				\draw[blue, thick] (-1,3) -- (0,4);
				\draw[blue, thick] (-1,3) -- (-2,4);
				\draw[blue, thick] (0,0) -- (-4,4);
				\filldraw[black] (0,0) circle (5pt)  {};
				\filldraw[black] (-1,3) circle (5pt)  {};
				\filldraw[black] (3,3) circle (5pt)  {};
				\node at (-0.6,1)[font=\fontsize{7pt}{0}]{$1$};
				\node at (0.4,1)[font=\fontsize{7pt}{0}]{$1$};
				\node at (3,3.7)[font=\fontsize{7pt}{0}]{$2$};
				\node at (-1,3.7)[font=\fontsize{7pt}{0}]{$3$};
				\end{tikzpicture}};
			\node () at (5,10) {\begin{tikzpicture}[scale=0.2,baseline=0pt]
				\draw[blue, thick] (0,-1) -- (0,0);
				\draw[blue, thick] (0,0) -- (4,4);
				\draw[blue, thick] (-1,1) -- (2,4);
				\draw[blue, thick] (-1,1) -- (-1,3);
				\draw[blue, thick] (-1,3) -- (0,4);
				\draw[blue, thick] (-1,3) -- (-2,4);
				\draw[blue, thick] (0,0) -- (-4,4);
				\filldraw[black] (0,0) circle (5pt)  {};
				\filldraw[black] (-1,3) circle (5pt)  {};
				\filldraw[black] (-1,1) circle (5pt)  {};
				\node at (-1.5,2)[font=\fontsize{7pt}{0}]{$2$};
				\node at (-0.5,2)[font=\fontsize{7pt}{0}]{$2$};
				\node at (0,0.7)[font=\fontsize{7pt}{0}]{$1$};
				\node at (-1,3.7)[font=\fontsize{7pt}{0}]{$3$};
				\end{tikzpicture}};
	\end{tikzpicture}};
\node at (10,0) {\begin{tikzpicture}[scale=0.48]
	\draw[ForestGreen, ultra thick] (0,1.5) -- (0,3.5);
	\draw[ForestGreen, ultra thick] (0,6.5) -- (0,8.5);
	
	\node () at (0,0) {$123$};
	\node () at (0,5) {$132$};
	\node () at (0,10) {$231$};
	\end{tikzpicture}};
\end{tikzpicture}
\end{center}
\caption{\sl Connected components of packed words and generalized Stirling permutations with underlying set partition $\{\{1,3\},\{2\},\{4\}\}$, and their image under $\delrpt$.}
\end{figure}
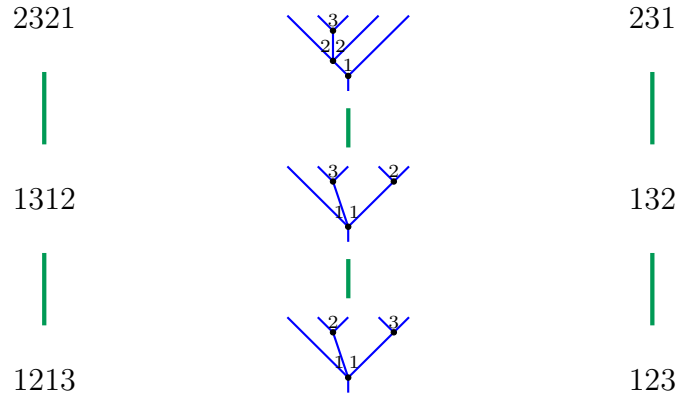

Since $\setpar(T_a(w))=\setpar(w)$, we have that $u$ and $w$ are incomparable if $\setpar(u)\neq\setpar(w)$. Hence, in each connected component, all elements have the same underlying set partition. But the converse is not true e.g. 121 and $212$ are not comparable, by Proposition \ref{prop:lpw}. The correct statement as it will be seen in Theorem~\ref{thm:lpw} is: two packed words are in the same connected component if and only if they have the same underlying set partition and the same nested inversions. 

\begin{Definition}
	Let $w$ be a packed word, $\delrpt(w)=\sigma$ and $\setpar(w)=\{B_1,\dots,B_k\}$ such that $\min B_1<\cdots<\min B_k$. Let $b_i=\min B_i$ for each $i$. We define the \textbf{inversion} set of $w$ to be $\Inv(w)=\Inv(\sigma)=\{(i,j):i<j\text{ and }w_{b_i}>w_{b_j}\}$. In particular, when $w$ is a permutation, its inversion set is consistent with the inversion set as a permutation.
	
	For $\setpar(w)=\{B_1,\dots,B_k\}$, two blocks $B_i,B_j$ are \textbf{nested} if $\min B_i<\min B_j<\max B_i$.
\end{Definition}

\begin{Lemma}\label{lem:iinv2}
	If $u$ and $w$ are packed words, then $u\leq_{Pw}w$ implies $\Inv(u)\subseteq\Inv(w)$.
\end{Lemma}

\begin{proof}
	Let $\setpar(u)=\{B_1,\dots,B_k\}$ such that $\min B_1<\cdots<\min B_k$. If $u<_{Pw} T_a(u)$ for some $a$, then $a\in B_i$, $a+1\in B_j$ for some $i<j$ and $\max B_i<\min B_j$. Then, $\Inv(T_a(u))=\Inv(u)\cup\{(i,j)\}$.
	The transitive closure of the relation implies the statement.
\end{proof}

\begin{Lemma}\label{lem:iinv}
	If $u$ and $w$ are packed words such that
	\begin{enumerate}
		\item $\setpar(u)=\setpar(w)=\{B_1,\dots,B_k\}$,
		\item If $B_i,B_j$ are nested, then either $(i,j)\in\Inv(u)\cap\Inv(w)$ or $(i,j)\notin\Inv(u)\cup\Inv(w)$.
	\end{enumerate}
	Then, $\Inv(u)\subseteq\Inv(w)$ implies $u\leq_{Pw}w$.
\end{Lemma}

\begin{proof} Assume $\min B_1 <\cdots< \min B_k$.

Let $\tau=\delrpt(w)$. If we have $\Inv(\sigma)=\Inv(u)\subseteq \Inv(w)=\Inv(\tau)$, then $\sigma$ is less than $\tau$ in the weak order. We proceed by induction on $m=\ell(\tau)-\ell(\sigma)$. If $m=0$ the result is obvious.
If $m>0$, then we can find $a$ such that $\sigma \lessdot (a,a+1)\sigma\leq \tau$ in the weak order and hence $$\Inv(\sigma)\subset  \Inv((a,a+1)\sigma)=\Inv(\sigma)\cup\{(\sigma^{-1}(a),\sigma^{-1}(a+1))\}\subseteq \tau.$$
Let $(i,j)=(\sigma^{-1}(a),\sigma^{-1}(a+1))$. The blocks $B_i$ and $B_j$ are not nested according to the assumption (2).
Since $i<j$, we have $\min B_i <\min B_j$, and we must have $\min B_i \leq \max B_i <\min B_j$, otherewise they would be nested.
Hence, using Proposition~\ref{prop:lpw}, we have that $u\lessdot_{Pw} T_a(u)$, and we can conlude from the induction hypothesis. 
\end{proof}

\begin{Proposition}\label{prop:delrpt-setpar}
	If $u$ and $w$ are packed words such that condition (2) in Lemma \ref{lem:iinv} does not hold, then $u$ and $w$ are incomparable.
\end{Proposition}

\begin{proof} 	 By assumption, there is some $B_i,B_j$ nested, such that $(i,j)\in\Inv(u)$ and $(i,j)\notin\Inv(w)$ (by exchanging the roles of $u$ and $w$ if necessary). By Lemma \ref{lem:iinv2}, $u \not \leq _{Pw} w$ because $\Inv(u)\nsubseteq\Inv(w)$. Since $B_i,B_j$ are nested, we can find $p<q<r$ such that $p,r\in B_i$ and $q\in B_j$. Then we have that $u_p=u_r>u_q$ and $w_p=w_r<w_q$. We claim that for all $w\leq_{Pw}v$, we must have $v_p=v_r<v_q$. It suffices to check for $v=T_a(w)$.
	
	Suppose not, i.e. $v_p=v_r>v_q$. Since $v=T_a(w)$, the only possibility is that $w_p=w_r=v_q=a$ and $w_q=v_p=v_r=a+1$. But then by Proposition \ref{prop:lpw}, $v$ and $w$ are incomparable. Therefore, $w$ is not less than $u$ either.
\end{proof}

\begin{Theorem}\label{thm:lpw}
	Fix a set partition $\{B_1,\dots,B_k\}$ with $\min B_1<\cdots<\min B_k$. Let $X=\{(i,j):i<j,B_i,B_j\text{ are nested}\}$. Then each $A\subseteq X$ induces a connected component (possibly empty) of the planar weak order consisting of packed words $w$ such that $A\subseteq\Inv(w)$, $(X\setminus A)\cap\Inv(w)=\emptyset$ and $\setpar(w)=\{B_1,\dots,B_k\}$.
	
	Hence each connected component in the planar weak order is a bounded lattice, and is isomorphic to an interval of the weak order. In particular,
	\begin{enumerate}
		\item $u\vee w=\pw(\setpar(u),\delrpt(u)\vee\delrpt(w))$,
		\item $u\land w=\pw(\setpar(u),\delrpt(u)\land\delrpt(w))$.
	\end{enumerate}
\end{Theorem}

\begin{proof}
	We only need to check that $u\vee w$ and $u\land w$ are in the connected component of $u$ and $w$. The rest follows from Propositions \ref{prop:lpw} and \ref{prop:delrpt-setpar} and the observation that $\Inv(\delrpt(w))=\Inv(w)$. A connected component $W$ in the planar weak order corrsponds to the interval $\{\delrpt(w):w\in W\}$ in the weak order.
	
	We show the case of $u\vee w$, and omit the case of $u\land w$. Suppose not, that means we can find $B_i,B_j\in\setpar(u)=\{B_1,\dots,B_k\}$ nested such that
	\begin{enumerate}
		\item $(i,j)\in \Inv(u)\cap\Inv(w)$ but $(i,j)\notin\Inv(u\vee w)$, or
		\item $(i,j)\notin \Inv(u)\cup\Inv(w)$ but $(i,j)\in\Inv(u\vee w)$.
	\end{enumerate}
	Case (1) cannot be true because $\Inv(u) = \Inv(\delrpt u) \subseteq \Inv( (\delrpt u) \vee (\delrpt w) )=\Inv(u\vee w)$. In case (2), recall the well-known fact that $\Inv(u\vee w)$ equals the transitive closure of $\Inv(u)\cup\Inv(w)$. Hence, we can find $i<h<j$ such that $(i,h)$ and $(h,j)$ are each in either $\Inv(u)$ or $\Inv(w)$ (or both). Without loss of generality, we can assume $(i,h)\in\Inv(u)$. If $(h,j)\in \Inv(u)$ also, then by transitivity $(i,j)\in \Inv(u)$, contradicting (2). So we must have $(h,j)\in \Inv(w)$, and by the same transitivity argument we must have $(i,h)\notin\Inv(w)$. Now since $B_i,B_j$ are nested, $\min B_i<\min B_{h}<\min B_j<\max B_i$. That means $B_i$ and $B_{h}$ are also nested. But then by Proposition \ref{prop:delrpt-setpar}, $u$ and $w$ are not comparable, which is a contradiction.
\end{proof}

\subsection{Planar weak order on generalized Stirling permutations}

Recall from the end of Section \ref{sec:gsp} the maps $\pi_\ptree:\GSP\to\PT$ that forgets the labels, and $\iota_\gsp:\PT\to\GSP$ sending a planar tree $t$ to the unique $213$-avoiding generalized Stirling permutation with underlying planar tree $t$.

\begin{Lemma}\label{lem:iota-max}
The generalized Stirling permutation $\iota_\gsp(t)$ is the maximal preimage of $t$ under $\pi_\ptree$.
\end{Lemma}

\begin{proof}
If $\pi_\ptree(u)=t$ and $u$ is not $213$-avoiding, we claim that we can find $a$ and $i<m<j$ such that $u_i=a$, $u_j=a+1$ and $u_m<a$. We can find $i<m<j$ such that $u_j>u_i>u_m$. We search iteratively on $u_j-u_i$. If $u_j-u_i>1$, then we can find $u_n=u_i+1$ for some $n$. If $n>m$, then the triple $(i,m,n)$ is what we want. If $n<m$, then the triple $(n,m,j)$ satisfies $u_j>u_n>u_m$ and $u_j-u_n<u_j-u_i$. The claim will be fulfilled by keep doing this process.

In this case, the node with label $a+1$ is not on top of the node with label $a$ i.e. $\pi_\ptree(T_a(u))=t$. Since $T_a(u)\geq_{Pw} u$, the statement follows.
\end{proof}

\begin{Proposition}\label{prop:lsb}
	Let $u,v$ be generalized Stirling permutations and $s,t$ be trees. Then
\begin{enumerate}
	\item if $u\leq_{Pw}v$, then $\pi_\ptree(u)\leq_{PT}\pi_\ptree(v)$;
	\item if $s\leq_{PT}t$, then $\iota_\gsp(s)\leq_{Pw}\iota_\gsp(t)$.
\end{enumerate}
\end{Proposition}

\begin{proof}
We begin with (1). Let $w$ be a $212$-avoiding packed word, and $u$ covered by $w$ in the planar weak order. It suffices to show that $\pi_\ptree(\mathbf{t}(u))\leq_{PT}\pi_\ptree(\mathbf{t}(w))$. Let $u=T_a(w)$, $x=\min\{i:w_i=a\}$, $y=\max\{i:w_i=a+1\}$. We must have $y<x$ because otherwise, $w$ and $u$ are not comparable by Proposition \ref{prop:lpw}. Then we have two cases.

Case 1. There exists $y<m<x$ such that $w_m<a$. In this case, the node with label $a+1$ is not on top of the node with label $a$. Then $u$ and $w$ have the same underlying tree.

Case 2. For any $y<m<x$, $w_m>a+1$. In this case, the node with label $a+1$ is on the left most leaf of the node with label $a$. Then, $\pi_\ptree(u)$ is covered by $\pi_\ptree(w)$ in the planar Tamari order, as shown in Definition \ref{def:lpt}.

To prove (2), let $s,t$ be trees and $s$ covered by $t$ in the planar Tamari order. Suppose $t$ is obtained from $s$ by a left rotation at node $x$, and let $y$ be the right most child of $x$ in $s$ i.e. $y$ is the left most child of $x$ in $t$. Let $\kappa_{\iota_{\gsp}(s)}(x)=a$, by construction, we must have $\kappa_{\iota_{\gsp}(s)}(y)=a+1$ because $\iota_{\gsp}(s)$ is $213$-avoiding. Let $u$ be the labeled tree such that $\pi_\ptree(u)=t$ and $\kappa_u(x)=a$, $\kappa_u(y)=a+1$ and $\kappa_u(i)=\kappa_{\iota_{\gsp}(s)}(i)$ for all internal nodes $i$. Then, $u$ is well defined and $u=T_a(\iota_{\gsp}(s))$.

By Lemma \ref{lem:iota-max}, we know that $\iota_{\gsp}(t)$ is maximal. Therefore, $\iota_{\gsp}(s)\leq_{Pw}u\leq_{Pw}\iota_{\gsp}(t)$.
\end{proof}

From Lemma \ref{lem:iota-max} and Proposition \ref{prop:lsb} we obtain the following corollary.
\begin{Corollary}\label{cor:Galois_connection}
The maps $\pi_{ptree}$ and $\iota_{gsp}$ form a   Galois connection between the weak order on Stirling permutations and the Tamari order on planar trees. 
\end{Corollary}
\begin{proof}
Assume $w \leq_{Pw} \iota_{gsp}(t)$. Then $\pi_{ptree}(w)\leq_{PT} \pi_{ptree}(\iota_{gsp} t)= t$, because $\pi_{ptree}$ is order-preserving.

Conversely, assume $\pi_{ptree}(w) \leq_{PT} t$. Then $w\leq_{Pw} \iota_{gsp}(\pi_{ptree} w)\leq_{Pw} \iota_{gsp}(t)$ - the first inequality is by Lemma \ref{lem:iota-max}, and the second because $\iota_{gsp}$ is order-preserving.
\end{proof}

\begin{Proposition}\label{prop:interval_multiplication_STSYM}

Let $F_u$ and $F_v$ be two basis elements in $\STSym$, and let $F_u^*$, $F_v^*$ be their dual basis elements in the graded dual $\STSym^*$. Then,
$$m(F_u^*\otimes F_v^*)=\sum_{u\backslash v\leq_{Pw} w\leq_{Pw} u/v}F_w^*.$$
\end{Proposition}

\begin{proof}
Recall from Proposition \ref{prop:wqsym-embedding} that
$$F_u^*\cdot F_v^*=\sum_{\substack{
		w=w_1\cdots w_{n+m},\\
		\pack(w_1\cdots w_n)=u,\\
		\pack(w_{n+1}\cdots w_{n+m})=v,\\
		\{w_1,\dots,w_n\}\cap\{w_{n+1},\dots,w_{n+m}\}=\emptyset}}F_w^*,$$
where $n=\deg(u)$ and $m=\deg(v)$.

For each $w$ in the summand above, clearly $\setpar(w)=\setpar(u\backslash v)=\setpar(u/v)$. Hence $w$, $u\backslash v$ and $u/v$ must be in the same connected component by Theorem \ref{thm:lpw}. 

Moreover, since $\pack(w_1\cdots w_n)=u$, for $i<j\leq n$, $(i,j)\in\Inv(w)$ if and only if $(i,j)\in\Inv(u)\cap\Inv(u\backslash v)\cap\Inv(u/v)$. Similarly, for $n<i<j$, $(i,j)\in\Inv(w)$ if and only if $(i-n,j-n)\in\Inv(v)$ if and only if $(i,j)\in\Inv(u\backslash v)\cap\Inv(u/v)$.

For $i\leq n<j$, $(i,j)\notin\Inv(u\backslash v)$ and $(i,j)\in\Inv(u/v)$ whenever $i=\min B_a$, $j=\min B_b$ for some $B_a, B_b\in\setpar(w)$. Therefore, $\Inv(u\backslash v)\subseteq\Inv(w)\subseteq\Inv(u/v)$. By Lemma \ref{lem:iinv}, we must have $u\backslash v\leq_{Pw} w\leq_{Pw} u/v$.
\end{proof}

\begin{Corollary}\label{cor:lst}
Let $F_s$ and $F_t$ be two basis elements in $\TSym$, and let $F_s^*$, $F_t^*$ be their dual basis elements in the graded dual $\TSym^*$. Then,
$$m(F_s^*\otimes F_t^*)=\sum_{s\backslash t\leq_{PT} r\leq_{PT} s/t}F_r^*.$$
\end{Corollary}

\begin{proof}
Fix a tree $r$ and let $w=\iota_\gsp(r)$. Since $\Pi_\ptree:\STSym\to\TSym$ is a Hopf morphism, by Proposition \ref{prop:interval_multiplication_STSYM}, we have
$$\Delta(F_r)=(\Pi_\ptree\otimes\Pi_\ptree)\circ\Delta(F_w)=\sum_{u\backslash v\leq_{Pw} w\leq_{Pw} u/v}F_{\pi_\ptree(u)}\otimes F_{\pi_\ptree(v)}.$$

Fix a pair $u,v$ such that $u\backslash v\leq_{Pw} w\leq_{Pw} u/v$, let $\pi_\ptree(u)=s$ and $\pi_{\ptree}(v)=t$. By Propositions \ref{prop:lsb} and \ref{prop:pi-property}, we have $s\backslash t\leq_{PT} r\leq_{PT}s/t$.

Conversely, fix a pair $s,t$ such that $s\backslash t\leq_{PT} r\leq_{Pw}s/t$, let $\iota_\gsp(s)=u$ and $\iota_{\gsp}(t)=v$. By Propositions \ref{prop:lsb} and \ref{prop:pi-property}, we have $u\backslash v\leq_{Pw} w\leq_{Pw}u/v$.

Finally, from the definition of comultiplication, we have $\langle F_u^*\otimes F_v^*,\Delta(F_w)\rangle=0$ or $1$ and $\langle F_s^*\otimes F_t^*,\Delta(F_r)\rangle=0$ or $1$. Therefore,
$$\Delta(F_r)=\sum_{u\backslash v\leq_{Pw} w\leq_{Pw} u/v}F_{\pi_\ptree(u)}\otimes F_{\pi_\ptree(v)}=\sum_{s\backslash t\leq_{PT} r\leq_{PT} s/t}F_s\otimes F_t.$$
\end{proof}

\subsection{Monomial basis}

The planar weak order allows the definition of a monomial basis for $\STSym$, exactly analogous to our previous algebras.

\begin{Definition}
The monomial basis of $\STSym$ are defined as the elements satisfying $\displaystyle F_u=\sum_{v\geq_{Pw} u}M_v$. The monomial basis is unique by triangularity and it forms a basis via M\"{o}bius inversion.
\end{Definition}

Note that, because of Theorem \ref{thm:monomialquotient}, the image
of monomial basis elements under the Hopf projection $\Pi_{\ptree}$
are either monomials or zero:
\[
\Pi(M_{u})=\begin{cases}
M_{\pi_{\ptree}(u)} & \text{if }u\text{ is $213$-avoiding};\\
0 & \text{otherwise}.
\end{cases}
\]

The rest of this section gives the usual formulas for the coproduct, product and antipode of the monomial basis.

\begin{Definition}
Let $u$ be a generalized Stirling permutation of degree $n$. Its global descent set is $\GD(u)=\{i\in[n-1]:u=v/w \text{ for some }v, w\text{ and }\deg(v)=i\}$. Equivalently, $\GD(u)=\{i\in[n-1]:u_a>u_b\text{ for all }a\leq i< b\}$.
\end{Definition}

The following two lemmas are useful for checking the axioms of Section \ref{sec:axioms}:

\begin{Lemma}\label{lem:iota-gd}
	The map $\iota_\gsp$ preserves global descents i.e. $\GD(t)=\GD(\iota_\gsp(t))$ for a tree $t$.
\end{Lemma}

\begin{proof}
Let $u=\iota_\gsp(t)$. $\GD(\iota_\gsp(t))\subseteq\GD(t)$ follows directly from the definitions of global descents. If $i\in\GD(t)$, since labels in generalized Stirling permutations are increasing along branches, for any $a\leq i$, we have $u_a>u_{i+1}$. Then, $213$-avoiding property implies that for any $a\leq i<b$, we must have $u_a>u_b$ i.e. $i\in\GD(u)$.
\end{proof}

\begin{Lemma}\label{lem:delrptshuffle}Given generalized Stirling permutations $u,v$ and a
fixed allowable splitting $v\mapsto(v_{1},\dots,v_{k})$, then $\delrpt(u\overleftarrow{\shuffle}(v_{1},\dots,v_{k}))=\delrpt(u)\overleftarrow{\shuffle}(\sigma_{1},\dots,\sigma_{l})$
where $\sigma=\delrpt(v)$ and $\sigma\mapsto(\sigma_{1},\dots,\sigma_{l})$
is some splitting where $\deg\sigma_{1},\dots,\deg\sigma_{l}$ depend
only on $\deg v_{1},\dots,\deg v_{k}$ and on $\setpar(u)$, $\setpar(v)$.
\end{Lemma}
\begin{proof}
If $u=u_{1}\cdots u_{k-1}$, then, as words, $u\overleftarrow{\shuffle}(v_{1},\dots,v_{k})$
is $v_{1}^{\uparrow}u_{1}v_{2}^{\uparrow}u_{2}\cdots v_{k-1}^{\uparrow}u_{k-1}v_{k}^{\uparrow}$,
where $v_{i}^{\uparrow}$ denotes that each number in $v_{i}$ is
shifted up by $\max u$. (Note that $u_{i}$ is a letter but $v_{i}$
is a word.) Since $v\mapsto(v_{1},\dots,v_{k})$ is allowable, the
letters $v_{i}$ and $v_j$ are distinct when $i\ne j$. Moreover, because of the shifting,
the letters in $v_{1}^{\uparrow},\dots,v_{k}^{\uparrow}$ are also
different from $u_{1},\dots,u_{k-1}$. Let $\setpar(u)=\{B_1,\ldots,B_\ell\}$ and let
$i_j=\min B_j$ assuming that $1=i_1<i_2<\cdots<i_\ell$.
In particular $\delrpt(u)=u_1u_{i_{2}}\cdots u_{i_{\ell}}$ for some $\ell\le k-1$.
We proceed in the same way with $v=v_1v_2\cdots v_k$ to get $\delrpt(v)=\sigma=\tau_1\cdots\tau_k$ where $\tau_j$ (a word, possibly empty) 
is the first occurrence of the distinct letters of $v$ that happen in the word $v_j$. We then have
\begin{align*}
\delrpt(u\overleftarrow{\shuffle}(v_{1},\dots,v_{k})) & =( \tau_{1}^{\uparrow})u_{1}(\tau_{2}^{\uparrow}\cdots \tau_{i_{2}}^{\uparrow})u_{i_{2}}(\tau_{i_{2}+1}^{\uparrow}\cdots \tau_{i_{3}}^{\uparrow})u_{i_{3}}\cdots
   u_{i_{\ell}}(\tau_{i_{\ell}+1}^{\uparrow}\cdots \tau_{k}^{\uparrow})\\
 & =\delrpt(u)\overleftarrow{\shuffle}(\tau_{1},\,\tau_{2}\cdots \tau_{i_{2}},\ \ldots\ ,\ \tau_{i_{\ell}+1}\cdots \tau_{k})
\end{align*}
If we let $\sigma_{1}= \tau_{1}$ and $\sigma_{j}=\tau_{i_{j-1}+1}\cdots \tau_{i_{j}}$, then
 $\sigma\mapsto(\sigma_{1},\dots,\sigma_{l})$ is a valid splitting and 
$\delrpt(u\overleftarrow{\shuffle}(v_{1},\dots,v_{k}))=\delrpt(u)\overleftarrow{\shuffle}(\sigma_{1},\dots,\sigma_{l})$ as desired.
The $\deg \sigma_j$ is the number of letters in $\tau_{i_{j-1}+1}\cdots \tau_{i_{j}}$. This depends on the $i_j$'s (built from $\setpar(u)$),
 the corresponding indexes in $v$ (built from $\setpar(v)$) and where they fall in the decomposition $v=v_1v_2\cdots v_k$
 (which depend on on the size of $v_i$, namely $\deg(v_i)$).
\end{proof}

\begin{Proposition}\label{prop:coproduct-monomial-stsym} The coproduct
in the monomial basis of $\STSym$ is:
\[
\Delta_{+}(M_{u})=\sum_{i\in\GD(u)}M_{{}^{i}u}\otimes M_{u^{i}}.
\]
\end{Proposition}
\begin{proof}
We check the axioms in Theorem \ref{thm:monomialcoproduct}. To see
($\Delta$1), note that $i\not\in\allow(u)$ if and only there is
$j,k$ with $j\leq i<k$ and $u_{j}=u_{k}$. Then, in $u'=T_{a}(u)$
covering $u$, we have $u'_{j}=u'_{k}$, so $i\not\in\allow(u')$
either. 

To check ($\Delta$2): consider $u'=T_{a}(u)$ covering $u$, and $i\in\allow(u)$,
so either $\max u^{-1}(a) \le i$ or $\min u^{-1}(a) >i$, 
and similarly for $a+1$. Thus there are four scenarios:
\begin{itemize}
\item $\max u^{-1}(a)\le i$ and $\max u^{-1}(a+1)\leq i$: then
${}^{i}(T_{a}u)=T_{b}({}^{i}u)$ for some $b$ (being the image of
$a$ under $\std({}^{i}u)$) and $(T_{a}u)^{i}=u^{i}$; 
\item $\max u^{-1}(a)>i$ and $\max u^{-1}(a+1)>i$: then ${}^{i}(T_{a}u)={}^{i}u$
and $(T_{a}u)^{i}=T_{b}(u^{i})$ for some $b$ (being the image of
$a$ under $\std(u^{i})$) ; 
\item $\max u^{-1}(a)\leq i$ and $\max u^{-1}(a+1)>i$: then ${}^{i}(T_{a}u)={}^{i}u$
and $(T_{a}u)^{i}=u^{i}$; 
\item $\max u^{-1}(a)>i$ and $\max u^{-1}(a+1)\leq i$:
this case is not possible by Proposition \ref{prop:lpw}. 
\end{itemize}
So, in all cases, ${}^{i}(T_{a}u)\geq_{Pw}{}^{i}u$ and $(T_{a}u)^{i}\geq_{Pw}u^{i}$.

To check ($\Delta$3): $u/v=\max\{w|{}^{i}w=u,w^{i}=v\}$
follows from Proposition \ref{prop:interval_multiplication_STSYM}.
And $/$ is order-preserving
because $(T_{a}u)/v=T_{a}(u/v)$, and $u/(T_{a}v)=T_{a+\max u}(u/v)$.
\end{proof}

\begin{Proposition}\label{prop:product-monomial-stsym} The product
of monomial basis elements in $\STSym$ is given by $M_{u}\cdot M_{v}=\sum_{w}\alpha_{u,v}^{w} M_{w}$,
where $\alpha_{u,v}^{w} =|A_{u,v}^{w}|$ is defined as in (\ref{eq:a-def}).
\end{Proposition}
\begin{proof}
We check the axioms in Theorem \ref{thm:monomialproduct}. Axiom ($m$0)
is true because of Theorem \ref{thm:lpw}. To see ($m$1): the set $Sh(\setpar(u),\setpar(v))$ correspond to the set
of sequences $0\le i_1\le i_2\le\cdots\le i_ {\deg u-1}\le \deg v$ of allowable split positions of
$v$. These allowable split positions depend only on $\setpar(v)=\{B_1,\ldots,B_k\}$,
as $i_j$ is not allowed if and only if $\min B_p < i_j < \max B_p$ for some $p$.

To see ($m$2): suppose $\deg u=k-1$ and $v\mapsto(v_{1},\dots,v_{k})$
is an allowable splitting. If $T_{a}(u)$ covers $u$, then $(T_{a}u)\overleftarrow{\shuffle}(v_{1},\dots,v_{k})=T_{a}(u\overleftarrow{\shuffle}(v_{1},\dots,v_{k}))$.
If $T_{a}(v)$ covers $v$, then $T_{a}(v)\mapsto(T_{a}(v_{1}),\dots,T_{a}(v_{k}))$
and $u\overleftarrow{\shuffle}(T_{a}(v_{1}),\dots,T_{a}(v_{k}))=T_{a+\max u}(u\overleftarrow{\shuffle}(v_{1},\dots,v_{k}))$.

To check ($m$3): consider comparable $u,u'$ with $k$ leaves, and
comparable $v,v'$. Fix $\zeta\in Sh(\setpar(u),\setpar(v))$, i.e.
fix $k-1$ allowable splitting positions (possibly repeated) in $v$,
and compute $v\mapsto(v_{1},\dots,v_{k})$, $v'\mapsto(v'_{1},\dots,v'_{k})$
and $v\vee v'\mapsto((v\vee v')_{1},\dots,(v\vee v')_{k})$ all at
this same multiset of lightening positions. Then $\zeta(u,v)=u\overleftarrow{\shuffle}(v_{1},\dots,v_{k})$,
and similarly for $u',v'$ and for $u\vee u',v\vee v'$. Because $\setpar(u)=\setpar(u')$
and $\setpar(v)=\setpar(v')$, and because the split positions are
fixed, so $\setpar(u\overleftarrow{\shuffle}(v_{1},\dots,v_{k}))=\setpar(u'\overleftarrow{\shuffle}(v'_{1},\dots,v'_{k}))=\setpar((u\vee u')\overleftarrow{\shuffle}((v\vee v')_{1},\dots,(v\vee v')_{k}))$.
So, by Theorem \ref{thm:lpw}, it suffices to show 
 $$\delrpt((u\vee u')\overleftarrow{\shuffle}((v\vee v')_{1},\dots,(v\vee v')_{k}))\leq\delrpt(u\overleftarrow{\shuffle}(v_{1},\dots,v_{k}))\vee\delrpt(u'\overleftarrow{\shuffle}(v'_{1},\dots,v'_{k})).$$
This follows from axiom ($m$3) in $\SSym$ and Lemma \ref{lem:delrptshuffle}. 
\end{proof}
\begin{Proposition}\label{prop:antipode-monomial-stsym} The
antipode of monomial basis elements in $\STSym$ is given by $\mathcal{S}(M_{u})=(-1)^{\GD(u)+1}\sum_{v}\beta_{u}^{v}M_{v}$
where $\beta_{f}^{g}$ is defined as in (\ref{eq:c-def-GD-2}).
\end{Proposition}
\begin{proof}
We check the axioms in Theorem \ref{thm:antipode}. Axioms ($\Delta$1)-($\Delta$3)
and ($m$0)-($m$3) were already checked in the previous two proofs.

Axiom ($\mathcal{S}$0) follows from Remark \ref{rem:s0trees}. To see
($\mathcal{S}$1): suppose $T_{a}(u)\geq_{Pw}u$ and $i\in\GD(u)$.
So $j\leq i<k\implies u_{j}>u_{k}$ . By Proposition \ref{prop:lpw},
$T_{a}(u)\geq_{Pw}u$ means that we cannot have $u_{j}=a+1$ and $u_{k}=a$.
Thus, $\{u_{j},u_{k}\}\neq\{a,a+1\}$ and so $T_{a}(u_{j})>T_{a}(u_{k})$.

To see ($\mathcal{S}$2): given $u'_{i}\geq_{Pw}u_{i}$ for each $i$,
first note that $\setpar(u'_{i})=\setpar(u_{i})$, so 
\begin{equation}
\setpar(u'_{1}/\cdots/u'_{k})=\setpar(u_{1}/\cdots/u_{k})=\setpar(u'_{1}\backslash\cdots\backslash u'_{k}).\label{eq:setparS3}
\end{equation}
 Also, $\delrpt(u'_{i})\geq\delrpt(u_{i})$, so by axiom ($\mathcal{S}$2)
in $\SSym$, 
\[
(\delrpt u'_{1})/\cdots/(\delrpt u'_{k})\leq_{w}[(\delrpt u_{1})/\cdots/(\delrpt u_{k})]\vee[(\delrpt u'_{1})\backslash\cdots\backslash(\delrpt u'_{k})].
\]
Clearly $\delrpt(v/w)=\delrpt(v)/\delrpt(w)$ and $\delrpt(v\backslash w)=\delrpt(v)\backslash\delrpt(w)$,
so 
\begin{align*}
\delrpt(u'_{1}/\cdots/u'_{k}) & \leq_{w}\delrpt(u_{1}/\cdots/u_{k})\vee\delrpt(u'_{1}\backslash\cdots\backslash u'_{k})\\
 & \leq_{w}\delrpt((u_{1}/\cdots/u_{k})\vee(u'_{1}\backslash\cdots\backslash u'_{k}))
\end{align*}
using Theorem \ref{thm:lpw} and (\ref{eq:setparS3}). Combining with
$\setpar(u'_{1}/\cdots/u'_{k})=\setpar((u_{1}/\cdots/u_{k})\vee(u'_{1}\backslash\cdots\backslash u'_{k}))$
gives the result.

Axiom ($\mathcal{S}$3) is proved similarly. 
\end{proof}

\section{Hopf algebra of parking functions}\label{sec:parking_functions}

\subsection{Parking functions}

Recall that a \emph{parking function} of degree $n$ is an integer sequence $(p_1,\dots,p_n)$ where $p_{j_i}\le i$ for all $i$, whenever $p_{j_1}\le p_{j_2}\le \cdots \le p_{j_n}$ is its nonincreasing rearrangement. An equivalent definition of a parking function is that of a Dyck path of length $2n$ with a bijection between north steps and $[n]$ such that the labels of consecutive $N$ steps are increasing from bottom to top. 

In this section, we have an alternative view of parking functions as labeled binary trees satisfying certain conditions.

\begin{Definition}
	A \textbf{parking function} of degree $n$ is a pair $(\sigma,t)\in\Perm_n\times\PBT_n$ such that $\Des(t)\subseteq\Des(\sigma)$. We write a parking function $f=(\sigma,t)$ as a binary tree $t$ with labels on the internal nodes such that the $i$-th internal node in-order traversal is labeled $\sigma(i)$. Recall that \emph{in-order traversal} of a binary tree is the one where recursively we visit first the left subtree in-order then the root and then the right subtree in-order. The set of parking functions is denoted by $\PF$.
\end{Definition}

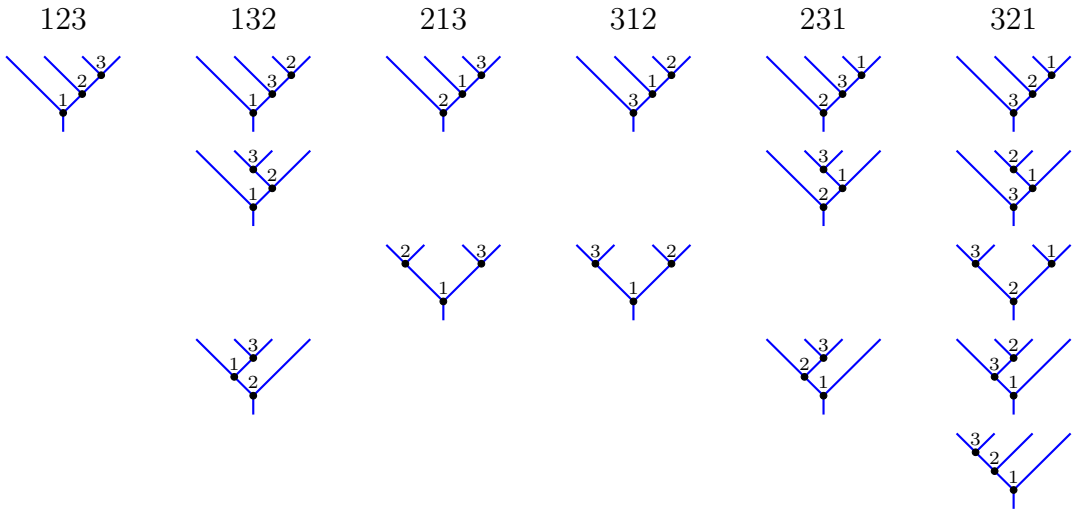
\begin{figure}
	\begin{center}
	\begin{tikzpicture}[scale=0.25,baseline=0pt]
	\node at (0,5) {$123$};
	\node at (10,5) {$132$};
	\node at (20,5) {$213$};
	\node at (30,5) {$312$};
	\node at (40,5) {$231$};
	\node at (50,5) {$321$};
	
	\draw[blue, thick] (0,-1) -- (0,0);
	\draw[blue, thick] (0,0) -- (3,3);
	\draw[blue, thick] (0,0) -- (-3,3);
	\draw[blue, thick] (1,1) -- (-1,3);
	\draw[blue, thick] (2,2) -- (1,3);
	\filldraw[black] (0,0) circle (5pt)  {};
	\filldraw[black] (1,1) circle (5pt)  {};
	\filldraw[black] (2,2) circle (5pt)  {};
	\node at (0,0.7)[font=\fontsize{7pt}{0}]{$1$};
	\node at (1,1.7)[font=\fontsize{7pt}{0}]{$2$};
	\node at (2,2.7)[font=\fontsize{7pt}{0}]{$3$};
	
	\draw[blue, thick] (10,-1) -- (10,0);
	\draw[blue, thick] (10,0) -- (13,3);
	\draw[blue, thick] (10,0) -- (7,3);
	\draw[blue, thick] (11,1) -- (9,3);
	\draw[blue, thick] (12,2) -- (11,3);
	\filldraw[black] (10,0) circle (5pt)  {};
	\filldraw[black] (11,1) circle (5pt)  {};
	\filldraw[black] (12,2) circle (5pt)  {};
	\node at (10,0.7)[font=\fontsize{7pt}{0}]{$1$};
	\node at (11,1.7)[font=\fontsize{7pt}{0}]{$3$};
	\node at (12,2.7)[font=\fontsize{7pt}{0}]{$2$};
	
	\draw[blue, thick] (20,-1) -- (20,0);
	\draw[blue, thick] (20,0) -- (23,3);
	\draw[blue, thick] (20,0) -- (17,3);
	\draw[blue, thick] (21,1) -- (19,3);
	\draw[blue, thick] (22,2) -- (21,3);
	\filldraw[black] (20,0) circle (5pt)  {};
	\filldraw[black] (21,1) circle (5pt)  {};
	\filldraw[black] (22,2) circle (5pt)  {};
	\node at (20,0.7)[font=\fontsize{7pt}{0}]{$2$};
	\node at (21,1.7)[font=\fontsize{7pt}{0}]{$1$};
	\node at (22,2.7)[font=\fontsize{7pt}{0}]{$3$};
	
	\draw[blue, thick] (30,-1) -- (30,0);
	\draw[blue, thick] (30,0) -- (33,3);
	\draw[blue, thick] (30,0) -- (27,3);
	\draw[blue, thick] (31,1) -- (29,3);
	\draw[blue, thick] (32,2) -- (31,3);
	\filldraw[black] (30,0) circle (5pt)  {};
	\filldraw[black] (31,1) circle (5pt)  {};
	\filldraw[black] (32,2) circle (5pt)  {};
	\node at (30,0.7)[font=\fontsize{7pt}{0}]{$3$};
	\node at (31,1.7)[font=\fontsize{7pt}{0}]{$1$};
	\node at (32,2.7)[font=\fontsize{7pt}{0}]{$2$};
	
	\draw[blue, thick] (40,-1) -- (40,0);
	\draw[blue, thick] (40,0) -- (43,3);
	\draw[blue, thick] (40,0) -- (37,3);
	\draw[blue, thick] (41,1) -- (39,3);
	\draw[blue, thick] (42,2) -- (41,3);
	\filldraw[black] (40,0) circle (5pt)  {};
	\filldraw[black] (41,1) circle (5pt)  {};
	\filldraw[black] (42,2) circle (5pt)  {};
	\node at (40,0.7)[font=\fontsize{7pt}{0}]{$2$};
	\node at (41,1.7)[font=\fontsize{7pt}{0}]{$3$};
	\node at (42,2.7)[font=\fontsize{7pt}{0}]{$1$};
	
	\draw[blue, thick] (50,-1) -- (50,0);
	\draw[blue, thick] (50,0) -- (53,3);
	\draw[blue, thick] (50,0) -- (47,3);
	\draw[blue, thick] (51,1) -- (49,3);
	\draw[blue, thick] (52,2) -- (51,3);
	\filldraw[black] (50,0) circle (5pt)  {};
	\filldraw[black] (51,1) circle (5pt)  {};
	\filldraw[black] (52,2) circle (5pt)  {};
	\node at (50,0.7)[font=\fontsize{7pt}{0}]{$3$};
	\node at (51,1.7)[font=\fontsize{7pt}{0}]{$2$};
	\node at (52,2.7)[font=\fontsize{7pt}{0}]{$1$};
	
	\draw[blue, thick] (10,-6) -- (10,-5);
	\draw[blue, thick] (10,-5) -- (13,-2);
	\draw[blue, thick] (10,-5) -- (7,-2);
	\draw[blue, thick] (11,-4) -- (9,-2);
	\draw[blue, thick] (10,-3) -- (11,-2);
	\filldraw[black] (10,-5) circle (5pt)  {};
	\filldraw[black] (11,-4) circle (5pt)  {};
	\filldraw[black] (10,-3) circle (5pt)  {};
	\node at (10,-4.3)[font=\fontsize{7pt}{0}]{$1$};
	\node at (11,-3.3)[font=\fontsize{7pt}{0}]{$2$};
	\node at (10,-2.3)[font=\fontsize{7pt}{0}]{$3$};
	
	\draw[blue, thick] (40,-6) -- (40,-5);
	\draw[blue, thick] (40,-5) -- (43,-2);
	\draw[blue, thick] (40,-5) -- (37,-2);
	\draw[blue, thick] (41,-4) -- (39,-2);
	\draw[blue, thick] (40,-3) -- (41,-2);
	\filldraw[black] (40,-5) circle (5pt)  {};
	\filldraw[black] (41,-4) circle (5pt)  {};
	\filldraw[black] (40,-3) circle (5pt)  {};
	\node at (40,-4.3)[font=\fontsize{7pt}{0}]{$2$};
	\node at (41,-3.3)[font=\fontsize{7pt}{0}]{$1$};
	\node at (40,-2.3)[font=\fontsize{7pt}{0}]{$3$};
	
	\draw[blue, thick] (50,-6) -- (50,-5);
	\draw[blue, thick] (50,-5) -- (53,-2);
	\draw[blue, thick] (50,-5) -- (47,-2);
	\draw[blue, thick] (51,-4) -- (49,-2);
	\draw[blue, thick] (50,-3) -- (51,-2);
	\filldraw[black] (50,-5) circle (5pt)  {};
	\filldraw[black] (51,-4) circle (5pt)  {};
	\filldraw[black] (50,-3) circle (5pt)  {};
	\node at (50,-4.3)[font=\fontsize{7pt}{0}]{$3$};
	\node at (51,-3.3)[font=\fontsize{7pt}{0}]{$1$};
	\node at (50,-2.3)[font=\fontsize{7pt}{0}]{$2$};

	\draw[blue, thick] (20,-11) -- (20,-10);
	\draw[blue, thick] (20,-10) -- (23,-7);
	\draw[blue, thick] (20,-10) -- (17,-7);
	\draw[blue, thick] (22,-8) -- (21,-7);
	\draw[blue, thick] (18,-8) -- (19,-7);
	\filldraw[black] (20,-10) circle (5pt)  {};
	\filldraw[black] (22,-8) circle (5pt)  {};
	\filldraw[black] (18,-8) circle (5pt)  {};
	\node at (20,-9.3)[font=\fontsize{7pt}{0}]{$1$};
	\node at (22,-7.3)[font=\fontsize{7pt}{0}]{$3$};
	\node at (18,-7.3)[font=\fontsize{7pt}{0}]{$2$};
	
	\draw[blue, thick] (30,-11) -- (30,-10);
	\draw[blue, thick] (30,-10) -- (33,-7);
	\draw[blue, thick] (30,-10) -- (27,-7);
	\draw[blue, thick] (32,-8) -- (31,-7);
	\draw[blue, thick] (28,-8) -- (29,-7);
	\filldraw[black] (30,-10) circle (5pt)  {};
	\filldraw[black] (32,-8) circle (5pt)  {};
	\filldraw[black] (28,-8) circle (5pt)  {};
	\node at (30,-9.3)[font=\fontsize{7pt}{0}]{$1$};
	\node at (32,-7.3)[font=\fontsize{7pt}{0}]{$2$};
	\node at (28,-7.3)[font=\fontsize{7pt}{0}]{$3$};
	
	\draw[blue, thick] (50,-11) -- (50,-10);
	\draw[blue, thick] (50,-10) -- (53,-7);
	\draw[blue, thick] (50,-10) -- (47,-7);
	\draw[blue, thick] (52,-8) -- (51,-7);
	\draw[blue, thick] (48,-8) -- (49,-7);
	\filldraw[black] (50,-10) circle (5pt)  {};
	\filldraw[black] (52,-8) circle (5pt)  {};
	\filldraw[black] (48,-8) circle (5pt)  {};
	\node at (50,-9.3)[font=\fontsize{7pt}{0}]{$2$};
	\node at (52,-7.3)[font=\fontsize{7pt}{0}]{$1$};
	\node at (48,-7.3)[font=\fontsize{7pt}{0}]{$3$};
	
	\draw[blue, thick] (10,-16) -- (10,-15);
	\draw[blue, thick] (10,-15) -- (13,-12);
	\draw[blue, thick] (10,-15) -- (7,-12);
	\draw[blue, thick] (9,-14) -- (11,-12);
	\draw[blue, thick] (10,-13) -- (9,-12);
	\filldraw[black] (10,-15) circle (5pt)  {};
	\filldraw[black] (9,-14) circle (5pt)  {};
	\filldraw[black] (10,-13) circle (5pt)  {};
	\node at (10,-14.3)[font=\fontsize{7pt}{0}]{$2$};
	\node at (9,-13.3)[font=\fontsize{7pt}{0}]{$1$};
	\node at (10,-12.3)[font=\fontsize{7pt}{0}]{$3$};
	
	\draw[blue, thick] (40,-16) -- (40,-15);
	\draw[blue, thick] (40,-15) -- (43,-12);
	\draw[blue, thick] (40,-15) -- (37,-12);
	\draw[blue, thick] (39,-14) -- (41,-12);
	\draw[blue, thick] (40,-13) -- (39,-12);
	\filldraw[black] (40,-15) circle (5pt)  {};
	\filldraw[black] (39,-14) circle (5pt)  {};
	\filldraw[black] (40,-13) circle (5pt)  {};
	\node at (40,-14.3)[font=\fontsize{7pt}{0}]{$1$};
	\node at (39,-13.3)[font=\fontsize{7pt}{0}]{$2$};
	\node at (40,-12.3)[font=\fontsize{7pt}{0}]{$3$};
	
	\draw[blue, thick] (50,-16) -- (50,-15);
	\draw[blue, thick] (50,-15) -- (53,-12);
	\draw[blue, thick] (50,-15) -- (47,-12);
	\draw[blue, thick] (49,-14) -- (51,-12);
	\draw[blue, thick] (50,-13) -- (49,-12);
	\filldraw[black] (50,-15) circle (5pt)  {};
	\filldraw[black] (49,-14) circle (5pt)  {};
	\filldraw[black] (50,-13) circle (5pt)  {};
	\node at (50,-14.3)[font=\fontsize{7pt}{0}]{$1$};
	\node at (49,-13.3)[font=\fontsize{7pt}{0}]{$3$};
	\node at (50,-12.3)[font=\fontsize{7pt}{0}]{$2$};

	\draw[blue, thick] (50,-21) -- (50,-20);
	\draw[blue, thick] (50,-20) -- (47,-17);
	\draw[blue, thick] (50,-20) -- (53,-17);
	\draw[blue, thick] (49,-19) -- (51,-17);
	\draw[blue, thick] (48,-18) -- (49,-17);
	\filldraw[black] (50,-20) circle (5pt)  {};
	\filldraw[black] (49,-19) circle (5pt)  {};
	\filldraw[black] (48,-18) circle (5pt)  {};
	\node at (50,-19.3)[font=\fontsize{7pt}{0}]{$1$};
	\node at (49,-18.3)[font=\fontsize{7pt}{0}]{$2$};
	\node at (48,-17.3)[font=\fontsize{7pt}{0}]{$3$};
	\end{tikzpicture}
	\end{center}
	\caption{\sl 
	Parking functions of degree $3$. Elements from the same column have the same underlying permutation.}\label{fig:pf}
\end{figure}

We present a bijection from our definition of parking functions to the classical one.

Given a parking function $(\sigma,t)$ of degree $n$, the Dyck path is constructed from $t$ as follows. We order the $2n+1$ nodes (including internal nodes and leaves) of $t$ in \emph{preorder}, that is, we visit first the root, then the left subtree in preorder and then the right subtree in preorder. 
Let $x_1x_2\cdots x_{2n}$ be the nodes of $t$ ignoring the right most leaf. From the order of the nodes, we construct Dyck path from $(n,n)$, the $n$ internal nodes correspond to west steps and the first $n$ leaves correspond to south steps. Then, we label of the $i$-th $N$ step, reading from top to bottom, by $\sigma(i)$.

\begin{figure}
	\begin{center}
		\begin{tikzpicture}[scale=0.25,baseline=0pt]
			\draw[gray, thin] (0,0) -- (0,3);
			\draw[gray, thin] (1,0) -- (1,3);
			\draw[gray, thin] (2,0) -- (2,3);
			\draw[gray, thin] (3,0) -- (3,3);
			\draw[gray, thin] (0,0) -- (3,0);
			\draw[gray, thin] (0,1) -- (3,1);
			\draw[gray, thin] (0,2) -- (3,2);
			\draw[gray, thin] (0,3) -- (3,3);
			\draw[red, ultra thick] (0,0) -- (0,1);
			\draw[red, ultra thick] (0,1) -- (1,1);
			\draw[red, ultra thick] (1,1) -- (1,2);
			\draw[red, ultra thick] (1,2) -- (2,2);
			\draw[red, ultra thick] (2,2) -- (2,3);
			\draw[red, ultra thick] (2,3) -- (3,3);
			\node at (0.5,0.4)[font=\fontsize{7pt}{0}]{$3$};
			\node at (1.5,1.4)[font=\fontsize{7pt}{0}]{$2$};
			\node at (2.5,2.4)[font=\fontsize{7pt}{0}]{$1$};
			
			\draw[gray, thin] (10,0) -- (10,3);
			\draw[gray, thin] (11,0) -- (11,3);
			\draw[gray, thin] (12,0) -- (12,3);
			\draw[gray, thin] (13,0) -- (13,3);
			\draw[gray, thin] (10,0) -- (13,0);
			\draw[gray, thin] (10,1) -- (13,1);
			\draw[gray, thin] (10,2) -- (13,2);
			\draw[gray, thin] (10,3) -- (13,3);
			\draw[red, ultra thick] (10,0) -- (10,1);
			\draw[red, ultra thick] (10,1) -- (11,1);
			\draw[red, ultra thick] (11,1) -- (11,2);
			\draw[red, ultra thick] (11,2) -- (12,2);
			\draw[red, ultra thick] (12,2) -- (12,3);
			\draw[red, ultra thick] (12,3) -- (13,3);
			\node at (10.5,0.4)[font=\fontsize{7pt}{0}]{$2$};
			\node at (11.5,1.4)[font=\fontsize{7pt}{0}]{$3$};
			\node at (12.5,2.4)[font=\fontsize{7pt}{0}]{$1$};
			
			\draw[gray, thin] (20,0) -- (20,3);
			\draw[gray, thin] (21,0) -- (21,3);
			\draw[gray, thin] (22,0) -- (22,3);
			\draw[gray, thin] (23,0) -- (23,3);
			\draw[gray, thin] (20,0) -- (23,0);
			\draw[gray, thin] (20,1) -- (23,1);
			\draw[gray, thin] (20,2) -- (23,2);
			\draw[gray, thin] (20,3) -- (23,3);
			\draw[red, ultra thick] (20,0) -- (20,1);
			\draw[red, ultra thick] (20,1) -- (21,1);
			\draw[red, ultra thick] (21,1) -- (21,2);
			\draw[red, ultra thick] (21,2) -- (22,2);
			\draw[red, ultra thick] (22,2) -- (22,3);
			\draw[red, ultra thick] (22,3) -- (23,3);
			\node at (20.5,0.4)[font=\fontsize{7pt}{0}]{$3$};
			\node at (21.5,1.4)[font=\fontsize{7pt}{0}]{$1$};
			\node at (22.5,2.4)[font=\fontsize{7pt}{0}]{$2$};
			
			\draw[gray, thin] (30,0) -- (30,3);
			\draw[gray, thin] (31,0) -- (31,3);
			\draw[gray, thin] (32,0) -- (32,3);
			\draw[gray, thin] (33,0) -- (33,3);
			\draw[gray, thin] (30,0) -- (33,0);
			\draw[gray, thin] (30,1) -- (33,1);
			\draw[gray, thin] (30,2) -- (33,2);
			\draw[gray, thin] (30,3) -- (33,3);
			\draw[red, ultra thick] (30,0) -- (30,1);
			\draw[red, ultra thick] (30,1) -- (31,1);
			\draw[red, ultra thick] (31,1) -- (31,2);
			\draw[red, ultra thick] (31,2) -- (32,2);
			\draw[red, ultra thick] (32,2) -- (32,3);
			\draw[red, ultra thick] (32,3) -- (33,3);
			\node at (30.5,0.4)[font=\fontsize{7pt}{0}]{$2$};
			\node at (31.5,1.4)[font=\fontsize{7pt}{0}]{$1$};
			\node at (32.5,2.4)[font=\fontsize{7pt}{0}]{$3$};
			
			\draw[gray, thin] (40,0) -- (40,3);
			\draw[gray, thin] (41,0) -- (41,3);
			\draw[gray, thin] (42,0) -- (42,3);
			\draw[gray, thin] (43,0) -- (43,3);
			\draw[gray, thin] (40,0) -- (43,0);
			\draw[gray, thin] (40,1) -- (43,1);
			\draw[gray, thin] (40,2) -- (43,2);
			\draw[gray, thin] (40,3) -- (43,3);
			\draw[red, ultra thick] (40,0) -- (40,1);
			\draw[red, ultra thick] (40,1) -- (41,1);
			\draw[red, ultra thick] (41,1) -- (41,2);
			\draw[red, ultra thick] (41,2) -- (42,2);
			\draw[red, ultra thick] (42,2) -- (42,3);
			\draw[red, ultra thick] (42,3) -- (43,3);
			\node at (40.5,0.4)[font=\fontsize{7pt}{0}]{$1$};
			\node at (41.5,1.4)[font=\fontsize{7pt}{0}]{$3$};
			\node at (42.5,2.4)[font=\fontsize{7pt}{0}]{$2$};
			
			\draw[gray, thin] (50,0) -- (50,3);
			\draw[gray, thin] (51,0) -- (51,3);
			\draw[gray, thin] (52,0) -- (52,3);
			\draw[gray, thin] (53,0) -- (53,3);
			\draw[gray, thin] (50,0) -- (53,0);
			\draw[gray, thin] (50,1) -- (53,1);
			\draw[gray, thin] (50,2) -- (53,2);
			\draw[gray, thin] (50,3) -- (53,3);
			\draw[red, ultra thick] (50,0) -- (50,1);
			\draw[red, ultra thick] (50,1) -- (51,1);
			\draw[red, ultra thick] (51,1) -- (51,2);
			\draw[red, ultra thick] (51,2) -- (52,2);
			\draw[red, ultra thick] (52,2) -- (52,3);
			\draw[red, ultra thick] (52,3) -- (53,3);
			\node at (50.5,0.4)[font=\fontsize{7pt}{0}]{$1$};
			\node at (51.5,1.4)[font=\fontsize{7pt}{0}]{$2$};
			\node at (52.5,2.4)[font=\fontsize{7pt}{0}]{$3$};

			\draw[gray, thin] (10,-5) -- (10,-2);
			\draw[gray, thin] (11,-5) -- (11,-2);
			\draw[gray, thin] (12,-5) -- (12,-2);
			\draw[gray, thin] (13,-5) -- (13,-2);
			\draw[gray, thin] (10,-5) -- (13,-5);
			\draw[gray, thin] (10,-4) -- (13,-4);
			\draw[gray, thin] (10,-3) -- (13,-3);
			\draw[gray, thin] (10,-2) -- (13,-2);
			\draw[red, ultra thick] (10,-5) -- (10,-3);
			\draw[red, ultra thick] (10,-3) -- (12,-3);
			\draw[red, ultra thick] (12,-3) -- (12,-2);
			\draw[red, ultra thick] (12,-2) -- (13,-2);
			\node at (10.5,-4.6)[font=\fontsize{7pt}{0}]{$2$};
			\node at (10.5,-3.6)[font=\fontsize{7pt}{0}]{$3$};
			\node at (12.5,-2.6)[font=\fontsize{7pt}{0}]{$1$};
			
			\draw[gray, thin] (40,-5) -- (40,-2);
			\draw[gray, thin] (41,-5) -- (41,-2);
			\draw[gray, thin] (42,-5) -- (42,-2);
			\draw[gray, thin] (43,-5) -- (43,-2);
			\draw[gray, thin] (40,-5) -- (43,-5);
			\draw[gray, thin] (40,-4) -- (43,-4);
			\draw[gray, thin] (40,-3) -- (43,-3);
			\draw[gray, thin] (40,-2) -- (43,-2);
			\draw[red, ultra thick] (40,-5) -- (40,-3);
			\draw[red, ultra thick] (40,-3) -- (42,-3);
			\draw[red, ultra thick] (42,-3) -- (42,-2);
			\draw[red, ultra thick] (42,-2) -- (43,-2);
			\node at (40.5,-4.6)[font=\fontsize{7pt}{0}]{$1$};
			\node at (40.5,-3.6)[font=\fontsize{7pt}{0}]{$3$};
			\node at (42.5,-2.6)[font=\fontsize{7pt}{0}]{$2$};
			
			\draw[gray, thin] (50,-5) -- (50,-2);
			\draw[gray, thin] (51,-5) -- (51,-2);
			\draw[gray, thin] (52,-5) -- (52,-2);
			\draw[gray, thin] (53,-5) -- (53,-2);
			\draw[gray, thin] (50,-5) -- (53,-5);
			\draw[gray, thin] (50,-4) -- (53,-4);
			\draw[gray, thin] (50,-3) -- (53,-3);
			\draw[gray, thin] (50,-2) -- (53,-2);
			\draw[red, ultra thick] (50,-5) -- (50,-3);
			\draw[red, ultra thick] (50,-3) -- (52,-3);
			\draw[red, ultra thick] (52,-3) -- (52,-2);
			\draw[red, ultra thick] (52,-2) -- (53,-2);
			\node at (50.5,-4.6)[font=\fontsize{7pt}{0}]{$1$};
			\node at (50.5,-3.6)[font=\fontsize{7pt}{0}]{$2$};
			\node at (52.5,-2.6)[font=\fontsize{7pt}{0}]{$3$};
			
			\draw[gray, thin] (20,-10) -- (20,-7);
			\draw[gray, thin] (21,-10) -- (21,-7);
			\draw[gray, thin] (22,-10) -- (22,-7);
			\draw[gray, thin] (23,-10) -- (23,-7);
			\draw[gray, thin] (20,-10) -- (23,-10);
			\draw[gray, thin] (20,-9) -- (23,-9);
			\draw[gray, thin] (20,-8) -- (23,-8);
			\draw[gray, thin] (20,-7) -- (23,-7);
			\draw[red, ultra thick] (20,-10) -- (20,-9);
			\draw[red, ultra thick] (20,-9) -- (21,-9);
			\draw[red, ultra thick] (21,-9) -- (21,-7);
			\draw[red, ultra thick] (21,-7) -- (23,-7);
			\node at (20.5,-9.6)[font=\fontsize{7pt}{0}]{$3$};
			\node at (21.5,-8.6)[font=\fontsize{7pt}{0}]{$1$};
			\node at (21.5,-7.6)[font=\fontsize{7pt}{0}]{$2$};
			
			\draw[gray, thin] (30,-10) -- (30,-7);
			\draw[gray, thin] (31,-10) -- (31,-7);
			\draw[gray, thin] (32,-10) -- (32,-7);
			\draw[gray, thin] (33,-10) -- (33,-7);
			\draw[gray, thin] (30,-10) -- (33,-10);
			\draw[gray, thin] (30,-9) -- (33,-9);
			\draw[gray, thin] (30,-8) -- (33,-8);
			\draw[gray, thin] (30,-7) -- (33,-7);
			\draw[red, ultra thick] (30,-10) -- (30,-9);
			\draw[red, ultra thick] (30,-9) -- (31,-9);
			\draw[red, ultra thick] (31,-9) -- (31,-7);
			\draw[red, ultra thick] (31,-7) -- (33,-7);
			\node at (30.5,-9.6)[font=\fontsize{7pt}{0}]{$2$};
			\node at (31.5,-8.6)[font=\fontsize{7pt}{0}]{$1$};
			\node at (31.5,-7.6)[font=\fontsize{7pt}{0}]{$3$};
			
			\draw[gray, thin] (50,-10) -- (50,-7);
			\draw[gray, thin] (51,-10) -- (51,-7);
			\draw[gray, thin] (52,-10) -- (52,-7);
			\draw[gray, thin] (53,-10) -- (53,-7);
			\draw[gray, thin] (50,-10) -- (53,-10);
			\draw[gray, thin] (50,-9) -- (53,-9);
			\draw[gray, thin] (50,-8) -- (53,-8);
			\draw[gray, thin] (50,-7) -- (53,-7);
			\draw[red, ultra thick] (50,-10) -- (50,-9);
			\draw[red, ultra thick] (50,-9) -- (51,-9);
			\draw[red, ultra thick] (51,-9) -- (51,-7);
			\draw[red, ultra thick] (51,-7) -- (53,-7);
			\node at (50.5,-9.6)[font=\fontsize{7pt}{0}]{$1$};
			\node at (51.5,-8.6)[font=\fontsize{7pt}{0}]{$2$};
			\node at (51.5,-7.6)[font=\fontsize{7pt}{0}]{$3$};

			\draw[gray, thin] (10,-15) -- (10,-12);
			\draw[gray, thin] (11,-15) -- (11,-12);
			\draw[gray, thin] (12,-15) -- (12,-12);
			\draw[gray, thin] (13,-15) -- (13,-12);
			\draw[gray, thin] (10,-15) -- (13,-15);
			\draw[gray, thin] (10,-14) -- (13,-14);
			\draw[gray, thin] (10,-13) -- (13,-13);
			\draw[gray, thin] (10,-12) -- (13,-12);
			\draw[red, ultra thick] (10,-15) -- (10,-13);
			\draw[red, ultra thick] (10,-13) -- (11,-13);
			\draw[red, ultra thick] (11,-13) -- (11,-12);
			\draw[red, ultra thick] (11,-12) -- (13,-12);
			\node at (10.5,-14.6)[font=\fontsize{7pt}{0}]{$2$};
			\node at (10.5,-13.6)[font=\fontsize{7pt}{0}]{$3$};
			\node at (11.5,-12.6)[font=\fontsize{7pt}{0}]{$1$};
			
			\draw[gray, thin] (40,-15) -- (40,-12);
			\draw[gray, thin] (41,-15) -- (41,-12);
			\draw[gray, thin] (42,-15) -- (42,-12);
			\draw[gray, thin] (43,-15) -- (43,-12);
			\draw[gray, thin] (40,-15) -- (43,-15);
			\draw[gray, thin] (40,-14) -- (43,-14);
			\draw[gray, thin] (40,-13) -- (43,-13);
			\draw[gray, thin] (40,-12) -- (43,-12);
			\draw[red, ultra thick] (40,-15) -- (40,-13);
			\draw[red, ultra thick] (40,-13) -- (41,-13);
			\draw[red, ultra thick] (41,-13) -- (41,-12);
			\draw[red, ultra thick] (41,-12) -- (43,-12);
			\node at (40.5,-14.6)[font=\fontsize{7pt}{0}]{$1$};
			\node at (40.5,-13.6)[font=\fontsize{7pt}{0}]{$3$};
			\node at (41.5,-12.6)[font=\fontsize{7pt}{0}]{$2$};
			
			\draw[gray, thin] (50,-15) -- (50,-12);
			\draw[gray, thin] (51,-15) -- (51,-12);
			\draw[gray, thin] (52,-15) -- (52,-12);
			\draw[gray, thin] (53,-15) -- (53,-12);
			\draw[gray, thin] (50,-15) -- (53,-15);
			\draw[gray, thin] (50,-14) -- (53,-14);
			\draw[gray, thin] (50,-13) -- (53,-13);
			\draw[gray, thin] (50,-12) -- (53,-12);
			\draw[red, ultra thick] (50,-15) -- (50,-13);
			\draw[red, ultra thick] (50,-13) -- (51,-13);
			\draw[red, ultra thick] (51,-13) -- (51,-12);
			\draw[red, ultra thick] (51,-12) -- (53,-12);
			\node at (50.5,-14.6)[font=\fontsize{7pt}{0}]{$1$};
			\node at (50.5,-13.6)[font=\fontsize{7pt}{0}]{$2$};
			\node at (51.5,-12.6)[font=\fontsize{7pt}{0}]{$3$};

			\draw[gray, thin] (50,-20) -- (50,-17);
			\draw[gray, thin] (51,-20) -- (51,-17);
			\draw[gray, thin] (52,-20) -- (52,-17);
			\draw[gray, thin] (53,-20) -- (53,-17);
			\draw[gray, thin] (50,-20) -- (53,-20);
			\draw[gray, thin] (50,-19) -- (53,-19);
			\draw[gray, thin] (50,-18) -- (53,-18);
			\draw[gray, thin] (50,-17) -- (53,-17);
			\draw[red, ultra thick] (50,-20) -- (50,-17);
			\draw[red, ultra thick] (50,-17) -- (53,-17);
			\node at (50.5,-19.6)[font=\fontsize{7pt}{0}]{$1$};
			\node at (50.5,-18.6)[font=\fontsize{7pt}{0}]{$2$};
			\node at (50.5,-17.6)[font=\fontsize{7pt}{0}]{$3$};
		\end{tikzpicture}
	\end{center}
	\caption{Labeled Dyck paths of degree $3$, the positions are in correspondence with the positions of parking functions in Figure \ref{fig:pf}.}\label{fig:pf_classic}
\end{figure}
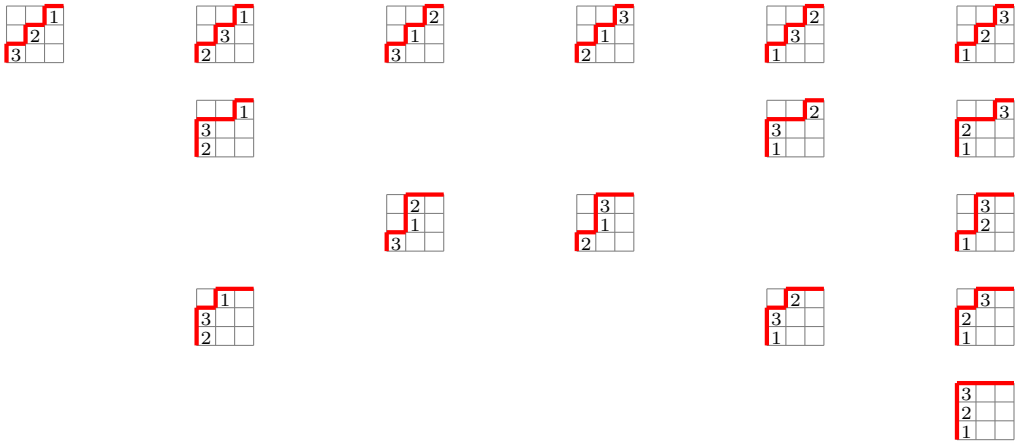

\subsection{Hopf algebra of parking functions}

We define a Hopf structure on parking functions.

\begin{Definition}
	As a graded vector space, the Hopf algebra of parking functions, $\PSym$ is  $\displaystyle\bigoplus_{n\geq 0}\Bbbk\PF_n$ where $\PF_n$ is the set of parking functions of degree $n$. By convention $\PF_0$ is the span of the empty parking function. The fundamental basis of $\PSym$ is defined as $\{F_f:f\in\PF\}$.
\end{Definition}

For $f\in \PF_n$ we let $\allow(f)=\{1,2,\ldots,n-1\}$. In particular,
the comultiplication and multiplication rules for $\PSym$ are
$$\Delta(F_f)=\sum_{i=0}^{\deg(f)}F_{^if}\otimes F_{f^i}=\sum_{f\mapsto(f_1,f_2)}F_{\std(f_1)}\otimes F_{\std(f_2)},$$
$$F_f\cdot F_g=\sum_{g\mapsto(g_1,\dots,g_{\deg(f)+1})}F_{f\overleftarrow{\shuffle}(g_1,\dots,g_{\deg(f)+1})}.$$

\begin{Example}
	Let $f=\begin{tikzpicture}[scale=0.2,baseline=0pt]
	\draw[blue, thick] (0,-1) -- (0,0);
	\draw[blue, thick] (0,0) -- (2,2);
	\draw[blue, thick] (0,0) -- (-2,2);
	\draw[blue, thick] (-1,1) -- (0,2);
	\filldraw[black] (0,0) circle (5pt)  {};
	\filldraw[black] (-1,1) circle (5pt)  {};
	\node at (-1,1.7)[font=\fontsize{7pt}{0}]{$2$};
	\node at (0,0.7)[font=\fontsize{7pt}{0}]{$1$};
	\end{tikzpicture}$ and $g=\begin{tikzpicture}[scale=0.2,baseline=0pt]
	\draw[red, thick] (0,-1) -- (0,0);
	\draw[red, thick] (0,0) -- (2,2);
	\draw[red, thick] (0,0) -- (-2,2);
	\draw[red, thick] (1,1) -- (0,2);
	\filldraw[black] (0,0) circle (5pt)  {};
	\filldraw[black] (1,1) circle (5pt)  {};
	\node at (0,0.7)[font=\fontsize{7pt}{0}]{$1$};
	\node at (1,1.7)[font=\fontsize{7pt}{0}]{$2$};
	\end{tikzpicture}$. Then
	$$F_f\cdot F_g=
	F_{\begin{tikzpicture}[scale=0.2,baseline=0pt]
	\draw[blue, thick] (0,-1) -- (0,0);
	\draw[blue, thick] (0,0) -- (2,2);
	\draw[blue, thick] (0,0) -- (-4,4);
	\draw[red, thick] (2,2) -- (0,4);
	\draw[red, thick] (2,2) -- (4,4);
	\draw[red, thick] (3,3) -- (2,4);
	\draw[blue, thick] (-3,3) -- (-2,4);
	\filldraw[black] (0,0) circle (5pt)  {};
	\filldraw[black] (2,2) circle (5pt)  {};
	\filldraw[black] (3,3) circle (5pt)  {};
	\filldraw[black] (-3,3) circle (5pt)  {};
	\node at (-3,3.7)[font=\fontsize{7pt}{0}]{$2$};
	\node at (0,0.7)[font=\fontsize{7pt}{0}]{$1$};
	\node at (2,2.7)[font=\fontsize{7pt}{0}]{$3$};
	\node at (3,3.7)[font=\fontsize{7pt}{0}]{$4$};
	\end{tikzpicture}}+
	F_{\begin{tikzpicture}[scale=0.2,baseline=0pt]
	\draw[blue, thick] (0,-1) -- (0,0);
	\draw[blue, thick] (0,0) -- (4,4);
	\draw[blue, thick] (0,0) -- (-4,4);
	\draw[blue, thick] (-1,1) -- (0,2);
	\draw[red, thick] (0,2) -- (-2,4);
	\draw[red, thick] (0,2) -- (2,4);
	\draw[red, thick] (1,3) -- (0,4);
	\filldraw[black] (0,0) circle (5pt)  {};
	\filldraw[black] (-1,1) circle (5pt)  {};
	\filldraw[black] (0,2) circle (5pt)  {};
	\filldraw[black] (1,3) circle (5pt)  {};
	\node at (-1,1.7)[font=\fontsize{7pt}{0}]{$2$};
	\node at (0,0.7)[font=\fontsize{7pt}{0}]{$1$};
	\node at (0,2.7)[font=\fontsize{7pt}{0}]{$3$};
	\node at (1,3.7)[font=\fontsize{7pt}{0}]{$4$};
	\end{tikzpicture}}+
	F_{\begin{tikzpicture}[scale=0.2,baseline=0pt]
	\draw[blue, thick] (0,-1) -- (0,0);
	\draw[blue, thick] (0,0) -- (4,4);
	\draw[blue, thick] (0,0) -- (-2,2);
	\draw[blue, thick] (-1,1) -- (2,4);
	\draw[red, thick] (-2,2) -- (0,4);
	\draw[red, thick] (-2,2) -- (-4,4);
	\draw[red, thick] (-1,3) -- (-2,4);
	\filldraw[black] (0,0) circle (5pt)  {};
	\filldraw[black] (-1,1) circle (5pt)  {};
	\filldraw[black] (-2,2) circle (5pt)  {};
	\filldraw[black] (-1,3) circle (5pt)  {};
	\node at (-2,2.7)[font=\fontsize{7pt}{0}]{$3$};
	\node at (-1,3.7)[font=\fontsize{7pt}{0}]{$4$};
	\node at (-1,1.7)[font=\fontsize{7pt}{0}]{$2$};
	\node at (0,0.7)[font=\fontsize{7pt}{0}]{$1$};
	\end{tikzpicture}}+
	F_{\begin{tikzpicture}[scale=0.2,baseline=0pt]
	\draw[blue, thick] (0,-1) -- (0,0);
	\draw[blue, thick] (0,0) -- (3,3);
	\draw[blue, thick] (0,0) -- (-4,4);
	\draw[red, thick] (3,3) -- (2,4);
	\draw[red, thick] (3,3) -- (4,4);
	\draw[blue, thick] (-2,2) -- (-1,3);
	\draw[red, thick] (-1,3) -- (-2,4);
	\draw[red, thick] (-1,3) -- (0,4);
	\filldraw[black] (0,0) circle (5pt)  {};
	\filldraw[black] (3,3) circle (5pt)  {};
	\filldraw[black] (-2,2) circle (5pt)  {};
	\filldraw[black] (-1,3) circle (5pt)  {};
	\node at (-2,2.7)[font=\fontsize{7pt}{0}]{$2$};
	\node at (-1,3.7)[font=\fontsize{7pt}{0}]{$3$};
	\node at (0,0.7)[font=\fontsize{7pt}{0}]{$1$};
	\node at (3,3.7)[font=\fontsize{7pt}{0}]{$4$};
	\end{tikzpicture}}+
	F_{\begin{tikzpicture}[scale=0.2,baseline=0pt]
	\draw[blue, thick] (0,-1) -- (0,0);
	\draw[blue, thick] (0,0) -- (3,3);
	\draw[blue, thick] (0,0) -- (-3,3);
	\draw[red, thick] (3,3) -- (2,4);
	\draw[red, thick] (3,3) -- (4,4);
	\draw[blue, thick] (-2,2) -- (0,4);
	\draw[red, thick] (-3,3) -- (-2,4);
	\draw[red, thick] (-3,3) -- (-4,4);
	\filldraw[black] (0,0) circle (5pt)  {};
	\filldraw[black] (3,3) circle (5pt)  {};
	\filldraw[black] (-2,2) circle (5pt)  {};
	\filldraw[black] (-3,3) circle (5pt)  {};
	\node at (-3,3.7)[font=\fontsize{7pt}{0}]{$3$};
	\node at (-2,2.7)[font=\fontsize{7pt}{0}]{$2$};
	\node at (0,0.7)[font=\fontsize{7pt}{0}]{$1$};
	\node at (3,3.7)[font=\fontsize{7pt}{0}]{$4$};
	\end{tikzpicture}}+
	F_{\begin{tikzpicture}[scale=0.2,baseline=0pt]
	\draw[blue, thick] (0,-1) -- (0,0);
	\draw[blue, thick] (0,0) -- (4,4);
	\draw[blue, thick] (0,0) -- (-3,3);
	\draw[blue, thick] (-1,1) -- (1,3);
	\draw[red, thick] (-3,3) -- (-2,4);
	\draw[red, thick] (-3,3) -- (-4,4);
	\draw[red, thick] (1,3) -- (0,4);
	\draw[red, thick] (1,3) -- (2,4);
	\filldraw[black] (0,0) circle (5pt)  {};
	\filldraw[black] (-1,1) circle (5pt)  {};
	\filldraw[black] (-3,3) circle (5pt)  {};
	\filldraw[black] (1,3) circle (5pt)  {};
	\node at (-3,3.7)[font=\fontsize{7pt}{0}]{$3$};
	\node at (-1,1.7)[font=\fontsize{7pt}{0}]{$2$};
	\node at (1,3.7)[font=\fontsize{7pt}{0}]{$4$};
	\node at (0,0.7)[font=\fontsize{7pt}{0}]{$1$};
	\end{tikzpicture}}$$
	Let $f=$
	\begin{tikzpicture}[scale=0.2,baseline=0pt]
	\draw[blue, thick] (0,-1) -- (0,0);
	\draw[blue, thick] (0,0) -- (5,5);
	\draw[blue, thick] (0,0) -- (-5,5);
	\draw[blue, thick] (2,2) -- (-1,5);
	\draw[blue, thick] (1,3) -- (3,5);
	\draw[blue, thick] (2,4) -- (1,5);
	\draw[blue, thick] (-4,4) -- (-3,5);
	\filldraw[black] (0,0) circle (5pt)  {};
	\filldraw[black] (2,2) circle (5pt)  {};
	\filldraw[black] (1,3) circle (5pt)  {};
	\filldraw[black] (2,4) circle (5pt)  {};
	\filldraw[black] (-4,4) circle (5pt)  {};
	\node at (-4,4.7)[font=\fontsize{7pt}{0}]{$4$};
	\node at (0,0.7)[font=\fontsize{7pt}{0}]{$2$};
	\node at (1,3.7)[font=\fontsize{7pt}{0}]{$3$};
	\node at (2,4.7)[font=\fontsize{7pt}{0}]{$5$};
	\node at (2,2.7)[font=\fontsize{7pt}{0}]{$1$};
	\end{tikzpicture}. Then
	$$\Delta_+(F_f)=
	F_{\begin{tikzpicture}[scale=0.2,baseline=0pt]
	\draw[blue, thick] (0,-1) -- (0,0);
	\draw[blue, thick] (0,0) -- (1,1);
	\draw[blue, thick] (0,0) -- (-1,1);
	\filldraw[black] (0,0) circle (5pt)  {};
	\node at (0,0.7)[font=\fontsize{7pt}{0}]{$1$};
	\end{tikzpicture}}\otimes
	F_{\begin{tikzpicture}[scale=0.2,baseline=0pt]
	\draw[blue, thick] (0,-1) -- (0,0);
	\draw[blue, thick] (0,0) -- (4,4);
	\draw[blue, thick] (0,0) -- (-4,4);
	\draw[blue, thick] (1,1) -- (-2,4);
	\draw[blue, thick] (0,2) -- (2,4);
	\draw[blue, thick] (1,3) -- (0,4);
	\filldraw[black] (0,0) circle (5pt)  {};
	\filldraw[black] (1,1) circle (5pt)  {};
	\filldraw[black] (0,2) circle (5pt)  {};
	\filldraw[black] (1,3) circle (5pt)  {};
	\node at (0,0.7)[font=\fontsize{7pt}{0}]{$2$};
	\node at (0,2.7)[font=\fontsize{7pt}{0}]{$3$};
	\node at (1,3.7)[font=\fontsize{7pt}{0}]{$4$};
	\node at (1,1.7)[font=\fontsize{7pt}{0}]{$1$};
	\end{tikzpicture}}+
	F_{\begin{tikzpicture}[scale=0.2,baseline=0pt]
	\draw[blue, thick] (0,-1) -- (0,0);
	\draw[blue, thick] (0,0) -- (2,2);
	\draw[blue, thick] (0,0) -- (-2,2);
	\draw[blue, thick] (-1,1) -- (0,2);
	\filldraw[black] (0,0) circle (5pt)  {};
	\filldraw[black] (-1,1) circle (5pt)  {};
	\node at (-1,1.7)[font=\fontsize{7pt}{0}]{$2$};
	\node at (0,0.7)[font=\fontsize{7pt}{0}]{$1$};
	\end{tikzpicture}}\otimes
	F_{\begin{tikzpicture}[scale=0.2,baseline=0pt]
	\draw[blue, thick] (0,-1) -- (0,0);
	\draw[blue, thick] (0,0) -- (3,3);
	\draw[blue, thick] (0,0) -- (-3,3);
	\draw[blue, thick] (-1,1) -- (1,3);
	\draw[blue, thick] (0,2) -- (-1,3);
	\filldraw[black] (0,0) circle (5pt)  {};
	\filldraw[black] (-1,1) circle (5pt)  {};
	\filldraw[black] (0,2) circle (5pt)  {};
	\node at (0,0.7)[font=\fontsize{7pt}{0}]{$1$};
	\node at (-1,1.7)[font=\fontsize{7pt}{0}]{$2$};
	\node at (0,2.7)[font=\fontsize{7pt}{0}]{$3$};
	\end{tikzpicture}}+
	F_{\begin{tikzpicture}[scale=0.2,baseline=0pt]
	\draw[blue, thick] (0,-1) -- (0,0);
	\draw[blue, thick] (0,0) -- (3,3);
	\draw[blue, thick] (0,0) -- (-3,3);
	\draw[blue, thick] (2,2) -- (1,3);
	\draw[blue, thick] (-2,2) -- (-1,3);
	\filldraw[black] (0,0) circle (5pt)  {};
	\filldraw[black] (2,2) circle (5pt)  {};
	\filldraw[black] (-2,2) circle (5pt)  {};
	\node at (-2,2.7)[font=\fontsize{7pt}{0}]{$3$};
	\node at (0,0.7)[font=\fontsize{7pt}{0}]{$1$};
	\node at (2,2.7)[font=\fontsize{7pt}{0}]{$2$};
	\end{tikzpicture}}\otimes
	F_{\begin{tikzpicture}[scale=0.2,baseline=0pt]
	\draw[blue, thick] (0,-1) -- (0,0);
	\draw[blue, thick] (0,0) -- (2,2);
	\draw[blue, thick] (0,0) -- (-2,2);
	\draw[blue, thick] (-1,1) -- (0,2);
	\filldraw[black] (0,0) circle (5pt)  {};
	\filldraw[black] (-1,1) circle (5pt)  {};
	\node at (-1,1.7)[font=\fontsize{7pt}{0}]{$2$};
	\node at (0,0.7)[font=\fontsize{7pt}{0}]{$1$};
	\end{tikzpicture}}+
	F_{\begin{tikzpicture}[scale=0.2,baseline=0pt]
	\draw[blue, thick] (0,-1) -- (0,0);
	\draw[blue, thick] (0,0) -- (4,4);
	\draw[blue, thick] (0,0) -- (-4,4);
	\draw[blue, thick] (2,2) -- (0,4);
	\draw[blue, thick] (3,3) -- (2,4);
	\draw[blue, thick] (-3,3) -- (-2,4);
	\filldraw[black] (0,0) circle (5pt)  {};
	\filldraw[black] (2,2) circle (5pt)  {};
	\filldraw[black] (3,3) circle (5pt)  {};
	\filldraw[black] (-3,3) circle (5pt)  {};
	\node at (-3,3.7)[font=\fontsize{7pt}{0}]{$3$};
	\node at (0,0.7)[font=\fontsize{7pt}{0}]{$1$};
	\node at (2,2.7)[font=\fontsize{7pt}{0}]{$2$};
	\node at (3,3.7)[font=\fontsize{7pt}{0}]{$4$};
	\end{tikzpicture}}\otimes
	F_{\begin{tikzpicture}[scale=0.2,baseline=0pt]
	\draw[blue, thick] (0,-1) -- (0,0);
	\draw[blue, thick] (0,0) -- (1,1);
	\draw[blue, thick] (0,0) -- (-1,1);
	\filldraw[black] (0,0) circle (5pt)  {};
	\node at (0,0.7)[font=\fontsize{7pt}{0}]{$1$};
	\end{tikzpicture}}$$
\end{Example}
\begin{Proposition}
	$\PSym$ with multiplication and comultiplication defined above is a Hopf algebra.
\end{Proposition}

\begin{proof}
	The multiplication and comultiplication are closed in $\PSym$. The comultiplication does not create any new descent in the underlying tree. The new descents created by the product $f\cdot g$ are at the positions connecting a part of $g$ to $f$. Hence the corresponding permutations always have a descent at those positions because the labels of $g$ get shifted up.
	
	The associativity, coassociativity and compatibility follow directly from those of $\YSym$ and $\SSym$.
\end{proof}

\begin{Remark}
	One may also define a similar structure on labeled binary trees whose internal nodes are labeled with $\{1,2,\dots,n\}$ with no restriction. This results in a Hopf algebra on pairs of permutations and binary trees with the same degree.
\end{Remark}

\begin{Remark}
	By forgetting the labelings and the binary trees, we have two maps $\pi_\tree:\PF\to\PBT$ and $\pi_\perm:\PF\to\Perm$. These lead to two Hopf morphisms $\Pi_\tree:\PSym\to\YSym$, $F_f\mapsto F_{\pi_\tree(f)}$ and $\Pi_\perm:\PSym\to\SSym$, $F_f\mapsto F_{\pi_\perm(f)}$.
	
	However, note that if $\Pi$ is the map from Proposition~\ref{prop:pi-property-bin} then we have that $\Pi_\tree\neq\Pi\circ\Pi_\perm$, since in a parking function $(\sigma, t)$, the underlying tree of $\sigma$ and $t$ do not have to be the same.
\end{Remark}

\subsection{Bidendriform structure}

In this section we show that $\PSym$ is free, self-dual and isomorphic to $\PQSym$ defined in \cite{NT07}. We achieve these by giving $\PSym$ a bidendriform structure.

We introduce the following operations on $\PSym$.

\begin{enumerate}
	\item $\displaystyle F_f\ll F_g=\sum_{\substack{h\in f\shuffle g \\ (\pi_\perm(h))(n)\leq\deg(f)}}F_h$
	\item $\displaystyle F_f\gg F_g=\sum_{{h\in f\shuffle g \\ (\pi_\perm(h))(n)>\deg(f)}}F_h$
	\item $\displaystyle\Delta_\ll(F_h)=\sum_{\substack{h\mapsto(f,g)\\ \deg(f)\geq(\pi_\perm(h))^{-1}(n)}}F_f\otimes F_g$
	\item $\displaystyle\Delta_\gg(F_h)=\sum_{\substack{h\mapsto(f,g)\\ \deg(f)<(\pi_\perm(h))^{-1}(n)}}F_f\otimes F_g$
\end{enumerate}
where $n=\deg(f)+\deg(g)$ in (1) and (2) while $n=\deg(h)$ in (3) and (4).

\begin{Proposition}\label{prop:PSYM_bidendriform}
The operations $\ll, \gg,\Delta_\ll,\Delta_\gg$ defined above puts a bidendriform structure on $\PSym$.	
\end{Proposition}

As for the case of $\STSym$, from results of Foissy \cite{F12} regarding Hopf algebras with bidendriform structures we have the following corollary.

\begin{Corollary}\label{cor:PSYM_free_selfdual}
	$\PSym$ is free, self-dual and isomorphic to $\PQSym$ defined in \cite{NT07}.
\end{Corollary}

\begin{proof}
	Because $\PSym$ and $\PQSym$ both admit bidendriform structures, by \cite{F07}, they are free, cofree and self-dual. They have the same dimension in each degree, hence they must be isomorphic by \cite{F12}.
\end{proof}

\begin{Problem}
	 A more interesting problem that we leave to the reader is to find an explicit isomorphism between $\PSym$ and $\PQSym$.
\end{Problem}

\subsection{The parking order and monomial basis}
We define a partial order, called the parking order $\leq_{P}$, on parking functions. For two parking functions $f,g$ of the same degree, $f\leq_{P} g$ if $\pi_{\tree}(f)\leq_{T}\pi_{\tree}(g)$ and $\pi_{\perm}(f)\leq_{w}\pi_{\perm}(g)$. For example see Figure~\ref{fig:PK2} for the parking functions of degree 2 and Figure~\ref{fig:PK3} for the parking functions of degree 3.

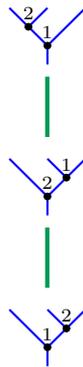
\begin{figure}
			\begin{center}
		\begin{tikzpicture}[scale=0.4]
		\draw[ForestGreen, ultra thick] (0,1.5) -- (0,3.5);
		\draw[ForestGreen, ultra thick] (0,6.5) -- (0,8.5);
		\node () at (0,0){\begin{tikzpicture}[scale=0.25,baseline=0pt]
			\draw[blue, thick] (0,-1) -- (0,0);
			\draw[blue, thick] (0,0) -- (2,2);
			\draw[blue, thick] (0,0) -- (-2,2);
			\draw[blue, thick] (1,1) -- (0,2);
			\filldraw[black] (0,0) circle (5pt)  {};
			\filldraw[black] (1,1) circle (5pt)  {};
			\node at (0,0.7)[font=\fontsize{7pt}{0}]{$1$};
			\node at (1,1.7)[font=\fontsize{7pt}{0}]{$2$};\end{tikzpicture}};
		\node () at (0,5){\begin{tikzpicture}[scale=0.25,baseline=0pt]
			\draw[blue, thick] (0,-1) -- (0,0);
			\draw[blue, thick] (0,0) -- (2,2);
			\draw[blue, thick] (0,0) -- (-2,2);
			\draw[blue, thick] (1,1) -- (0,2);
			\filldraw[black] (0,0) circle (5pt)  {};
			\filldraw[black] (1,1) circle (5pt)  {};
			\node at (0,0.7)[font=\fontsize{7pt}{0}]{$2$};
			\node at (1,1.7)[font=\fontsize{7pt}{0}]{$1$};\end{tikzpicture}};
		\node () at (0,10){\begin{tikzpicture}[scale=0.25,baseline=0pt]
			\draw[blue, thick] (0,-1) -- (0,0);
			\draw[blue, thick] (0,0) -- (2,2);
			\draw[blue, thick] (0,0) -- (-2,2);
			\draw[blue, thick] (-1,1) -- (0,2);
			\filldraw[black] (0,0) circle (5pt)  {};
			\filldraw[black] (-1,1) circle (5pt)  {};
			\node at (-1,1.7)[font=\fontsize{7pt}{0}]{$2$};
			\node at (0,0.7)[font=\fontsize{7pt}{0}]{$1$};\end{tikzpicture}};
	\end{tikzpicture}
	\end{center}
	 \caption{\sl Parking order on parking functions of degree 2.} \label{fig:PK2}
\end{figure}

\begin{figure}	
	\begin{center}
		\begin{tikzpicture}[scale=0.2,baseline=0pt]
		\draw[ForestGreen, ultra thick] (3,3) -- (8,8);
		\draw[ForestGreen, ultra thick] (-3,3) -- (-8,8);
		\draw[ForestGreen, ultra thick] (-10,13) -- (-10,17);
		\draw[ForestGreen, ultra thick] (-7,23) -- (-2,28);
		\draw[ForestGreen, ultra thick] (7,23) -- (2,28);
		\draw[ForestGreen, ultra thick] (10,13) -- (10,17);
		\draw[ForestGreen, ultra thick] (13,13) -- (18,18);
		\draw[ForestGreen, ultra thick] (13,23) -- (18,28);
		\draw[ForestGreen, ultra thick] (3,32) -- (18,39);
		\draw[ForestGreen, ultra thick] (-13,13) -- (-18,18);
		\draw[ForestGreen, ultra thick] (-13,23) -- (-18,28);
		\draw[ForestGreen, ultra thick] (-3,32) -- (-18,39);
		\draw[ForestGreen, ultra thick] (-23,23) -- (-28,28);
		\draw[ForestGreen, ultra thick] (-23,33) -- (-28,38);
		\draw[ForestGreen, ultra thick] (-23,43) -- (-28,48);
		\draw[ForestGreen, ultra thick] (-20,23) -- (-20,27);
		\draw[ForestGreen, ultra thick] (-20,33) -- (-20,37);
		\draw[ForestGreen, ultra thick] (-30,33) -- (-30,37);
		\draw[ForestGreen, ultra thick] (-30,43) -- (-30,47);
		\draw[ForestGreen, ultra thick] (20,23) -- (20,27);
		\draw[ForestGreen, ultra thick] (20,33) -- (20,37);
		\draw[ForestGreen, ultra thick] (17,43) -- (2,58);
		\draw[ForestGreen, ultra thick] (-26,52) -- (-3,59);
		\node () at (0,0){\begin{tikzpicture}[scale=0.2,baseline=0pt]
			\draw[blue, thick] (0,-1) -- (0,0);
			\draw[blue, thick] (0,0) -- (3,3);
			\draw[blue, thick] (0,0) -- (-3,3);
			\draw[blue, thick] (1,1) -- (-1,3);
			\draw[blue, thick] (2,2) -- (1,3);
			\filldraw[black] (0,0) circle (5pt)  {};
			\filldraw[black] (1,1) circle (5pt)  {};
			\filldraw[black] (2,2) circle (5pt)  {};
			\node at (0,0.7)[font=\fontsize{7pt}{0}]{$1$};
			\node at (1,1.7)[font=\fontsize{7pt}{0}]{$2$};
			\node at (2,2.7)[font=\fontsize{7pt}{0}]{$3$};\end{tikzpicture}};
		\node () at (-10,10){\begin{tikzpicture}[scale=0.2,baseline=0pt]
			\draw[blue, thick] (0,-1) -- (0,0);
			\draw[blue, thick] (0,0) -- (3,3);
			\draw[blue, thick] (0,0) -- (-3,3);
			\draw[blue, thick] (1,1) -- (-1,3);
			\draw[blue, thick] (2,2) -- (1,3);
			\filldraw[black] (0,0) circle (5pt)  {};
			\filldraw[black] (1,1) circle (5pt)  {};
			\filldraw[black] (2,2) circle (5pt)  {};
			\node at (0,0.7)[font=\fontsize{7pt}{0}]{$1$};
			\node at (1,1.7)[font=\fontsize{7pt}{0}]{$3$};
			\node at (2,2.7)[font=\fontsize{7pt}{0}]{$2$};\end{tikzpicture}};
		\node () at (10,10){\begin{tikzpicture}[scale=0.2,baseline=0pt]
			\draw[blue, thick] (0,-1) -- (0,0);
			\draw[blue, thick] (0,0) -- (3,3);
			\draw[blue, thick] (0,0) -- (-3,3);
			\draw[blue, thick] (1,1) -- (-1,3);
			\draw[blue, thick] (2,2) -- (1,3);
			\filldraw[black] (0,0) circle (5pt)  {};
			\filldraw[black] (1,1) circle (5pt)  {};
			\filldraw[black] (2,2) circle (5pt)  {};
			\node at (0,0.7)[font=\fontsize{7pt}{0}]{$2$};
			\node at (1,1.7)[font=\fontsize{7pt}{0}]{$1$};
			\node at (2,2.7)[font=\fontsize{7pt}{0}]{$3$};\end{tikzpicture}};
			\node () at (10,20){\begin{tikzpicture}[scale=0.2,baseline=0pt]
			\draw[blue, thick] (0,-1) -- (0,0);
			\draw[blue, thick] (0,0) -- (3,3);
			\draw[blue, thick] (0,0) -- (-3,3);
			\draw[blue, thick] (1,1) -- (-1,3);
			\draw[blue, thick] (2,2) -- (1,3);
			\filldraw[black] (0,0) circle (5pt)  {};
			\filldraw[black] (1,1) circle (5pt)  {};
			\filldraw[black] (2,2) circle (5pt)  {};
			\node at (0,0.7)[font=\fontsize{7pt}{0}]{$3$};
			\node at (1,1.7)[font=\fontsize{7pt}{0}]{$1$};
			\node at (2,2.7)[font=\fontsize{7pt}{0}]{$2$};\end{tikzpicture}};
		\node () at (-10,20){\begin{tikzpicture}[scale=0.2,baseline=0pt]
			\draw[blue, thick] (0,-1) -- (0,0);
			\draw[blue, thick] (0,0) -- (3,3);
			\draw[blue, thick] (0,0) -- (-3,3);
			\draw[blue, thick] (1,1) -- (-1,3);
			\draw[blue, thick] (2,2) -- (1,3);
			\filldraw[black] (0,0) circle (5pt)  {};
			\filldraw[black] (1,1) circle (5pt)  {};
			\filldraw[black] (2,2) circle (5pt)  {};
			\node at (0,0.7)[font=\fontsize{7pt}{0}]{$2$};
			\node at (1,1.7)[font=\fontsize{7pt}{0}]{$3$};
			\node at (2,2.7)[font=\fontsize{7pt}{0}]{$1$};\end{tikzpicture}};
		\node () at (0,30){\begin{tikzpicture}[scale=0.2,baseline=0pt]
			\draw[blue, thick] (0,-1) -- (0,0);
			\draw[blue, thick] (0,0) -- (3,3);
			\draw[blue, thick] (0,0) -- (-3,3);
			\draw[blue, thick] (1,1) -- (-1,3);
			\draw[blue, thick] (2,2) -- (1,3);
			\filldraw[black] (0,0) circle (5pt)  {};
			\filldraw[black] (1,1) circle (5pt)  {};
			\filldraw[black] (2,2) circle (5pt)  {};
			\node at (0,0.7)[font=\fontsize{7pt}{0}]{$3$};
			\node at (1,1.7)[font=\fontsize{7pt}{0}]{$2$};
			\node at (2,2.7)[font=\fontsize{7pt}{0}]{$1$};	\end{tikzpicture}};
		\node () at (20,20){\begin{tikzpicture}[scale=0.2,baseline=0pt]
			\draw[blue, thick] (0,-1) -- (0,0);
			\draw[blue, thick] (0,0) -- (3,3);
			\draw[blue, thick] (0,0) -- (-3,3);
			\draw[blue, thick] (2,2) -- (1,3);
			\draw[blue, thick] (-2,2) -- (-1,3);
			\filldraw[black] (0,0) circle (5pt)  {};
			\filldraw[black] (2,2) circle (5pt)  {};
			\filldraw[black] (-2,2) circle (5pt)  {};
			\node at (-2,2.7)[font=\fontsize{7pt}{0}]{$2$};
			\node at (0,0.7)[font=\fontsize{7pt}{0}]{$1$};
			\node at (2,2.7)[font=\fontsize{7pt}{0}]{$3$};\end{tikzpicture}};
		\node () at (20,30){\begin{tikzpicture}[scale=0.2,baseline=0pt]
			\draw[blue, thick] (0,-1) -- (0,0);
			\draw[blue, thick] (0,0) -- (3,3);
			\draw[blue, thick] (0,0) -- (-3,3);
			\draw[blue, thick] (2,2) -- (1,3);
			\draw[blue, thick] (-2,2) -- (-1,3);
			\filldraw[black] (0,0) circle (5pt)  {};
			\filldraw[black] (2,2) circle (5pt)  {};
			\filldraw[black] (-2,2) circle (5pt)  {};
			\node at (-2,2.7)[font=\fontsize{7pt}{0}]{$3$};
			\node at (0,0.7)[font=\fontsize{7pt}{0}]{$1$};
			\node at (2,2.7)[font=\fontsize{7pt}{0}]{$2$};\end{tikzpicture}};
		\node () at (20,40){\begin{tikzpicture}[scale=0.2,baseline=0pt]
			\draw[blue, thick] (0,-1) -- (0,0);
			\draw[blue, thick] (0,0) -- (3,3);
			\draw[blue, thick] (0,0) -- (-3,3);
			\draw[blue, thick] (2,2) -- (1,3);
			\draw[blue, thick] (-2,2) -- (-1,3);
			\filldraw[black] (0,0) circle (5pt)  {};
			\filldraw[black] (2,2) circle (5pt)  {};
			\filldraw[black] (-2,2) circle (5pt)  {};
			\node at (-2,2.7)[font=\fontsize{7pt}{0}]{$3$};
			\node at (0,0.7)[font=\fontsize{7pt}{0}]{$2$};
			\node at (2,2.7)[font=\fontsize{7pt}{0}]{$1$};\end{tikzpicture}};
		\node () at (-20,20){\begin{tikzpicture}[scale=0.2,baseline=0pt]
			\draw[blue, thick] (0,-1) -- (0,0);
			\draw[blue, thick] (0,0) -- (3,3);
			\draw[blue, thick] (0,0) -- (-3,3);
			\draw[blue, thick] (1,1) -- (-1,3);
			\draw[blue, thick] (0,2) -- (1,3);
			\filldraw[black] (0,0) circle (5pt)  {};
			\filldraw[black] (1,1) circle (5pt)  {};
			\filldraw[black] (0,2) circle (5pt)  {};
			\node at (0,0.7)[font=\fontsize{7pt}{0}]{$1$};
			\node at (0,2.7)[font=\fontsize{7pt}{0}]{$3$};
			\node at (1,1.7)[font=\fontsize{7pt}{0}]{$2$};\end{tikzpicture}};
		\node () at (-20,30){\begin{tikzpicture}[scale=0.2,baseline=0pt]
			\draw[blue, thick] (0,-1) -- (0,0);
			\draw[blue, thick] (0,0) -- (3,3);
			\draw[blue, thick] (0,0) -- (-3,3);
			\draw[blue, thick] (1,1) -- (-1,3);
			\draw[blue, thick] (0,2) -- (1,3);
			\filldraw[black] (0,0) circle (5pt)  {};
			\filldraw[black] (1,1) circle (5pt)  {};
			\filldraw[black] (0,2) circle (5pt)  {};
			\node at (0,0.7)[font=\fontsize{7pt}{0}]{$2$};
			\node at (0,2.7)[font=\fontsize{7pt}{0}]{$3$};
			\node at (1,1.7)[font=\fontsize{7pt}{0}]{$1$};\end{tikzpicture}};
		\node () at (-20,40){\begin{tikzpicture}[scale=0.2,baseline=0pt]
			\draw[blue, thick] (0,-1) -- (0,0);
			\draw[blue, thick] (0,0) -- (3,3);
			\draw[blue, thick] (0,0) -- (-3,3);
			\draw[blue, thick] (1,1) -- (-1,3);
			\draw[blue, thick] (0,2) -- (1,3);
			\filldraw[black] (0,0) circle (5pt)  {};
			\filldraw[black] (1,1) circle (5pt)  {};
			\filldraw[black] (0,2) circle (5pt)  {};
			\node at (0,0.7)[font=\fontsize{7pt}{0}]{$3$};
			\node at (0,2.7)[font=\fontsize{7pt}{0}]{$2$};
			\node at (1,1.7)[font=\fontsize{7pt}{0}]{$1$};\end{tikzpicture}};
		\node () at (-30,30){\begin{tikzpicture}[scale=0.2,baseline=0pt]
			\draw[blue, thick] (0,-1) -- (0,0);
			\draw[blue, thick] (0,0) -- (3,3);
			\draw[blue, thick] (0,0) -- (-3,3);
			\draw[blue, thick] (-1,1) -- (1,3);
			\draw[blue, thick] (0,2) -- (-1,3);
			\filldraw[black] (0,0) circle (5pt)  {};
			\filldraw[black] (-1,1) circle (5pt)  {};
			\filldraw[black] (0,2) circle (5pt)  {};
			\node at (-1,1.7)[font=\fontsize{7pt}{0}]{$1$};
			\node at (0,2.7)[font=\fontsize{7pt}{0}]{$3$};
			\node at (0,0.7)[font=\fontsize{7pt}{0}]{$2$};\end{tikzpicture}};
		\node () at (-30,40){\begin{tikzpicture}[scale=0.2,baseline=0pt]
			\draw[blue, thick] (0,-1) -- (0,0);
			\draw[blue, thick] (0,0) -- (3,3);
			\draw[blue, thick] (0,0) -- (-3,3);
			\draw[blue, thick] (-1,1) -- (1,3);
			\draw[blue, thick] (0,2) -- (-1,3);
			\filldraw[black] (0,0) circle (5pt)  {};
			\filldraw[black] (-1,1) circle (5pt)  {};
			\filldraw[black] (0,2) circle (5pt)  {};
			\node at (-1,1.7)[font=\fontsize{7pt}{0}]{$2$};
			\node at (0,2.7)[font=\fontsize{7pt}{0}]{$3$};
			\node at (0,0.7)[font=\fontsize{7pt}{0}]{$1$};\end{tikzpicture}};
		\node () at (-30,50){\begin{tikzpicture}[scale=0.2,baseline=0pt]
			\draw[blue, thick] (0,-1) -- (0,0);
			\draw[blue, thick] (0,0) -- (3,3);
			\draw[blue, thick] (0,0) -- (-3,3);
			\draw[blue, thick] (-1,1) -- (1,3);
			\draw[blue, thick] (0,2) -- (-1,3);
			\filldraw[black] (0,0) circle (5pt)  {};
			\filldraw[black] (-1,1) circle (5pt)  {};
			\filldraw[black] (0,2) circle (5pt)  {};
			\node at (-1,1.7)[font=\fontsize{7pt}{0}]{$3$};
			\node at (0,2.7)[font=\fontsize{7pt}{0}]{$2$};
			\node at (0,0.7)[font=\fontsize{7pt}{0}]{$1$};\end{tikzpicture}};
		\node () at (0,60){\begin{tikzpicture}[scale=0.2,baseline=0pt]
			\draw[blue, thick] (0,-1) -- (0,0);
			\draw[blue, thick] (0,0) -- (-3,3);
			\draw[blue, thick] (0,0) -- (3,3);
			\draw[blue, thick] (-1,1) -- (1,3);
			\draw[blue, thick] (-2,2) -- (-1,3);
			\filldraw[black] (0,0) circle (5pt)  {};
			\filldraw[black] (-1,1) circle (5pt)  {};
			\filldraw[black] (-2,2) circle (5pt)  {};
			\node at (-2,2.7)[font=\fontsize{7pt}{0}]{$3$};
			\node at (-1,1.7)[font=\fontsize{7pt}{0}]{$2$};
			\node at (0,0.7)[font=\fontsize{7pt}{0}]{$1$};\end{tikzpicture}};
		\end{tikzpicture}
	\end{center}
	\caption{\sl Parking order on parking functions of degree 3.}\label{fig:PK3}
\end{figure}
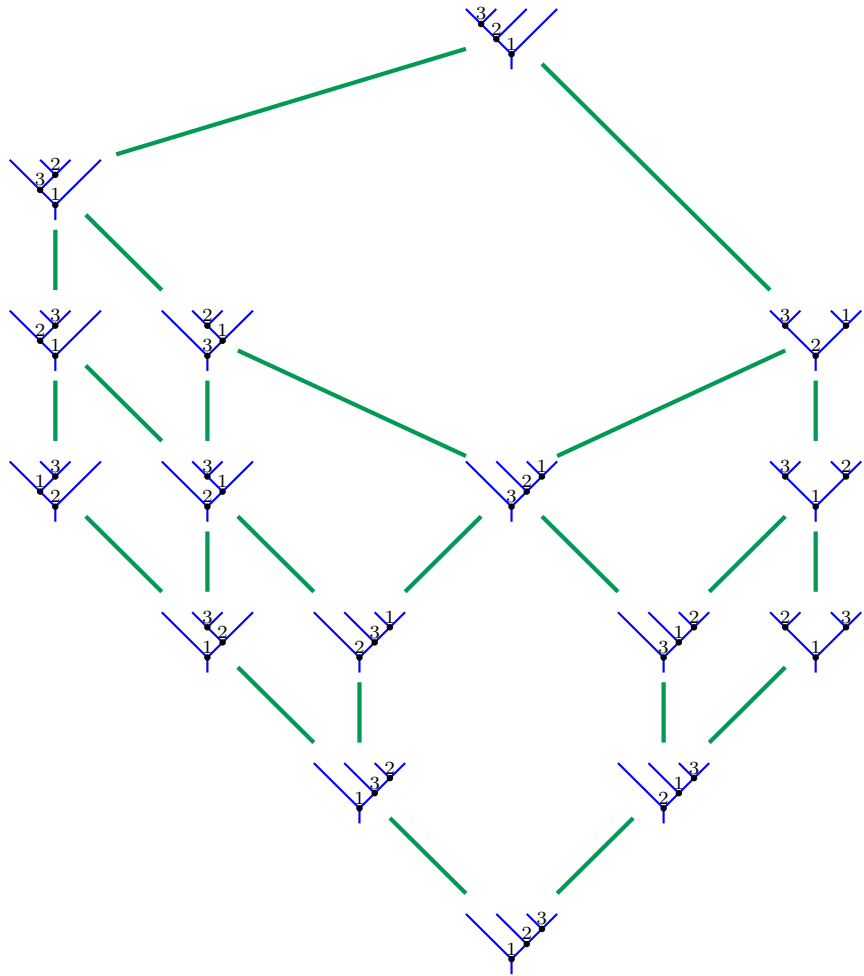

\begin{Lemma}\label{lem:parking-order}
The poset $\leq_{P}$ admits least upper bounds. More specifically, given two parking functions $f$ and $g$ of the same degree, $f\vee g$ is the parking function $(\pi_{\perm}f \vee \pi_{\perm}g,\pi_{\tree}f \vee \pi_{\tree}g)$.
\end{Lemma}

\begin{proof}
We first make the following claims.

\begin{enumerate}
	\item If $s,t$ are two binary trees of the same degree, then $\Des(s\vee t)=\Des(s)\cup\Des(t)$,
	\item If $\sigma,\tau$ are two permutations of the same degree, then $\Des(\sigma\vee \tau)=\Des(\sigma)\cup\Des(\tau)$.
\end{enumerate}

First note that $(2)\implies(1)$: let $\sigma=\iota(s)$ and $\tau=\iota(t)$ so that $\Des(\sigma)=\Des(\pi \sigma)= \Des(s)$ and similarly $\Des(\tau)=\Des(t)$. So $\Des(s)\cup \Des(t) = \Des(\sigma) \cup \Des(\tau)=\Des(\sigma\vee \tau)=\Des(\pi(\sigma\vee \tau))=\Des(\pi \sigma \vee \pi \tau)=\Des(s\vee t)$, using Proposition \ref{prop:pi-property-bin}.(3) in the second to last equality.

To see $(2)$, first note that from the definition of weak order we have that $\Inv(\sigma)\cup\Inv(\tau)\subseteq\Inv(\sigma\vee\tau)$. In particular, it must also hold that  $\Des(\sigma)\cup\Des(\tau)\subseteq\Des(\sigma\vee\tau)$.

Conversely, as was noted in \cite{M94}, $\Inv(\sigma\vee\tau)$ is the transitive closure of $\Inv(\sigma)\cup\Inv(\tau)$ i.e. $(i,j),(j,k)\in\Inv(\sigma\vee\tau)\implies(i,k)\in\Inv(\sigma\vee\tau)$ . Since $(i,i+1)$ cannot come from transitivity of other inversions, we must have $(i,i+1)\in\Inv(\sigma\vee\tau)\implies (i,i+1)\in\Inv(\sigma)\cup\Inv(\tau)$.

Let $f=(\sigma,s)$ and $g=(\tau,t)$. By definition, we know that $\Des(s)\subseteq\Des(\sigma)$ and $\Des(t)\subseteq\Des(\tau)$. Therefore, by the claims, $\Des(s\vee t)\subseteq\Des(\sigma\vee\tau)$ i.e. $(\sigma\vee\tau,s\vee t)$ is a parking function. Then it is obviously $f\vee g$.
\end{proof}

Let $f,g$ be two parking functions. It is not hard to see that $f/g$ and $f\backslash g$ are parking functions, since the only new descent in the underlying tree is $\deg(f)$ in $f/g$, and it is also a descent in the corresponding permutation because the labels in $f$ are shifted up. In $f\backslash g$ there are no new descents.

\begin{Proposition}\label{prop:interval_multiplication_PSYM}
	Let $f$ and $g$ be two parking functions, and let $F_f^*$, $F_g^*$ be the dual basis element in the graded dual $\PSym^*$. Then,
		$$m(F_f^*\otimes F_g^*)=\sum_{f\backslash g\leq_{P} h\leq_{P} f/g}F_h^*.$$
\end{Proposition}

\begin{proof}
	Fix a parking function $h$ and $0\leq i\leq\deg(h)$. Let $\Delta_i(F_h)=F_{{}^ih}\otimes F_{h^i}$. It follows immediately from the definition that we have $\langle F_f^*\otimes F_g^*,\Delta_i(F_h) \rangle=1$ for a unique pair $(f,g)$, or $\langle F_f^*\otimes F_g^*,\Delta_i(F_h) \rangle=0$ otherwise. We claim that the following are equivalent.
	\begin{enumerate}
		\item $\langle F_f^*\otimes F_g^*,\Delta_i(F_h) \rangle=1$.
		\item \begin{enumerate}
			\item $\langle F_{\pi_{\perm}(f)}^*\otimes F_{\pi_{\perm}(g)}^*,\Delta_i(F_{\pi_{\perm}(h)}) \rangle=1$ and
			\item $\langle F_{\pi_{\tree}(f)}^*\otimes F_{\pi_{\tree}(g)}^*,\Delta_i(F_{\pi_{\tree}(h)}) \rangle=1$.
		\end{enumerate}
		\item \begin{enumerate}
			\item $\pi_\perm(f)\backslash \pi_\perm(g)\leq_{w} \pi_\perm(h)\leq_{w} \pi_\perm(f)/\pi_\perm(g)$ and
			\item $\pi_\tree(f)\backslash \pi_\tree(g)\leq_{T} \pi_\tree(h)\leq_{T} \pi_\tree(f)/\pi_\tree(g)$.
			\item $\deg(f)=i$
		\end{enumerate}
		\item $f\backslash g\leq_{P} h\leq_{P} f/g$ and $\deg(f)=i$.
	\end{enumerate}
	The equivalence of (2) and (3) follows from Propositions \ref{prop:tree} and \ref{prop:perm}. The equivalence of (3) and (4) follows from the easy facts $\pi_\perm(f/g)=\pi_\perm(f)/\pi_\perm(g)$ and other similar relation for binary trees. $(1)\Rightarrow(2)$ follows immediately from the definition of $\Delta_i(h)$. Given $(2)$, the 4-tuple of permutations and binary trees $(\pi_\perm(f),\pi_\perm(g),\pi_\tree(f),\pi_\tree(g))$ is unique. The uniqueness leads to (1) because we know that there exists $f,g$ satisfying (1) and therefore satisfying (2).
	
	The equivalent of (1) and (4), together with the relation $\langle m(F_f^*\otimes F_g^*),F_h \rangle=\langle F_f^*\otimes F_g^*,\Delta(F_h) \rangle=0$ or $1$, completes the proof.
\end{proof}

\begin{Definition}
	The monomial basis of parking functions are defined as the elements satisfying $\displaystyle F_f=\sum_{g\geq_{P} f}M_g$. The monomial basis is unique by triangularity and it forms a basis via M\"{o}bius inversion.
\end{Definition}

\begin{Definition} \label{def:parking-globaldescent}
	Let $f$ be a parking function of degree $n$. Its global descent set $\GD(f)$ is $\{i\in[n-1]:i\in \GD(\pi_\perm(f))\cap \GD(\pi_\tree(f))\}$.
\end{Definition}

\begin{Proposition}\label{prop:coproduct-monomial-parking} The coproduct
in the monomial basis of $\PSym$ is:
\[
\Delta_{+}(M_{f})=\sum_{i\in\GD(f)}M_{^{i}f}\otimes M_{f^{i}}.
\]
\end{Proposition}
\begin{proof}
We check the axioms in Theorem \ref{thm:monomialcoproduct}. Axiom
($\Delta$1) is immediate because $\allow(f)=\{1,\dots,\deg f-1\}$).

To check ($\Delta$2): if $f\leq_{P}f'$, then $\pi_{\perm}(f)\leq_{w}\pi_{\perm}(f')$,
so by condition $\Delta3$ on $\SSym$, we have $\pi_{\perm}({}^{i}f)={}^{i}(\pi_{\perm}f)\leq_{w}{}^{i}(\pi_{\perm}f')=\pi_{\perm}({}^{i}f')$,
and similarly $\pi_{\tree}({}^{i}f)\leq_{T}\pi_{\tree}({}^{i}f')$,
so ${}^{i}f\leq_{P}{}^{i}f'$. By the same argument, we also
have $f^{i}\leq_{P}f'^{i}$.

To check ($\Delta$3): it follows from Proposition \ref{prop:interval_multiplication_PSYM} that $g/h=\max\{f|{}^{i}f=g,f^{i}=h\}$. That $/$ is order-preserving is easily proved using a similar argument to ($\Delta$2) above.

Finally, we check that Definition \ref{def:parking-globaldescent}
of a global descent agrees with the definition in Section \ref{sec:axioms-coproduct}.
For this, the following statements are equivalent:
\begin{enumerate}
\item $i\in\GD(f)$.
\item $i\in\GD(\pi_{\perm}f)$ and $i\in\GD(\pi_{\tree}f)$.
\item $\pi_{\perm}f={}^{i}(\pi_{\perm}f)/(\pi_{\perm}f)^{i}=\pi_{\perm}({}^{i}f/f^{i})$, $\pi_{\tree}f={}^{i}(\pi_{\tree}f)/(\pi_{\tree}f)^{i}=\pi_{\tree}({}^{i}f/f^{i})$.
\item $f={}^{i}f/f^{i}$.
\end{enumerate}
The equivalence of (1) and (2) is Definition \ref{def:parking-globaldescent}.
The equivalence of (2) and (3) is the definition of global descent
from Theorem \ref{thm:monomialcoproduct}, in $\SSym$ and $\YSym$.
The equivalence of (3) and (4) is because $f$ is uniquely determined
by its images under $\pi_{\perm}$ and $\pi_{\tree}$.
\end{proof}

\begin{Proposition}\label{prop:product-monomial-parking} The product
of monomial basis elements in $\PSym$ is given by $M_{f}\cdot M_{g}=\sum_{h}\alpha_{f,g}^{h} M_{h}$,
where $\alpha_{f,g}^{h} =|A_{f,g}^{h}|$ is defined as in (\ref{eq:a-def}).
\end{Proposition}

\begin{proof}
We check the axioms in Theorem \ref{thm:monomialproduct}. Axioms
($m$0) follows from Lemma \ref{lem:parking-order}. Axiom ($m$1) follows from 
the definition of multiplication in $\PSym$ using shuffles as defined in Example~\ref{ex:shuffle}. Axiom ($m$2) may be verified in the
same manner as for ($\Delta$2) in the proof of Proposition \ref{prop:coproduct-monomial-parking},
i.e. by applying axiom ($m$2) for $\SSym$ and $\YSym$ to the images
of $f$ and $g$ under $\pi_{\perm}$ and $\pi_{\tree}$.

To check ($m$3), first observe that $\pi_{\perm}\zeta(f,g)=\zeta(\pi_{\perm}f,\pi_{\perm}g)$
and similarly for $\pi_{\tree}$. Using this, together with Lemma
\ref{lem:parking-order} in the second and fifth lines, and condition $m$4 for $\SSym$ in the third line, 
\begin{align*}
\pi_{\perm}\zeta(f_{1}\vee f_{2},g_{1}\vee g_{2}) & =\zeta(\pi_{\perm}(f_{1}\vee f_{2}),\pi_{\perm}(g_{1}\vee g_{2}))\\
 & =\zeta(\pi_{\perm}f_{1}\vee\pi_{\perm}f_{2},\pi_{\perm}g_{1}\vee\pi_{\perm}g_{2})\\
 & \leq_{w}\zeta(\pi_{\perm}f_{1},\pi_{\perm}g_{1})\vee\zeta(\pi_{\perm}f_{2},\pi_{\perm}g_{2})\\
 & =\pi_{\perm}\zeta(f_{1},g_{1})\vee\pi_{\perm}\zeta(f_{2},g_{2})\\
 & =\pi_{\perm}(\zeta(f_{1},g_{1})\vee\zeta(f_{2},g_{2})).
\end{align*}
By the same argument, $\pi_{\tree}\zeta(f_{1}\vee f_{2},g_{1}\vee g_{2})\leq_{T}\pi_{\tree}(\zeta(f_{1},g_{1})\vee\zeta(f_{2},g_{2}))$,
so $\zeta(f_{1}\vee f_{2},g_{1}\vee g_{2})\leq_{P}\zeta(f_{1},g_{1})\vee\zeta(f_{2},g_{2})$.
\end{proof}
\begin{Proposition}\label{prop:antipode-monomial-parking} The
antipode of monomial basis elements in $\PSym$ is given by $\mathcal{S}(M_{f})=(-1)^{\GD(f)+1}\sum_{g}\beta_{f}^{g} M_{g}$
where $\beta_{f}^{g}$ is defined as in (\ref{eq:c-def-GD-2}).
\end{Proposition}
\begin{proof}
We check the axioms for Theorem \ref{thm:antipode}. Axioms ($\Delta$1)-($\Delta$3)
and ($m$0)-($m$3) were already checked above.

Axiom ($\mathcal{S}$0) follows from Remark \ref{rem:s0trees}. Axiom  ($\mathcal{S}$1)
follows readily from Definition \ref{def:parking-globaldescent} of
$\GD(f)$, together with condition ($\mathcal{S}$1) for $\SSym$ and
$\YSym$. Axioms ($\mathcal{S}$2) and ($\mathcal{S}$3) can be checked like
for ($\Delta$2) in the proof of Proposition \ref{prop:coproduct-monomial-parking}:
apply the corresponding axioms in $\SSym$ and $\YSym$ to the images
of $f_{i}$ under $\pi_{\perm}$ and $\pi_{\tree}$.
\end{proof}

\begin{Proposition}The Hopf projections $\Pi_{\tree}:\PSym\rightarrow\YSym$
and $\Pi_{\perm}:\PSym\rightarrow\SSym$  on the monomial basis is given
by:
\[
\Pi_{\tree}(M_{f})=\begin{cases}
M_{\pi_{\tree}(f)} & \text{if }\pi_{\perm}(f)=n\ n-1\ \cdots\ 1;\\
0 & \text{otherwise}.
\end{cases}
\]
\[
\Pi_{\perm}(M_{f})=\begin{cases}
M_{\pi_{\perm}(f)} & \text{if }f=\max\{g:\pi_{\perm}(g)=\pi_{\perm}(f)\};\\
0 & \text{otherwise}.
\end{cases}
\]
\end{Proposition}
\begin{proof}
Use Theorem \ref{thm:monomialquotient}: by definition of $\leq_{P}$,
both $\pi_{\tree}$ and $\pi_{\perm}$ are order-preserving. Following
axiom ($\pi$0), 
\begin{align*}
   \iota_{\tree}(t)&=\max\{(\sigma,t)\text{ a parking function}\}\\&=(\max\{\sigma:\Des(t)\subseteq\Des(\sigma)\},t)=(n\ n-1\ \cdots\ 1,t),
   \end{align*}
and $\iota_{\tree}$ is clearly order-preserving. Similarly,
 \begin{align*}
\iota_{\perm}(\sigma)&=\max\{(\sigma,t)\text{ a parking function}\}\\&=(\sigma,\max\{t:\Des(t)\subseteq\Des(\sigma)\}),
\end{align*}
which exists by claim (1) in the proof of Lemma \ref{lem:parking-order}.
To show $\iota_{\perm}$ is order-preserving, we show
that, if $\sigma\geq_{w}\tau$, then $\max\{t:\Des(t)\subseteq\Des(\sigma)\}\geq_{T}\max\{t:\Des(t)\subseteq\Des(\tau)\}$.
This is true because $\sigma\geq_{w}\tau$ means $\Inv(\sigma)\supseteq\Inv(\tau)$
so $\Des(\sigma)\supseteq\Des(\tau)$.
\end{proof}

\begin{Remark}
Our bijection from classical parking functions to our labeled-tree visualization can be extended to $m$-parking functions; the resulting $m+1$-ary trees have labels satisfying $\Des(t)\subseteq \Des(\sigma)$ for some generalized definition of descent. With a multiplication and comultiplication analogous to that of $\STSym$, these $m$-parking functions form a Hopf algebra, that is likely to be isomorphic to the one defined in \cite{NT20}. By imposing the condition $\Des(t)\subseteq \Des(\sigma)$ on labeled trees of any arity, we obtain a Hopf algebra of ``planar parking functions", which contains the $m$-parking function algebras.
\end{Remark}

\section{Proofs of the monomial bases properties}\label{sec:proofs}

Throughout this section, let $P_{n}$ be a poset for each integer
$n\geq0$. Let $\mathcal{H}$ be a graded vector space with a fundamental
basis $\{F_{f}:f\in P_{n}\}$ and a monomial basis $\{M_{f}:f\in P_{n}\}$,
related by $F_{f}=\sum_{g\geq f}M_{g}$. Let $C_{j}$ denote a connected
component of a $P_{n}$.

\subsection{Proof of the coproduct formulas in the monomial basis}
Proving the iterated coproduct formula (\ref{eq:mcoproductformula-multi})
and the antipode formula require the following $k$-factor generalizations
of the comultiplication axioms. Recall that for $S\subseteq\allow(f)$, we denote by $f|S$  the splitting of $f$ at $S$.
\begin{enumerate}
\item[($\Delta1'$).]  For each $f\in P_{n}$,
\begin{equation}
\Delta_{+}^{[k]}(F_{f})=\sum_{S\subseteq\allow(f),\ |S|=k-1}F_{(f|S)_{1}}\otimes\cdots\otimes F_{(f|S)_{k}}.\label{eq:fcoproductformula-multi}
\end{equation}
\item[($\Delta2'$).]  If $f\leq f'\in P_{n}$, then $f|S\leq f'|S$ componentwise.
\item[($\Delta3'$).]  Given $f_{1},f_{2},\dots,f_{k}$ with $f_{i}\in P_{n_{i}}$, let
$S=\{n_{1},n_{1}+n_{2},\dots,n_{1}+\cdots+n_{k-1}\}$, and $n=n_{1}+\cdots+n_{k}$.
Then 
\[f_{1}/\cdots/f_{k}=\max\{f\in P_{n}:f|S=(f_{1},f_{2},\dots,f_{k})\}.\]
If $f_{i}\leq f'_{i}$ for all $i$, then \[f_{1}/\cdots/f_{k}\leq f_{1}'/\cdots/f_{k}';\]
and $/$ is an associative operation. 
\item[($\Delta4'$).]  Given $f\in P_{n}$ and $S\subseteq\allow(f)$ with $|S|=k-1$ and
$f|S=(f_{1},\dots f_{k})$, we have $S\subseteq\GD(f)$ if and only
if $f=f_{1}/\cdots/f_{k}$.
\end{enumerate}

\begin{Lemma}\label{lem:kfactor-coproductaxiom} Suppose $\mathcal{H}$
satisfies axioms ($\Delta$1)-($\Delta$3). Then, for each $f\in P_{n}$ and each
$i<j$ with $j\in\allow(f)$, we have $i\in\allow(f)$ if and only
if $i\in\allow({}^{j}f)$. And $\mathcal{H}$
satisfies ($\Delta1'$)-($\Delta4'$) for all $k\geq1$. \end{Lemma}

\begin{proof}
Calculate $\Delta_{+}^{[3]}(F_{f})$ in two ways:
by coassociativity,
\begin{align}
\Delta_{+}^{[3]}(F_{f}) & =\sum_{j\in\allow(f)}\Delta_{+}F_{^{j}f}\otimes F_{f^{j}}\label{eq:coassociativity1}\\ 
 & =\sum_{j\in\allow(f)}\sum_{i\in\allow({}^{j}f)}F_{^{i}\left(^{j}f\right)}\otimes F_{\left(^{j}f\right)^{i}}\otimes F_{f^{j}},\nonumber 
\end{align}
where the degrees of $F_{{}^{i}\left(^{j}f\right)}$, $F_{\left({}^{j}f\right)^{i}}$, $F_{f^{j}}$
are respectively $i$, $j-i$, $n-j$, and
\begin{align} 
\Delta_{+}^{[3]}(F_{f}) & =\sum_{i\in\allow(f)}F_{^{i}f}\otimes\Delta_{+}F_{f^{i}}\label{eq:coassociativity2}  \\  
 & =\sum_{i\in\allow(f)}\sum_{j'\in\allow(f^{i})}F_{^{i}f}\otimes F_{^{j'}\left(f^{i}\right)}\otimes F_{\left(f^{i}\right)^{j'}},\nonumber 
\end{align}
where the degrees of $F_{{}^{i}f}$, $F_{{}^{j'}\left(f^{i}\right)}$, $F_{\left(f^{i}\right)^{j'}}$
are respectively $i$, $j'$, $n-i-j'$.  We have $i\in\allow(f)$ if and only if there is a non-term with first tensor of degree $i$ in \eqref{eq:coassociativity2}.
The equivalence of \eqref{eq:coassociativity1} and \eqref{eq:coassociativity2} shows that $i\in\allow(f)$ if and only if $i\in\allow({}^{j}f)$.

To see ($\Delta1'$): because $\Delta_{+}^{[k]}=(\Delta_{+}^{[k-1]}\otimes\iota)\circ\Delta_{+}$:
\begin{align*}
\Delta_{+}^{[k]}(F_{f}) & =\sum_{j\in\allow(f)}\Delta_{+}^{[k-1]}F_{{}^{j}f}\otimes F_{f^{j}}\\
 & =\sum_{j\in\allow(f)}\sum_{i_{1},\dots,i_{k-2}\in \allow({}^{j}f)}F_{({}^{j}f|\{i_{1},\dots,i_{k-2}\})_{1}}\otimes\cdots\otimes F_{({}^{j}f|\{i_{1},\dots,i_{k-2}\})_{k-1}}\otimes F_{f^{j}},
\end{align*}
by induction on $k$.
By the equivalence of \eqref{eq:coassociativity1} and \eqref{eq:coassociativity2}, we have the equivalent conditions \[j\in\allow(f)
\text{ and } i_{1},\dots,i_{k-2}\in\allow({}^{j}f) \Longleftrightarrow i_{1},\dots,i_{k-2},j\in\allow(f).\]
Now use definition (\ref{eq:deconcatenation-defn}).

Axiom ($\Delta2'$) is easy to show by induction on $|S|$, using (\ref{eq:deconcatenation-defn}).

For ($\Delta3'$), to see the associativity of $/$,
first consider the case $k=3$, so 
\begin{equation}
(f_{1}/f_{2})/f_{3}=\max\left\{ f:{}^{n_{1}+n_{2}}f=(f_{1}/f_{2}),f^{n_{1}+n_{2}}=f_{3}\right\} .\label{eq:/associativity}
\end{equation}
Consider $g=\max\left\{ f:f|S=(f_{1},f_{2},f_{3})\right\} $. Since
$((f_{1}/f_{2})/f_{3})|S=(f_{1},f_{2},f_{3})$, so $g\geq(f_{1}/f_{2})/f_{3}$.
By ($\Delta$2), this implies \[{}^{n_{1}+n_{2}}g\geq{}^{n_{1}+n_{2}}((f_{1}/f_{2})/f_{3})=f_{1}/f_{2}.\]
Separately, by (\ref{eq:coassociativity1}), ${}^{n_{1}}({}^{n_{1}+n_{2}}g)=f_{1}$
and $({}^{n_{1}+n_{2}}g)^{n_{1}}=f_{2}$, so by definition of $f_{1}/f_{2}$,
we have ${}^{n_{1}+n_{2}}g\leq f_{1}/f_{2}$, which implies ${}^{n_{1}+n_{2}}g=f_{1}/f_{2}$.
Hence, $g$ satisfies the condition in (\ref{eq:/associativity}),
and thus $g=(f_{1}/f_{2})/f_{3}$. A similar argument shows $g=f_{1}/(f_{2}/f_{3})$.

For $k>3$, it is clear that $(f_{1}/\cdots/f_{k})|S=((f_{1}/\cdots/f_{k-1})/f_{k})|S=(f_{1},f_{2},\dots,f_{k})$,
using (\ref{eq:deconcatenation-defn}) and induction on $k$. Now take
$f'$ with $f'|S=(f_{1},f_{2},\dots,f_{k})$, and let $j$ be the
maximal element of $S$. By (\ref{eq:deconcatenation-defn}), $f'|\{j\}=({}^{j}f,f_{k})$
and ${}^{j}f|(S\backslash\{j\})=(f_{1},\dots,f_{k-1})$. So, by inductive
hypothesis, $f'\leq{}^{j}f/f_{k}$ and ${}^{j}f\leq f_{1}/\cdots/f_{k-1}$.
Because $/$ is order-preserving, this means 
\[f'\leq{}^{j}f/f_{k}\leq(f_{1}/\cdots/f_{k-1})/f_{k}=f_{1}/\cdots/f_{k}.\] 
Hence $f_{1}/\cdots/f_{k}=\max\{f\in P_{n}:f|S=(f_{1},f_{2},\dots,f_{k})\}$.

That $f_{1}/\cdots/f_{k}$ is order-preserving is a straightforward
induction on $k$.

To see ($\Delta4'$): if $f=f_{1}/\cdots/f_{k-1}/f_{k}=(f_{1}/\cdots/f_{i})/(f_{i+1}/\cdots/f_{k})$,
then clearly $\deg f_{1}+\cdots+\deg f_{i}\in\GD(f)$, so $S\subseteq\GD(f)$.
We show the converse by induction on $k$. Given $S\subseteq\GD(f)$,
let $j$ be the maximal element of $S$, so $f={}^{j}f/f_{k}$.
We claim that $S\backslash\{j\}\subseteq\GD({}^{j}f)$, so, by inductive
hypothesis, ${}^{j}f=f_{1}/\cdots/f_{k-1}$, thus $f=f_{1}/\cdots/f_{k-1}/f_{k}$. 

We are left to show that for $1\le i<j$
\begin{equation} \label{eq:GDrestriction}
  i\in \GD(f)\quad \iff \quad i\in \GD({}^{j}f)
\end{equation}
This is very similar to what we did in \eqref{eq:coassociativity1} and \eqref{eq:coassociativity2} above but with the basis $M$.
For this we have to establish that the comultiplication on $M$ is as in (\ref{eq:mcoproductformula}).
 For $f\in P_n$ and $i\in \allow(f)$, define
\begin{equation}
\Delta_{i}'(M_{f}):=\begin{cases}
M_{{}^{i}f}\otimes M_{f^{i}} & \text{if }i\in \GD(f);\\
0 & \text{otherwise}.
\end{cases}.\label{eq:delta'}
\end{equation}
We show now that $\Delta_+$ agrees with $\sum_{i\in \allow(f)}\Delta_{i}'$ on the $F$ basis:
\begin{align*}
\Delta_{i}'(F_{f})&=\Delta_{i}'\big(\sum_{g\ge f} M_g\big)=\sum_{g\ge f} \Delta_i'(M_g)\\ 
  &=\sum_{\substack{g\ge f \\ i\in \GD(g)}} M_{^ig}\otimes M_{g^i} \\
 &=\sum_{\substack{ ^ig\ge ^if \\ g^i\ge f^i}} M_{^ig}\otimes M_{g^i} = F_{{}^{i}f}\otimes F_{f^{i}}
\end{align*}
The last equality follows from the fact that  for $i\in\allow(f)$, we have a bijection 
 $$ \{g:g\geq f\text{ and }i\in\GD(g)\} \quad \leftrightarrow \quad \{(^ig,g^i):{}^ig\geq{}^{i}f,g^i\geq f^{i}\}.$$
 The forward map given by $g\mapsto({}^{i}g,g^{i})$ is well defined by axiom ($\Delta$2), since 
 $g\geq f$ implies ${}^ig\geq{}^{i}f$ and $g^i\geq f^{i}$. The reverse map given by $(^ig,g^i)\mapsto {}^ig/g^i$
 is well defined by axiom ($\Delta$3), since $^ig/g^i\ge {}^if/f^i \ge f$ and $i\in \GD(^ig/g^i)$. It is straightforward to check that those two maps are inverse to each other.
 This establishes (\ref{eq:mcoproductformula}) and a proof of \eqref{eq:GDrestriction} is similar to the first paragraph of the present proof using the basis $M$ instead of $F$
 in  \eqref{eq:coassociativity1} and \eqref{eq:coassociativity2}.
\end{proof}

\begin{proof}[Proof of Theorem \ref{thm:monomialcoproduct}]
We have established the formula for
$\Delta_{+}(M_{f})$ (\ref{eq:mcoproductformula}) in the final paragraph of the proof of Lemma~\ref{lem:kfactor-coproductaxiom}.
To see (\ref{eq:mcoproductformula-multi}), run the proof again with
$k$ factors instead of 2 factors, using ($\Delta1'$)-($\Delta4'$) of Lemma~\ref{lem:kfactor-coproductaxiom}.
\end{proof}

\subsection{Proof of the product formulas for monomial basis}

To prove the $k$-factor product formula, and the antipode formula,
we need the following $k$-factor generalizations of axioms ($m$1)-($m$3):
\begin{enumerate}
\item[($m1'$).]  For $f_{1}\in C_{j_{1}},f_{2}\in C_{j_{2}},\dots,f_{k}\in C_{j_{k}}$,
\begin{equation}
F_{f_{1}}\cdots F_{f_{k}}=\sum_{\zeta\in Sh(C_{j_{1}},\dots,C_{j_{k}})}F_{\zeta(f_{1},\dots,f_{k})},\label{eq:fproduct-multi}
\end{equation}
with $Sh(C_{j_{1}},\dots,C_{j_{k}})$ defined as in (\ref{eq:shuffle-multi}).
\item[($m2'$).]  If $f_{i}\leq f'_{i}\in C_{j_{i}}$ for all $i$, and $\zeta\in Sh(C_{j_{1}},\dots,C_{j_{k}})$,
then $\zeta(f_{1},\dots,f_{k})\leq\zeta(f'_{1},\dots,f'_{k})$.
\item[($m3'$).]  For $f_{i},f'_{i}\in C_{j_{i}}$ for all $i$, and $\zeta\in Sh(C_{j_{1}},\dots,C_{j_{k}})$,
we have 
\begin{equation}
\zeta(f_{1}\vee f'_{1},\dots,f_{k}\vee f_{k}')\leq\zeta(f_{1},\dots,f_{k})\vee\zeta(f'_{1},\dots,f'_{k}).
\end{equation}
\end{enumerate}

\

\begin{Lemma}\label{lem:kfactor-productaxiom} If $\mathcal{H}$
satisfies axioms ($m$0)-($m$3), then it satisfies ($m1'$)-($m3'$). \end{Lemma}
\begin{proof}
($m1'$) is clear from (\ref{eq:shuffle-multi}) and associativity.
The axiom ($m2'$) is proved by induction on $k$, the base case of $k=2$ being
($m$2). To prove ($m3'$) by induction on $k$: by (\ref{eq:shuffle-multi}),
\begin{align*}
(\zeta,\zeta')(f_{1}\vee f'_{1},\dots,f_{k}\vee f_{k}',) & =\zeta'(\zeta(f_{1}\vee f'_{1},f_{2}\vee f'_{2}),f_{3}\vee f'_{3},\dots,f_{k}\vee f_{k}')\\
 & \leq\zeta'(\zeta(f_{1},f_{2})\vee\zeta(f_{1},f_{2}),f_{3}\vee f'_{3},\dots,f_{k}\vee f_{k}'),
\end{align*}
using ($m$3) in the first factor, and that $\zeta'$ is order-preserving
in the first factor by ($m2'$). Now apply the inductive hypothesis.
\end{proof}

\begin{proof}[Proof of Theorem \ref{thm:monomialproduct}]
 We follow the argument of \cite[Th. 4.1]{AS05}. 
 Fix $f\in C_{i}\subseteq P_{n}$ and $g\in C_{j}\subseteq P_{m}$,
and write $Sh$ for the shuffle set $Sh(C_{i},C_{j})$. To multiply
$M_{f}$ with $M_{g}$, first express them in terms of fundamental
basis elements, then use condition ($m$2) for the product of fundamentals:
\begin{align*}
M_{f}\cdot M_{g} & =\sum_{f'\geq f,g'\geq g}\mu(f,f')\mu(g,g')F_{f'}\cdot F_{g'}\\
 & =\sum_{f'\geq f,g'\geq g}\:\sum_{\zeta\in Sh}\mu(f,f')\mu(g,g')F_{\zeta(f',g')}\\
 & =\sum_{h}\:\sum_{f'\geq f,g'\geq g}\:\sum_{\substack{\zeta\in Sh,\\ \zeta(f',g')\leq h}}\mu(f,f')\mu(g,g')M_{h}.
\end{align*}
Define
\begin{equation}
B_{f,g}^{h}:=\left\{ \zeta\in Sh :\zeta(f,g)\leq h\right\} ,\label{eq:b-def}
\end{equation}
so that
\[
M_{f}\cdot M_{g}=\sum_{h}\:\sum_{f'\geq f,g'\geq g}\mu(f,f')\mu(g,g')|B_{f',g'}^{h}|M_{h}.
\]
Thus we wish to show 
\begin{equation}
|A_{f,g}^{h}|=\sum_{f'\geq f,g'\geq g}\mu(f,f')\mu(g,g')|B_{f',g'}^{h}|.
\end{equation}

By
M\"obius inversion on $C_{i}\times C_{j}$, it suffices to show the disjoint union
\begin{equation}
B_{f,g}^{h}=\bigsqcup_{f'\geq f,g'\geq g}A_{f',g'}^{h}.\label{eq:ba-partition}
\end{equation}

To facilitate the checking of (\ref{eq:ba-partition}), we rewrite
conditions ($m$2) and ($m$3) in terms of the $B$ sets:
\begin{enumerate}
\item[($m$2{*}).]  if $f'\geq f$, $g'\geq g$, then $B_{f',g'}^{h}\subseteq B_{f,g}^{h}$; 
\item[($m$3{*}).]   $B_{f_{1},g_{1}}^{h}\cap B_{f_{2},g_{2}}^{h}\subseteq B_{f_{1}\vee f_{2},g_{1}\vee g_{2}}^{h}$.
\end{enumerate}
And the definition of $A_{f,g}^{h}$ may be written as: 
\begin{equation}
\begin{array}{c}
\zeta\in A_{f,g}^{h}\\ \Updownarrow\\ \zeta\in B_{f,g}^{h}\mbox{ and, for all }f'\geq f,g'\geq g,\mbox{ with }\zeta\in B_{f,'g'}^{h},\mbox{ we have }f'=f,g'=g.\label{eq:ba-max}
\end{array}
\end{equation}

The check of (\ref{eq:ba-partition}) is done as follows:
\begin{itemize}
\item $A_{f_{1},g_{1}}^{h}\cap A_{f_{2},g_{2}}^{h}=\emptyset$.

\vspace{.1in}

Indeed, take $\zeta\in A_{f_{1},g_{1}}^{h}\subseteq B_{f_{1},g_{1}}^{h}$
and also $\zeta\in A_{f_{2},g_{2}}^{h}\subseteq B_{f_{2},g_{2}}^{h}$.
Then, by ($m$3{*}), $\zeta\in B_{f_{1}\vee f_{2},g_{1}\vee g_{2}}^{h}$.
But, by (\ref{eq:ba-max}), $f_{1}\vee f_{2}=f_{1}$ and $g_{1}\vee g_{2}=g_{1}$,
and also $f_{1}\vee f_{2}=f_{2}$ and $g_{1}\vee g_{2}=g_{2}$. That is,
$f_{1}=f_{2}$ and $g_{1}=g_{2}$.

\

\item If $f'\geq f,g'\geq g$, then $A_{f',g'}^{h}\subseteq B_{f,g}^{h}$.

\vspace{.1in}

This is because $A_{f',g'}^{h}\subseteq B_{f',g'}^{h}\subseteq B_{f,g}^{h}$,
by ($m$2{*}).

\

\item If $\zeta\in B_{f,g}^{h}$, then $\zeta\in A_{f',g'}^{h}$ for some
$f'\geq f,g'\geq g$. 

\vspace{.1in}

From the maximality condition in the definition
of $A_{f',g'}^{h}$, it's clear that we need to choose $f',g'$ maximal
such that $\zeta\in B_{f',g'}^{h}$ - i.e. $(f',g')=\vee\{(f,g):\:\zeta\in B_{f,g}^{h}\}$.
This choice of $f',g'$ indeed ensures $\zeta\in B_{f',g'}^{h}$ because
of ($m$3{*}). And the maximality condition, in order to deduce $\zeta\in A_{f',g'}^{h}$,
is obvious from the definition of $f',g'$. 
\end{itemize}
This finishes the proof of (\ref{eq:a-def}). And the proof of (\ref{eq:a-def-multi})
is a $k$-factor version of the above argument, using axioms ($m1'$)-($m3'$)
of Lemma \ref{lem:kfactor-productaxiom}. 
\end{proof}

\subsection{Proof of the antipode formula for monomial basis}

In this section, fix $f\in P_{n}$ and $S\subseteq\GD(f)$. Let $f|S=(f_{1},\dots,f_{k})$
and let $f_{i}$ be in the connected component $C_{i}$.

First, we rewrite $A_{f_{1},\dots,f_{k}}^{h}$. 

\begin{Lemma}\label{lem:product-multi} Suppose conditions ($\Delta$1)-($\Delta$3)
and ($m$0)-($m$3) are satisfied. Then, for $S\subseteq\GD(f)$, the set~\eqref{eq:a-def-multi}
producing the coefficient $\alpha_{f_{1},\dots,f_{k}}^{h}$ in Theorem~\ref{thm:monomialproduct} may be written as

\begin{equation}
A_{f|S}^{h}:=\left\{ \zeta\in Sh(\cmpts(f|S))\middle|\:\begin{array}{c}
\zeta(f|S)\leq h,\\
\mbox{and if }f'\geq f\mbox{ satisfy }\zeta(f'|S)\leq h,\\
\mbox{then }f'=f.
\end{array}\right\} .\label{eq:a-def-multi-antipode}
\end{equation}
\end{Lemma}
\begin{proof}
By comparing with (\ref{eq:a-def-multi}), it suffices to show that, given
$\zeta\in Sh(\cmpts(f|S))$
and $\zeta(f|S)\leq h$, then the following are equivalent:
\begin{enumerate}
\item[(i)] If $f_{i}'\geq f_{i}$ for all $i$ and $\zeta(f_{1}',\dots,f_{k}')\leq h$,
then $f_{i}'=f_{i}$ for all $i$;
\item[(ii)] If $f'\geq f$ satisfy $\zeta(f'|S)\leq h$, then $f'=f$.
\end{enumerate}

To see (ii) $\implies$ (i): given $f_{i}'\geq f_{i}$, let $f'=f_{1}'/f_{2}'/\cdots/f_{k}'$.
By condition ($\Delta3'$), $/$ is order-preserving, so $f'\geq f$.
Hence, by (ii), $f'=f$; so by ($\Delta2'$), $f'|S=f|S$ partwise, i.e.
$f_{i}'=f_{i}$ for all $i$.

To see (i) $\implies$ (ii): suppose $f'\geq f$, and let $f'|S=(f_{1}',\dots,f_{k}')$.
 By ($\Delta2'$), $f_{i}'\geq f_{i}$. Hence, by (i), $f_{i}'=f_{i}$
for all $i$. By definition of $/$, we have $f'\leq f_{1}'/f_{2}'/\cdots/f_{k}'=f_{1}/\cdots/f_{k}=f$, where the last equality follows from $S\subseteq\GD(f)$. Hence $f'\leq f$,
but by assumption $f'\geq f$, so $f'=f$.
\end{proof}
Following \cite[Th. 5.5]{AS05}, the proof strategy of the antipode
formula is a M\"obius inversion, analogous to the proof of the product
formula. By Takeuchi's formula (details below), the coefficient of
$M_{h}$ in $\mathcal{S}(M_{f})$ will turn out to be $\mu(S,\GD(f))|A_{f|S}^{h}|$,
where $\mu$ is the M\"obius function on the Boolean poset of subsets
of $\GD(f)$. Notices that, in this case,  the $A$ sets are now playing the same role that the $B$ sets had in the product formula, and the subsets of $\GD(f)$
are in the role that the $f$'s had. This suggests doing M\"obius inversion
on this Boolean poset, i.e. defining a new set $C_{f|S}^{h}$ (analogous
to the $A$'s for the product formula) such that $A_{f|S}^{h}=\amalg_{R\subseteq S}C_{f|R}^{h}$.
However, there is a major obstacle in doing so: the $A$ sets are
not nested, meaning $A_{f|S}^{h}\subseteq Sh(\cmpts(f|S))$ and $A_{f|R}^{h}\subseteq Sh(\cmpts(f|R))\neq Sh(\cmpts(f|S))$,
so we cannot say $A_{f|R}^{h}\subseteq A_{f|S}^{h}$ when $R\subseteq S$.

This problem is not obvious from \cite[Th. 5.5]{AS05} because the set
 $Sh(\cmpts(f|S))$ is considered as a subset of the same symmetric group, regardless
of $S$. Translating this symmetric group view to the current situation,
we define a set of compatible inclusions 
\begin{equation}
\iota_{R,S}:Sh(\cmpts(f|R))\hookrightarrow Sh(\cmpts(f|S)),
\end{equation}
whenever $R\subseteq S\subseteq GD(f)$, in order to get \begin{equation} A_{f|S}^{h}=\bigsqcup_{R\subseteq S}\iota_{R,S}(C_{f|R}^{h}).
\end{equation}

\

\begin{Lemma}\label{lem:nesting-for-antipode} Suppose axioms ($\Delta$1)-($\Delta$3),
($m$0)-($m$3), are all satisfied, and there is a binary associative operation
$(f,g)\mapsto f\backslash g$ such that, for all connected components $C_{i}$
and $C_{j}$, there is some $\zeta_*\in Sh(C_{i},C_{j})$ satisfying
$f\backslash g=\zeta_*(f,g)$, for all $f\in C_{i}$ and $g\in C_{j}$.
Further, assume that $f\backslash g\leq f/g$. Then there are compatible
inclusions $\iota_{R,S}:Sh(\cmpts(f|R))\hookrightarrow Sh(\cmpts(f|S))$
whenever $R\subseteq S\subseteq\GD(f)$. Furthermore, for $R\subseteq S\subseteq\GD(f)$,
and $\zeta\in Sh(\cmpts(f|R))$, we have $(\iota_{R,S}\zeta)(f|S)\leq\zeta(f|R)$.
\end{Lemma}

\begin{proof}
To set up the inclusion $\iota_{R,S}$, focus on the case where $R$
and $S$ differ by a single element, since the associativity of the
$\backslash$ operation means that we can compose such injections to obtain
the general case. So, let $R$ consist of all but the $i$th element
of $S$.

Since $f|S=(f_{1},\dots,f_{k})$ and $S\subseteq\GD(f)$, then 
 $$f=f_{1}/\cdots/f_{k}=f_{1}/\cdots/f_{i-1}/(f_{i}/f_{i+1})/f_{i+2}/\cdots/f_{k},$$
by associativity of $/$ as in ($\Delta3'$). So $f|R=(f_{1},\dots,f_{i-1},f_{i}/f_{i+1},f_{i+2},\dots,f_{k})$.
Hence 
\begin{align*} 
 \cmpts(f|R)&=(C_{1},\dots,C_{i-1},\cmpt(f_{i}/f_{i+1}),C_{i+2},\dots,C_{k})\\
 &=(C_{1},\dots,C_{i-1},\cmpt(f_{i}\backslash f_{i+1}),C_{i+2},\dots,C_{k}),
 \end{align*}
because $f_{i}/f_{i+1}>f_{i}\backslash f_{i+1}$. Then, as noted at
the start of Section \ref{sec:axioms-antipode} (after Remark \eqref{rem:translationrequirements}), 
we have an injection 
$$Sh(C_{1},\dots,C_{i-1},\cmpt(\zeta_*(C_{i},C_{i+1})),C_{i+2},\dots,C_{k})\hookrightarrow Sh(C_{1},\dots,C_{k}),$$
from the associativity
in (\ref{eq:shuffle-multi}), where the left-hand side is $Sh(\cmpts(f|R))$ and the right-hand side is $Sh(\cmpts(f|S))$.

To verify that $(\iota_{R,S}\zeta)(f|S)\leq\zeta(f|R)$, note that $(\iota_{R,S}\zeta)(f|S)=\zeta(f_{1},\dots,f_{i-1},f_{i}\backslash f_{i+1},f_{i+2},\dots,f_{k})$,
and $\zeta(f|R)=\zeta(f_{1},\dots,f_{i-1},f_{i}/f_{i+1},f_{i+2},\dots,f_{k})$.
Now $f_{i}\backslash f_{i+1}<f_{i}/f_{i+1}$, and $\zeta$ is order-preserving
by axiom ($m2'$).
\end{proof}

One more lemma is necessary. We need to convert conditions ($\mathcal{S}$2) and ($\mathcal{S}$3) into
analogues of conditions ($m$2{*}) and ($m$3{*}) for the $A$ sets instead
of the $B$ sets.

\begin{Lemma}\label{lem:antipodecondition}Suppose conditions ($\Delta$1)-($\Delta$3),
($m$0)-($m$3), ($\mathcal{S}$0)-($\mathcal{S}$3) are all satisfied. Then,
\begin{enumerate}
\item[($\mathcal{S}$2{*}).]  If $R\subseteq S\subseteq\GD(f)$, then $\iota_{R,S}(A_{f|R}^{h})\subseteq A_{f|S}^{h}$; 
\item[($\mathcal{S}$3{*}).] For $R_{1},R_{2}\subseteq S\subseteq\GD(f)$, we have $\iota_{R_{1},S}(A_{f|R_{1}}^{h})\cap\iota_{R_{2},S}(A_{f|R_{2}}^{h})\subseteq\iota_{R_{1}\cap R_{2},S}(A_{f|R_{1}\cap R_{2}}^{h})$.
\end{enumerate}
\end{Lemma}
\begin{proof}
The argument below is loosely based on the proofs of Claims 1 and
2 within the proof of \cite[Th. 5.5]{AS05}. Following the idea in \eqref{eq:ba-max}, for each statement ($\mathcal{S}$2{*}) and ($\mathcal{S}$3{*}),
we separate the definition of $A$'s into two parts, namely a condition on $B$'s
and then the maximality of the $f$'s, so these are the four claims to check:
\begin{enumerate}
\item[($\mathcal{S}$2{*}a).]  If $R\subseteq S\subseteq\GD(f)$, then $\iota_{R,S}(B_{f|R}^{h})\subseteq B_{f|S}^{h}$. 

\vspace{.1in}
Equivalently,
 if $\zeta(f|R)\leq h$, then $(\iota_{R,S}\zeta)(f|S)\leq h$, which follows from Lemma \ref{lem:nesting-for-antipode}.
 
\
 
\item[($\mathcal{S}$2{*}b).]  Suppose $R\subseteq S\subseteq\GD(f)$. We have to show:
\begin{equation} \label{eq:statementS2b}
\begin{array}{c}
\text{if $\zeta(f|R)\leq h$ and $\zeta(f'|R)\leq h$ for some $f'\geq f$, then $f'=f$,}\\[3pt]
\Downarrow\\[3pt]
\text{if $\zeta(f|R)\leq h$ and $(\iota_{R,S}\zeta)(f'|S)\leq h$ for some $f'\geq f$, then $f'=f$.}
\end{array}
\end{equation}
Let $S=\{i_{1},\dots,i_{k-1}\}$ and $R=\{i_{j_{1}},\dots,i_{j_{l}}\}$.
If $f|S$ and $f'|S$ are such that $f|S=(f_{1},\dots,f_{k})$ and $f'|S=(f'_{1},\dots,f'_{k})$, we write $f'\geq f$ if $f'_{i}\geq f_{i}$ for every $i$. Since $S\subseteq\GD(f)\subseteq\GD(f')$,
this means $f=f_{1}/\cdots/f_{k}$, and $f|R=(f_{1}/\cdots/f_{j_{1}},f_{j_{1}+1}/\cdots/f_{j_{2}},\dots,f_{j_{l-1}+1}/\cdots/f_{j_{l}})$,
and similarly for $f'$. Then 
\begin{align*}
\zeta(f|R) & =\zeta(f_{1}/\cdots/f_{j_{1}},f_{j_{1}+1}/\cdots/f_{j_{2}},\dots,f_{j_{l-1}+1}/\cdots/f_{j_{l}});\\
\iota_{R,S}\zeta(f'|S) & =\zeta(f'_{1}\backslash\cdots\backslash f'_{j_{1}},f'_{j_{1}+1}\backslash\cdots\backslash f'_{j_{2}},\dots,f'_{j_{l-1}+1}\backslash\cdots\backslash f'_{j_{l}}).
\end{align*}
Using axiom ($m$3), and then axiom ($\mathcal{S}$2) to each input, together
with axiom ($m$2) (that $\zeta$ is order-preserving), we have:
\begin{align*}
\zeta(f|R)\vee\iota_{R,S}\zeta(f'|S) & \geq\zeta((f_{1}/\cdots/f_{j_{1}})\vee(f'_{1}\backslash\cdots\backslash f'_{j_{1}}),\dots,(f_{j_{l-1}+1}/\cdots/f_{j_{l}})\vee(f'_{j_{l-1}+1}\backslash\cdots\backslash f'_{j_{l}}))\\
 & \geq\zeta(f'_{1}/\cdots/f'_{j_{1}},\dots,f'_{j_{l-1}+1}/\cdots/f'_{j_{l}})\\
 & =\zeta(f'|R).
\end{align*}
Thus, if $\zeta(f|R)\leq h$ and $(\iota_{R,S}\zeta)(f'|S)\leq h$,
we must have $\zeta(f'|R)\leq h$. So, using the hypothesis of \eqref{eq:statementS2b}, we conclude that $f'=f$.

\

\item[($\mathcal{S}$3{*}a).] We will show that 
\[\iota_{R_{1},S}(B_{f|R_{1}}^{h})\cap\iota_{R_{2},S}(B_{f|R_{2}}^{h})\subseteq\iota_{R_{1}\cap R_{2},S}(B_{f|R_{1}\cap R_{2}}^{h}),\]
i.e. if $\zeta=\iota_{R_{1},S}(\zeta_{1})=\iota_{R_{2},S}(\zeta_{2})$
satisfies $\zeta_{1}(f|R_{1})\leq h$ and $\zeta_{2}(f|R_{2})\leq h$,
then $\zeta=\iota_{R_{1}\cap R_{2},S}(\zeta')$ and $\zeta'(f|R_{1}\cap R_{2})\leq h$. 

\vspace{.1in}

The claim $\zeta=\iota_{R_{1}\cap R_{2},S}(\zeta')$ is given by ($\mathcal{S}$0).
We note here that $\zeta_{1}=\iota_{R_{1}\cap R_{2},R_{1}}(\zeta')$
and $\zeta_{2}=\iota_{R_{1}\cap R_{2},R_{2}}(\zeta')$, since
\[\iota_{R_{1},S}(\zeta_{1})=\zeta=\iota_{R_{1}\cap R_{2},S}(\zeta')=\iota_{R_{1},S}\iota_{R_{1}\cap R_{2},R_{1}}(\zeta'),\]
and the fact that $\iota_{R_{1},S}$ is injective. 

We explain with an example how condition ($\mathcal{S}$3)
implies 
\[\zeta'(f|R_{1}\cap R_{2})\leq\zeta_{1}(f|R_{1})\vee\zeta_{2}(f|R_{2})\]
(and thus $\zeta'(f|R_{1}\cap R_{2})\leq h$ as required). Suppose
$S=\{i_{1},\dots,i_{6}\}$, $R_{1}=\{i_{1},i_{3,}i_{4}\}$, $R_{2}=\{i_{1},i_{2},i_{6}\}$,
and let $f|S=(f_{1},\dots,f_{7})$. Then $f=f_{1}/\cdots/f_{7}$ and
\begin{align*}
f|R_{1}\cap R_{2} & =(f_{1},f_{2}/f_{3}/f_{4}/f_{5}/f_{6}/f_{7});\\
f|R_{1} & =(f_{1},f_{2}/f_{3},f_{4},f_{5}/f_{6}/f_{7});\\
f|R_{2} & =(f_{1},f_{2},f_{3}/f_{4}/f_{5}/f_{6},f_{7}).
\end{align*}
This implies:
\begin{align}
\zeta'(f|R_{1}\cap R_{2}) & =\zeta'(f_{1},f_{2}/f_{3}/f_{4}/f_{5}/f_{6}/f_{7});\nonumber \\
\zeta_{1}(f|R_{1})=\iota_{R_{1}\cap R_{2},R_{1}}(f|R_{1}) & =\zeta'(f_{1},(f_{2}/f_{3})\backslash f_{4}\backslash(f_{5}/f_{6}/f_{7}));\label{eq:axiomS3R1}\\
\zeta_{2}(f|R_{2})=\iota_{R_{1}\cap R_{2},R_{2}}(f|R_{2}) & =\zeta'(f_{1},f_{2}\backslash(f_{3}/f_{4}/f_{5}/f_{6})\backslash f_{7}).\label{eq:axiomS3R2}
\end{align}
Examine the second input of $\zeta'$ in (\ref{eq:axiomS3R1}) and
(\ref{eq:axiomS3R2}). First note that $f_{5}$ and $f_{6}$ appear
only as $f_{5}/f_{6}$ (because $i_{5}\notin R_{1}\cup R_{2}$), so
regard $f_{5}/f_{6}$ as one object, instead of $f_{5}$ and $f_{6}$
independently. Next, because $i_{2},i_{3},i_{4,}i_{6}\in(R_{1}\cup R_{2})\backslash(R_{1}\cap R_{2})$,
there is a $/$ sign in one of the identities (\ref{eq:axiomS3R1}) and (\ref{eq:axiomS3R2}), and a $\backslash$ sign in the same position of remaining identity (between $f_{2}$ and $f_{3}$; between $f_3$ and $f_4$; between $f_4$ and $f_5/f_6$, and
between $f_{5}/f_{6}$ and $f_{7}$). Thus condition
$S$3 applies to give 
\[
f_{2}/f_{3}/f_{4}/f_{5}/f_{6}/f_{7}\leq(f_{2}/f_{3})\backslash f_{4}\backslash(f_{5}/f_{6}/f_{7})\vee f_{2}\backslash(f_{3}/f_{4}/f_{5}/f_{6})\backslash f_{7}.
\]
(In the general case, we need one such equation for each input position
where the inputs of $\zeta'$ in (\ref{eq:axiomS3R1}) and (\ref{eq:axiomS3R2})
disagree.) So, by conditions ($m$2) and ($m$3):
\[
\zeta'(f_{1}/f_{2}/f_{3}/f_{4}/f_{5}/f_{6}/f_{7})\leq\zeta'(f_{1},(f_{2}/f_{3})\backslash f_{4}\backslash(f_{5}/f_{6}/f_{7}))\vee\zeta'(f_{1},f_{2}\backslash(f_{3}/f_{4}/f_{5}/f_{6})\backslash f_{7}).
\]

\

\item[($\mathcal{S}$3{*}b).]  Suppose $\zeta,\zeta_{1},\zeta_{2},\zeta'$ are as above. We will show that
\begin{equation} \label{eq:statementS3b}
\begin{array}{c}
\text{if $\zeta_{1}(f'|R_{1})\leq h$ and $\zeta_{2}(f'|R_{2})\leq h$ for some $f'\geq f$, then $f'=f$,}\\[3pt]
 \Downarrow\\[3pt]
\text{if $\zeta'(f'|R_{1}\cap R_{2})\leq h$ for some $f'\geq f$, then $f'=f$.}
\end{array}
\end{equation}
Assume $\zeta'(f'|R_{1}\cap R_{2})\leq h$. By Lemma \ref{lem:nesting-for-antipode},
$R_{1}\cap R_{2}\subseteq R_{1}$ means $\zeta_{1}(f'|R_{1})\leq\zeta'(f'|R_{1}\cap R_{2})\leq h$,
and similarly $\zeta_{2}(f'|R_{2})\leq h$. So, by the hypothesis of \eqref{eq:statementS3b},
$f'=f$. 
\end{enumerate}
\end{proof}

\

Now we are finally ready to prove the antipode formula.
\begin{proof}[Proof of Theorem \ref{thm:antipode}]
Takeuchi's formula says 
\[
\mathcal{S}(x)=\sum_{k=1}^{\deg x}(-1)^{k+1}(m\otimes\cdots\otimes m)\circ\Delta_{+}^{[k]}(x).
\]
Using (\ref{eq:mcoproductformula-multi}) and Lemma \ref{lem:product-multi},
\begin{align*}
\mathcal{S}(M_{f}) & =\sum_{S\subseteq\GD(f)}(-1)^{|S|+1}\sum_{h}|A_{f|S}^{h}|M_{h}\\
 & =(-1)^{|\GD(f)|+1}\sum_{h}\sum_{S\subseteq\GD(f)}\mu(S,\GD(f))|A_{f|S}^{h}|M_{h},
\end{align*}

where $\mu$ is relative to the Boolean poset of subsets of $\GD(f)$.
Now interpret ${\displaystyle \sum_{S\subseteq T}\mu(S,T)|A_{f|S}^{h}|}$
by M\"obius inversion on the Boolean poset with $A_{f|S}^{h}$. That is,
we partition $A_{f|S}^{h}$ according to where $S$ is minimal - hence
the third condition in (\ref{eq:c-def}) compared to (\ref{eq:a-def-multi-antipode}).
More specifically, we show that 
\begin{equation}
A_{f|S}^{h}=\bigsqcup_{R\subseteq S}\iota_{R,S}(C_{f|R}^{h}),\label{eq:ac-partition}
\end{equation}
where $C_{f|S}^{h}$ are the shuffles in $Sh(\cmpts(f|S))$ satisfying
\begin{equation}
\begin{array}{ll}
\text{(i)} & \zeta(f|S)\leq h;\\
\text{(ii)} & \mbox{if }f'\geq f\mbox{ satisfies }\zeta(f'|S)\leq h,\mbox{ then }f=f';\\
\text{(iii)} & \mbox{if }\zeta=\iota_{R,S}\zeta'\text{ for some }R\subseteq S\text{ and some }\zeta'\in Sh(\cmpts f|R),\\
 & \text{ and }\zeta'(f|R)\leq h\mbox{, then }R=S.
\end{array}\label{eq:c-def}
\end{equation}

Observe that the definition of $C_{f|S}^{h}$ relative to $A_{f|S}^{h}$
is not exactly like the definition of $B_{f,g}^{h}$ relative to $A_{f,g}^{h}$:
the third condition in the definition of $C_{f|S}^{h}$ only specifies
the minimality of $S$ for $\zeta(f|S)\leq h$, not also for the second
line regarding the maximality of $f$. In other words, 
\begin{equation}\label{eq:ac-min}
\begin{array}{c}
\zeta\in C_{f|S}^{h}\\
\Updownarrow\\
\zeta\in A_{f|S}^{h}\mbox{ and, if }\zeta=\iota_{R,S}(\zeta')\text{ for some }R\subseteq S\text{ with }\zeta\in B_{f|R}^{h},\mbox{ we have }R=S.
\end{array}
\end{equation}
To prove the partitioning analogous to the proof of the product formula
(Theorem \ref{thm:monomialproduct}), we need to change the $B_{f|R}^{h}$
to $A_{f|R}^{h}$ in this definition of $C_{f|S}^{h}$. So we need
to prove that the following are equivalent:
\begin{enumerate}
\item[(i)] $\zeta\in A_{f|S}^{h}$ and, if $\zeta=\iota_{R,S}(\zeta')$ for
some $R\subseteq S$ with $\zeta'\in B_{f|R}^{h}$, we have $R=S$.
\item[(ii)] $\zeta\in A_{f|S}^{h}$ and, if $\zeta=\iota_{R,S}(\zeta')$ for some
$R\subseteq S$ with $\zeta'\in A_{f|R}^{h}$, we have $R=S$.
\end{enumerate}

Since $B_{f|R}^{h}\supseteq A_{f|R}^{h}$, (i) $\implies$ (ii) is immediate.
To show (ii) $\implies$ (i): since the ``conclusion'' for (i) and (ii) are
equal (``we have $R=S$''), it suffices to show that the hypotheses
for (i) imply the hypotheses for (ii). \par
Recall $B_{f|R}^{h}=\amalg_{f'\geq f}A_{f'|R}^{h}$.
So we show
\[
\mbox{if }\zeta\in A_{f|S}^{h}\mbox{, and }\zeta=\iota_{R,S}(\zeta')\text{ for some }R\subseteq S\mbox{ and }\zeta'\in A_{f'|R}^{h}\mbox{ for some }f'\geq f,\mbox{ then }\zeta\in A_{f|R}^{h}.
\]
Here's the proof: $\zeta'\in A_{f'|R}^{h}$ and $\iota_{R,S}(A_{f'|R}^{h})\subseteq A_{f'|S}^{h}$
by condition ($\mathcal{S}$2{*}). So $\zeta\in A_{f|S}^{h}$ and $\zeta\in A_{f'|S}^{h}$.
By the maximality of $f$ in the definition of $A$ (\ref{eq:a-def-multi-antipode}),
this implies $f=f'$. So $\zeta'\in A_{f'|R}^{h}$ actually means
$\zeta'\in A_{f|R}^{h}$.

Having established this equivalent and more useful definition of $C_{f|S}^{h}$,
we can continue the proof of the antipode formula in exactly the same
manner as for the product formula. 
\begin{itemize}
\item For any $R_{1},R_{2}\subseteq S\subseteq\GD(f)$, we have the empty intersection
$\iota_{R_{1},S}(C_{f|R_{1}}^{h})\cap\iota_{R_{2},S}(C_{f|R_{2}}^{h})=\emptyset$.

\vspace{.1in}

Indeed, take $\zeta \in \iota_{R_{1},S}(C_{f|R_{1}}^{h})\cap\iota_{R_{2},S}(C_{f|R_{2}}^{h})$,
so $\zeta=\iota_{R_{1},S}(\zeta_{1})=\iota_{R_{2},S}(\zeta_{2})$
with $\zeta\in C_{f|S}^{h}\subseteq A_{f|S}^{h}$. By condition
($\mathcal{S}$3{*}), there is some $\zeta'\in A_{f|R_{1}\cap R_{2}}^{h}$ with
$\zeta=\iota_{R_{1}\cap R_{2},S}(\zeta')$, and recall from the proof
of ($\mathcal{S}$3{*}a) within Lemma \ref{lem:antipodecondition} that $\zeta_{1}=\iota_{R_{1}\cap R_{2},R_{1}}(\zeta')$
and $\zeta_{2}=\iota_{R_{1}\cap R_{2},R_{2}}(\zeta')$. As $\zeta_{1}\in C_{f|R_{1}}^{h}$,
the minimality condition means $R_{1}\cap R_{2}=R_{1}$, and similarly
$R_{1}\cap R_{2}=R_{2}$. So $R_{1}=R_{2}$.

\

\item if $R\subseteq S$, then $\iota_{R,S}(C_{f|R}^{h})\subseteq A_{f|S}^{h}$.

\vspace{.1in}

This is because $C_{f|R}^{h}\subseteq A_{f|R}^{h}$ and $\iota_{R,S}(A_{f|S}^{h})\subseteq A_{f|S}^{h}$,
by condition ($\mathcal{S}$2{*}).

\

\item if $\zeta\in A_{f|S}^{h}$, then $\zeta\in\iota_{R,S}(C_{f|R}^{h})$
for some $R\subseteq S$.

\vspace{.1in}

From the minimality condition in the definition of $C_{f|R}^{h}$,
it's clear that we need to choose $R$ minimal such that $\zeta\in\iota_{R,S}(C_{f|R}^{h})$
- i.e. $R=\cap\{R':\:\zeta\in\iota_{R',S}(C_{f|R'}^{h})\}$. This
choice of $R$ indeed ensures $\zeta\in\iota_{R,S}(A_{f|R}^{h})$
because of condition ($\mathcal{S}$3{*}). So $\zeta=\iota_{R,S}(\zeta')$ for
some $\zeta'\in A_{f|R}^{h}$. To deduce that $\zeta'\in C_{f|R}^{h}$
(i.e. to see the minimality of $R$), suppose $\zeta'=\iota_{R',R}(\zeta'')$
for some $R'\subseteq R$ and some $\zeta''\in A_{f|R'}^{h}$. Then
\[\zeta=\iota_{R,S}(\zeta')=\iota_{R,S}\iota_{R',R}(\zeta'')=\iota_{R',S}(\zeta''),\]
so $R\subseteq R'$ by definition of $R$. Hence $R'=R$ as required. 
\end{itemize}
Finally, to establish that $\beta_f^h=|C_{f|\GD(f)}^{h}|$, we show the equivalence of the last condition in the definition
of $C_{f|\GD(f)}^{h}$ in (\ref{eq:c-def}) with the last condition in (\ref{eq:c-def-GD-2}).
Clearly, the last condition in (\ref{eq:c-def}) is equivalent to the following:
if $\zeta=\iota_{R,\GD(f)}\zeta'$ for some $R\subsetneq\GD(f)$ and
some $\zeta'\in Sh(\cmpts(f|R))$, then $\zeta'(f|R)\not\leq h$. The statement in (\ref{eq:c-def-GD-2})
says that it is sufficient to check this condition when $R=\GD(f)\backslash\{i\}$,
for some $i$. The reason is that, by Lemma \ref{lem:nesting-for-antipode},
$(\iota_{R,\GD(f)\backslash\{i\}}\zeta')(f|\GD(f)\backslash\{i\})\leq\zeta'(f|R)$,
for any $R\subsetneq\GD(f)$ and $i\in \GD(f)\backslash R$. So, if $\zeta'(f|R)\leq h$,
then necessarily $(\iota_{R,\GD(f)\backslash\{i\}}\zeta')(f|\GD(f)\backslash\{i\})\leq h$.
\end{proof}

\subsection{Proof of the inheritance of axioms through quotienting}

\begin{proof}[Proof of Theorem \ref{thm:monomialquotient}]
Let $\Pi':\mathcal{H}\rightarrow\bar{\mathcal{H}}$ be defined by
(\ref{eq:quotientmonomial}). Then, for $f\in P_{n}$:
\begin{align*}
\Pi'(F_{f}) & =\sum_{g\geq f}\Pi'(M_{g})\\
 & =\sum_{\iota_{n}(t)\geq f}M_{t}.
\end{align*}
We claim this equals $\Pi(F_{f})=F_{\pi_{n}(f)}=\sum_{t\geq\pi_{n}(f)}M_{t}$.
It suffices to show that $\iota_{n}(t)\geq f$ if and only if $t\geq\pi_{n}(f)$, i.e. $\pi_n$ and $\iota_n$ form a Galois connection, which can be proved similar to Corollary \ref{cor:Galois_connection}.

In the first scenario under ``Furthermore'': it follows from ($\pi\Delta$1),
($\pi\Delta$2) and ($\Delta$1) on $\mathcal{H}$ that $\Pi$ is a coalgebra
morphism. To check $\bar{\mathcal{H}}$ satisfies ($\Delta$1): if $s\leq t$,
then $\allow(s)=\allow(\iota_{n}(s))=\allow(\iota_{n}(t))=\allow(t)$,
where the first equality is by ($\pi\Delta$1), because $s=\pi_{n}(\iota_{n}s)$,
and the middle equality is because $\iota_{n}(s)\leq\iota_{n}(t)$.

To check $\bar{\mathcal{H}}$ satisfies ($\Delta$3):
given $s\in\bar{P}_{i},t\in\bar{P}_{n-i}$, first note that 
$${}^{i}(s/t)={}^{i}\pi_n(\iota_{i}(s)/\iota_{n-i}(t))=\pi_n[{}^{i}(\iota_{i}(s)/\iota_{n-i}(t))],$$
by definition of $/$ on $\bar{P}$ and by ($\pi\Delta$2). By definition
of $/$ on $P$, $\pi_n[{}^{i}(\iota_{i}(s)/\iota_{n-i}(t))]=\pi_i\iota_{i}(s)=s$,
and similarly $(s/t)^{i}=t$. To check the maximality, suppose $r\in\bar{P}_{n}$
satisfies ${}^{i}r=s$ and $r^{i}=t$, and let $f=\iota_{n}(r)$.
Then $\pi_{i}({}^{i}f)={}^{i}\pi_{n}(f)={}^{i}r=s$, so the definition
of $\iota_{i}$ implies ${}^{i}f\leq\iota_{i}(s)$. Similarly, $f^{i}\leq\iota_{n-i}(t)$.
So $f\leq{}^{i}f/f^{i}\leq\iota_{i}(s)/\iota_{n-i}(t)$. And $\pi_{n}$
is order-preserving, so $r=\pi_{n}(f)\leq\pi_{n}(\iota_{i}(s)/\iota_{n-i}(t))=s/t$
by definition. That $/$ is order-preserving on $\bar{P}$, and axiom
($\Delta$2), can be proved similarly.

Similarly, in the second scenario, it follows from ($\pi m$1), ($\pi m$2)
and ($m$1) on $\mathcal{H}$ that $\Pi$ is an algebra morphism. Axiom
($m$0) holds for $\bar{\mathcal{H}}$ by Lemma \ref{lem:pipreserve-leastupperbound}.
($m$1) holds for $\bar{\mathcal{H}}$ because of ($\pi m$1). Axiom
($m$2) can be checked in a similar way as ($\Delta$3) above. To deduce
($m$3): given $s_{1},s_{2}\in\bar{C}_{i}\subseteq\bar{P}_{n}$, $t_{1},t_{2}\in\bar{C}_{j}\subseteq\bar{P}_{m}$,
consider $f_{l}=\iota_{n}(s_{l})$, $g_{l}=\iota_{m}(t_{l})$ for
$l=1,2$. Take $\bar{\zeta}\in Sh(\bar{C}_{i},\bar{C}_{j})$ corresponding
to $\zeta\in Sh(C_{i},C_{j})$. Then 
\[ \bar{\zeta}(s_{1},t_{1})=\bar{\zeta}(\pi_{n}(f_{1}),\pi_{m}(g_{1}))=\pi_{n+m}\zeta(f_{1},g_{1})\]
and similarly for $s_{2},t_{2}$,
so 
\[\bar{\zeta}(s_{1},t_{1})\vee\bar{\zeta}(s_{2},t_{2})=\pi_{n+m}\zeta(f_{1},g_{1})\vee\pi_{n+m}\zeta(f_{2},g_{2})=\pi_{n+m}(\zeta(f_{1},g_{1})\vee\zeta(f_{2},g_{2}))\]
by Lemma \ref{lem:pipreserve-leastupperbound}. Because ($m$3) holds
on $\mathcal{H}$ and $\pi_{n+m}$ is order-preserving, this is $\geq\pi_{n+m}(\zeta(f_{1}\vee f_{2},g_{1}\vee g_{2}))=\bar{\zeta}(\pi_{n+m}(f_{1}\vee f_{2}),\pi_{n+m}(g_{1}\vee g_{2}))=\bar{\zeta}(s_{1}\vee t_{1},s_{2}\vee t_{2})$,
using Lemma \ref{lem:pipreserve-leastupperbound} again in the last
equality.

In the third scenario: to show ($\mathcal{S}$0) on $\bar{\mathcal{H}}$,
let $\backslash\in Sh(\bar{C}_{i},\bar{C}_{j})$ be the shuffle corresponding
to $\backslash\in Sh(C_{i},C_{j})$. By ($\pi m$2), we have $\pi_{n}(g\backslash h)=\pi_{i}(g)\backslash\pi_{n-i}(h)$,
so in particular $s\backslash t=\pi_{n}(\iota_{i}(s)\backslash\iota_{n-i}(t))$.
The requirement $s\backslash t\leq s/t$, and ($\mathcal{S}$2) and
($\mathcal{S}$3), can be checked for $\bar{\mathcal{H}}$ by lifting
to $\mathcal{H}$, as for axioms ($\Delta$3) and ($m$3) above. And ($\mathcal{S}$1)
on $\bar{\mathcal{H}}$ follows from ($\pi\mathcal{S}$1) and ($\mathcal{S}$1)
on $\mathcal{H}$.
\end{proof}

\section{Acknowledgements}
We are grateful to the Fields Institute since this project started in the Algebraic Combinatorics Workshop at the Fields institute during the Fall of 2017. We thank Marcelo Aguiar and Aaron Lauve for helpful conversations.


\begin{thebibliography}{99}

	\bibitem[AS05]{AS05} M. Aguiar, F. Sottile,
	\newblock \textit{Structure of the Malvenuto-Reutenauer Hopf algebra of permutations},
	\newblock Advances in Mathematics, 191(2):225--275, 2005.
	
	\bibitem[AS06]{AS06} M. Aguiar, F. Sottile,
	\newblock \textit{Structure of the Loday-Ronco Hopf algebra of trees},
	\newblock Journal of Algebra, 295(2):473--511, 2006.
	
	\bibitem[BP12]{BP12} F. Bergeron, L-F. Pr\'eville-Ratelle, 
	\newblock \textit{Higher trivariate diagonal harmonics via generalized Tamari posets},
	\newblock Journal of Combinatorics, 3(3):317--341, 2012.
		
	\bibitem[BZ05]{BZ09} N. Bergeron, M. Zabrocki,
	\newblock \textit{The Hopf algebras of symmetric functions and quasi-symmetric functions in non-commutative variables are free and co-free},
	\newblock Journal of Algebra and its Applications, 8(4):581--600, 2009. [see also https://arxiv.org/pdf/math/0509265 (2005)]
	
	\bibitem[CG19]{CG19} C. Ceballos, Cesar, R. S. Gonz\'{a}lez D'Le\'{o}n, 
	\newblock \textit{Signature {C}atalan combinatorics},
	\newblock Journal of Combinatorics, 10(4):725--773, 2019.	

	\bibitem[CS08]{CS08} H. Crapo, W Schmitt,
	\newblock \textit{Primitive elements in the matroid-minor {H}opf algebra},
	\newblock Journal of Algebraic Combinatorics, 28(1):43--64, 2008.
	
	\bibitem[F02]{F02} L. Foissy,
	\newblock \textit{Les alg\`ebres de Hopf des arbres enracin\'es, I},
	\newblock Bulletin of Mathematical Sciences, 126:193--239, 2002.
	
	\bibitem[F07]{F07} L. Foissy,
	\newblock \textit{Bidendriform bialgebras, trees, and free quasi-symmetric functions},
	\newblock Journal of Pure and Applied Algebra, 209:439--459, 2007.
	
	\bibitem[F12]{F12} L. Foissy,
	\newblock \textit{Free and cofree Hopf algebras},
	\newblock Journal of Pure and Applied Algebra, 216:480--494, 2012.
	
	\bibitem[FLS13]{FLS13} S. Forcey, A. Lauve, F. Sottile,
	\newblock \textit{Cofree compositions of coalgebras},
	\newblock Annals of Combinatorics, 17:105--130, 2013.	
	
	\bibitem[GS78]{GS78} I. Gessel, R. Stanley,
	\newblock \textit{Stirling polynomials},
	\newblock Journal of Combinatorial Theory, Series A, 24:24--33, 1978.
	
	\bibitem[G16]{G16} R. S. Gonz\'{a}lez D'Le\'{o}n,
	\newblock \textit{On the free {L}ie algebra with multiple brackets},
	\newblock Advances in Applied Mathematics, 79:37--97, 2016.	

	\bibitem[J99]{J99} A. J\"ollenbeck
	\newblock \textit{Nichtkommutative Charaktertheorie der symmetrischen Gruppen},
	\newblock 
 Bayreuther Mathematischen Schriften 56:1-41, 1999.
	
	\bibitem[LR98]{LR98} J-L. Loday, M. Ronco,
	\newblock \textit{Hopf algebra of the planar binary trees},
	\newblock Advances in Mathematics, 139:293--309, 1998.
	
	\bibitem[LR02]{LR02} J-L. Loday, M. Ronco,
	\newblock \textit{Order structure on the algebra of permutations and of planar binary trees},
	\newblock Journal of Algebraic Combinatorics, 15(3):253-270, 2002.
	
	\bibitem[M94]{M94} G. Markowsky,
	\newblock \textit{Permutation lattices revisited},
	\newblock Mathematical Social Sciences, 27:59--72, 1994.
	
	\bibitem[MR95]{MR95} C. Malvenuto, C. Reutenauer,
	\newblock \textit{Duality between quasi-symmetric functions and the Solomon descent algebra},
	\newblock Journal of Algebra, 177:967--982, 1995.	
	
	\bibitem[MR21]{MR21} C. Malvenuto, C. Reutenauer,
	\newblock \textit{Primitive Elements of the Hopf Algebras of Tableaux},
	\newblock https://arxiv.org/abs/2010.06731.	

	\bibitem[NT06]{NT06} J-C. Novelli, J-Y. Thibon,
	\newblock \textit{Polynomial realizations of some trialgebras},
	\newblock Proc. FPSAC'06, San Diego, http://arxiv.org/abs/math.CO/0605061
	
	\bibitem[NT07]{NT07} J-C. Novelli, J-Y. Thibon,
	\newblock \textit{Hopf algebras and dendriform structures arising from parking functions},
	\newblock Fundamenta Mathematicae, 193:189--241, 2007.
	
	\bibitem[NT20]{NT20} J-C. Novelli, J-Y. Thibon,
	\newblock \textit{Hopf algebras of $m$-permutations, $(m+1)$-ary trees, and $m$-parking functions},
	\newblock Advances in Applied Mathematics, 117:A. 102019, 2020.
	
	\bibitem[PP18]{PP18} V. Pilaud, V. Pons,
	\newblock \textit{Permutrees},
	\newblock Algebriac Combinatorics, 1(2):173-224, 2008.
	
	\bibitem[S20]{S20} M. Sanchez,
	\newblock \textit{Poset Hopf Monoids},
	\newblock http://arxiv.org/pdf/2005.13707.
	
	\bibitem[T71]{T71} M. Takeuchi,
	\newblock \textit{Free Hopf algebras generated by coalgebras},
	\newblock Journal of the Mathematical Society of Japan, 23:561--582, 1971.
	
	\bibitem[V15]{V15} Y. Vargas,
	\newblock \textit{Alg\`ebres d\'e Hopf de Mots Tass\'es et de Fonctions Motifs},
	\newblock Thesis.
	
	\bibitem[V20]{V20} Y. Vargas,
	\newblock \textit{Structure of the J\"ollenbeck Hopf algebras of packed words},
	\newblock In preparation.
	



\end{thebibliography}
\end{document}